\documentclass[11pt,b5paper,openany]{book}
\usepackage{palatino}
\usepackage{combelow}
\usepackage[english]{babel}
\usepackage{color}
\usepackage{amsmath,amsthm}
\usepackage{amsfonts}
\usepackage{amssymb}
\usepackage{longtable}
\usepackage{hyperref}
\usepackage{natbib}
\usepackage{indentfirst}
\usepackage{graphicx}
\usepackage{varioref}
\usepackage{makeidx}
\usepackage{fancyhdr}
\usepackage{titletoc}
\usepackage{titlesec}
\newtheorem{thm}{Theorem}[chapter]

\newtheorem{lem}[thm]{Lemma}

\theoremstyle{definition}
\newtheorem{defn}[thm]{Definition}
\newtheorem{cons}[thm]{Consequence}
\newtheorem{exem}[thm]{Example}
\theoremstyle{remark}
\newtheorem{rem}[thm]{Remark}
\newtheorem{obs}[thm]{Observation}
\newtheorem{prog}[thm]{Program}

\newtheorem{alg}[thm]{Algorthm}

\newcommand{\abs}[1]{\left\vert#1\right\vert}
\newcommand{\set}[1]{\left\{#1\right\}}
\newcommand{\Real}{\mathbb{R}}
\newcommand{\Int}{\mathbb{Z}}

\newcommand{\Na}{\mathbb{N}}
\newcommand{\Ns}{\mathbb{N}^*}

\newcommand{\I}[1]{I_{#1}}
\newcommand{\md}[1]{\left(mod\ #1\right)}

\newcommand{\desp}[2][\alpha]{\ensuremath{p_1^{#1_1}\cdot p_2^{#1_2}\cdots p_#2^{#1_{#2}}}}

\makeindex
\begin{document}
\sloppy
\title{Solving Diophantine Equations}
\author{Octavian Cira and Florentin Smarandache}
\date{2014}
\maketitle
\frontmatter
\maketitle
\chapter*{Preface}

In recent times the we witnessed an explosion of Number Theory problems that are solved using mathematical software and powerful computers. The observation that the number of transistors packed on integrated circuits doubles every two years made by Gordon E. Moore in 1965 is still accurate to this day. With ever increasing computing power more and more mathematical problems can be tacked using brute force. At the same time the advances in mathematical software made tools like Maple, Mathematica, Matlab or Mathcad widely available and easy to use for the vast majority of the mathematical research community. This tools don't only perform complex computations at incredible speeds but also serve as a great tools for symbolic computation, as proving tools or algorithm design.

The online meeting of the two authors lead to lively exchange of ideas, solutions and observation on various Number Theory problems. The ever increasing number of results, solving techniques, approaches, and algorithms led to the the idea presenting the most important of them in in this volume. The book offers solutions to a multitude of $\eta$--Diophantine equation proposed by Florentin Smarandache\index{Smarandache F.} in previous works \citep{Smarandache1993,Smarandache1999a,Smarandache2006} over the past two decades. The expertise in tackling Number Theory problems with the aid of mathematical software such as \citep{Cira+Cira2010}, \citep{Cira2013,Cira2014a,Cira+Smarandache2014,Cira2014c,Cira2014d,Cira2014e,Cira2014f} played an important role in producing the algorithms and programs used to solve over $62$ $\eta$--Diophantine equation. There are numerous other important publications related to Diophantine Equations that offer various approaches and solutions. However, this book is different from other books of number theory since it dedicates most of its space to solving Diophantine Equations involving the Smarandache function. A search for similar results in online resources like \emph{The On-Line Encyclopedia of Integer Sequences} reveals the lack of a concentrated effort in this direction.

The brute force approach for solving $\eta$--Diophantine equation is a well known technique that checks all the possible solutions against the problem constrains to select the correct results. Historically, the proof of concept was done by \cite{Appel+Haken1977}\index{Appel K.}\index{Hanken W.} when they published the proof for the \emph{four color map} theorem. This is considered to be the the first theorem that was proven using a computer. The approach used both the computing power of machines as well as theoretical results that narrowed down infinite search space to $1936$ map configurations that had to be check. Despite some controversy in the '80 when a masters student discovered a series of errors in the discharging procedure, the initial results was correct. Appel and Haken went on to publish a book \citep{Appel+Haken1989} that contained the entire and correct prof that \emph{every planar map is four-colorable}.

Recently, in 2014 an empirical results of Goldbach conjecture was published in Mathematics of Computation where \cite{Oliveira-e-Silva2013}, \citep{Oliveira-e-Silva2014}\index{Oliveira e Silva T.}, confirm the theorem to be true for all even numbers not larger than $4\times10^{18}$.

The use of Smarandache function $\eta$ that involves the set of all prime numbers constitutes one of the main reasons why, most of the problems proposed in this book do not have a finite number of cases. It could be possible that the unsolved problems from this book could be classified in classes of unsolved problems, and thus solving a single problem will help in solving all the unsolved problems in its class. But the authors could not classify them in such classes. The interested readers might be able to do that. In the given circumstances the authors focused on providing the most comprehensive partial solution possible, similar to other such solutions in the literature like:
\begin{itemize}
  \item Goldbach's conjecture. In 2003 Oliveira e Silva\index{Oliveira e Silva T.} announced that all even numbers $\le2\times10^{16}$ can be expressed as a sum of two primes. In 2014 the partial result was extended to all even numbers smaller then $4\times10^{18}$, \citep{Oliveira-e-Silva2014}.
  \item For any positive integer $n$, let $f(n)$ denote the number of solutions to the Diophantine equation $4/n=1/x+1/y+1/z$ with $x$, $y$, $z$ positive integers. The \emph{Erd\H{o}s-Straus conjecture}\index{Erd\"{o}s P.}\index{Straus E. G.}, \citep{Oblath1950,Rosati1954,Bernstein1962,Tao2011}, asserts that $f(n)\ge1$ for every $n\ge2$. \cite{Swett2006}\index{Swett A.} established that the conjecture is true for all integers for any $n\le10^{14}$. \cite{Elsholtz+Tao2012} established some related results on $f$ and related quantities, for instance established the bound $f(p)\ll p^{3/5}+O\big(1/\log(\log(p))\big)$ for all primes $p$.
  \item \cite{Tutescu1996}\index{Tutescu L.} stated that $\eta(n)\neq\eta(n+1)$ for any $n\in\Ns$. On March 3rd, 2003 Weisstein\index{Weisstein E. W.} published a paper stating that all the relation is valid for all numbers up to $10 ^ 9$, \citep{Sondow+Weisstein}.
  \item A number $n$ is $k$--hyperperfect for some integers $k$ if $n=1+k\cdot s(n)$, where $s(n)$ is the sum of the proper divisors of $n$. All $k$--hyperperfect numbers less than $10^{11}$ have been computed. It seems that the conjecture "\emph{all} $k$--\emph{hyperperfect numbers for odd} $k>1$ \emph{are of the form} $p^2\cdot q$, \emph{with} $p=(3k+4)/4$ \emph{prime and} $q=3k+4=2p+3$ \emph{prime}" is false  \citep{McCranie2000}.
\end{itemize}
This results do not offer the solutions to the problems but they are important contributions worth mentioning.

The emergence of mathematical software generated a new wave of mathematical research aided by computers. Nowadays it is almost impossible to conduct research in mathematics without using software solutions such as  Maple, Mathematica, Matlab or Mathcad, etc. The authors used extensively Mathcad to explore and solve various Diophantine equations because of the very friendly nature of the interface and the powerful programming tools that this software provides. All the programs presented in the following chapters are in their complete syntax as used in Mathcad. The compact nature of the code and ease of interpretation made the choice of this particular software even more appropriate for use in a written presentation of solving techniques.

The empirical search programs in this book where developed and executed in Mathcad. The source code of this algorithms can be interpreted as pseudo code (the Mathcad syntax allows users to write code that is very easy to read) and thus translated to other programming languages.

Although the intention of the authors was to provide the reader with a comprehensive book some of the notions are presented out of order. For example the book the primality test that used Smarandache's function is extensively used. The first occurrences of this test preceded the definition the actual functions and its properties. However, overall, the text covers all definition and proves for each mathematical construct used. At the same time the references point to the most recent publications in literature, while results are presented in full only when the number of solutions is reasonable. For all other problems, that generate in excess of $100$ double, triple or quadruple pairs, only partial results are contained in the sections of this book. Nevertheless, anyone interested in the complete list should contact the authors to obtain a electronic copy of it. Running the programs in this book will also generate the same complete list of possible solutions for any odd the problems in this book.
\begin{flushright}
  Authors
\end{flushright}

\subsection*{Acknowledgments}
We would like to thank all the collaborators that helped putting together this book, especially to Codru\c{t}a Stoica and Cristian Mihai Cira, for the important comments and observations.

\renewcommand\indexname{{\LARGE Preface}}
\addcontentsline{toc}{chapter}{Preface}
\tableofcontents
\renewcommand\indexname{{\LARGE Contents}}
\addcontentsline{toc}{chapter}{Contents}
\listoffigures
\renewcommand\indexname{{\LARGE List of figure}}
\addcontentsline{toc}{chapter}{List of figure}
\listoftables
\renewcommand\indexname{{\LARGE List of table}}
\addcontentsline{toc}{chapter}{List of table}
\chapter{Introduction}

A Diophantine equation is a linear equation $ax+by=c$ where $a,b,c\in\Int$ and the solutions $x$ and  $y$ are also integer numbers. This equation can be completely solved by the well known algorithm proposed by Brahmagupta\index{Brahmagupta} \citep{WeissteinDiophantineEquation}.

In 1900, Hilbert\index{Hilbert D.} wondered if there is an universal algorithm that solves the Diophantine equation, but \cite{Matiyasevich1970}\index{Matiyasevich Y. V.} proved that such an algorithm does not exist for the first order solution.

The function $\eta$ relates to each natural number $n$ the smallest natural number $m$ such that $m!$ is a multiple of $n$. In this book we aim to find analytical or empirical solutions to Diophantine and $\eta$--Diophantine equation, namely Diophantine equation that contain the Smarandache\textquoteright{s}\index{Smarandache F.} $\eta$ function, \cite{Smarandache1980a}.

An analytical solution implies a general solution that completely solves the problem. For example, the general solution for the equation $a\cdot x-b\cdot y=c$, with $a,b,c\in\Ns$ is $x_k=b\cdot k+x_0$ and $y_k=a\cdot k+y_0$, where $(x_0,y_0)$ is a particular solution, and $k$ is an integer, $k\ge\max\set{\lceil-x_0/b\rceil,\lceil-y_0/a\rceil}$.

By and empirical solution we understand a set of algorithms that determine the solutions of the Diophantine equation within a finite domain of integer numbers, dubbed \emph{the search domain} to dimension $d$. For example, the $\eta$--Diophantine equation $\eta(m\cdot x+n)=x$ over the valid \emph{search domain} of dimension $d=3$, the solutions could be the triplets $(m,n,x)\in D_c=\set{1,2,\ldots,1000}\times\set{1,2,\ldots1000}\times\set{1,2,\ldots,999}$.

The first chapter introduces concepts about prime numbers, primality tests, decomposition algorithms for natural numbers, counting algorithms for all natural numbers up to a real one, etc. This concepts are fundamental for validating the empirical solutions of the $\eta$--Diophantine equations.

The second chapter introduces the function $\eta$ along side its known properties. This concepts allow the description of \cite{Kempner1918}\index{Kempner A. J.} algorithm that computes the $\eta$ function. The latter sections contain the set of commands and instructions that generate the file $\eta.prn$ which contains the $\eta(n)$ values for $n=1,2,\ldots,10^6$.

The third chapter describes the division functions $\sigma_0$, $\sigma_1$, usually denoted by $\sigma$, $\sigma_2$ and $s$. The $\sigma_0(n)$ function counts the number of divisors of $n$, while $\sigma(n)=\sigma_1(n)$ returns the sum of all those divisors. Consequently $\sigma_2(n)$ computed the sum of squared divisors of $n$ while, in general $\sigma_k(n)$ add all divisors to the power of $k$. We call divisors of $n$ all natural numbers that divide $n$ including $1$ and $n$, thus the proper divisors are considered all natural divisors excluding $n$ itself. In this case the function $s(n)=\sigma(n)-n$ is , in fact, the sum of all proper divisors. Along side the the definition, this third chapter also contains the properties and computing algorithms that generate the files $\sigma0.prn$, $\sigma1.prn$, $\sigma2.prn$, $s.prn$ that contain all the values for functions $\sigma_0(n)$, $\sigma(n)$, $\sigma_2(n)$ and $s(n)$ for $n=1,2,\ldots,10^6$. The last section describes the $k$--perfect numbers.

Euler\index{Euler L.}'s totient function also known as the $\varphi$ function that counts the natural numbers less than or equal to $n$ that are relatively prime is described in chapter $4$. As an example, for $n=12$ the relatively prime factors are $1$, $5$, $7$, and $11$ because $(1,12)=1$\footnote{where $(m, n)$ is $gcd(m,n)$ that is the greatest common divisor of $n$ and $m$}, $(5,12)=1$, $(7,12)=1$, and $(11,12)=1$, thus $\varphi(12)=4$. The chapter also describes the most important properties of this function. The latter section of the chapter contain the algorithm that generates the file $\varphi.prn$ that contains the values of the function $\varphi$ for $n$ raging from $1,2,\ldots,10^6$. Also, in this chapter the describes a generalization of Euler\index{Euler L.} theorem relative to the totient function $\varphi$ and the algorithm the computes the pair $(s,m_s)$ that verifies the Diophantine equation $a^{\varphi(m_s)}\equiv a^s\md{m}$, where $a,m\in\Ns$.

In chapter 5 we define a function $L$ which will allow us to (separately or simultaneously) generalize many theorems from Number Theory obtained by Wilson\index{Wilson J.}, Fermat\index{Fermat P.}, Euler\index{Euler L.}, Gauss\index{Gauss C. F.},
Lagrange\index{Lagrange J. L.}, Leibniz\index{Leibniz G.}, Moser\index{Moser L.}, and Sierpinski\index{Sierpinski W.}.

Various analytical solutions to Diophantine equations such as: the second degree equation, the linear equation with $n$ unknown, linear systems, the $n$ degree equation with one unknown, Pell\index{Pell J.} general equation, and the equation $x^2-2y^4=1$. For each of this cases, in chapter six we present symbolic computation that ensure the detection of the solutions for the particular Diophantine equations.

Chapter seven describes the solutions to the $\eta$--Diophantine equations using the search algorithms in the search domains.

The Conclusions and Index section conclude the book.

\mainmatter
\chapter{Prime numbers}

A prime number (or a prime) is a natural number greater than 1 that has no positive divisors other than $1$ and itself. A natural number greater than 1 that is not a prime number is called a composite number. For example, $7$ is prime because $1$ and $7$ are its only positive integer factors, whereas $10$ is composite because it has the divisors $2$ and $5$ in addition to $1$ and $10$. The fundamental theorem of Arithmetics, \citep[p. 2-3]{Hardy+Wright2008}, establishes the central role of primes in the Number Theory: any integer greater than $1$ can be expressed as a product of primes that is unique up to ordering. The uniqueness in this theorem requires excluding $1$ as a prime because one can include arbitrarily many instances of $1$ in any factorization, e.g., $5$, $1\cdot5$, $1\cdot1\cdot5$, etc. are all valid factorizations of $5$, \citep{Estermann1952,Vinogradov1955}.

The property of being prime (or not) is called primality. A simple but slow method of verifying the primality of a given number $n$ is known as trial division. It consists of testing whether $n$ is a multiple of any integer between $2$ and $\lfloor\sqrt{n}\rfloor$.
The floor function $\lfloor x\rfloor$, also called the greatest integer function or integer value \citep{Spanier+Oldham1987},
gives the largest integer less than or equal to $x$. The name and symbol for the floor function were coined by Iverson\index{Iverson K. E.}, \citep{Graham+Knuth+Patashnik1994}. Algorithms much more efficient than trial division have been devised to test the primality of large numbers. Particularly fast methods are available for numbers of special forms, such as Mersenne\index{Mersenne M.} numbers. As of April 2014, the largest known prime number $2^{57885161}-1$ has $17425170$ decimal digits \citep{Caldwell2014}.

There are infinitely many primes, as demonstrated by Euclid\index{Euclid} around 300 BC. There is no known useful formula that sets apart all of the prime numbers from composites. However, the distribution of primes, that is to say, the statistical behavior of primes in the large, can be modelled. The first result in that direction is the prime number theorem, proven at the end of the 19th century, which says that the probability that a given, randomly chosen number $n$ is prime is inversely proportional to its number of digits, or to $\log(n)$.

Many questions around prime numbers remain open, such as Goldbach\textquoteright{s}\index{Goldbach C.} conjecture, and the
twin prime conjecture, Diophantine equations that have integer functions. Such questions spurred the development of various
branches of the Number Theory, focusing on analytic or algebraic aspects of numbers. Prime numbers give rise to various
generalizations in other mathematical domains, mainly algebra, such as prime elements and prime ideals.

\section{Generating prime numbers}

The generation of prime numbers can be done by means of several deterministic algorithms, known in the literature, as sieves:
Sieve of Eratosthenes\index{Eratosthenes}, Sieve of Euler\index{Euler L.}, Sieve of Sundaram\index{Sundaram S. P.}, Sieve of Atkin\index{Atkin A. O. L.}, etc. In this volume we will detail only the most efficient prime number generating algorithms.

The Sieve of Eratosthenes\index{Eratosthenes} is an algorithm that allows the generation of all prime numbers up to a given limit
$L\in\Ns$. The algorithm was given by Eratosthenes\index{Eratosthenes} around 240 BC.

\begin{prog}\label{CEPritchard}
Let us consider the origin of vectors and matrices $1$, which can be defined in Mathcad by assigning $ORIGIN:=1$. The Sieve of Eratosthenes\index{Eratosthenes} in the linear variant of Pritchard\index{Pritchard P.}, presented in pseudo code in the article \citep{Pritchard1987}, written in Mathcad is:
\begin{tabbing}
  $CEP(L):=$\=\vline\ $fo$\=$r\ k\in1..L$\\
  \>\vline\ \>\ $is\_prime_k\leftarrow1$\\
  \>\vline\ $k\leftarrow2$\\
  \>\vline\ $w$\=$hile\ k^2\le L$\\
  \>\vline\>\vline\ $j\leftarrow k^2$\\
  \>\vline\>\vline\ $w$\=$hile\ j\le L$\\
  \>\vline\>\vline\>\vline\ $is\_primep_j\leftarrow0$\\
  \>\vline\>\vline\>\vline\ $j\leftarrow j+k$\\
  \>\vline\>\vline\ $k\leftarrow k+1$\\
  \>\vline\ $j\leftarrow1$\\
  \>\vline\ $fo$\=$r\ k\in1..L$\\
  \>\vline\ \>\vline\ $if$\=$\ is\_prime_k\textbf{=}1$\\
  \>\vline\ \>\vline\ \>\vline\ $prime_j\leftarrow k$\\
  \>\vline\ \>\vline\ \>\vline\ $j\leftarrow j+1$\\
  \>\vline\ $return\ prime$
\end{tabbing}
It is well known that the segmented version of the Sieve of Eratosthenes\index{Eratosthenes}, with basic optimizations, uses
$O(L)$ operations and
\[
 O\left(\sqrt{L}\frac{\log(\log(L))}{\log(L)}\right)
\]
bits of memory, \citep{Pritchard1987,Pritchard1994}.
\end{prog}

\begin{table}[h]
  \centering
  \begin{multline*}
  \begin{tabular}{|l|cccccccccccc}
    \hline
    $q\backslash is\_prime$ & 1 & 2 & 3 & 4 & 5 & 6 & 7 & 8 & 9 & 10 & 11 & 12 \\ \hline\hline
      & 0 & 1 & 1 & 1 & 1 & 1 & 1 & 1 & 1 &  1 &  1 &  1 \\ \hline
    2 &   &   &   & 0 &   & 0 &   & 0 &   &  0 &    &  0 \\ \hline
    3 &   &   &   &   &   &   &   &   & 0 &    &    &  0 \\ \hline
    4 &   &   &   &   &   &   &   &   &   &    &    &    \\ \hline
    5 &   &   &   &   &   &   &   &   &   &    &    &    \\ \hline\hline
      & 0 & 1 & 1 & 0 & 1 & 0 & 1 & 0 & 0 &  0 &  1 &  0 \\ \hline
  \end{tabular}\\
  \begin{tabular}{ccccccccccccc|}
     \hline
     13 & 14 & 15 & 16 & 17 & 18 & 19 & 20 & 21 & 22 & 23 & 24 & 25 \\ \hline\hline
      1 &  1 &  1 &  1 &  1 &  1 &  1 &  1 &  1 &  1 &  1 &  1 &  1 \\ \hline
        &  0 &    &  0 &    &  0 &    &  0 &    &  0 &    &  0 &    \\ \hline
        &    &  0 &    &    &  0 &    &    &  0 &    &    &  0 &    \\ \hline
        &    &    &  0 &    &    &    &  0 &    &    &    &  0 &    \\ \hline
        &    &    &    &    &    &    &    &    &    &    &    &  0 \\ \hline\hline
      1 &  0 &  0 &  0 &  1 &  0 &  1 &  0 &  0 &  0 &  1 &  0 &  0 \\ \hline
   \end{tabular}
\end{multline*}
  \caption{The vector $is\_prime$ in the code \ref{CEPritchard}}\label{AlgEratostene}
\end{table}

The linear variant of the Sieve of Eratosthenes\index{Eratosthenes} implemented by Pritchard\index{Pritchard P.}, given by the code \ref{CEPritchard}, has the inconvenience that is repeats uselessly operations. For example, for $L=25$, in table (\ref{AlgEratostene}) is given the binary vector $is\_prime$ which contains at each position the values $1$ or $0$. On the first line is the index of the vector.
\begin{enumerate}
  \item Initially, all the positions of vector $is\_prime$ have the value $1$.
  \item For $q=2$ the algorithm puts $0$ on all the positions $is\_prime_k$ multiple of $2$, for $k\ge q^2=4$.
  \item For $q=3$ the algorithm puts $0$ on all the positions $is\_prime_k$ multiple of $3$, for $k\ge q^2=9$, which means positions $9$, $12$, $15$, $18$,
   $21$ and $24$ but positions $12$, $18$ and $24$ were already annulated in the previous step.
  \item For $q=4$ the algorithm puts $0$ on all the positions $is\_prime_k$ multiple of $4$, for $k\ge q^2=16$, which means positions $16$, $20$, $24$,
  but these positions were annulated also in the second step, and on position $24$ is taken $0$ for the third time.
  \item For $q=5$ one takes $is\_prime_{q^2}=0$.
\end{enumerate}

Eventually, vector $is\_prime$ is read. The index of vector $is\_prime$, which has the value $1$, is a prime number. If we count the number of attributing the value $0$, we remark that this operation was made $21$ time. It is obvious that these repeated operations make the algorithm less efficient.

\begin{prog}\label{CEPritchardImpar}
This program is a better version of program \ref{CEPritchard} because it puts $0$ only on the odd positions of the vector $is\_prime$.
\begin{tabbing}
  $CEPi(L):=$\=\vline\ $fo$\=$r\ k\in3,5..L$\\
  \>\vline\ \>\ $is\_prime_k\leftarrow1$\\
  \>\vline\ $fo$\=$r\ k\in3,5..floor(\sqrt{L})$\\
  \>\vline\ \>\ $fo$\=$r\ j\in k^2,k^2+2k..L$\\
  \>\vline\ \>\ \>\ $is\_prime_j\leftarrow0$\\
  \>\vline\ $prime_1\leftarrow2$\\
  \>\vline\ $j\leftarrow2$\\
  \>\vline\ $fo$\=$r\ k\in1,3..L$\\
  \>\vline\ \>\ $if$\=$\ is\_prime_k\textbf{=}1$\\
  \>\vline\ \>\ \>\vline\ $prime_j\leftarrow k$\\
  \>\vline\ \>\ \>\vline\ $j\leftarrow j+1$\\
  \>\vline\ $return\ prime$
\end{tabbing}
\end{prog}

\begin{prog}\label{CEPritchardMemorieMinima}
This program is a better version of program \ref{CEPritchardImpar} because it uses a minimal memory space.
  \begin{tabbing}
  $CEPm(L):=$\=\vline\ $\lambda\leftarrow floor\left(\frac{L}{2}\right)$\\
  \>\vline\ $fo$\=$r\ k\in1..\lambda$\\
  \>\vline\ \>\ $is\_prime_k\leftarrow1$\\
  \>\vline\ $fo$\=$r\ k\in3,5..floor(\sqrt{L})$\\
  \>\vline\ \>\ $fo$\=$r\ j\in k^2,k^2+2k..L$\\
  \>\vline\ \>\ \>\ $is\_prime_{\frac{j-1}{2}}\leftarrow0$\\
  \>\vline\ $prime_1\leftarrow2$\\
  \>\vline\ $j\leftarrow2$\\
  \>\vline\ $fo$\=$r\ k\in1..\lambda-1$\\
  \>\vline\ \>\ $if$\=$\ is\_prime_k\textbf{=}1$\\
  \>\vline\ \>\ \>\vline\ $prime_j\leftarrow 2\cdot k+1$\\
  \>\vline\ \>\ \>\vline\ $j\leftarrow j+1$\\
  \>\vline\ $return\ prime$
\end{tabbing}
\end{prog}

Even the execution time of the program \ref{CEPritchardMemorieMinima} is a little longer than of the program \ref{CEPritchardImpar}, the best linear variant of the Sieve of Eratosthenes is the program \ref{CEPritchardMemorieMinima}, as it provides an important memory
economy ($11270607$ memory locations instead of $21270607$, the amount of memory locations used by programs \ref{CEPritchard} and
\ref{CEPritchardImpar}).

\begin{prog}\label{CECira}
The program for the Sieve of Eratosthenes\index{Eratosthenes}, Pritchard\index{Pritchard P.} variant, was improved in order to allow the number of repeated operations to diminish. The improvement consists in the fact that attributing $0$ is done for only odd multiples of prime numbers. The program has a restriction, but which won't cause inconveniences, namely $L$ must be a integer greater than $14$.
\begin{tabbing}
  $CEPb(L):=$\=\vline\ $f$\=$or\ k\in3,5..L$\\
  \>\vline\ \>\ $is\_prime_k\leftarrow1$\\
  \>\vline\ $prime\leftarrow(2\ 3\ 5\ 7)^\textrm{T}$\\
  \>\vline\ $i\leftarrow last(prime)+1$\\
  \>\vline\ $f$\=$or\ j\in9,15..L$\\
  \>\vline\ \>\ $is\_prime_j\leftarrow0$\\
  \>\vline\ $k\leftarrow3$\\
  \>\vline\ $s\leftarrow (prime_{k-1})^2$\\
  \>\vline\ $t\leftarrow (prime_k)^2$\\
  \>\vline\ $w$\=$hile\ t\le L$\\
  \>\vline\ \>\vline\ $f$\=$or\ j\in t,t+2\cdot prime_k..L$\\
  \>\vline\ \>\vline\ \>\ $is\_prime_j\leftarrow0$\\
  \>\vline\ \>\vline\ $f$\=$or\ j\in s+2,s+4..t-2$\\
  \>\vline\ \>\vline\ \>\ $if$\=\ $is\_prime_j=1$\\
  \>\vline\ \>\vline\ \>\ \>\vline\ $prime_i\leftarrow j$\\
  \>\vline\ \>\vline\ \>\ \>\vline\ $i\leftarrow i+1$\\
  \>\vline\ \>\vline\ $s\leftarrow t$\\
  \>\vline\ \>\vline\ $k\leftarrow k+1$\\
  \>\vline\ \>\vline\ $t\leftarrow (prime_k)^2$\\
  \>\vline\ $f$\=$or\ j\in s+2,s+4..L$\\
  \>\vline\ \> $if$\=$\ is\_prime_j\textbf{=}1$\\
  \>\vline\ \>\ \>\vline\ $prime_i\leftarrow j$\\
  \>\vline\ \>\ \>\vline\ $i\leftarrow i+1$\\
  \>\vline\ $return\ prime$
\end{tabbing}

We remark that it is not necessary to put $0$ on each positions $(prime_k)^2+prime_k$, as in the original version of the program \ref{CEPritchard}, because the sum of two odd numbers is an even number and the even positions are not considered. In this moment
of the program we are sure that the positions from $(prime_{k-1})^2+2$ to $(prime_k)^2-2$ of the vector \emph{is\_prime} (from 2 in 2) which were left on $1$ (which means that their indexes are prime numbers), can be added to the prime numbers vector. Hence, instead of building the vector $prime$ at the end of the markings, we do it in intermediary steps. The advantage consists on the fact that we have a list of prime numbers which can be used to obtain the other primes, up to the given limit $L$.
\end{prog}

\begin{prog}\label{CECiraMemorieMinima}
The program that improves the program CEPb by halving the used memory space.
\begin{tabbing}
  $CEPbm(L):=$\=\vline\ $\lambda\leftarrow floor\left(\frac{L}{2}\right)$\\
  \>\vline\ $f$\=$or\ k\in1..\lambda$\\
  \>\vline\ \>\ $is\_prime_k\leftarrow1$\\
  \>\vline\ $prime\leftarrow(2\ 3\ 5\ 7)^\textrm{T}$\\
  \>\vline\ $i\leftarrow last(prime)+1$\\
  \>\vline\ $f$\=$or\ j\in4,7..\lambda$\\
  \>\vline\ \>\ $is\_prime_j\leftarrow0$\\
  \>\vline\ $k\leftarrow3$\\
  \>\vline\ $s\leftarrow (prime_{k-1})^2$\\
  \>\vline\ $t\leftarrow (prime_k)^2$\\
  \>\vline\ $w$\=$hile\ t\le L$\\
  \>\vline\ \>\vline\ $f$\=$or\ j\in t,t+2\cdot prime_k..L$\\
  \>\vline\ \>\vline\ \>\ $is\_prime_{\frac{j-1}{2}}\leftarrow0$\\
  \>\vline\ \>\vline\ $f$\=$or\ j\in s+2,s+4..t-2$\\
  \>\vline\ \>\vline\ \>\ $if$\=\ $is\_prime_{\frac{j-1}{2}}=1$\\
  \>\vline\ \>\vline\ \>\ \>\vline\ $prime_i\leftarrow j$\\
  \>\vline\ \>\vline\ \>\ \>\vline\ $i\leftarrow i+1$\\
  \>\vline\ \>\vline\ $s\leftarrow t$\\
  \>\vline\ \>\vline\ $k\leftarrow k+1$\\
  \>\vline\ \>\vline\ $t\leftarrow (prime_k)^2$\\
  \>\vline\ $f$\=$or\ j\in s+2,s+4..L$\\
  \>\vline\ \> $if$\=$\ is\_prime_{\frac{j-1}{2}}\textbf{=}1$\\
  \>\vline\ \>\ \>\vline\ $prime_i\leftarrow j$\\
  \>\vline\ \>\ \>\vline\ $i\leftarrow i+1$\\
  \>\vline\ $return\ prime$
\end{tabbing}
\end{prog}

The performances of the 5 programs can be observed on the following execution sequences (the call of the programs have been done on the same computer and in similar conditions):
\begin{enumerate}
  \item Call of the program CEP\ref{CEPritchard}, i.e. the Sieve of Eratosthenes in the linear version of
     Pritchard\index{Pritchard P.}
      \[
       L:=2\cdot10^7\ \ \ \ t_0:=time(0)\ \ \ \ p:=CEP(L)\ \ \ \ t_1:=time(1)
      \]
      \[
       (t_1-t_0)sec=28.238s\ \ \ \ last(p)=1270607\ \ \ \ p_{last(p)}=19999999~,
      \]
  \item Call of the program CEPm\ref{CEPritchardMemorieMinima},
      \[
       L:=2\cdot10^7\ \ \ \ t_0:=time(0)\ \ \ \ p:=CEPm(L)\ \ \ \ t_1:=time(1)
      \]
      \[
       (t_1-t_0)sec=10.920s\ \ \ \ last(p)=1270607\ \ \ \ p_{last(p)}=19999999~.
      \]
  \item Call of the program CEPi\ref{CEPritchardImpar},
      \[
       L:=2\cdot10^7\ \ \ \ t_0:=time(0)\ \ \ \ p:=CEPi(L)\ \ \ \ t_1:=time(1)
      \]
      \[
       (t_1-t_0)sec=7.231s\ \ \ \ last(p)=1270607\ \ \ \ p_{last(p)}=19999999~,
      \]
  \item Call of the program CEPb\ref{CECira},
      \[
       L:=2\cdot10^7\ \ \ \ t_0:=time(0)\ \ \ \ p:=CEPb(L)\ \ \ \ t_1:=time(1)
      \]
      \[
       (t_1-t_0)sec=5.064s\ \ \  last(p)=1270607\ \ \ p_{last(p)}=19999999
      \]
  \item Call of the program CEPbm\ref{CECiraMemorieMinima},
      \[
       L:=2\cdot10^7\ \ \ \ t_0:=time(0)\ \ \ \ p:=CEPb(L)\ \ \ \ t_1:=time(1)
      \]
      \[
       (t_1-t_0)sec=7.133s\ \ \  last(p)=1270607\ \ \ p_{last(p)}=19999999
      \]
\end{enumerate}

In the comparative table \ref{TabelComparativ} are presented the attributions of $0$, the execution times on a computer with a
processor Intel of 2.20GHz with RAM of 4.00GB (3.46GB usable) and the number of memory units for the programs \ref{CEPritchard},
\ref{CEPritchardMemorieMinima}, \ref{CEPritchardImpar}, \ref{CECira} and \ref{CECiraMemorieMinima}.
\begin{table}[h]
  \centering
  \begin{tabular}{|r|r|r|r|}
     \hline
     program & Attributions of $0$  & Execution time & Memory used \\ \hline
     \ref{CEPritchard} &       71\ 760\ 995 &      28.238\ sec & 21\ 270\ 607\\ \hline
     \ref{CEPritchardMemorieMinima} &      35\ 881\ 043 &      10.920\ sec & 11\ 270\ 607\\ \hline
     \ref{CEPritchardImpar} &     35\ 881\ 043 &     7.231\ sec & 21\ 270\ 607\\ \hline
     \ref{CECira} &    18\ 294\ 176 &    5.064\ sec & 21\ 270\ 607\\ \hline
     \ref{CECiraMemorieMinima} &    18\ 294\ 176 &    7.133\ sec & 11\ 270\ 607\\ \hline
   \end{tabular}
  \caption{Comparative table}\label{TabelComparativ}
\end{table}

The Sieve of Sundaram is a simple deterministic algorithm for finding the prime numbers up to a given natural number. This algorithm was presented by \cite{Sundaram+Aiyar1934}\index{Sundaram S. P.}. As it is known, the Sieve of Sundaram uses $O(L\log(L))$ operations in order to find the prime numbers up to $L$. The algorithm of the Sieve of Sundaram in pseudo code Mathcad is:
\begin{tabbing}
  $CS(L):=$\=\vline\ $m\leftarrow floor\left(\dfrac{L}{2}\right)$\\
  \>\vline\ $fo$\=$r\ k\in1..m$\\
  \>\vline\ \>\ $is\_prime_k\leftarrow1$\\
  \>\vline\ $fo$\=$r\ k\in1..m$\\
  \>\vline\ \>\ $fo$\=$r\ j\in 1..ceil\left(\dfrac{m-k}{2\cdot k+1}\right)$\\
  \>\vline\ \>\ \>\ $is\_prime_{k+j+2\cdot k\cdot j}\leftarrow0$\\
  \>\vline\ $prime_1\leftarrow2$\\
  \>\vline\ $j\leftarrow1$\\
  \>\vline\ $fo$\=$r\ k\in1..m$\\
  \>\vline\ \>\ $if$\=$\ is\_prime_k\textbf{=}1$\\
  \>\vline\ \>\ \>\vline\ $j\leftarrow j+1$\\
  \>\vline\ \>\ \>\vline\ $prime_j\leftarrow2\cdot k+1$\\
  \>\vline\ $return\ prime$
\end{tabbing}
The Call of the program CS
\[
 L:=2\cdot10^7\ \ \ \ t_0:=time(0)\ \ \ \ p:=CS(L)\ \ \ \ t_1:=time(1)
\]
\[
 (t_1-t_0)sec=32.706s\ \ \  last(p)=1270607\ \ \ p_{last(p)}=19999999
\]

Until recently, i.e. till the appearance of the Sieve of Atkin\index{Atkin A. O. L.}, \citep{Atkin+Bernstein2004}, the Sieve of Eratosthenes\index{Eratosthenes} was considered the most efficient algorithm that generates all the prime numbers up to a limit $L$. The figure \ref{RapOEpeOA} emphasize the graphic representation of the ratio between the number of operations needed for the Sieve of Eratosthenes, $OE(L):=O(L\cdot\log(\log(L)))$, and the number of operations needed for the Sieve of Atkin, $OA(L):=O(L/\log(\log(L)))$, for $L=10^2,10^3,\ldots,10^{20}$. In this figure one can see that the Sieve of Atkin is better (relative to the number of operations
needed by the program) then the Sieve of Eratosthenes, for $L>10^{10}$.
\begin{figure}
  \centering
  \includegraphics[scale=0.7]{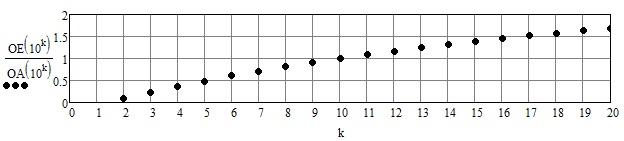}\\
  \caption{The ratio of the numbers of operations}\label{RapOEpeOA}
\end{figure}
\begin{prog}\label{CApsedocod}
The Sieve of Atkin in pseudo code presented in Mathcad is:
\begin{tabbing}
  $Atkin(L):=$\=\vline\ $fo$\=$r\ k\in5..L$\\
  \>\vline\ \>\ $is\_prime_k\leftarrow0$\\
  \>\vline\ $fo$\=$r\ x\in1..\sqrt{L}$\\
  \>\vline\ \>\vline\ $fo$\=$r\ y\in1..\sqrt{L}$\\
  \>\vline\ \> \vline\ \>\vline\ $n\leftarrow4x^2+y^2$\\
  \>\vline\ \> \vline\ \>\vline\ $if$\=$\ n\le L\wedge\big(mod(n,12)\textbf{=}1\vee mod(n,12)\textbf{=}5\big)$\\
  \>\vline\ \> \vline\ \>\vline\ \>\ $is\_prime_n\leftarrow\neg is\_prime_n$\\
  \>\vline\ \> \vline\ \>\vline\ $n\leftarrow3x^2+y^2$\\
  \>\vline\ \> \vline\ \>\vline\ $if$\=$\ n\le L\wedge mod(n,12)\textbf{=}7$\\
  \>\vline\ \> \vline\ \>\vline\ \>\ $is\_prime_n\leftarrow\neg is\_prime_n$\\
  \>\vline\ \> \vline\ \>\vline\ $n\leftarrow3x^2+y^2$\\
  \>\vline\ \> \vline\ \>\vline\ $if$\=$\ x\neq y \wedge n\le L\wedge mod(n,12)\textbf{=}11$\\
  \>\vline\ \> \vline\ \>\vline\ \>\ $is\_prime_n\leftarrow\neg is\_prime_n$\\
  \>\vline\ $fo$\=$r\ n\in5..\sqrt{L}$\\
  \>\vline\ \>\ $if$\=$\ is\_prime_n$\\
  \>\vline\ \>\ \>\ $fo$\=$r\ k\in1..\left\lfloor\frac{L}{n^2}\right\rfloor$\\
  \>\vline\ \>\ \>\ \>\ $is\_prime_{k\cdot n^2}\leftarrow0$\\
  \>\vline\ $prime_1\leftarrow2$\\
  \>\vline\ $prime_2\leftarrow3$\\
  \>\vline\ $j\leftarrow3$\\
  \>\vline\ $fo$\=$r\ n\in5..L$\\
  \>\vline\ \>\ $if$\=$\ is\_prime_n$\\
  \>\vline\ \>\ \>\ $prime_j\leftarrow n$\\
  \>\vline\ $return\ prime$\\
  \end{tabbing}

As it is known, this algorithm uses only $O(L/\log(\log(L)))$ simple operations and $O(L^{1/2+o(1)})$ memory locations,
\citep{Atkin+Bernstein2004}.
\end{prog}

Our implementation, in Mathcad, of Atkin's algorithm contains some remarks that make more performance program than the original
algorithm.
\begin{enumerate}
  \item Except $2$ all even numbers are not prime, it follows that, with the initialization $is\_prime_{2k}\leftarrow0$ for
     $k\in\set{2,3,\ldots,L/2}$, there is no need to change the values of these components. Consequently, we will change only the odd components.
  \item If $j$ is odd then $4k^2+j^2$ is always odd. It follows that the sequence
        \begin{equation}\label{ObsAtkin2}
          j\in\set{1,3..\left\lfloor\sqrt{L}\right\rfloor}\ \ \textnormal{and}\ \ k\in\set{1,2..\left\lfloor\frac{\sqrt{L-j^2}}{2}\right\rfloor}
        \end{equation}
        assures that the number $4k^2+j^2$ is always odd.
  \item If $j$ and $k$ have different parities, Then $3k^2+j^2$ is odd. Then the sequence
        \begin{multline}\label{ObsAtkin3}
          j\in\set{1,2,..\left\lfloor\sqrt{L}\right\rfloor}\\
          \textnormal{and}\ \ k_0=mod(j,2)+1~,\ \ k\in\set{k_0,k_0+2..\left\lfloor\sqrt{\frac{L-j^2}{3}}\right\rfloor}
        \end{multline}
        assures that $3k^2+j^2$ is odd.
  \item If $k>j$ and $k$ and $j$ have different parities, then $3k^2-j^2$ is odd. Then the sequence
         \begin{equation}\label{ObsAtkin4}
           j\in\set{1,2,..\left\lfloor\sqrt{L}\right\rfloor}\ \ \textnormal{and}\ \ k\in\set{j+1,j+3..\left\lfloor\sqrt{\frac{L+j^2}{3}}\right\rfloor}
         \end{equation}
      assures that $3k^2-j^2$ is odd.
  \item Similarly as in 1, we will eliminate only the perfect squares for odd numbers $\ge5$, because only these are odd.
\end{enumerate}

\begin{prog}\label{ProgAtkin}
AO program (Atkin optimized) of generating prime numbers up to $L$.
\begin{tabbing}
  $AO(L):=$\=\vline\ $is\_prime_L\leftarrow0$\\
  \>\vline\ $\lambda\leftarrow floor\big(\sqrt{L}\big)$\\
  \>\vline\ $fo$\=$r\ j\in1..ceil(\lambda)$\\
  \>\vline\ \>\vline\ $fo$\=$r\ k\in1..ceil\left(\frac{\sqrt{L-j^2}}{2}\right)$\\
  \>\vline\ \>\vline\ \>\vline\ $n\leftarrow4k^2+j^2$\\
  \>\vline\ \>\vline\ \>\vline\ $m\leftarrow mod(n,12)$\\
  \>\vline\ \>\vline\ \>\vline\ $is\_prime_n\leftarrow\neg is\_prime_n\ if\ n\le L\wedge(m\textbf{=}1\vee m\textbf{=}5)$\\
  \>\vline\ \>\vline\ $fo$\=$r\ k\in1..ceil\left(\sqrt{\frac{L-j^2}{3}}\right)$\\
  \>\vline\ \>\vline\ \>\vline\ $n\leftarrow3k^2+j^2$\\
  \>\vline\ \>\vline\ \>\vline\ $is\_prime_n\leftarrow\neg is\_prime_n\ if\ n\le L\wedge mod(n,12)\textbf{=}7$\\
  \>\vline\ \>\vline\ $fo$\=$r\ k\in j+1..ceil\left(\sqrt{\frac{L+j^2}{3}}\right)$\\
  \>\vline\ \>\vline\ \>\vline\ $n\leftarrow3k^2-j^2$\\
  \>\vline\ \>\vline\ \>\vline\ $is\_prime_n\leftarrow\neg is\_prime_n\ if\ n\le L\wedge mod(n,12)\textbf{=}11$\\
  \>\vline\ $fo$\=$r\ j\in5,7..\lambda$\\
  \>\vline\ \>\ $fo$\=$r\ k\in1,3..\frac{L}{j^2}\ if\ is\_prime_j$\\
  \>\vline\ \>\ \>\ $is\_prime_{k\cdot j^2}\leftarrow0$\\
  \>\vline\ $prime_1\leftarrow2$\\
  \>\vline\ $prime_2\leftarrow3$\\
  \>\vline\ $fo$\=$r\ n\in5,7..L$\\
  \>\vline\ \>\ $if$\=$\ is\_prime_n$\\
  \>\vline\ \>\ \>\vline\ $prime_j\leftarrow n$\\
  \>\vline\ \>\ \>\vline\ $j\leftarrow j+1$\\
  \>\vline\ $return\ prime$
\end{tabbing}
In this program function $ceil$ was used (which means $\lceil\cdot\rceil$) instead of function $floor$ (which means $\lfloor\cdot\rfloor$) in formulas (\ref{ObsAtkin2}), (\ref{ObsAtkin3}) and (\ref{ObsAtkin4}), in order to avoid errors
of floating comma which could determine the loss of cases at limit $L$, for example, when $L$ is a perfect square.

\begin{enumerate}
  \item Call of the program \ref{CApsedocod} the Sieve of $Atkin$
      \[
       L:=2\cdot10^7\ \ t_0:=time(0)\ \ p:=Atkin(L)\ \ t_1:=time(1)
      \]
      \[
       (t_1-t_0)s=23.531s\ \ p_{last(p)}=19999999\ \ last(p)=1270607~,
      \]
  \item Call of the program \ref{ProgAtkin} the optimized Sieve of $Atkin$
      \[
       L:=2\cdot10^7\ \ t_0:=time(0)\ \ p:=AO(L)\ \ t_1:=time(1)
      \]
      \[
       (t_1-t_0)s=19.45s\ \ p_{last(p)}=19999999\ \ last(p)=1270607~,
      \]
\end{enumerate}
\end{prog}

There exists an implementation for the Sieve of Atkin, due to \cite{BernsteinPrimgen}\index{Bernstein D. J.} under the name \emph{Primgen}. \emph{Primegen} is a library of programs for fast generating prime numbers, increasingly. \emph{Primegen} generates
all $50847534$ prime numbers up to $10^9$ in only 8 seconds on a computer with a Pentium II-350 processor. \emph{Primegen} can
generate prime numbers up to $10^{15}$.

\section{Primality tests}

A central problem in the Number Theory is to determine weather an odd integer is prime or not. The test than can establish this is
called primality test.

Primality tests can be deterministic or non-deterministic. The deterministic ones establish exactly if a number is prime, while
the non-deterministic ones can falsely determine that a composite number is prime. These test are much more faster then the
deterministic ones. The numbers that pass a non-deterministic primality test are called \emph{probably prime} (this is denoted
by \emph{prime?}) until their primality is deterministically proved. A list of \emph{probably prime} numbers are
Mersenne's\index{Mersenne M.} numbers, \citep{CaldwellThePrimePages}:
\begin{itemize}
  \item[] $M_{43}=2^{30402457}-1$, Dec. 2005 -- Curtis Cooper\index{Cooper C.} and Steven Boone\index{Boone S.},
  \item[] $M_{44}=2^{32582657}-1$, Sept. 2006 -- Curtis Cooper\index{Cooper C.} and Steven Boone\index{Boone S.},
  \item[] $M_{45}=2^{37156667}-1$, Sept. 2008 -- Hans-Michael Elvenich\index{Elvenich H.-M.},
  \item[] $M_{46}=2^{42643801}-1$, Apr. 2009 -- Odd Magnar Strindmo\index{Strindmo O. M.},
  \item[] $M_{47}=2^{43112609}-1$, Aug. 2008 -- Edson Smith\index{Smith E.},
  \item[] $M_{48}=2^{57885161}-1$, Jan. 2013 -- Curtis Cooper\index{Cooper C.}.
\end{itemize}

\subsection{The test of primality $\eta$}

As seen in Theorem \ref{TPrimalitateEta}, we can use as primality test the computing of the value of $\eta$ function. For $n>4$, if
relation $\eta(n)=n$ is satisfied, it follows that $n$ is prime. In other words, the prime numbers (to which number $4$ is added)
are fixed points for $\eta$ function. In this study we will use this primality test.

\begin{prog}\label{ProgTpEta}
The program for $\eta$ primality test. The program returns the value $0$ if the number is not prime and the value $1$ if the number is prime. File $\eta.prn$ is read and assigned to vector $\eta$~.
\[
 ORIGIN:=1\ \ \ \ \eta:=READPRN("\ldots\backslash\eta.prn")
\]
\begin{tabbing}
  $Tp\eta(n):=$\=\ \vline\ $return\ "Error\ n<1\ or\ not\ integer"\ if\ n<1\vee n\neq trunc(n)$\\
  \>\ \vline\ $i$\=$f\ n>4$\\
  \>\ \vline\ \>\vline\ $return\ 0\ if\ \eta_n\neq n$\\
  \>\ \vline\ \>\vline\ $return\ 1\ otherwise$\\
  \>\ \vline\ $o$\=$therwise$\\
  \>\ \vline\ \>\vline\ $return\ 0\ if\ n\textbf{=}1\vee n\textbf{=}4$\\
  \>\ \vline\ \>\vline\ $return\ 1\ otherwise$
\end{tabbing}
By means of the program \ref{ProgTpEta} was realized the following test.
\[
 n:=499999\ \ k:=1..n\ \ v_k:=2\cdot k+1
\]
\[
 last(v)=499999\ \ v_1=3\ \ v_{last(v)}=999999
\]
\[
 t_0:=time(0)\ \ w_k:=Tp\eta(v_k)\ \ t_1:=time(1)
\]
\[
 (t_1-t_0)sec=0.304s\ \ \ \sum w=78497~.
\]
The number of prime numbers up to $10^6$ is $78798$, and the sum of non-zero components (equal to 1) is $78797$, as $2$ was not
counted as prime number because it is an even number. We remark that the time needed by the primality test of all odd numbers is
$0.304s$ a much more better time than the $8s$ necessary for the primality test \ref{ProgramCautareBinara} on a computer with an
Intel processor of 2.20GHz with RAM of 4.00GB (3.46GB usable).
\end{prog}

\subsection{Deterministic tests}

Proving that an odd number $n$ is prime can be done by testing sequentially the vector $p$ that contains prime numbers.

The browsing of the list of prime numbers can be improved by means of the function that counts the prime numbers
\citep{WeissteinPrimeCountingFunction}. Traditionally, by $\pi(x)$ is denoted the function that indicates the number of prime numbers
$p\le x$, \cite[p. 15]{Shanks1962,Shanks1993}. The notation for the function that counts the prime numbers is a little bit
inappropriate as it has nothing to do with $\pi$, The universal constant that represents the ratio between the length of a circle
and its diameter. This notation was introduced by the number theorist Edmund Landau\index{Landau E.} in 1909 and has now become
standard, \citep{Landau1958} \citep[p. 38]{Derbyshire2004}.
We will give a famous result of \cite{Rosser+Schoenfeld1962}\index{Rosser B. J.}\index{Schoenfeld L.}, related to function $\pi(x)$. Let functions $\pi_s,\ \pi_d:(1,+\infty)\to\mathbb{R}_+$ given by formulas
\begin{equation}
  \pi_s(x)=\frac{x}{\ln(x)}\left(1+\frac{1}{2\ln(x)}\right)
\end{equation}
and
\begin{equation}
  \pi_d(x)=\frac{x}{\ln(x)}\left(1+\frac{3}{2\ln(x)}\right)~.
\end{equation}
\begin{thm}\label{T1Rosser}
  Following inequalities
  \begin{equation}\label{InegalitatiPiN}
    \pi_s(x)<\pi(x)<\pi_d(x)~,
  \end{equation}
  hold, for all $x>1$, the right side inequality, and for all $x\ge59$ the left side inequality.
\end{thm}
\begin{proof}
  See \cite[T. 1]{Rosser+Schoenfeld1962}.
\end{proof}
Let functions $f,\ \pi_m,\ \pi_M:\Ns\to\Ns$ be defined by formulas:
\[
 f(n)=\left\lfloor\frac{n}{\ln(n)}\left(1+\frac{1}{2\ln(n)}\right)\right\rfloor~,
\]
\begin{equation}
  \pi_m(n)=\left\{\begin{array}{ll}
                    f(n)-2 & \textnormal{if}\ n<11 \\ \\
                    f(n)-1 & \textnormal{if}\ 11\le n\le39 \\ \\
                    f(n) & \textnormal{if}\ n>39
                  \end{array}\right.~,
\end{equation}
\begin{equation}
  \pi_M(n)=\left\lceil\frac{n}{\ln(n)}\left(1+\frac{3}{2\ln(n)}\right)\right\rceil~,
\end{equation}
where $\lfloor\cdot\rfloor$ is the lower integer part function and $\lceil\cdot\rceil$ is the upper integer part function. As a
consequence of Theorem \ref{T1Rosser} we have
\begin{thm}\label{T1Cira}
  Following inequalities
  \begin{equation}\label{InegalitatiPimPiM}
    \pi_m(n)<\pi(x)<\pi_M(n)
  \end{equation}
  hold, for all $n\in2\Ns+1$, where by $2\Ns+1$ is denoted the set of natural odd numbers.
\end{thm}
\begin{proof}
  As function $\pi_d(n)\le\pi_M(n)$ for all $n\in\Ns$, it results, according to Theorem \ref{T1Rosser}, that the right side inequality is true for all  $n\in\Ns$, hence, also for $n\in2\Ns+1$.
\begin{figure}[h]
  \centering
  \includegraphics[scale=0.65]{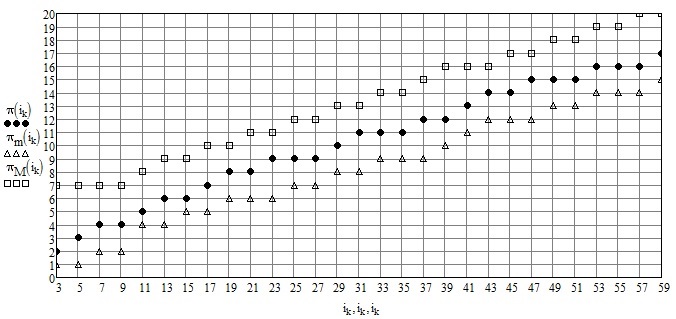}\\
  \caption{Functions $\pi_M(n)$, $\pi(n)$ and $\pi_m(n)$}\label{FigT1Cira}
\end{figure}

As $\pi_m(n)\le\pi_s(n)$ for all $n\in\Ns$, and the left side inequality (\ref{InegalitatiPiN}) holds for all $n\ge59$, it
follows that the left side inequality (\ref{InegalitatiPimPiM}) holds for all $n\ge59$.

For $n\in\set{3,5,7,\ldots,59}$ we have:
 \begin{equation}\label{pi-pim}
   \begin{array}{rcr}
   \pi(3)-\pi_m(3) &=& 1 \\
   \pi(5)-\pi_m(5) &=& 1 \\
   \pi(7)-\pi_m(7) &=& 2 \\
   \pi(9)-\pi_m(9) &=& 1 \\
   \pi(11)-\pi_m(11) &=& 1 \\
   \pi(13)-\pi_m(13) &=& 1 \\
   \pi(15)-\pi_m(15) &=& 1 \\
   \pi(17)-\pi_m(17) &=& 1 \\
   \pi(19)-\pi_m(19) &=& 2 \\
   \pi(21)-\pi_m(21) &=& 1 \\
   \pi(23)-\pi_m(23) &=& 2 \\
   \pi(25)-\pi_m(25) &=& 2 \\
   \pi(27)-\pi_m(27) &=& 1 \\
   \pi(29)-\pi_m(29) &=& 2 \\
 \end{array}\ \ \
 \begin{array}{rcr}
   \pi(31)-\pi_m(31) &=& 2 \\
   \pi(33)-\pi_m(33) &=& 2 \\
   \pi(35)-\pi_m(35) &=& 1 \\
   \pi(37)-\pi_m(37) &=& 2 \\
   \pi(39)-\pi_m(39) &=& 1 \\
   \pi(41)-\pi_m(41) &=& 1 \\
   \pi(43)-\pi_m(43) &=& 2 \\
   \pi(45)-\pi_m(45) &=& 1 \\
   \pi(47)-\pi_m(47) &=& 2 \\
   \pi(49)-\pi_m(49) &=& 1 \\
   \pi(51)-\pi_m(51) &=& 1 \\
   \pi(53)-\pi_m(53) &=& 1 \\
   \pi(55)-\pi_m(55) &=& 1 \\
   \pi(57)-\pi_m(57) &=& 1 \\
   \pi(59)-\pi_m(59) &=& 1 \\
 \end{array}
\end{equation}
we analyze table \ref{pi-pim} (see also \ref{FigT1Cira}) we can say that the left side inequality (\ref{InegalitatiPimPiM}) holds
for all $n\in2\Ns+1$.
\end{proof}

Theorem \ref{T1Cira} allows us to find a lower and an upper margin for the number of prime numbers up to the given odd number. Using
the bisection method, one can efficiently determine if the given odd numbers is in the list of prime numbers or not.

The function that counts the maximum number of necessary tests for the bisection algorithm to decide if number $N$ is prime, is given
by the formula:
\begin{equation}
  n_t(N)=\lceil\log_2\big(\pi_M(N)-\pi_m(N)\big)\rceil
\end{equation}
\begin{figure}[h]
  \centering
  \includegraphics[scale=0.75]{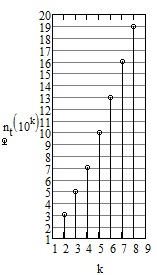}\\
  \caption{The graph of function $n_t(10^n)$ for $n=2,3,\ldots,8$}\label{FigFunctiaNt}
\end{figure}
The algorithm is efficient. For example, for numbers $N$, $10^7<N<10^8$, the algorithm will proceed between $16$ and $19$ necessary tests for the bisection algorithm, at the worst (see figure \ref{FigFunctiaNt}).

For all programs we have considered $ORIGIN:=1$~. By means of the algorithm \ref{CECira} (The Sieve of Eratosthenes, Pritchard's
improved version) and of command
\[
 p:=CEPb(2\cdot10^7)
\]
all prime numbers up to $2\cdot10^7$ are generated in vector $p$.

\begin{prog}\label{ProgramCautareBinara} The program is an efficient primality test for $N$. A binary search is used (the bisection algorithm), i.e., if $N$, which finds itself between the prime numbers $p_\ell$ and $p_r$, is in the list of prime numbers $p$.
\begin{tabbing}
  $Cb(N,\ell,r):=$\=\ \vline\ $w$\=$hile\ \ell<r$\\
  \>\ \vline\ \>\ \vline\ $M\leftarrow\dfrac{\ell+r}{2}$\\
  \>\ \vline\ \>\ \vline\ $m\leftarrow ceil\left(M\right)$\\
  \>\ \vline\ \>\ \vline\ $return\ 1\ \ if\ N\textbf{=}p_m$\\
  \>\ \vline\ \>\ \vline\ $\ell\leftarrow m\ \ if\ N>p_m$\\
  \>\ \vline\ \>\ \vline\ $r\leftarrow floor\left(M\right)\ \ if\ N<p_m$\\
  \>\ \vline\ $return\ 0$
\end{tabbing}
The subprogram \ref{ProgramCautareBinara} calls the components $p_k$ of the vector that contains the prime numbers. If $N$ is
prime, the subprogram returns 1, if $N$ is not prime, it returns $0$. The necessary time to test the primality of all odd numbers up
to $10^6$ is $8.283sec$ on a 2.2 GHz processor.
\end{prog}

Other deterministic tests:
\begin{enumerate}
  \item Pepin's\index{P\'{e}pin T.} test or the $p-1$ test. If we study attentively a list that contains the greatest known prim numbers, $p$, we will remark that most of them has a particular form, namely, $p-1$ or $p+1$ and can be decomposed very fast. This result is not unexpected as there exist deterministic primality tests for such numbers. In 1891,
      Lucas\index{Lucas F. E. A.}, \citep{Williams1998}, has converted the Fermat's Little Theorem\index{Fermat P.} into a practical primality test, improved afterwards by Kraitchik\index{Kraitchik M.} and Lehmer\index{Lehmer D. N.} \citep{Brillhart+Lehmer+Selfridge1975}, \citep{Dan2005}.
  \item n+1 tests or Lucas-Lehmer test for Mersenne numbers. Approximately half of the prime numbers in the list that contains the greatest known prim numbers are of the form $N-1$, where $N$ can be easily factorized.
      \begin{prog}
      The program for Lucas-Lehmer algorithm is:
      \begin{tabbing}
        $LL(n):=$\=\vline\ $return\ "Error\ n<3\ or\ n>53"\ if\ n\le2\vee n\ge54$\\
        \>\vline\ $M\leftarrow2^n-1$\\
        \>\vline\ $f\leftarrow Fa(n)$\\
        \>\vline\ $return\ (M\ "is\ not\ prime")\ if\ (f_{1,1})^2<n$\\
        \>\vline\ $s\leftarrow4$\\
        \>\vline\ $f$\=$or\ k\in1..n-2$\\
        \>\vline\ \>\vline\ $S\leftarrow s^2-2$\\
        \>\vline\ \>\vline\ $s\leftarrow mod(S,M)$\\
        \>\vline\ \>\vline\ $return\ "Error"\ if\ floor\left(\frac{S}{M}\right)\cdot M+s\neq S$\\
        \>\vline\ $return\ (M\ "is\ prime")\ if\ s\textbf{=}0$\\
        \>\vline\ $return\ (M\ "is\ prime")\ otherwise$
      \end{tabbing}
      Run examples:
      \[
       LL(11)=(2047\ "is\ not\ prime")\ \ \ LL(13)=(8191\ "is\ prime")\ \ \
      \]
      \[
       LL(19)=(524287\ "is\ prime")\ \ LL(23)=(8388607\ "is\ not\ prime")~.
      \]
      \end{prog}
  \item The Miller-Rabin test. If we apply the Miller's\index{Miller G. L.} test for numbers lesser than $2.5\cdot10^{10}$ but different from $3215031751$, and they pass the test for basis $2$, $3$, $5$ and $7$, they are prime. Similarly, if we apply a test in seven steps, the previously obtained results allow to verify the primality of all prime numbers up to $3.4\cdot10^{14}$. If we choose 25 iterations for Miller's algorithm applied to a number, the probability that this is not composite is lesser than $2^{-50}$. Hence, the Miller-Rabin\index{Rabin M. O.} test becomes a deterministic test for numbers lesser than $3.4\cdot10^{10},$\citep{Dan2005}.
      \begin{prog} The program for Miller-Rabin test is:
        \begin{tabbing}
          $MR(n):=$\=\vline\ $return\ "Error\ n<2\ or\ n\ even"\ if\ n<2\vee mod(n,2)=1$\\
          \>\vline\ $s\leftarrow0$\\
          \>\vline\ $t\leftarrow n-1$\\
          \>\vline\ $w$\=$hile\ mod(t,2)=0$\\
          \>\vline\ \>\vline\ $s\leftarrow s+1$\\
          \>\vline\ \>\vline\ $t\leftarrow\frac{t}{2}$\\
          \>\vline\ $\lambda\leftarrow\frac{\sqrt{n}}{2}$\\
          \>\vline\ $f$\=$or\ k\in1..25$\\
          \>\vline\ \>\vline\ $b\leftarrow2+2\cdot floor(rnd(\lambda))+1$\\
          \>\vline\ \>\vline\ $y\leftarrow RRP(b,t,n)$\\
          \>\vline\ \>\vline\ $if$\=$\ y\neq1\wedge y\neq n-1$\\
          \>\vline\ \>\vline\ \>\vline\ $j\leftarrow1$\\
          \>\vline\ \>\vline\ \>\vline\ $w$\=$hile\ j\le s-1\wedge j\neq n-1$\\
          \>\vline\ \>\vline\ \>\vline\ \>\vline\ $y\leftarrow mod(y^2,n)$\\
          \>\vline\ \>\vline\ \>\vline\ \>\vline\ $return\ 0\ if\ y\textbf{=}1$\\
          \>\vline\ \>\vline\ \>\vline\ \>\vline\ $j\leftarrow j+1$\\
          \>\vline\ \>\vline\ \>\vline\ $return\ 0\ if\ y\neq n-1$\\
          \>\vline\ $return\ 1$
        \end{tabbing}
        The test of the program ha been made for $n=2^{47}-1>3.4\cdot10^{10}$ and cu $n=2^{19}-1$.
        \[
         MR(2^{47}-1)=0\ \ \ MR(2^{19}-1)=1
        \]
        $n=2^{47}-1$ is indeed a composite number
        \[
         Fa(2^{47}-1)=\left(
                        \begin{array}{rr}
                          2351 & 1 \\
                          4513 & 1 \\
                          13264529 & 1 \\
                        \end{array}
                      \right)~,
        \]
       and $2^{19}-1=524287$ is a prime number. For factorization of a natural numbers has been done with the programs $Fa$, \ref{ProgFa}, emphasized in Section \ref{FD}~.

       The program $MR$ calls the program $RRP$ for repeatedly squaring modulo $m$, i.e. it calculates $mod(b^n,m)$ for great
       numbers.
        \begin{tabbing}
          $RRP(b,n,m):=$\=\vline\ $N\leftarrow1$\\
          \>\vline\ $return\ N\ if\ n\textbf{=}0$\\
          \>\vline\ $A\leftarrow b$\\
          \>\vline\ $a\leftarrow Cb2(n)$\\
          \>\vline\ $N\leftarrow b\ if\ a_0\textbf{=}1$\\
          \>\vline\ $f$\=$or\ k\in1..last(a)$\\
          \>\vline\ \>\vline\ $A\leftarrow mod(A^2,m)$\\
          \>\vline\ \>\vline\ $N\leftarrow mod(A\cdot N,m)\ if\ a_k\textbf{=}1$\\
          \>\vline\ $return\ N$
        \end{tabbing}
        The test of this program has been made on following example:
        \[
         RRP(5,596,1234)=1013~,
        \]
        provided in the paper \citep[p. 60]{Dan2005}. Concerning this program, it calls a program for finding the digits of basis $2$ for a decimal number.
        \begin{tabbing}
          $Cb2(n):=$\=\vline\ $j\leftarrow0$\\
          \>\vline\ $c_0\leftarrow n$\\
          \>\vline\ $w$\=$hile\ trunc\left(\dfrac{c_j}{2}\right)\textbf{=}0$\\
          \>\vline\ \>\vline\ $r_j\leftarrow mod(c_j,2)$\\
          \>\vline\ \>\vline\ $j\leftarrow j+1$\\
          \>\vline\ \>\vline\ $c_j\leftarrow trunc\left(\dfrac{c_{j-1}}{2}\right)$\\
          \>\vline\ $r_j\leftarrow c_j$\\
          \>\vline\ $return\ r$
        \end{tabbing}
        The test of this program is made by following example:
        \[
         Cb2(107)=\left(\begin{array}{c}
                          1 \\
                          1 \\
                          0 \\
                          1 \\
                          0 \\
                          1 \\
                          1 \\
                    \end{array}\right)~.
        \]
      \end{prog}

  \item AKS test. Agrawal\index{Agrawal M.}, Kayal\index{Kayal N.} and Saxena\index{Saxena N.}, \citep{Agrawal+Kayal+Saxena2004}, have found a deterministic algorithm, relative easy, that isn't based on any unproved statement. The idea of AKS test results form a simple version of the Fermat's Little Theorem \index{Fermat P.}. The AKS algorithm is:
      \begin{enumerate}
        \item[INPUT] a natural number $>2$;
        \item[OUTPUT] $0$ if $n$ is not prime, $1$ if $n$ is prime;
        \item[1.] If $n$ is of the form $a^b$, with $b>1$, then return: $n$ is not prime and stop the algorithm.
        \item[2.] Let $r\leftarrow2$.
        \item[3.] As long as $r<n$; execute:
            \begin{enumerate}
              \item[3.1.] If $(n,r)\neq1$, return: $n$ is not prime and stop the algorithm.
              \item[3.2.] If $r\ge2$ and it is prime, then execute: let $q$ be the greatest factor of$r-1$, then, if $q>4\sqrt{r}\lg(n)$ and
              $n^{(r-1)/q}\neq 1$ $(mod\ r)$, then go to item 4.
              \item[3.3.] Let $r\leftarrow r+1$.
            \end{enumerate}
        \item[4.] For $a$ from 1 to $2\sqrt{r}\lg(n)$, execute:
            \begin{enumerate}
              \item[4.1.] If $(x-a)^n\neq x^n-a$ $(mod\ x^r-1,n)$, then return: $n$ is not prime and stop the algorithm.
            \end{enumerate}
        \item[5.] Return: $n$ is prime and stop the algorithm.
      \end{enumerate}
\end{enumerate}

\subsection{Smarandache's criteria of primality}

In this section we present four necessary and sufficient conditions for a natural number to be prime, \citep{Smarandache1981a}.

\begin{defn}
  We say that integers $a$ are $b$ congruent modulo $m$ \big(denoted $a\equiv b\md{m}$\big) if and only if $m\mid a-b$ \big(i.e. $m$ divides $a-b$\big) or $a-b=k\cdot m$, where $k\in\Int$, $k\neq1$ and $k\neq m$ \big(i.e. $m$ is a proper factor of $a-b$\big). Therefore, we have
  \begin{equation}
    a\equiv b\md{m} \Leftrightarrow mod(a-b,m)=0~,
  \end{equation}
  where $mod(x,y)$ is the function that returns the rest of the division of $x$ by $y$, with $x,y\in\Int$.
\end{defn}

In 1640 Fermat\index{Fermat P.} shows without demonstrate the following theorem:
\begin{thm}[Fermat]
  If $a\in\Na$ and $p$ is prime and $p\nmid a$, then
  \[
   a^{p-1}\equiv1\md{p}~.
  \]
\end{thm}
The first proof of the this theorem was given in 1736 by Euler\index{Euler L.}.

\begin{thm}[Wilson]\label{TWilson1}
  If $p$ is prime, then
  \[
   (p-1)!+1\equiv0\md{p}~.
  \]
\end{thm}
The theorem Wilson\index{Wilson J.} \ref{TWilson1} was published by \cite{Warnig1770}\index{Waring E.}, but it was known long before even Leibniz\index{Leibniz G.}.

\begin{thm}\label{T1CP}
  Let $p\in\Ns$, $p\ge3$, then $p$ is prime if and only if
  \begin{equation}\label{CP1}
    (p-3)!\equiv\frac{p-1}{2} \md{p}~.
  \end{equation}
\end{thm}
\begin{proof}\

  \emph{Necessity}: $p$ is prime $\Rightarrow$ $(p-1)!\equiv -1 \md{p}$ conform to Wilson\textquoteright{s} theorem \ref{TWilson1}. It results that $(p-1)(p-2)(p-3)!\equiv -1 \md{p}$, or $2(p-3)!\equiv p-1 \md{p}$. But $p$ being a prime number $\ge3$ it results that $(2,p)=1$ and $(p-1)/2\in\Int$. It has sense the division of the congruence by $2$, and therefore we obtain the conclusion.

  \emph{Sufficiency}: We multiply the congruence $(p-3)!\equiv(p-1)/2 \md{p}$ with $(p-1)(p-2)\equiv2 \md{p}$,
  \citep[pp. 10-16]{Popovici1973}, and it results that $(p-1)!\equiv-1\md{p}$ from Wilson\textquoteright{s} theorem \ref{TWilson1}, which makes that $p$ is prime.
\end{proof}

\begin{prog} The primality criterion (\ref{CP1}), given by Theorem \ref{T1CP} can be implemented in Mathcad as follows:
  \begin{tabbing}
    $CSP1(p):=$\=\vline\ $return\ -1\ if\ p<3\vee p\neq trunc(p)$\\
    \>\vline\ $return\ 1\ if\ mod\left[(p-3)!-\dfrac{p-1}{2},p\right]\textbf{=}0$\\
    \>\vline\ $return\ 0\ otherwise$
  \end{tabbing}
\end{prog}

The call of this criterion using the symbolic computation is:
\[
 \begin{array}{lcr}
   CSP1(2) & \rightarrow & -1~,\\
   CSP1(3.5) & \rightarrow & -1~,\\
   CSP1(61) & \rightarrow & 1~, \\
   CSP1(87) & \rightarrow & 0~, \\
   CSP1(127) & \rightarrow & 1~, \\
   CSP1(1057) & \rightarrow & 0~,
 \end{array}
\]
where $1$ indicates that the number is prime, $0$ the contrary and $-1$ error, i.e. $p<3$ or $p$ is not integer.

\begin{lem}\label{Lemma1CP}
  Let $m$ be a natural number $>4$. Then $m$ is a composite number if and only if $(m-1)!\equiv0 \md{m}$.
\end{lem}
\begin{proof}\

  The \emph{sufficiency} is evident conform to Wilson\textquoteright{s} theorem \ref{TWilson1}.

  \emph{ Necessity}: $m$ can be written as $m=\desp[\alpha]{s}$ where $p_i$ prime numbers, two by two distinct and $\alpha_i\in\Ns$, for any $i\in\I{s}=\set{1,2,\ldots,s}$.

  If $s\neq1$ then $p_i^{\alpha_i}<m$, for any $i\in\I{s}$. Therefore $\desp[\alpha]{s}$ are distinct factors in the product $(m-1)!$, thus $(m-1)!\equiv0 \md{m}$.

  If $s=1$ then $m=p^\alpha$ with $\alpha\ge2$ (because non-prime). When $\alpha=2$ we have $p<m$ and $2p<m$ because $m>4$. It results that $p$ and $2p$ are different factors in $(m-1)!$ and therefore $(m-1)!\equiv0 \md{m}$. When $\alpha>2$, we have $p<m$ and $p^{\alpha-1}<m$, and $p$ and $p^{\alpha-1}$ are different factors in product $(m-1)!$.

  Therefore $(m-1)!\equiv0 \md{m}$ and the lemma is proved for all cases.
\end{proof}

\begin{thm}\label{T2CP}
  Let $p$ be a natural number $p>4$. Then $p$ is prime if and only if
  \begin{equation}\label{CP2}
    (p-4)!\equiv(-1)^{\left[\frac{p}{3}\right]+1}\cdot\left[\frac{p+1}{6}\right] \md{p}~,
  \end{equation}
  where $[x]$ is the integer part of $x$, i.e. the largest integer less than or equal to $x$.
\end{thm}
\begin{proof}\

  \emph{Necessity}: $(p-4)!(p-3)(p-2)(p-1)\equiv-1\md{p}$ from Wilson\textquoteright{s} theorem \ref{TWilson1}, or $6(p-4)!\equiv1\md{p}$; $p$ being prime and greater than $4$, it results that $(6,p)=1$. It results that $p=6k\pm1$, with $k\in\Ns$.
  \begin{enumerate}
    \item If $p=6k-1$, then $6\mid(p+1)$ and $(6,p)=1$, and dividing the congruence $6(p-4)!\equiv p+1\md{p}$, which is equivalent with the initial one, by $6$ we obtain:
        \[
         (p-4)!\equiv\frac{p+1}{6}\equiv(-1)^{\left[\frac{p}{3}\right]+1}\cdot\left[\frac{p+1}{6}\right]\md{p}~.
        \]
    \item If $p=6k+1$, then $6\mid(1-p)$ and $(6,p)=1$, and dividing the congruence $6(p-4)!\equiv1-p\md{p}$ , which is equivalent to the initial one, by $6$ it results:
        \[
         (p-4)!\equiv\frac{1-p}{6}\equiv-k\equiv(-1)^{\left[\frac{p}{3}\right]+1}\cdot\left[\frac{p+1}{6}\right]\md{p}~.
        \]
  \end{enumerate}
  \emph{Sufficiency}: We must prove that $p$ is prime. First of all we\textquoteright{l}l show that $p\neq \mathcal{M}6$. Let\textquoteright{s} suppose by absurd that $p=6k$, $k\in\Ns$. By substituting in the congruence from hypothesis, it results
  $(6k-4)!\equiv-k\md{6k}$. From the inequality $6k-5\ge k$ for $k\in\Ns$, it results that $k\mid(6k-5)!$. From $22\mid(6k-4)$,
  it results that $2k\mid(6k-5)!(6k-4)$. Therefore $2k\mid(6k-4)!$ and $2k\mid6k$, it results (conform with the congruencies\textquoteright{} property), \citep[pp. 9-26]{Popovici1973}, that $2k\mid(-k)$, which is not true; and therefore $p\neq \mathcal{M}6$.

  From $(p-1)(p-2)(p-3)\equiv-6\md{p}$ by multiplying it with the initial congruence it results that:
  \[
   (p-1)!\equiv(-1)^{\left[\frac{p}{3}\right]}\cdot6\cdot\left[\frac{p+1}{6}\right]\md{p}~.
  \]

  Let\textquoteright{s} consider lemma \ref{Lemma1CP}, for $p>4$ we have:
  \[
   (p-1)!\equiv\left\{
   \begin{array}{rcl}
      0 & \md{p} & \textnormal{if $p$ is not prime}; \\
     -1 & \md{p} & \textnormal{if $p$ is prime};
   \end{array}\right.
  \]
  \begin{enumerate}
    \item If $p=6k+2\Rightarrow(p-1)!\equiv6k\not\equiv0\md{p}$~.
    \item If $p=6k+3\Rightarrow(p-1)!\equiv-6k\not\equiv0\md{p}$~.
    \item If $p=6k+4\Rightarrow(p-1)!\equiv-6k\not\equiv0\md{p}$~.
  \end{enumerate}
  Thus $p\neq\mathcal{M}6+r$ with $r\in\set{0,2,3,4}$. It results that $p$ is of the form: $p=6k\pm1$, $k\in\Ns$ and then we have:
  $(p-1)!\equiv-1\md{p}$, which means that $p$ is prime.
\end{proof}

\begin{prog} The primality criterion (\ref{CP2}), given by Theorem \ref{T2CP} can be implemented in Mathcad as follows:
  \begin{tabbing}
    $CSP2(p):=$\=\vline\ $return\ -1\ if\ p<5\vee p\neq trunc(p)$\\
    \>\vline\ $m\leftarrow trunc\left(\dfrac{p}{3}\right)+1$\\
    \>\vline\ $n\leftarrow trunc\left(\dfrac{p+1}{6}\right)$\\
    \>\vline\ $return\ 1\ if\ mod\left[(p-4)!-(-1)^m\cdot n,p\right]\textbf{=}0$\\
    \>\vline\ $return\ 0\ otherwise$
  \end{tabbing}
\end{prog}
The call of this criterion using the symbolic computation is:
\[
 \begin{array}{lcr}
   CSP2(4) & \rightarrow & -1~,\\
   CSP2(5.5) & \rightarrow & -1~,\\
   CSP2(61) & \rightarrow & 1~, \\
   CSP2(87) & \rightarrow & 0~, \\
   CSP2(127) & \rightarrow & 1~, \\
   CSP2(1057) & \rightarrow & 0~,
 \end{array}
\]
where $1$ indicates that the number is prime, $0$ the contrary and $-1$ error, i.e. $p<5$ or $p$ is not integer.

\begin{thm}\label{T3CP}
  If $p$ is a natural number $p\ge5$, then $p$ is prime if and only if
  \begin{equation}\label{CP3}
    (p-5)!\equiv r\cdot h+\frac{r^2-1}{24}\md{p}~,
  \end{equation}
  where
  \[
   h=\left[\frac{p}{24}\right]\ \ \textnormal{and}\ \ r=p-24h~.
  \]
\end{thm}
\begin{proof}\

  \emph{Necessity}: if $p$ is prime, it results that:
  \[
   (p-5)!(p-4)(p-3)(p-2)(p-1)\equiv-1\md{p}
  \]
  or
  \[
   24(p-5)!\equiv-1\md{p}~.
  \]
  But $p$ could be written as $p=24h+r$, with $r\in\set{1,5,7,11,13,17,19,23}$, because it is prime. It can be easily verified that
  \[
   \frac{r^2-1}{24}\in\set{0,1,2,5,7,12,15,22}\subset\Int~.
  \]
  \[
   24(p-5)!\equiv-1+r(24h+r)\equiv24rh+r^2-1\md{p}
  \]
  Because $(24,p)=1$ and $24\mid(r^2-1)$ we can divide the congruence by $24$, obtaining:
  \[
   (p-5)!\equiv rh+\frac{r^2-1}{24}\md{p}~.
  \]

  \emph{Sufficiency}: $p$ can be written $p=24h+r$ , $h,r\in\Na$, $0\le r<24$. Multiplying the congruence $(p-4)(p-3)(p-2)(p-1)\equiv24\md{p}$ with the initial one, we obtain: $(p-1)!\equiv r(24h+r)-1\equiv-1\md{p}$.
\end{proof}

\begin{prog}
  The implementation of the primality criterion (\ref{CP3}) given by Theorem \ref{T3CP} is:
\begin{tabbing}
  $CSP3(p):=$\=\vline\ $return\ -1\ if\ p<5\vee p\neq trunc(p)$\\
  \>\vline\ $h\leftarrow trunc\left(\dfrac{p}{24}\right)$\\
  \>\vline\ $r\leftarrow p-24\cdot h$\\
  \>\vline\ $return\ 1\ if\ mod\left[(p-5)!-\left(r\cdot h+\dfrac{r^2-1}{24}\right),p\right]\textbf{=}0$\\
  \>\vline\ $return\ 0\ otherwise$
\end{tabbing}
\end{prog}

The call of this criterion using the symbolic computation is:
\[
 \begin{array}{lcr}
   CSP3(4) & \rightarrow & -1~,\\
   CSP3(5.5) & \rightarrow & -1~,\\
   CSP3(61) & \rightarrow & 1~, \\
   CSP3(87) & \rightarrow & 0~, \\
   CSP3(127) & \rightarrow & 1~, \\
   CSP3(1057) & \rightarrow & 0~,
 \end{array}
\]
where $1$ indicates that the number is prime, $0$ the contrary and $-1$ error, i.e. $p<5$ or $p$ is not integer.

\begin{thm}
  Let\textquoteright{s} consider $p=(k-1)!\cdot h\pm1$, with $k>2$ a natural number. Then $p$ is prime if and only if
  \begin{equation}\label{CP4}
    (p-k)!\equiv(-1)^{k+\left[\frac{p}{h}\right]+1}\cdot h\md{p}~.
  \end{equation}
\end{thm}
\begin{proof}\

  \emph{Necessity}: If $p$ is prime then, according to Wilson\textquoteright{s} theorem \ref{TWilson1}, results that $(p-1)!\equiv-1\md{p}\Leftrightarrow(-1)^{k-1}(p-k)!(k-1)!\equiv-1\md{p}\Leftrightarrow(p-k)!(k-1)!\equiv(-1)^k\md{p}$.
  We have:
  \begin{equation}\label{Rel1T4CP}
    ((k-1)!,p)=1~.
  \end{equation}
  \begin{enumerate}
    \item[(A)] $p=(k-1)!\cdot h-1$.
        \begin{enumerate}
          \item  $k$ is an even number $\Rightarrow(p-k)!(k-1)!\equiv1+ p\md{p}$, and because of the relation (\ref{Rel1T4CP}) and $(k-1)!\mid(1+p)$, by dividing with $(k-1)!$ we have: $(p-k)!\equiv h\md{p}$.
          \item $k$ is an odd number $\Rightarrow(p-k)!(k-1)!\equiv-1-p\md{p}$, and because of the relation (\ref{Rel1T4CP}) and
          $(k-1)!\mid(-1-p)$, by dividing with $(k-1)!$ we have: $(p-k)!\equiv-h\md{p}$.
        \end{enumerate}
    \item[(B)] $p=(k-1)!\cdot h+1$.
        \begin{enumerate}
          \item $k$ is an even number $\Rightarrow(p-k)!(k-1)!\equiv1-p\md{p}$, and because $(k-1)!\mid(1-p)$ and of the relation (\ref{Rel1T4CP}), by dividing with $(k-1)!$ we have: $(p-k)!\equiv-h\md{p}$.
          \item $k$ is an odd number $\Rightarrow(p-k)!(k-1)!\equiv-1+ p\md{p}$, and because $(k-1)!\mid(-1+p)$ and of the relation (\ref{Rel1T4CP}), by dividing with $(k-1)!$ we have $(p-k)!\equiv h\md{p}$.
        \end{enumerate}
  \end{enumerate}
  Putting together all these cases, we obtain: if p is prime, $p=(k-1)!\cdot h\pm1$, with $k>2$ and $h\in\Ns$, then the relation (\ref{CP4}) is true.

  \emph{Sufficiency}: Multiplying the relation (\ref{CP4}) by $(k-1)!$ it results that:
  \[
   (p-k)!(k-1)!\equiv(k-1)!\cdot h\cdot(-1)^{\left[\frac{p}{h}\right]+1}\cdot(-1)^k\md{p}~.
  \]
  Analyzing separately each of these cases:
  \begin{enumerate}
    \item[(A)] $p=(k-1)!\cdot h-1$ and
    \item[(B)] $p=(k-1)!\cdot h+1$, we obtain for both, the congruence:
  \end{enumerate}
  \[
   (p-k)!(k-1)!\equiv(-1)^k\md{p}
  \]
  which is equivalent (as we showed it at the beginning of this proof) with $(p-1)!\equiv-1\md{p}$ and it results that $p$ is prime.
\end{proof}

\begin{prog}\label{ProgCSP4}
  The implementation of the primality criterion given by (\ref{CP4}) using the symbolic computation is:
  \begin{tabbing}
    $CSP4(p):=$\=\vline\ $return\ -1\ if\ p<2\vee p\neq trunc(p)$\\
    \>\vline\ $return\ 1\ if\ p\textbf{=}2$\\
    \>\vline\ $h\leftarrow0$\\
    \>\vline\ $j\leftarrow3$\\
    \>\vline\ $w$\=$hile\ (j-1)!\le p+1$\\
    \>\vline\ \>\vline\ $if$\=$\ mod[p+1,(j-1)!]\textbf{=}0$\\
    \>\vline\ \>\vline\ \>\vline\ $h\leftarrow\dfrac{p+1}{(j-1)!}$\\
    \>\vline\ \>\vline\ \>\vline\ $k\leftarrow j$\\
    \>\vline\ \>\vline\ $j\leftarrow j+1$\\
    \>\vline\ $return\ 0\ if\ h\textbf{=}0$\\
    \>\vline\ $return\ 1\ if\ mod\left[(p-k)!-(-1)^{k+trunc\left(\frac{p}{h}\right)+1}\cdot h,p\right]\textbf{=}0$\\
    \>\vline\ $return\ 0\ otherwise$
  \end{tabbing}
\end{prog}

The test of the program \ref{ProgCSP4} has been done as follows. We know that we have 24 odd prime numbers up to $99$. Vector $I$
of odd numbers from $3$ to $99$ was generated with the sequence:
\[
 ORIGIN:=2\ \ \ j:=2..50\ \ \ I_j:=2\cdot j-1~.
\]
For each component of vector $I$ program \ref{ProgCSP4} was called and the result was assigned to vector $v$. As the values of vector
$v$ are $1$ for prime numbers and $0$ for non-prime numbers, it follows that the sum of the components of vector $v$ will give the
number of prime numbers. If this sum is $24$, it follows that criterion \ref{CP4} and program \ref{ProgCSP4} are correct for all
odd numbers up to $99$.
\[
 v_j:=CSP4(I_j)\ \ \ \sum v=24~.
\]

The call of this criterion using the symbolic computation is:
\[
 \begin{array}{lcr}
   CSP4(1) & \rightarrow & -1~,\\
   CSP4(2) & \rightarrow & 1~,\\
   CSP4(3.5) & \rightarrow & -1~,\\
   CSP4(47) & \rightarrow & 1~, \\
   CSP4(147) & \rightarrow & 0~, \\
   CSP4(149) & \rightarrow & 1~, \\
   CSP4(150) & \rightarrow & 0~.
 \end{array}
\]
where $1$ indicates that the number is prime, $0$ the contrary and $-1$ error, i.e. $p<2$ or $p$ is not integer.

\section{Decomposition product of prime factors}

The factorization problem of integers is: given a positive integer $n$ let find its prime factors, which means the pairs
$(p_i,\alpha_i)$, $p_i$ are distinct prime numbers and $\alpha_i$ are positive integers, such that $n=\desp[\alpha]{s}$.

In the Number Theory, the factorization of integers is the process of finding the divisors of a given composite number. This seems to
be a trivial problem, but for huge numbers there doesn't exist any efficient factorization algorithm, the most efficient algorithm
has an exponential complexity, relative to the numbers of digits. Hence, a factorization experiment of a number containing 200
decimal digits was successfully ended only after several months. In this experiment were used 80 computers Opteron processor of 2.2
GHz, connected in a network of Gigabit type.

Many algorithms were conceived to determine the prime factors of a given number. They can vary very little in sophistication and
complexity. It is very difficult to build a general algorithm for this "complex" computing problem, such that any additional
information about the number or its factors can be often useful to save an important amount of time.

The algorithms for factorizing an integer $n$ can be divided into two types:
\begin{enumerate}
  \item General algorithms. Algorithm trial division is:
      \begin{enumerate}
        \item[INPUT] $n\in\Na$, $n\ge3$, $n$ is neither prime nor perfect square and $b\in\Ns$.
        \item[OUTPUT] Smallest prime factor $n$ if it is $<b$, otherwise failure.
        \item[1.] for $q\in\set{2,3,5,7,11,\ldots,p}$, $p\le b$.
        \begin{enumerate}
          \item[1.1.] Return $q$ if $mod(n,q)=0$.
          \item[1.2.] Otherwise continue.
        \end{enumerate}
        \item[2.] Return failure.
      \end{enumerate}
      The number of steps for trial division is $O\sim(\sqrt[3]{n})$ most of the time, \citep{Myasnikov+Backes2008}.
  \item Special algorithms. Their execution time depends on the special properties of number $n$, as, for example, the size of the greatest prime factor. This category includes:
      \begin{enumerate}
        \item The $rho$ algorithm of Pollard\index{Pollard J. M.},
        \citep{Pollard1975,Brent1980,WeissteinPollardrhoFactorizationMethod};
        \item The $p-1$ algorithm of Pollard\index{Pollard J. M.}
        \citep{Cormen+Leiserson+Rivest+Stein2001};
        \item The algorithm based on elliptic curves
        \citep{Galbraith2012};
        \item The Pollard-Strassen\index{Strassen V.} algorithm
        \citep{Pomerance1982,Hardy+Muskat+Williams1990,WeissteinPrimeFactorizationAlgorithms},
        which was proved to be the fastest factorization algorithm.
            For $a\in\Na$ we denote $\overline{a}=mod(a,n)$. Let $c$, $1\le c\le\sqrt{n}$
            \[
             F(x)=(x+1)(x+2)\cdots(x+c)\in\Int[x]
            \]
            and
            \[
             f(x)=\overline{F}(x)\in\Int_{N}[x]
            \]
            then
            \[
             \overline{c^2!}=\prod_{k=0}^cf(\overline{k\cdot c})~.
            \]
            This algorithm has the following steps:
            \begin{enumerate}
              \item[INPUT] $n\in\Na$, $n\ge3$, $n$ is neither prime, nor perfect square, $b\in\Ns$.
              \item[OUTPUT] If the smallest prime factor of $n$ is $<b$, otherwise failure.
              \item[1.] Compute $c\leftarrow\left\lceil\sqrt{b}\right\rceil$.
              \item[2.] Determine the coefficients of polynomial $f\in\Int_N[x]$:
              \[
               f(x)=\prod_{k=1}^c(x+\overline{k})~.
              \]
              \item[3.] Compute $g_k\in\set{0,1,\ldots,n-1}$
              such that
              \[
               g_k=mod\big(f(\overline{k\cdot c}),n\big)\ \ \textnormal{for}\ \ 0\le k<c~.
              \]
              \begin{enumerate}
                \item[4.1.] If $\gcd(g_k,n)=1$ for $\forall k\in\set{0,1,\ldots c-1}$ then return failure.
                \item[4.2.] On the contrary, let
                \[
                 k=\min\set{0\le k<c;\ \gcd(g_k,n)>1}~.
                \]
              \end{enumerate}
              \item[5.] Return $\min\set{d\ ;\ mod(n,d)=0,\ k\cdot c+1\le d\le k\cdot c+c}$.
            \end{enumerate}
        Pollard\textquoteright{s}\index{Pollard J. M.} and Strassen\textquoteright{s}\index{Strassen V.} integer factoring algorithm works correctly and uses $O(M(\sqrt{b})M(log(n))(log(b)+log(log(n)))$ word operations, where $M$ is the time for multiplication, and space for $O(\sqrt{b}\cdot log(n))$ words, \citep{Myasnikov+Backes2008,Gathen+Gerhard2013}.
        \begin{prog}\label{Horner} This program uses the Schema of \cite{Horner1819}\index{Horner, W. G.}, the fastest algorithm to compute the value of a polynomial, \citep{CiraMetodeNumerice2005}. The input variables are the vector $a$ which defines the polynomial $a_mx^m+a_{m-1}x^{m-1}+\ldots+a_1x+a_0$ and $x$.
          \begin{tabbing}
            $Horner(a,x):=$\=\vline\ $m\leftarrow last(a)$\\
            \>\vline\ $f\leftarrow a_m$\\
            \>\vline\ $f$\=$or\ k\in m-1..0$\\
            \>\vline\ \>\ $f\leftarrow f\cdot x+a_k$\\
            \>\vline\ $return\ f$
          \end{tabbing}
        \end{prog}
        \begin{prog}\label{Prod} Computation program for the coefficients of the polynomial $(x+1)(x+2)\cdots(x+c)$.
          \begin{tabbing}
             $Prod(c):=$\=\vline\ $v\leftarrow(1\ 1)^\textrm{T}$\\
             \>\vline\ $return\ v\ if\ c\textbf{=}1$\\
             \>\vline\ $f$\=$or\ k\in 2..c$\\
             \>\vline\ \>\ $v\leftarrow stack(0,v)+stack(k\cdot v,0)$\\
             \>\vline\ $return\ v$
           \end{tabbing}
        \end{prog}
        \begin{prog} This program applies the Pollard-Strassen\index{Pollard J. M.}\index{Strassen V.} algorithm for finding the smallest prime factor, not greater than $b$, of number $n$.
          \begin{tabbing}
            $PS(n,b):=$\=\vline\ $c\leftarrow ceil\big(\sqrt{b}\big)$\\
            \>\vline\ $C\leftarrow Prod(c)$\\
            \>\vline\ $f$\=$or\ k\in 0..c$\\
            \>\vline\ \>\ $a_k\leftarrow mod(C_k,n)$\\
            \>\vline\ $f$\=$or\ k\in 0..c-1$\\
            \>\vline\ \>\vline\ $g_k\leftarrow mod(Horner(a,mod(k\cdot c,n)),n)$\\
            \>\vline\ \>\vline\ $d_k\leftarrow \gcd(g_k,n)$\\
            \>\vline\ \>\vline\ $return\ d_k\ if\ d_k>1$\\
            \>\vline\ $return\ "Fail"$
          \end{tabbing}
          This program calls programs \ref{Prod} and \ref{Horner}. The program was tested by means of following examples:
          \[
           n:=143\ \ \ b:=floor(\sqrt{n})=11\ \ \ PS(n,b)=11
          \]
          \[
           n:=667\ \ \ b:=floor(\sqrt{n})=25\ \ \ PS(n,b)=23
          \]
          \[
           n:=4009\ \ \ b:=floor(\sqrt{n})=63\ \ \ PS(n,b)=19
          \]
          \[
           n:=10097\ \ \ b:=floor(\sqrt{n})=100\ \ \ PS(n,b)=23
          \]
        \end{prog}
      \end{enumerate}
\end{enumerate}

\subsection{Direct factorization}\label{FD}

The most easy method to find factors is the so-called "direct search". In this method, all possible factors are systematically
tested using a division of testings to see if they really divide the given number. This algorithm is useful only for small numbers
($<10^6$).
\begin{prog}\label{ProgFa}
  The program of factorization of a natural number. This program uss the vector of prime numbers $p$ generated by the Sieve of
  Eratosthenes\index{Eratosthenes}, the fastest program that generates prime numbers up to a given limit. The Call of the Sieve of Eratosthenes,  the program \ref{CECira}, is made using the sequence:
  \[
  \begin{array}{l}
    L:=2\cdot10^7\ \ \ t_0=time(0) \ \ \ p:=CEPb(L)\ \ \ \ t_1=time(1)\\ \\
    (t_1-t_0)s=5.064s\ \ \ last(p)=1270607\ \ \ p_{last(p)}=19999999
  \end{array}
  \]
\begin{tabbing}
    $Fa(m):=$\=\ \vline\ $return\ ("m="\ m\ ">that\ the\ last\ p^2")\ if\ m>(p_{last(p)})^2$\\
    \>\ \vline\ $j\leftarrow1$\\
    \>\ \vline\ $k\leftarrow0$\\
    \>\ \vline\ $f\leftarrow(1\ 1)$\\
    \>\ \vline\ $w$\=$hile\ m\ge p_j$\\
    \>\ \vline\ \>\ \vline\ $i$\=$f\ \mod(m,p_j)\textbf{=}0$\\
    \>\ \vline\ \>\ \vline\ \>\ \vline\ $k\leftarrow k+1$\\
    \>\ \vline\ \>\ \vline\ \>\ \vline\ $m\leftarrow\dfrac{m}{p_j}$\\
    \>\ \vline\ \>\ \vline\ $otherwise$\\
    \>\ \vline\ \>\ \vline\ \>\ \vline\ $f\leftarrow stack[f,(p_j,k)]\ if\ k>0$\\
    \>\ \vline\ \>\ \vline\ \>\ \vline\ $j\leftarrow j+1$\\
    \>\ \vline\ \>\ \vline\ \>\ \vline\ $k\leftarrow0$\\
    \>\ \vline\ $f\leftarrow stack[f,(p_j,k)]\ if\ k>0$\\
    \>\ \vline\ $return\ submatrix(f,2,rows(f),1,2)$
  \end{tabbing}
\end{prog}
We give a remark that can simplify the primality test in some cases.
\begin{obs}
  If $p$ is the first prime factor of $n$ and $p^2>q=\frac{n}{p}$, then $q$ is a prime number. Hence, the decomposition in prime factors of number $n$ is $p\cdot q$.
  \begin{proof}
    Let us suppose that $q$ is a composite number, which means $q=a\cdot b$~. As $p$ is the first prime factor of $n$, it follows that $a,b>p$~. Hence, a contradiction is obtained, namely $n=p\cdot q=p\cdot a\cdot b>p^3>n$. Therefore, $q$ is a prime
    number. Hence, the decomposition in prime factors of $n$ is $p\cdot q$~.
  \end{proof}
\end{obs}

Examples of factorization:
\[
 Fa(2^{36}-1)=\left(
                \begin{array}{rr}
                    3 & 3 \\
                    5 & 1 \\
                    7 & 1 \\
                   13 & 1 \\
                   19 & 1 \\
                   37 & 1 \\
                   73 & 1 \\
                  109 & 1
                \end{array}
              \right)~,\ \ \
 Fa(3^{20}-1)=\left(
                \begin{array}{rr}
                     2 & 4 \\
                     5 & 2 \\
                    11 & 2 \\
                    61 & 1 \\
                  1181 & 1 \\
                \end{array}
              \right)~,
\]
\[
 Fa(11^7-1)=\left(
              \begin{array}{cc}
                    2 & 1 \\
                    5 & 1 \\
                   43 & 1 \\
                45319 & 1 \\
              \end{array}
            \right)~,\ \
 Fa(7^{11}-1)=\left(
              \begin{array}{cc}
                     2 & 1 \\
                     3 & 1 \\
                  1123 & 1 \\
                293459 & 1 \\
              \end{array}
            \right)~.
\]

\subsection{Other methods of factorization}

\begin{enumerate}
  \item The method of Fermat\index{Fermat P.} and the generalized method of Fermat are recommended
     for the case where $n$ has two factors of similar extension. For a natural number $n$, two integers are searched, $x$ and $y$ such that $n=x^2-y^2$. Then $n=(x-y)(x+y)$ and we obtain a first decomposition of $n$, where one factor is very small. This factorization may be inefficient if the factors $a$ and $b$ do not have close values, it is possible to be necessary $\frac{n+1}{2}-\sqrt{n}$ verifications for testing if the generated numbers are squares. In this situation we can use a generalized Fermat ethod \index{Fermat P.} which applies better in such cases, \citep{Dan2005}.
  \item The method of Euler\index{Euler L.} of factorization can be applied for odd numbers $n$
     that can be written as the sum of two squares in two different ways
      \[
       n=a^2+b^2=c^2+d^2
      \]
     where $a,c$ are even and $c,d$ odd.
  \item The method of Pollard-$rho$ or the Monte Carlo method. We suppose that a great number $n$ is composite. The simplest test,
     much more faster than the method of divisions, is due to \cite{Pollard1975}\index{Pollard J. M.}. It is also called the $rho$ method, or the Monte Carlo method. This test has a special purpose, used to find the small prime factors for a composite number.

     For the Pollard-$rho$ algorithm, a certain function $f:\Int_n\to\Int_n$ is chosen, such that, for example, its values to be determined easily. Hence, $f$ is usually a polynomial function; for example $f(x)=x^2+a$, where $a\neq\set{0,2}$.

     Pollard-$rho$ algorithm with the chosen function $f(x)=x^2+1$, is:
      \begin{enumerate}
        \item[INPUT:] A composite number $n>2$, which is not the power of a prime number.
        \item[OUTPUT:] A proper divisor of $n$.
        \item[1.] Let $a\leftarrow2$ and $b\leftarrow2$.
        \item[2.] For $k=1,2,\ldots$, run:
            \begin{enumerate}
              \item[2.1] Compute $a\leftarrow mod(f(a),n)$ and $b\leftarrow mod(f(b),n)$.
              \item[2.2] Compute $d=(a-b,n)$.
              \item[2.3] If $1<d<n$, then return $d$ proper divisor of $n$ and stop the
              algorithm.
              \item[2.4] If $d=n$, then return the message "\emph{Failure, another function must be chosen}".
            \end{enumerate}
      \end{enumerate}
  \item Pollard $p-1$ method. This method has a special purpose, being used for the factorization of numbers $n$ which have a prime factor $p$
  with the property that $p-1$ is a product of prime factors smaller than a relative small number. Pollard $p-1$ algorithm is:
      \begin{enumerate}
        \item[INPUT:] A composite number $n>2$, which is not the power of a prime number.
        \item[OUTPUT:] A proper divisor of $n$.
        \item[1.] Choose a margin $B$.
        \item[2.] Choose, randomly, an $a$, $2\le a\le n-1$ and compute $d=(a,n)$. If $d\ge2$, return $d$ proper divisor of $n$ and stop the algorithm.
        \item[3.] For every prim $q\le B$, run:
            \begin{enumerate}
              \item[3.1.] Compute $\ell=\left[\ln(n)/\ln(q)\right]$.
              \item[3.2.] Compute $a\leftarrow mod\left(a^{q^\ell},n\right)$.
            \end{enumerate}
        \item[4.] Compute $d=(n-1,a)$.
        \item[5.] If $d=1$ or $d=n$, then return the message "\emph{E'sec}", else, return $d$ proper divisor of $n$ and stop the algorithm.
      \end{enumerate}
\end{enumerate}

\section{Counting of the prime numbers}

\subsection{Program of counting of the prime numbers}

If we have the list of prime numbers, we can, obviously, write a program to count them up to a given number $x\in\Ns$. We read the
file of prime numbers available on the site \citep{CaldwellThePrimePages} and we assign it to vector $p$ with the sequence:
\[
 p:=READPRN("\ldots\backslash Prime.prn")
\]
\[
 last(p)=6\cdot10^6\ \ p_{last(p)}=104395301~.
\]
Command $last(p)$ states that vector $p$ contains the first $6\cdot10^6$ prime numbers, and the last prime number of vector
$p$ is $104395301$.
\begin{prog} Program for counting the prime numbers up to a natural number $x$.
  \begin{tabbing}
    $\pi(x):=$\=\ \vline\ $f$\=$or\ n\in1..last(p)$\\
    \>\ \vline\ \>\ $if$\=$\ p_n\ge x$\\
    \>\ \vline\ \>\ \>\vline\ $return\ n-1\ if\ p_n>x$\\
    \>\ \vline\ \>\ \>\vline\ $return\ n\ otherwise$\\
  \end{tabbing}
  For example, let us count the prime numbers up to $10^n$, for $n=1,2,\ldots,8$. This counting can be made by using  following commands:
  \[
   n:=1..8\ \ \ \ \pi(10^n)=\left(\begin{array}{c}
                                    4 \\
                                    25 \\
                                    168 \\
                                    1229 \\
                                    9592 \\
                                    78498 \\
                                    664579 \\
                                    5761455 \\
                                  \end{array}\right)~.
  \]
\end{prog}

\subsection{Formula of counting of the prime numbers}
By means of Smarandache's\index{Smarandache F.} function we obtain a formula for counting the prime numbers less or equal to $n$,
\citep{Seagull1995}.
\begin{thm}
If $n$ is an integer $\ge4$, then
\begin{equation}\label{FormulaPi}
  \pi(n)=-1+\sum_{k=2}^n\left\lfloor\frac{\eta(k)}{k}\right\rfloor
\end{equation}
\end{thm}
\begin{proof}
Knowing the $\eta(n)$ has the property that if $p>4$ then $\eta(p)=p$ if only if $p$ is prime, and $\eta(n)<n$ for any $n$,
and $\eta(4)=4$ (the only exception from the first rule), then
\[
 \left\lfloor\frac{\eta(k)}{k}\right\rfloor=\left\{\begin{array}{l}
                                              1~,\ \textnormal{if $k$ is prime} \\
                                              0~,\ \textnormal{if $k$ is not prime}
                                            \end{array}\right.~.
\]
We easily find an exact formula for the number of primes less than or equal to $n$.
\end{proof}

If we read the file $\eta.prn$ and attribute to the values of the vector $\eta$ the sequence
\[
 ORIGIN:=1\ \ \ \ \eta:=READPRN("\ldots\backslash\eta.prn")
\]
then formula (\ref{FormulaPi}) becomes:
\begin{equation}\label{FormulaPiEta}
  \pi(n):=\left\{\begin{array}{ll}
                   return\ "Error\ n<1\ or\ n \notin\Int"\ if\ n<1\vee n\neq trunc(n)\\ \\
                   return\ -1+\displaystyle\sum_{k=2}^n\left\lfloor\frac{\eta_k}{k}\right\rfloor\ if\ n\ge4 \\ \\
                   return\ 2\ if\ n=3 \\ \\
                   return\ 1\ if\ n=2 \\ \\
                   return\ 0\ if\ n=1
                 \end{array}\right.
\end{equation}

Using this formula, the number of primes up to $n=10$, $n=10^2$, \ldots, $n=10^6$ has been determined and the obtained results are:
\[
 \pi(10)=4\ \ \ \ \pi(10^2)=25\ \ \ \ \pi(10^3)=168\ \ \ \ \pi(10^4)=1229
\]
\[
 \pi(10^5)=9592\ \ \ \ \pi(10^6)=78498~.
\]

\chapter{Smarandache\textquoteright{s} function $\eta$}

The function that associates to each natural number $n$ the smallest natural number $m$ which has the property that $m!$ is a
multiple of $n$ was considered for the first time by \cite{Lucas1883}\index{Lucas F. E. A.}. Other authors who have
considered this function in their works are: \cite{Neuberg1887}\index{Neuberg J.}, \cite{Kempner1918}\index{Kempner A. J.}. This function was rediscovered by \cite{Smarandache1980}\index{Smarandache F.}. The function is denoted by Smarandache with $S$ or $\eta$, and on the site \emph{Wolfram MathWorld}, \citep{Sondow+Weisstein}, it is denoted $\mu$. In this volume we have adopted the notation $\eta$ found in the paper \citep{Smarandache1999a}.

Therefore, function $\eta:\Ns\to\Ns$, $\eta(n)=m$, where $m$ is the smallest natural that has the property that $n$ divides $m!$,
(or $m!$ is a multiple of $n$) is known in the literature as \emph{Smarandache\textquoteright{s} function}.
\begin{figure}[h]
  \centering
  \includegraphics[scale=0.75]{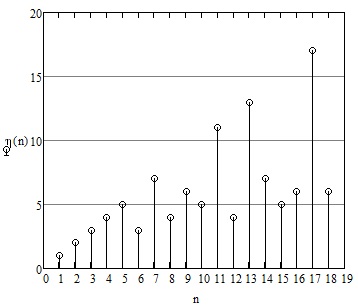}\\
  \caption{$\eta$ function}
\end{figure}
The values of the function, for $n=1,2,\ldots,18$, are: $1$, $2$, $3$, $4$, $5$, $3$, $7$, $4$, $6$, $5$, $11$, $4$, $13$, $7$, $5$,
$6$, $17$, $6$ obtained by means of an algorithm that results from the definition of function $\eta$, as follows:

\begin{prog}\label{ProgramEta}\
  \begin{tabbing}
   $\eta(n)=$\=\ \vline\ $f$\=$or\ m=1..n$\\
   \>\ \vline\ \>\ $return\ m\ if\ mod(m!,n)\textbf{=}0$
 \end{tabbing}
\end{prog}

The program \ref{ProgramEta} can not be used for $n\ge19$ as the numbers $19!$, $20!$, \ldots has much more than $17$ decimal
digits and in the classic computation approach (without an arithmetics of random precisions \citep{Uznanski2014}) will be generated errors due to the classic representation in the memory of computers.

\cite{Kempner1918}\index{Kempner A. J.}, gave an algorithm to compute $\eta(n)$ using the classic factorization $\desp[\alpha]{s}$ with prime numbers of $n$, and the generalized base $(\alpha_i)_{[p_i]}$, for $i=\overline{1,s}$. Partial solutions for algorithms that compute $\eta(n)$ were given previously by Lucas\index{Lucas F. E. A.} and Neuberg\index{Neuberg J.}, \citep{Sondow+Weisstein}~.

We give Kempner\textquoteright{s}\index{Kempner A. J.} algorithm, that computes Smarandache\textquoteright{s} function $\eta$. At
the beginning, let us define the recursive sequence
\[
 a_{j+1}=p\cdot a_j+1\ \ \textnormal{with}\ \ j=1,2,\ldots\ \ \textnormal{and}\ \ a_1=1~,
\]
where $p$ is a prime number. This sequence represents the generalized base of $p$. As $a_2=p+1$, $a_3=p^2+p+1$, \ldots we
can prove by induction that
\[
 a_j=1+p+\ldots+p^{j-1}=\frac{p^j-1}{p-1}\ \ \textnormal{for}\ \ \forall j\ge2~.
\]
The value of $\nu$, such that $a_\nu\le\alpha<a_{\nu+1}$, is given by the formula
\begin{equation}
  \nu=\lfloor\log_p\big(1+\alpha(p-1)\big)\rfloor~,
\end{equation}
where $\lfloor\cdot\rfloor$ is the function \emph{lower integer part}. With the help Euclid\textquoteright{s}\index{Euclid} algorithm we can determine the unique sequences $\kappa_i$ and $r_i$, as follows
\begin{eqnarray}
  \alpha &=& \kappa_\nu\cdot a_\nu+r_\nu~, \\
  r_\nu &=& \kappa_{\nu-1}\cdot a_{\nu-1}+r_{\nu-1}~, \\
        & \vdots &  \nonumber\\
  r_{\nu-(\lambda-2)} &=& \kappa_{\nu-(\lambda-1)}\cdot a_{\nu-(\lambda-1)}+r_{\nu-(\lambda-1)}~, \\
  r_{\nu-(\lambda-1)} &=& \kappa_{\nu-\lambda}\cdot a_{\nu-\lambda}~.
\end{eqnarray}
which means, until the rest $r_{\nu-\lambda}=0$. At each step $\kappa_i$ is the integer part of the ratio $r_i/a_i$ and $r_i$ is
the rest of the division. For example, for the first step we have $\kappa_\nu=\lfloor\alpha/a_\nu\rfloor$ and $r_\nu=\alpha-\kappa_\nu\cdot a_\nu$. Then, we have
\begin{equation}
  \eta(p^\alpha)=(p-1)\alpha+\sum_{i=\nu}^\lambda\kappa_i~.
\end{equation}
In general, for
\begin{equation}\label{DescompunereaLuiN}
  n=\desp[\alpha]{s}~,
\end{equation}
the value of function $\eta$ is given by the formula:
\begin{equation}\label{L06}
  \eta(n)=\max\set{\eta(p_1^{\alpha_1}),\eta(p_2^{\alpha_2}),\ldots,\eta(p_s^{\alpha_s})}~,
\end{equation}
formula due to \cite{Kempner1918}\index{Kempner A. J.}.

\begin{rem}\label{RemarcaSmarandache}
If $n\in\Ns$ has the decomposition in product of prime numbers (\ref{DescompunereaLuiN}), where $p_i$ are prime numbers such that
$p_1<p_2<\ldots<p_s$, and $s\ge1$, then the Kempner\textquoteright{s}\index{Kempner A. J.} algorithm of
computing function $\eta$ is
\begin{multline}\label{FormulaAlgoritmuluiKempner}
   \eta(n)=\max\left\{p_1\cdot\left(\alpha_{1_{[p_1]}}\right)_{(p_1)},\ p_2\cdot\left(\alpha_{2_{[p_2]}}\right)_{(p_2)},\right.\\     \left.\ldots,\ p_s\cdot\left(\alpha_{s_{[p_s]}}\right)_{(p_s)}\right\}~,
\end{multline}
where $\left(\alpha_{[p]}\right)_{(p)}$ means that $\alpha$ is "\emph{written}" in the generalized base $p$ (denoted $\alpha_{[p]}$) and is "\emph{read}" in base $p$ (denoted $\beta_{(p)}$, where $\beta=\alpha_{[p]}$), \citep[p.39]{Smarandache1999}.
\end{rem}

On the site \emph{The On-Line Encyclopedia of Integer Sequences}, \citep[A002034]{Sloane2014}, is given a list of 1000 values of function $\eta$, due to T. D. Noe\index{Noe T. D.}. We remark that on the site \emph{The On-Line Encyclopedia of Integer Sequences}
it is defined $\eta(1)=1$, while \cite{Ashbacher1995}\index{Ashbacher C.} and \cite{Russo2000}\index{Russo F. A.} consider that $\eta(1)=0$.

\section{The properties of function $\eta$}

The greater values for function $\eta$ are obtained for $4$ and for the prime numbers and are $\eta(p)=p$, \citep[A046022]{Sloane2014}.

The smallest values for $n$ are, \citep[A094371]{Sloane2014}:
\[
  \frac{\eta(n)}{n}=1,\ \frac{1}{2},\ \frac{1}{3},\ \frac{1}{4},\ \frac{1}{6},\ \frac{1}{8},\ \frac{1}{12},\ \frac{3}{40},\ \frac{1}{15},\ \frac{1}{16},\ \frac{1}{24},\ \frac{1}{30},\ \ldots
\]
for the values \citep[A002034]{Sloane2014}:
\[
 n=1,\ 6,\ 12,\ 20,\ 24,\ 40,\ 60,\ 80,\ 90,\ 112,\ 120,\ 180,\ \ldots
\]

This function is important because it characterize the prime numbers -- by the following fundamental property.
\begin{thm}\label{TPrimalitateEta}
Let $p$ be an integer $>4$. Then $p$ is prime if and only if $\eta(p)=p$.
\end{thm}
\begin{proof}
  See \citep[p. 31]{Smarandache1999}.
\end{proof}
Hence, the fixed points of this function are prime numbers (to which $4$ is added). Due to this property, function $\eta$ is used
as a primality test.

The formula (\ref{FormulaAlgoritmuluiKempner}) used to compute Smarandache\textquoteright{s}\index{Smarandache F.} function $\eta$ allows us to give several values of the function for particular numbers $n$
\begin{equation}\label{L01}
  \begin{array}{l}
     \eta(1)=1~,\\
     \eta(n!)=n~, \\
     \eta(p)=p~,\\
     \eta(p_1\cdot p_2\cdots p_s)=p_1\cdot p_2\cdots p_s~, \\
     \eta(p^\alpha)=p\cdot\alpha~,
   \end{array}
\end{equation}
where $p$ and $p_i$ are distinct prime numbers with $p_1<p_2<\ldots<p_s$ and $\alpha\le p$, \citep{Kempner1918}.

Other special numbers for which we can give the values of function $\eta$ are:
\begin{equation}
  \eta(P_p)=M_p~,
\end{equation}
where $P_2=6$, $P_3=28$, $P_5=496$, $P_7=8128$, \ldots, \citep[A000396]{Sloane2014}, are the perfect numbers corresponding to the prime numbers $2$, $3$, $5$, $7$, \ldots, and $M_2=2^2-1=3$, $M_3=2^3-1=7$, $M_5=2^5-1=31$, $M_7=2^7-1=127$, \ldots, \citep[A000668]{Sloane2014}, are Mersenne\index{Mersenne M.} numbers corresponding to the prime numbers prime $2$, $3$, $5$, $7$, \ldots, see the papers \citep{Ashbacher1997,Ruiz1999a}.

Function $\eta$ has following properties:
\begin{equation}
 \eta(n_1\cdot n_2)\le\eta(n_1)+\eta(n_2)~,
\end{equation}
\begin{equation}
  \max\set{\eta(n_1),\eta(n_2)}\le\eta(n_1\cdot n_2)\le\eta(n_1)\cdot\eta(n_2)~,
\end{equation}
where $n_1,n_2\in\Ns$.

If $p$ is a prime number and $\alpha\ge2$ an integer, then
\begin{equation}
  \eta\left(p^{p^\alpha}\right)=p^{\alpha+1}-p^\alpha+p~.
\end{equation}
This result is due to \cite{Ruiz1999b}\index{Ruiz S. M.}.

The case $p^\alpha$ with $\alpha>p$ is more complicated to which applies the Kempner\textquoteright{s}\index{Kempner A. J.}
algorithm.

According to formula (\ref{L06}), it results that for all $n\in\Ns$ we have
\begin{equation}
  \eta(n)\ge g_{pf}(n)~,
\end{equation}
where $g_{pf}(n)$ is the function \emph{the greatest prime factor} of $n$. Therefore, $\eta(n)$ can be computed by determining
$g_{pf}(n)$ and testing if $n\mid g_{pf}(n)!$~. If $n\mid g_{pf}(n)!$ then $\eta(n)=g_{pf}(n)$, if $n\nmid g_{pf}(n)!$ then
$\eta(n)>g_{pf}(n)$ and we call Kempner\textquoteright{s}\index{Kempner A. J.} algorithm.

Let $A\subset\Ns$ a set of strictly nondecreasing positive integers. We denote by $A(n)$ the number of numbers of the set $A$
up to $n$. In what follows we give the definition of the density of a set of natural numbers, \citep[p. 199]{Guy1994}.
\begin{defn}
 We name \emph{density of a set} $A\subset\Ns$, the number
 \[
  \lim_{n\to\infty}\frac{A(n)}{n},
 \]
 if it exists.
\end{defn}
For example, \emph{the density of the set of the even natural numbers} is $1/2$ because
\[
 \lim_{n\to\infty}\frac{\lfloor\frac{n}{2}\rfloor}{n}=\frac{1}{2}~.
\]

The set of numbers $n\in\Ns$ with the property that $n\nmid g_{pf}(n)!$ has \emph{zero density}, such as
\cite{Erdos1991}\index{Erd\"{o}s P.} supposed and \cite{Kastanas1994}\index{Kastanas I.} proved.

The first numbers with the property that $n\nmid g_{pf}(n)!$ are: $4$, $8$, $12$, $16$, $18$, $24$, $25$, $27$, $32$, $36$, $45$,
$48$, $49$, $50$, \ldots \citep[A057109]{Sloane2014}.

If we denote by $N(x)$ the number of numbers $n\in\Ns$ which have the properties $2\le n\le x$ and $n\nmid g_{pf}(n)!$, then we
obtain the estimation
\begin{equation}
  N(x)\ll x\cdot e^{-\frac{1}{4}\sqrt{\ln(x)}}~,
\end{equation}
due to \cite{Akbik1999}\index{Akbik S.}, where the notation $f(x)\ll g(x)$ means that there exists $c\in\Real_+$ such that
$\abs{f(x)}<c\cdot\abs{g(x)}$, $\forall x$. As
\[
 \frac{N(x)}{x}\ll e^{-\frac{1}{4}\sqrt{\ln(x)}}~,\ \ \textnormal{and}\ \ \lim_{x\to\infty}e^{-\frac{1}{4}\sqrt{\ln(x)}}=0~,
\]
we may say that the set has \emph{zero density}.

This result was later improved by \cite{Ford1999}\index{Ford K.} and by the authors \cite{DeKoninck+Doyon2003}
\index{De Koninck J.-M.}\index{Doyon N.}. Ford\index{Ford K.} proposed following asymptotic formula:
\begin{equation}\label{L08}
  N(x)\approx\frac{\sqrt{\pi}\big(1+\ln(2)\big)}{\sqrt[4]{2^3}}\sqrt[4]{\ln(x)^3\ln(\ln(x))^3}\cdot x^{1-\frac{1}{u_0}}\cdot\rho(u_0)~,
\end{equation}
where $\rho(u)$ is the Dickman's\index{Dickman K.} function, \citep{Dickman1930,WeissteinDickmanFunction}, and $u_0$ is defined
implicitly by equation
\begin{equation}
  \ln(x)=u_0\left(x^\frac{1}{u_0^2}-1\right)~.
\end{equation}

The estimation made in formula (\ref{L08}) was rectified by \cite{Ivic2003}\index{Ivi\v{c} A.}, in two consecutive postings,
\begin{equation}
  N(x)=x\left(2+O\left(\sqrt{\frac{\ln(\ln(x))}{\ln(x)}}\right)\right)
  \int_2^x\rho\left(\frac{\ln(x)}{\ln(t)}\right)\frac{\ln(t)}{t^2}dt~,
\end{equation}
or, by means of elementary functions
\begin{equation}
  N(x)=x\cdot\exp\left[-\sqrt{2\ln(x)\ln(\ln(x))}\left(1+O\left(\frac{\ln(\ln(\ln(x)))}{\ln(\ln(x))}\right)\right)\right]~.
\end{equation}

\cite{Tutescu1996}\index{Tutescu L.} assumed that function $\eta$ does not have the same value for two consecutive values of the
argument, which means
\[
 \forall\ n\in\Ns~,\ \ \eta(n)\neq \eta(n+1)~.
\]

Weisstein\index{Weisstein E. W.} published, on the 3rd of March 2004, \citep{Sondow+Weisstein}, the fact that he has verified this
result, by means of a program, up to $10 ^ 9$.

Several numbers $n\in\Ns$ may have the same value for $\eta$ function, i.e. function $\eta$ is not injective. In table
\ref{Tabel01} we emphasize numbers $n$ for which $\eta(n)=k$.
\begin{table}
  \centering
  \begin{tabular}{|l|l|}
    \hline
    $k$ & $n$ for which we have $\eta(n)=k$ \\ \hline
    1 & 1 \\ \hline
    2 & 2 \\ \hline
    3 & 3, 6 \\ \hline
    4 & 4, 8, 12, 24 \\ \hline
    5 & 5, 10, 15, 20, 30, 40, 60, 120 \\ \hline
    6 & 6, 16, 18, 36, 45, 48, 72, 80, 90, 144, 240, 360, 720 \\
    \hline
  \end{tabular}
  \caption{Values $n$ for which $\eta(n)=k$}\label{Tabel01}
\end{table}

Let $a(k)$ be the smallest inverse of $\eta(n)$, i.e. the smallest $n$ for which $\eta(n)=k$. Then $a(k)$ is given by
\begin{multline}
  a(k)=g_{pf}(n)^\omega~,\\
  \textnormal{where}\ \ \omega=\sum_{i=1}^L\left\lfloor\frac{n-1}{g_{pf}(k)^i}\right\rfloor~,
  \ \ \textnormal{and}\ \ L=\left\lfloor\log_{g_{pf}(k)}(n-1)\right\rfloor~.
\end{multline}
This result was published by \cite{Sondow2005}\index{Sondow J.}. For $k=1,2,\ldots$, function $a(k)$ is equal to $1$, $2$, $3$,
$4$, $5$, $9$, $7$, $32$, $27$, $25$, $11$, $243$, \ldots as seen in \citep[A046021]{Sloane2014}.

Some values of $\eta(n)$ function are obtained for huge values of $n$. An increasing sequence of great values of $a(k)$ is $1$, $2$,
$3$, $4$, $5$, $9$, $32$, $243$, $4096$, $59049$, $177147$, $134217728$, $31381059609$, \ldots, (see \citep[A092233]{Sloane2014}), the sequence that corresponds to $n=$ $1$, $2$, $3$, $4$, $5$, $6$, $8$, $12$, $24$, $27$, $32$, $48$, $54$, \ldots (see \citep[A092232]{Sloane2014}).

In the process of finding number $n$ for which $\eta(n)=k$, we remark that $n$ is a divisor of $\eta(n)!$ but not of $\eta(n-1)!$. Therefore, in order to find all the numbers $n$ for each $\eta(n)$ has a value, we consider all $n$ with $\eta(n)=k$, where $n$ is in the set of all divisors of $k!$ minus the divisors of $(k-1)!$. In particular, $b(k)$ of $n$ for which $\eta(n)=k$, for $k>1$ is
\begin{equation}\label{L14}
  b(k)=\sigma_0(k!)-\sigma_0\big((k-1)!\big)~,
\end{equation}
where $\sigma_0(m)$ is the \emph{divisors counting} function of $m$. Hence, the number of integers $n$ with $\eta(n)=1$, $2$, \ldots are given by the sequence $1$, $1$, $2$, $4$, $8$, $14$, $30$, $36$, $64$, $110$, \ldots (see \citep[A038024]{Sloane2014}).

Particularly, equation (\ref{L14}) shows that the inverse of Smarandache's function, $a(n)$, exists always, as for each $n$ there exist an $m$ such that $\eta(n)=m$ \big(i.e. the smallest $a(n)$\big), because
\[
 \sigma_0(n!)-\sigma_0\big((n-1)!\big)>0~,
\]
for $n>1$.

\cite{Sondow2006}\index{Sondow J.} showed that $\eta(n)$ appears unexpectedly in an irrational limit for $e$ and it suppose that
the inequality $n^2<\eta(n)!$ holds for "\emph{almost every} $n$", where "\emph{almost every} $n$" means the set of integers minus an
exception set of \emph{zero density}. The exception set is $2$, $3$, $6$, $8$, $12$, $15$, $20$, $24$, $30$, $36$, $40$, $45$,
$48$, $60$, $72$, $80$, \ldots, (see \citep[A122378]{Sloane2014}).

As equation $g_{pf}(n)=\eta(n)$, considered by \cite{Erdos1991,Kastanas1994} for "\emph{almost every} $n$", is equivalent with the inequality $n^2<g_{pf}(n)!$ for "\emph{almost every} $n$" of Sondow's\index{Sondow J.} conjecture, it results that the conjecture of Erd\"{o}s\index{Erd\"{o}s P.} and Kastanas\index{Kastanas I.} is equivalent with Sondow's conjectures. The exception set, in this case, of \emph{zero density} is: $2$, $3$, $4$, $6$, $8$, $9$, $12$, $15$, $16$, $18$, $20$, $24$, $25$, $27$, $30$, $32$, $36$, \ldots, (see
\citep[A122380]{Sloane2014}).

D. Wilson\index{Wilson D.}, underlines, in the case where
\begin{equation}
  I(n,p)=\frac{n-\Sigma(n,p)}{p-1}~,
\end{equation}
is a power of $p$ prime in $n!$, where $\Sigma(n,p)$ is the function \emph{sum in base} $p$ of $n$, then following relation
\begin{equation}
  a(n)=\min_{p\mid n}p^{I(n-1,p)+1}~,
\end{equation}
hold, where the minimum is reached for every prime number $p$ that divides $n$. This minimum seems to be always attainable when
$p=g_{pf}(n)$.

\section{Programs for Kempner\textquoteright{s} algorithm}

In this section we emphasize Kempner\textquoteright{s}\index{Kempner A. J.} algorithm by means of the Mathcad programs necessary to the algorithm.
\begin{prog}
  The function that counts the digits in base $p$ of $n$
  \begin{tabbing}
    $ncb(n,p):=$\=\ \vline\ $return\ ceil(\log(n,p))\ if\ n>1$\\
    \>\ \vline\ $return\ 1\ otherwise$
  \end{tabbing}
  where the utility function Mathcad $ceil(\cdot)$ is the upper integer part function.
\end{prog}

\begin{prog}
  The program that generates the generalized base $p$ \big(denoted by Smarandache\index{Smarandache F.} $[p]$\big) for a number with
  $m$ digits
  \begin{tabbing}
    $a(p,m):=$\=\ \vline\ $f$\=$or\ i\in1..m$\\
    \>\ \vline\ \>\ $a_i\leftarrow\dfrac{p^i-1}{p-1}$\\
    \>\ \vline\ $return\ a$
  \end{tabbing}
\end{prog}

\begin{prog}
  The program that generates the base $p$ \big(denoted by Smarandache\index{Smarandache F.} $(p)$\big) to write number $\alpha$
  \begin{tabbing}
    $b(\alpha,p):=$\=\ \vline\ $return\ (1)\ if\ p=1$\\
    \>\ \vline\ $f$\=$or\ i\in1..ncb(\alpha,p)$\\
    \>\ \vline\ \>\ $b_i\leftarrow p^{i-1}$\\
    \>\ \vline\ $return\ b$
  \end{tabbing}
\end{prog}

\begin{prog}
  The program of finding the digits of the generalized base $[p]$ for number $n$
  \begin{tabbing}
    $Nbg(n,p):=$\=\ \vline\ $m\leftarrow ncb(n,p)$\\
    \>\ \vline\ $a\leftarrow a(p,m)$\\
    \>\ \vline\ $return\ (1)\ if\ m\textbf{=}0$\\
    \>\ \vline\ $f$\=$or\ i\in m..1$\\
    \>\ \vline\ \>\ \vline\ $c_i\leftarrow trunc\left(\dfrac{n}{a_i}\right)$\\
    \>\ \vline\ \>\ \vline\ $n\leftarrow \mod(n,a_i)$\\
    \>\ \vline\ $return\ c$
  \end{tabbing}
\end{prog}

\begin{prog}\label{ProgEta}
  The program for Smarandache\textquoteright{s} function
  \begin{tabbing}
    $\eta(n):=$\=\ \vline\ $return\ "Error\ n\ is\ not\ integer"\ if\ n\neq trunc(n)$\\
    \>\ \vline\ $return\ "Error\ n<1"\ if\ n<1$\\
    \>\ \vline\ $return\ (1)\ if\ n\textbf{=}1$\\
    \>\ \vline\ $f\leftarrow Fa(n)$\\
    \>\ \vline\ $p\leftarrow f^{\langle1\rangle}$\\
    \>\ \vline\ $\alpha\leftarrow f^{\langle2\rangle}$\\
    \>\ \vline\ $f$\=$or\ k=1..rows(p)$\\
    \>\ \vline\ \>\ $\eta_k\leftarrow p_k\cdot Nbg(\alpha_k,p_k)\cdot b(\alpha_k,p_k)$\\
    \>\ \vline\ $return\ \max(\eta)$
  \end{tabbing}
  This program calls the program $Fa(n)$ of factorization by prime numbers. The program uses Smarandache\textquoteright{s} remark
  \ref{RemarcaSmarandache} relative to Kempner\textquoteright{s} algorithm.
\end{prog}

If we introduce number $n$ as a product of prime numbers $p_i$ raised at power $\alpha_i$ ($\alpha_i$ integer $\ge0$) it will result a variant of the program \ref{ProgEta} which can compute the values of $\eta$ function for huge numbers.
\begin{prog}
  The program for computing the values of $\eta$ function for huge numbers.
  \begin{tabbing}
    $\eta_s(f):=$\=\ \vline\ $Prop\leftarrow"Matrix\ f\ is\ not\ at\ least\ one\ row\ with\ two\ columns"$\\
    \>\ \vline\ $return\ Prop\ if\ \neg(IsArray(f)\wedge rows(f)\ge1\wedge cols(f)\textbf{=}2)$\\
    \>\ \vline\ $p\leftarrow f^{\langle1\rangle}$\\
    \>\ \vline\ $\alpha\leftarrow f^{\langle2\rangle}$\\
    \>\ \vline\ $f$\=$or\ k=1..rows(p)$\\
    \>\ \vline\ \>\ $\eta_k\leftarrow p_k\cdot Nbg(\alpha_k,p_k)\cdot b(\alpha_k,p_k)$\\
    \>\ \vline\ $return\ \max(\eta)$
  \end{tabbing}
\end{prog}

\begin{prog}\label{ProgEK} Program that generates the matrix that contains al values $n$ for which $\eta(n)=k$.
  \begin{tabbing}
    $EK(N):=$\=\ \vline\ $f$\=$or\ n\in2..N$\\
    \>\ \vline\ \>\ $K_n\leftarrow\eta(n)$\\
    \>\ \vline\ $f$\=$or\ q\in2..max(K)$\\
    \>\ \vline\ \>\ \vline\ $j\leftarrow1$\\
    \>\ \vline\ \>\ \vline\ $f$\=$or\ k\in2..N$\\
    \>\ \vline\ \>\ \vline\ \>\ $i$\=$f\ K_k\textbf{=}q$\\
    \>\ \vline\ \>\ \vline\ \>\ \>\ \vline\ $EK_{q,j}\leftarrow k$\\
    \>\ \vline\ \>\ \vline\ \>\ \>\ \vline\ $j\leftarrow j+1$\\
    \>\ \vline\ $return\ EK$
  \end{tabbing}
\end{prog}

\subsection{Applications}
Several applications for the given programs are given in what follows:
\begin{enumerate}
  \item \emph{Compute the values of $\eta$ function for numbers $n_1$, $n_2$ given as products of prime numbers raised at a positive integer power}.
    \begin{enumerate}
      \item Let $n_1=2^{12}\cdot7^{13}\cdot11^{23}=895430243255334261261034$, then
      \[
       n_1:=\left(
          \begin{array}{cc}
            2 & 12 \\
            7 & 13 \\
            11 & 23 \\
          \end{array}
        \right)\ \ \eta_s(n_1)=242~~.
      \]
      \item Let $n_2=3^{33}\cdot5^{55}\cdot7^{51}\cdot11^{11}=$
      \[
        12589532854288041315477068297463914028063002~,
      \]
      then
      \[
       n_2:=\left(
          \begin{array}{cc}
            3 & 33 \\
            5 & 55 \\
            7 & 51 \\
            11 & 11 \\
          \end{array}
        \right)\ \ \ \eta_s(n_2)=315~,
      \]
    \end{enumerate}
  \item \emph{Find the number whose factorial ends in 1000 zeros}.

  To answer this question we remark that for $n=10^{1000}$ we have $\eta(n)!=M\cdot10^{1000}$ and this $\eta(n)$ is the smallest natural number   whose factorial ends in 1000 zeros. We have $\eta(n)=\eta(2^{1000}\cdot5^{1000})$, then
  \[
   n:=\left(
        \begin{array}{cc}
          2 & 1000 \\
          5 & 1000 \\
        \end{array}
      \right)\ \ \ \eta_s(n)=4005
  \]
  and, hence, the number whose factorial ends in 1000 zeros is 4005. The numbers $4006$, $4007$, $4008$, $4009$ have also the required property, but 4010 has the property that its factorial has 1001 zeros.
  \item \emph{Determine all values $n$ for which $\eta(n)=7$}.

  With the help of program \ref{ProgEK} we can generate the matrix that contains all values $n$ for which $\eta(n)=k$. Line $7$ of the matrix is the answer to the problem:
  \begin{multline*}
    7=\eta(n), \textnormal{for }n=7,14,21,28,35,42,56,63,70,84,105,\\
    112,126,140,168,210,252,280,315,336,420,504,560,630,\\
    840,1008,1260,1680,2520,5040~.
  \end{multline*}
\end{enumerate}

\subsection{Calculation the of values $\eta$ function}

Generating the file $\eta.prn$ once and reading the generated file in Mathcad documents that determine solutions of the Diophantine
equations lead to an important saving of the execution time for the program that searches the solutions.
\begin{prog}\label{ProgGenVfS}
  The program by means of which the file $\eta.prn$ is generated is:
  \begin{tabbing}
    $ValFS(N):=$\=\vline\ $\eta_1\leftarrow1$\\
    \>\vline\ $fo$\=$r\ n\in2..N$\\
    \>\vline\ \>\ $\eta_n\leftarrow\eta(n)$\\
    \>\vline\ $return\ \eta$\\
  \end{tabbing}
  This program calls the program \ref{ProgEta} which calculates the values of $\eta$ function. The generating sequence of the file $\eta.prn$ is:
  \[
   t_0:=time(0)\ \ WRITEPRN("\eta.prn"):=ValFS(10^6)\ \ t_1:=time(1)
  \]
  \[
   (t_1-t_0)sec="1:7:32.625"hhmmss
  \]
  The execution time of generating the values of $\eta$ function up to $10^6$ exceeds one hour on a computer with an Intel processor of 2.20GHz with RAM of 4.00GB (3.46GB usable).
\end{prog}

\begin{figure}
  \centering
  \includegraphics[scale=0.3]{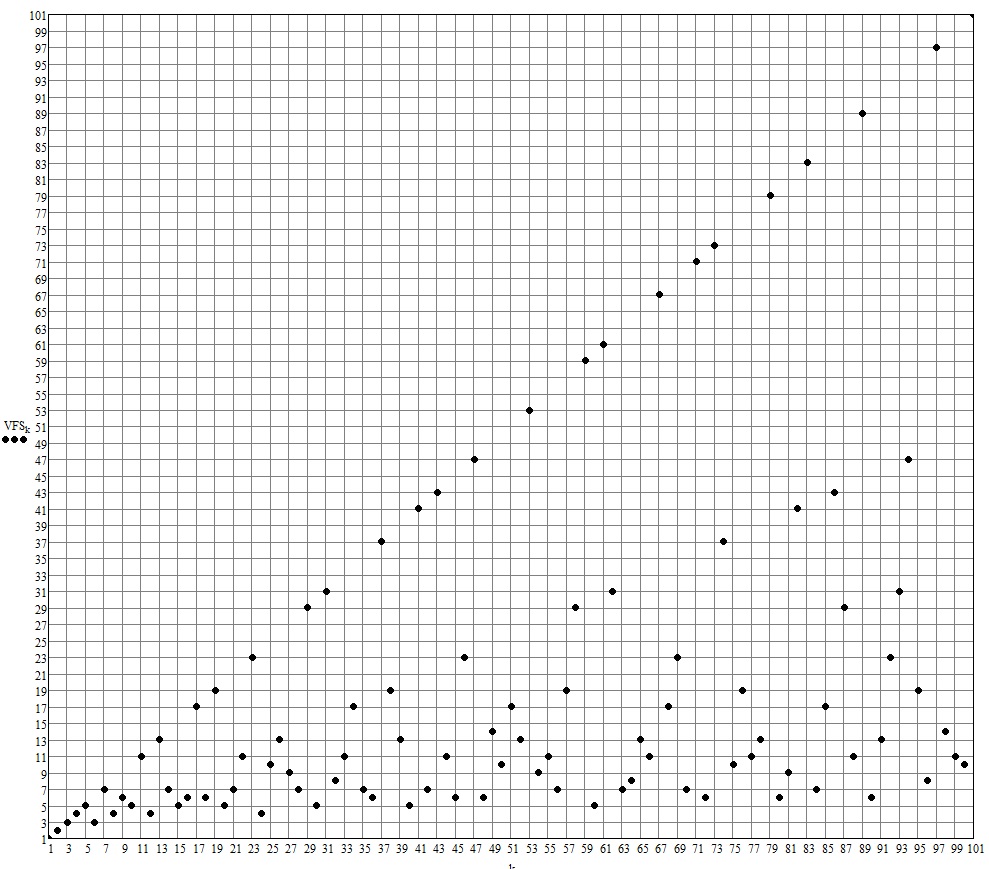}\\
  \caption{The graph of $\eta$ function on the set $\set{1,2,\ldots,101}$}
\end{figure}

We give the list of the first 400 and the last 256 values of $\eta$ function:\\
1,\,2,\ 3,\ 4,\ 5,\ 3,\ 7,\ 4,\ 6,\ 5,\ 11,\ 4,\ 13,\ 7,\ 5,\ 6,\
17,\ 6,\ 19,\ 5,\ 7,\ 11,\ 23,\ 4,\ 10,\ 13,\ 9,\ 7,\ 29,\ 5,\
31,\ 8,\ 11,\ 17,\ 7,\ 6,\ 37,\ 19,\ 13,\ 5,\ 41,\ 7,\ 43,\ 11,\
6,\ 23,\ 47,\ 6,\ 14,\ 10,\ 17,\ 13,\ 53,\ 9,\ 11,\ 7,\ 19,\ 29,\
59,\ 5,\ 61,\ 31,\ 7,\ 8,\ 13,\ 11,\ 67,\ 17,\ 23,\ 7,\ 71,\ 6,\
73,\ 37,\ 10,\ 19,\ 11,\ 13,\ 79,\ 6,\ 9,\ 41,\ 83,\ 7,\ 17,\ 43,\
29,\ 11,\ 89,\ 6,\ 13,\ 23,\ 31,\ 47,\ 19,\ 8,\ 97,\ 14,\ 11,\
10,\ 101,\ 17,\ 103,\ 13,\ 7,\ 53,\ 107,\ 9,\ 109,\ 11,\ 37,\ 7,\
113,\ 19,\ 23,\ 29,\ 13,\ 59,\ 17,\ 5,\ 22,\ 61,\ 41,\ 31,\ 15,\
7,\ 127,\ 8,\ 43,\ 13,\ 131,\ 11,\ 19,\ 67,\ 9,\ 17,\ 137,\ 23,\
139,\ 7,\ 47,\ 71,\ 13,\ 6,\ 29,\ 73,\ 14,\ 37,\ 149,\ 10,\ 151,\
19,\ 17,\ 11,\ 31,\ 13,\ 157,\ 79,\ 53,\ 8,\ 23,\ 9,\ 163,\ 41,\
11,\ 83,\ 167,\ 7,\ 26,\ 17,\ 19,\ 43,\ 173,\ 29,\ 10,\ 11,\ 59,\
89,\ 179,\ 6,\ 181,\ 13,\ 61,\ 23,\ 37,\ 31,\ 17,\ 47,\ 9,\ 19,\
191,\ 8,\ 193,\ 97,\ 13,\ 14,\ 197,\ 11,\ 199,\ 10,\ 67,\ 101,\
29,\ 17,\ 41,\ 103,\ 23,\ 13,\ 19,\ 7,\ 211,\ 53,\ 71,\ 107,\ 43,\
9,\ 31,\ 109,\ 73,\ 11,\ 17,\ 37,\ 223,\ 8,\ 10,\ 113,\ 227,\ 19,\
229,\ 23,\ 11,\ 29,\ 233,\ 13,\ 47,\ 59,\ 79,\ 17,\ 239,\ 6,\
241,\ 22,\ 12,\ 61,\ 14,\ 41,\ 19,\ 31,\ 83,\ 15,\ 251,\ 7,\ 23,\
127,\ 17,\ 10,\ 257,\ 43,\ 37,\ 13,\ 29,\ 131,\ 263,\ 11,\ 53,\
19,\ 89,\ 67,\ 269,\ 9,\ 271,\ 17,\ 13,\ 137,\ 11,\ 23,\ 277,\
139,\ 31,\ 7,\ 281,\ 47,\ 283,\ 71,\ 19,\ 13,\ 41,\ 8,\ 34,\ 29,\
97,\ 73,\ 293,\ 14,\ 59,\ 37,\ 11,\ 149,\ 23,\ 10,\ 43,\ 151,\
101,\ 19,\ 61,\ 17,\ 307,\ 11,\ 103,\ 31,\ 311,\ 13,\ 313,\ 157,\
7,\ 79,\ 317,\ 53,\ 29,\ 8,\ 107,\ 23,\ 19,\ 9,\ 13,\ 163,\ 109,\
41,\ 47,\ 11,\ 331,\ 83,\ 37,\ 167,\ 67,\ 7,\ 337,\ 26,\ 113,\
17,\ 31,\ 19,\ 21,\ 43,\ 23,\ 173,\ 347,\ 29,\ 349,\ 10,\ 13,\
11,\ 353,\ 59,\ 71,\ 89,\ 17,\ 179,\ 359,\ 6,\ 38,\ 181,\ 22,\
13,\ 73,\ 61,\ 367,\ 23,\ 41,\ 37,\ 53,\ 31,\ 373,\ 17,\ 15,\ 47,\
29,\ 9,\ 379,\ 19,\
127,\ 191,\ 383,\ 8,\ 11,\ 193,\ 43,\ 97,\ 389,\ 13,\ 23,\ 14,\ 131,\ 197,\ 79,\ 11,\ 397,\ 199,\ 19,\ 10,\\
\vdots\\
\begin{figure}[h]
  \centering
  \includegraphics[scale=0.3]{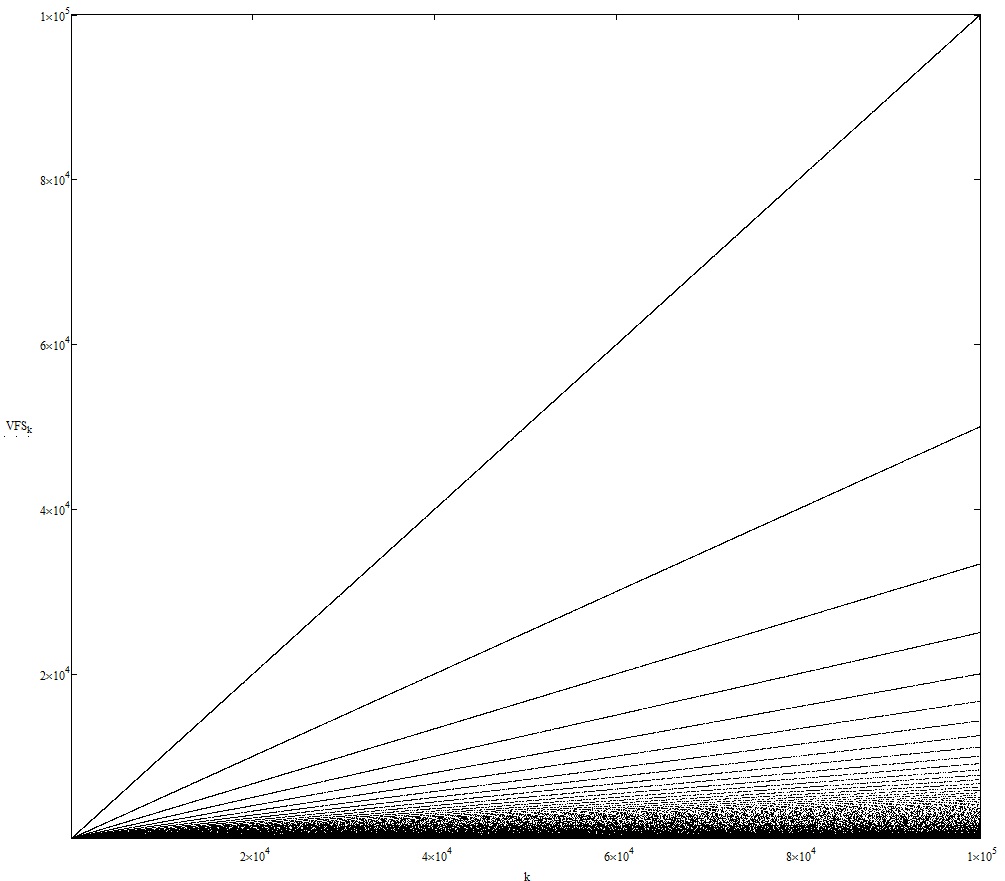}\\
  \caption{The graph of $\eta$ function on the set $\set{1,2,\ldots,10^5}$}
\end{figure}
607,\ 389,\ 1669,\ 83311,\ 569,\ 1193,\ 83,\ 7351,\ 239,\ 55541,\ 1451,\ 193,\ 14489,\ 26309,\ 1531,\ 127,\
2531,\ 1567,\ 2267,\ 2293,\ 999749,\ 43,\ 14081,\ 9613,\ 19603,\ 71411,\ 3389,\ 9257,\ 90887,\ 499879,\ 333253,\ 12497,\
7517,\ 166627,\ 999763,\ 10867,\ 1709,\ 499883,\ 5779,\ 541,\ 999769,\ 5881,\ 2207,\ 249943,\ 999773,\ 829,\ 197,\ 199,\
9007,\ 38453,\ 937,\ 877,\ 32251,\ 71413,\ 12343,\ 2659,\ 4877,\ 166631,\ 2557,\ 249947,\ 47609,\ 149,\ 76907,\ 131,\
23251,\ 499897,\ 66653,\ 5101,\ 2309,\ 593,\ 521,\ 4999,\ 10099,\ 1319,\ 613,\ 29,\ 199961,\ 499903,\ 333269,\ 107,\
999809,\ 46,\ 2053,\ 733,\ 333271,\ 229,\ 557,\ 41659,\ 10987,\ 317,\ 111091,\ 49991,\ 571,\ 2347,\ 8263,\ 113,\
13331,\ 137,\ 7193,\ 27773,\ 5987,\ 7691,\ 1013,\ 124979,\ 1499,\ 15149,\ 199967,\ 5813,\ 1949,\ 4201,\ 3533,\ 2083,\
14923,\ 3067,\ 827,\ 1381,\ 53,\ 55547,\ 2141,\ 124981,\ 333283,\ 19997,\ 443,\ 11903,\ 999853,\ 499927,\ 1307,\ 23,\
9173,\ 166643,\ 142837,\ 49993,\ 333287,\ 17239,\ 999863,\ 1543,\ 643,\ 71419,\ 739,\ 249967,\ 76913,\ 33329,\ 7873,\ 919,\
269,\ 16127,\ 421,\ 859,\ 1447,\ 967,\ 337,\ 3571,\ 13697,\ 4273,\ 999883,\ 249971,\ 349,\ 499943,\ 142841,\ 563,\
5347,\ 99989,\ 1277,\ 249973,\ 14083,\ 179,\ 15383,\ 124987,\ 333299,\ 9433,\ 3257,\ 101,\ 1721,\ 21737,\ 4219,\ 31247,\
6451,\ 9803,\ 999907,\ 67,\ 461,\ 99991,\ 90901,\ 683,\ 52627,\ 499957,\ 107,\ 547,\ 999917,\ 18517,\ 1009,\ 431,\
25639,\ 151,\ 449,\ 809,\ 47,\ 1901,\ 111103,\ 124991,\ 677,\ 33331,\ 999931,\ 223,\ 193,\ 38459,\ 881,\ 31,\
8849,\ 499969,\ 2237,\ 173,\ 1733,\ 166657,\ 20407,\ 1033,\ 823,\ 499973,\ 76919,\ 3623,\ 87,\ 2857,\ 331,\ 62497,\
999953,\ 761,\ 18181,\ 249989,\ 2801,\ 499979,\ 999959,\ 641,\ 999961,\ 13513,\ 811,\ 503,\ 4651,\ 139,\ 32257,\ 31249,\
333323,\ 277,\ 6211,\ 197,\ 97,\ 29411,\ 199,\ 523,\ 90907,\ 821,\ 999979,\ 49999,\ 111109,\ 6329,\ 999983,\ 251,\
28571,\ 38461,\ 1297,\ 22727,\ 52631,\ 271,\ 997,\ 2551,\ 333331,\ 21739,\ 199999,\ 499,\ 1321,\ 254,\ 37,\ 25~.

\chapter{Divisor functions $\sigma$}

\section{The divisor function $\sigma$}

The divisor function of order $k$ is given by the formula:
\begin{equation}
  \sigma_k(n)=\sum_{d|n}d^k~.
\end{equation}

For $k=0$, we have function $\sigma_0(n)$ (see figure \ref{FunctiaSigma0}) which counts the number of divisors of $n$. For example, $12$ has $1$, $2$, $3$, $4$, $6$, $12$ as divisors and, hence, their number is $6$.
\begin{figure}
  \centering
  \includegraphics[scale=0.5]{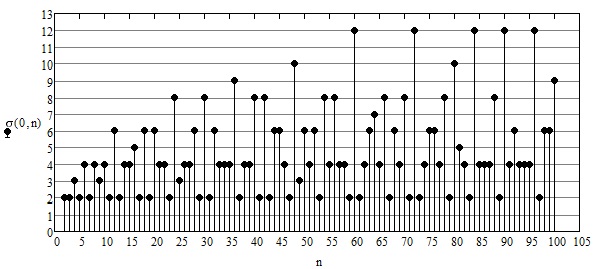}\\
  \caption{Function $\sigma_0(n)$}\label{FunctiaSigma0}
\end{figure}

For $k=1$ we have function $\sigma_1(n)$, (see figure \ref{FunctiaSigma1}) the function sum of the divisors of $n$. For example, $\sigma_1(12)=1+2+3+4+6+12=28$.

Function $\sigma_1(n)$, which gives the sum of the divisors of $n$, is usually written without index, i.e. $\sigma(n)$.
\begin{figure}
  \centering
  \includegraphics[scale=0.8]{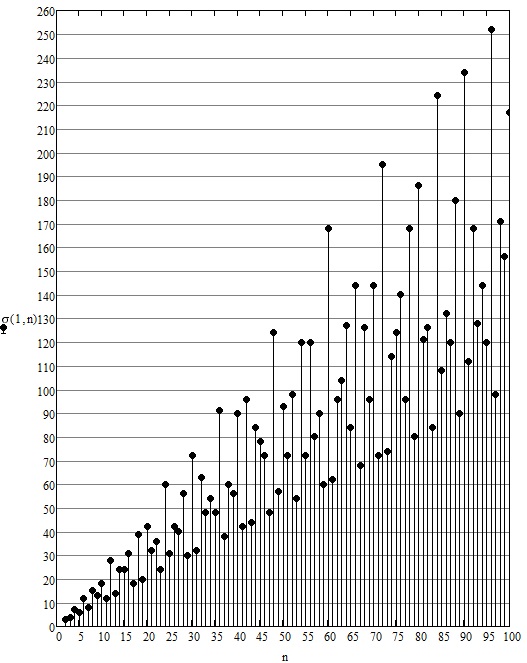}\\
  \caption{Function $\sigma(n)$}\label{FunctiaSigma1}
\end{figure}

The function sum of the proper divisors of $s$, \citep{Madachy1979}, and is given by the formula:
\begin{equation}
  s(n)=\sigma(n)-n~.
\end{equation}
For example, $s(12)=1+2+3+4+6=16$.

For $k=2$ function $\sigma_2(n)$ is the sum of the squares of the divisors. Fo examples, $\sigma_2(12)=1^2+2^2+3^2+4^2+6^2+12^2=210$.

Let $n$ be a natural number whose decomposition into prime factors is
\begin{equation}\label{Descompunerea}
  n=\desp[\alpha]{s}~,
\end{equation}
where $p_1<p_2<\ldots p_s$ are prime numbers, and $\alpha_j\in\Na$ for $j=1,2,\ldots,s$.

\begin{thm}\label{TeoremaSigmaNM}
  For two positive natural numbers $n$ and $m$, relative prime, $(n,m)=1$, then
  \begin{equation}\label{sigma(nm)=sigma(n)sigma(m)}
    \sigma(n\cdot m)=\sigma(n)\cdot\sigma(m)~.
  \end{equation}
\end{thm}
\begin{proof}
For each divisor $d_j$ of $n\cdot m$ we have $d_j=n_{j_1}\cdot m_{j_2}$, where $n_{j_1}|n$ and $m_{j_2}|m$. The numbers $1$,
$n_1$, $n_2$, \ldots, $n$ are the divisors of $n$ and $1$, $m_1$, $m_2$, \ldots, $m$ are the divisors of $m$. Then we have
\[
  \sigma(n)=1+n_1+n_2+\ldots+n
\]
and
\[
  \sigma(m)=1+m_1+m_2+\ldots+m~.
\]
According to the previous relations we can write $n_{j_1}(1+m_1+m_2+\ldots+m)=n_{j_1}\cdot\sigma(m)$. If we sum relative to $n_{j_1}$ it follows that $(1+n_1+n_2+\ldots+n)\sigma(m)=\sigma(n)\cdot\sigma(m)$, i.e. relation (\ref{sigma(nm)=sigma(n)sigma(m)}) holds.
\end{proof}

We owe to Berndt, \cite[p. 94]{Berndt1985}, \citep{WeissteinDivisorFunction}, next result.
\begin{thm}
  For every natural number $n$, whose decomposition into prime factors is \emph{(\ref{Descompunerea})}, we have that
  \begin{equation}\label{FormulaSigma}
    \sigma(n)=\prod_{j=1}^s\frac{p_j^{\alpha_j+1}-1}{p_j-1}~.
  \end{equation}
\end{thm}
\begin{proof}
According to relations (\ref{Descompunerea}) and (\ref{sigma(nm)=sigma(n)sigma(m)}) it follows that
\[
  \sigma(n)=\sigma(p_1^{\alpha_1})\sigma(p_2^{\alpha_2})\cdots\sigma(p_s^{\alpha_s})~.
\]
The divisors of $p_j^{\alpha_j}$ are $1$, $p_j$, $p_j^2$, \ldots, $p_j^{\alpha_j}$, therefore, the sum of the divisors of
$p_j^{\alpha_j}$ is
\[
 \sigma(p_j^{\alpha_j})=1+p_j+p_j^2+\ldots+p_j^{\alpha_j}=\frac{p_j^{\alpha_j+1}-1}{p_j-1}~.
\]
Hence, according to Proposition \ref{TeoremaSigmaNM} we can state that
\[
 \sigma(n)=\sigma(p_1^{\alpha_1}\cdot p_2^{\alpha_2}\cdots p_r^{\alpha_s})=\prod_{j=1}^s\frac{p_j^{\alpha_j+1}-1}{p_j-1}~.
\]
\end{proof}

By generalizing formula (\ref{FormulaSigma}) it results a relation for function $\sigma_k$. Function $\sigma_k:\Ns\to\Na$,
\citep{WeissteinDivisorFunction}, is given by the relations:
\begin{equation}
  \sigma_0(n)=\prod_{j=1}^s(\alpha_j+1)
\end{equation}
and
\begin{equation}
  \sigma_k(n)=\prod_{j=1}^s\frac{p_j^{(\alpha_j+1)k}-1}{p_j^k-1}~.
\end{equation}

\subsection{Computing the values of $\sigma_k$ functions}

\begin{prog}\label{ProgSigmak} The program for computing the values of function $\sigma_k$, for $k=0,1,\ldots$.
   \begin{tabbing}
     $\sigma(k,n):=$\=\vline\ $f\leftarrow Fa(n)$\\
     \>\vline\ $return\ \displaystyle\prod_{j=1}^{rows(f)}\left(f_{j,2}+1\right)\ if\ k\textbf{=}0$\\
     \>\vline\ $return\ \displaystyle\prod_{j=1}^{rows(f)}\dfrac{\left(f_{j,1}\right)^{(f_{j,2}+1)k}-1}{\left(f_{j,1}\right)^k-1}\ if\ k>0$\\
   \end{tabbing}
   The program \ref{ProgSigmak} calls the program \ref{ProgFa} of factorization in product of prime factors.
\end{prog}

\begin{prog}\label{ProgGenSigmak} The program by means of which the files $\sigma k.prn$ are generated is:
  \begin{tabbing}
    $G\sigma(k,N):=$\=\vline\ $f\varphi_1\leftarrow1$\\
    \>\vline\ $fo$\=$r\ n\in2..N$\\
    \>\vline\ \>\ $f\sigma_n\leftarrow \sigma(k,n)$\\
    \>\vline\ $return\ f\sigma$\\
  \end{tabbing}
  Obviously this program calls the program \ref{ProgSigmak} for computing the values of function $\sigma_k$. The sequence for generating the file $\sigma0.prn$ is:
  \[
   t_0:=time(0)\ \ WRITEPRN("\sigma0.prn"):=G\sigma(0,10^6)\ \ t_1:=time(1)
  \]
  \[
   (t_1-t_0)sec="0:0:2.833"hhmmss
  \]
  The sequences for generating the files $\sigma1.prn$ and $\sigma2.prn$ are similar.
\end{prog}

\section{$k$--hyperperfect numbers }

A number $n\in\Ns$ is called $k$--\emph{hyperperfect} if following identity
\[
  n=1+k\sum_{j}d_j
\]
holds, or
\[
 n=1+k(\sigma(n)-n-1)~,\ \ \ \ \ \ n=1+k(s(n)-1)~,
\]
where $\sigma(n)=\sigma_1(n)$ is the function that represents sum of the divisors $d_j$ of $n$ and $s(n)$ the sum of the proper
divisors of $n$, where $1<d_j<n$. After rearranging, we obtain relation
\[
  k\sigma(n)=(k+1)n+k-1
\]
which, if it is verified, means that $n$ is $k$--\emph{hyperperfect} number. If $k=1$ we say that $n$ is a \emph{perfect} number.

The conjecture of \cite{McCranie2000}\index{McCranie J. S.} states: \emph{the number $n=p^2q$ is a $k$--hyperperfect number if
$k\in2\Ns+1$, $p=\frac{3k+1}{2}$, $q=3k+4$, $p$ and $q$ prime numbers}.

If $p$ and $q$ are distinct odd prime numbers such that $k(p + q)=pq-1$ for a $k\in\Ns$, then $n=pq$ is $k$--\emph{hyperperfect}.

If $k\in\Ns$ and $p=k+1$ is prime, then, if there exists a $j\in\Ns$ such that $q =p^j-p+1$ prime, then $n=p^{j-1}q$ is
$k$--\emph{hyperperfect}.

The first $k$-\emph{hyperperfect} numbers are: 6, 21, 28, 301, 325, 496, 697, 1333, \ldots \citep[A034897]{Sloane2014}, which
correspond to the values of $k$: 1, 2, 1, 6, 3, 1, 12, 18, \ldots~. \cite{McCranie2000} gave the list of all \emph{hyperperfect} numbers up to $10^{11}$.

\chapter{Euler\textquoteright{s} totient function $\varphi$}

Euler\textquoteright{s}\index{Euler L.} totient function, denoted $\varphi$, counts the number of factors relative prime to $n$,
where $1$ is considered relative prime to every natural number. For example, factors relative prime to $36$ are $1$, $5$, $7$,
$11$, $13$, $17$, $19$, $23$, $25$, $29$, $31$, $35$ and, therefore, it results that $\varphi(36)=12$. By convention, we
have $\varphi(0)=1$.
\begin{prog} The program for computing the values of Euler\textquoteright{s}\index{Euler L.} totient function which applies the definition of the function is
  \begin{tabbing}
    $\varphi(n):=$\=\vline\ $return\ 1\ if\ n\textbf{=}0$\\
    \>\vline\ $j\leftarrow0$\\
    \>\vline\ $f$\=$or\ k\in1..n$\\
    \>\vline\ \>\ $j\leftarrow j+1\ if\ \gcd(k,n)\textbf{=}1$\\
    \>\vline\ $return\ j$
  \end{tabbing}
  This program can not be used for computing the values of Euler\textquoteright{s}\index{Euler L.} totient function for great numbers.
\end{prog}

Function $n-\varphi(n)$ is called cototient function.
\begin{figure}[h]
  \centering
  \includegraphics[scale=0.6]{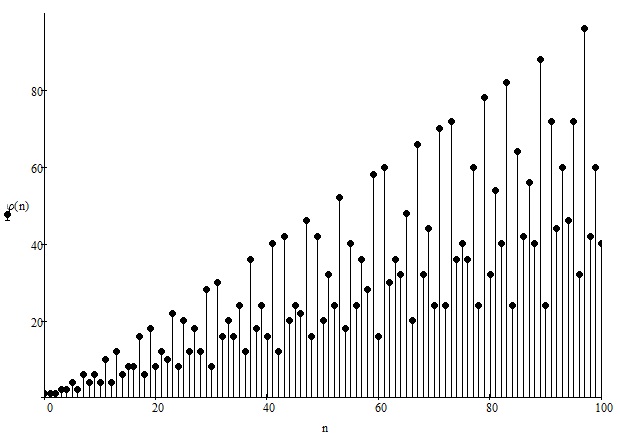}\\
  \caption{Euler\textquoteright{s} totient function}
\end{figure}

\section{The properties of function $\varphi$}

For $p$ prime number we have $\varphi(p)=p-1$, and
\[
 \varphi(p^\alpha)=p^{\alpha-1}(p-1)=p^\alpha\left(1-\frac{1}{p}\right)~.
\]

Let $m$ be a prime multiple of $p$. We define function $\varphi_p(m)$ which counts the positive integers $\le m$ which are not divisible by $p$. As $p$, $2p$, \ldots, $\frac{m}{p}p$ have common factor $p$, it follows that
\begin{equation}\label{phip}
  \varphi_p(m)=m-\frac{m}{p}=m\left(1-\frac{1}{p}\right)~.
\end{equation}
Let $q$ be another prime number that divides $m$, or let $m$ be a multiple of $q$. Then $q$, $2q$, \ldots, $\frac{m}{q}q$ have $q$
common factor, but there exist also duplicate common factors $pq$, $2pq$, \ldots, $\frac{m}{pq}pq$. Therefore, the number of terms
that have to be subtracted from $\varphi_p(m)$ to obtain $\varphi_{pq}(m)$ is
\begin{equation}\label{Deltaphiq}
  \frac{m}{q}-\frac{m}{pq}=\frac{m}{q}\left(1-\frac{1}{p}\right)~.
\end{equation}
Then, from (\ref{phip}) and (\ref{Deltaphiq}) it results that
\begin{equation}
  \varphi_{pq}(m)=m\left(1-\frac{1}{p}\right)-\frac{m}{q}\left(1-\frac{1}{p}\right)=m\left(1-\frac{1}{p}\right)\left(1-\frac{1}{q}\right)~.
\end{equation}
Similarly, by mathematical induction it can be proved that if $n$ is divisible by $p_1$, $p_2$, \ldots, $p_s$, prime numbers (or $n$
is a multiple of $p_1$, $p_2$, \ldots, $p_s$, prime numbers), then we have
\begin{equation}
  \varphi(n)=n\prod_{k=1}^s\left(1-\frac{1}{p_k}\right)~.
\end{equation}
We have an interesting identity, due to \cite{Olofsson2004}\index{Olofsson A.}, regarding $\varphi(n^\ell)$ and $\varphi(n)$, given by relation
\begin{equation}\label{phink}
  \varphi(n^\ell)=n^{\ell-1}\varphi(n)~.
\end{equation}

Euler's\index{Euler L.} totient function satisfies the inequality $\varphi(n)>\sqrt{n}$ for all $n\in\Ns$ excepting $2$ and $6$,
\citep{Kendall+Osborn1965}, \citep[p. 9]{Mitrinovic+Sandor1995}. Consequently, $\varphi(n)=2$ only for $n=3$, $n=4$ and $n=6$.
Also, in the monograph \citep{Sierpinski1988}, was proved that $\varphi(n)\le n-\sqrt{n}$.

The solutions of the $\varphi$--Diophantine equation $\varphi(n)=\varphi(n+1)$ are: $1$, $3$, $15$, $104$, $164$, $194$, $255$, $495$, $584$, $975$, \ldots \citep[A003275]{Sloane2014}.

In the search domain $D_c=\set{1,2,\ldots,10^{10}}$ there exists only one solution $n=5186=2^5\cdot3^4$ for which the double identity $\varphi(n)=\varphi(n+1)=\varphi(n+2)$ holds, \citep[p. 139]{Guy2004}.

The smallest three close numbers (the difference between them is $\le6$), for which the double equality $\varphi(n_1)=\varphi(n_2)=\varphi(n_3)$ holds, are: $404471$, $404473$ and $404477$. These numbers verify the equalities:
\[
 \varphi(404471)=\varphi(404473)=\varphi(404477)=403200~.
\]

The smallest four close numbers (the difference between them is $\le12$), for which the triple equality
$\varphi(n_1)=\varphi(n_2)=\varphi(n_3)=\varphi(n_4)$ hold, are: $25930$, $25935$, $24940$ and $25942$. They verify the equalities:
\[
 \varphi(25930)=\varphi(25935)=\varphi(25940)=\varphi(25942)=10368~.
\]
These results were published in \citep[p. 139]{Guy2004}.

\cite{McCranie2000} found the arithmetic progression $a_k=a_0+k\cdot r$, where the first term is $a_0=583200$ and
$r=30$ is the ratio, for which we have
\[
 \varphi(a_k)=155520\ \ \textnormal{for all}\ \ k=0,1,\ldots,5~.
\]
Other arithmetic progressions with six consecutive terms, with $a_0=1166400$ and $r=583200$, which have the same property, are
also known \citep[A050518]{Sloane2014}.

An interesting conjecture due to \cite{Guy2004} has following predication. \emph{If Goldach's conjecture holds, then, for every
$m\in\Ns$, there exist the prime numbers $p$ and $q$ such that $\varphi(p)+\varphi(q)=2m$}. Erd\"{o}s wondered if this statement
also holds for $p$ and $q$ not necessarily primes, but this "relaxed" conjecture remains unproved.

\cite{Guy2004} considered the $\varphi$--$\sigma$--Diophantine equation $\varphi(\sigma(n))=n$. F. Helenius\index{Helenius F.}
found 365 solutions, of which the first are: $2$, $8$, $12$, $128$, $240$, $720$, $6912$, $32768$, $142560$, $712800$, \ldots,
\citep[A001229]{Sloane2014}.

\subsection{Computing the values of $\varphi$ function}

\begin{prog}\label{Progphif} Considering formula (\ref{phink}), an efficient program for computing the values of function $\varphi$ can be written.
  \begin{tabbing}
    $\varphi(n):=$\=\vline\ $return\ 1\ if\ n\textbf{=}0\vee n\textbf{=}1$\\
    \>\vline\ $f\leftarrow Fa(n)$\\
    \>\vline\ $\phi\leftarrow n$\\
    \>\vline\ $f$\=$or\ k\in1..rows(f)$\\
    \>\vline\ \>\ $\phi\leftarrow\phi\cdot\dfrac{f_{k,1}-1}{f_{k,1}}$\\
    \>\vline\ $return\ \phi$
  \end{tabbing}
  This program calls the program \ref{ProgFa} for factorization of a number.
\end{prog}

\begin{prog}\label{ProgGenPhi} The program by means of which the file $\varphi.prn$ is generated is:
  \begin{tabbing}
    $G\varphi(N):=$\=\vline\ $f\varphi_1\leftarrow1$\\
    \>\vline\ $fo$\=$r\ n\in2..N$\\
    \>\vline\ \>\ $f\varphi_n\leftarrow \varphi(n)$\\
    \>\vline\ $return\ f\varphi$\\
  \end{tabbing}
  This program calls the program \ref{Progphif} for computing the values of Euler's totient function. The sequence for generating the file $\varphi.prn$ is:
  \[
   t_0:=time(0)\ \ WRITEPRN("\varphi.prn"):=G\varphi(10^6)\ \ t_1:=time(1)
  \]
  \[
   (t_1-t_0)sec="5:30:33.558"hhmmss
  \]
  The execution time for generating the values of function $\varphi$ up to $10^6$ is of $5$ hours and $30$ minutes on a computer
  with an Intel processor of 2.20GHz with RAM of 4.00GB (3.46GB usable).
\end{prog}

\section{A generalization of Euler\textquoteright{s} theorem}

In the sections which follow we will prove a result which replaces the theorem of Euler: "If $(a,m)=1$, then $a^{\varphi(m)}\equiv1 \md{m}$", for the case when $a$ and $m$ are not relatively primes.

One supposes that $m>0$. This assumption will not affect the generalization, because Euler\textquoteright{s} indicator satisfies the equality: $\varphi(m)=\varphi(-m)$, \citep{Popovici1973}, and that the congruencies verify the following property: $a\equiv b\md{m}\Leftrightarrow a\equiv b\md{-m}$, \citep[pp. 12--13]{Popovici1973}.

In the case of congruence modulo $0$, there is the relation of equality. One denotes $\big(a,b\big)=gcf(a,b)$ greatest common factor of the two integers $a$ and $b$, and one chooses $\big(a,b\big)>0$. Note $gcf$ is the same as $\gcd$ for numbers, so $\big(a,b\big)=gcf(a,b)=\gcd(a,b)$.

\begin{lem}\label{Lemma1}
  Let be $a$ an integer and $m$ a natural number $>0$. The exist $d_0,m_0\in\Na$ such that $a=a_0d_0$, $m=m_0d_0$ and $\big( a_0,m_0\big)=1$.
\end{lem}
\begin{proof}
  It is sufficient to choose $d_0=\big(a,m\big)$. In accordance with the definition of the $gcf$ (greatest common factor), the quotients of $a_0$ and $m_0$ of $a$ and $m$ by their $gcf$ are relatively primes, \citep[pp. 25--26]{Creanga+Cazacu+Mihut+Opait+Reischer1965}.
\end{proof}

\begin{lem}\label{Lemma2}
  With the notations of lemma \ref{Lemma1}, if $d_0\neq1$ and if: $d_0=d_0^1d_1$, $m_0=m_1d_1$, $\big(d_0^1,m_1\big)=1$ and $d_1\neq1$, then $d_0>d_1$ and $m_0>m_1$, and if $d_0=d_1$, then after a limited number of steps $i$ one has $d_0>d_{i+1}=\big( d_i,m_i\big)$.
\end{lem}
\begin{proof}
  \[
   (0)\left\{\begin{array}{lcl}
               a=a_0d_0 &;& \big(a_0,m_0\big)=1 \\
               m=m_0d_0 &;& d_0\neq1
             \end{array}\right.~,
  \]
  \[
   (1)\left\{\begin{array}{lcl}
               d_0=d_0^1d_1 &;& \big(d_0^1,m_1\big)=1 \\
               m_0=m_1d_1 &;& d_1\neq1
             \end{array}\right.~.
  \]
  From $(0)$ and from $(1)$ it results that $a=a_0d_0=a_0d_0^1d_1$ therefore $d_0=d_0^1d_1$ thus $d>d_1$ if $d_0^1\neq1$.

  From $m_0=m_1d_1$ we deduct that $m_0>m_1$. If $d_0=d_1$ then $m_0=m_1d_0=k\cdot d_0^z$, where $z\in\Ns$ and $d_0\nmid k$. Therefore $m_1=k\cdot d_0^{k-1}$; $d_2=\big(d_1,m_1\big)=\big(d_0,k\cdot d_0^{z-1}\big)$. After $i=z$ steps, it results $d_{i+1}=\big( d_0,k\big)<d_0$.
\end{proof}

\begin{lem}\label{Lemma3}
  For each integer $a$ and for each natural number $m>0$ one can build the following sequence of relations:
   \[
   (0)\left\{\begin{array}{lcl}
               a=a_0d_0 &;& \big(a_0,m_0\big)=1 \\
               m=m_0d_0 &;& d_0\neq1
             \end{array}\right.~,
   \]
   \[
    (1)\left\{\begin{array}{lcl}
               d_0=d_0^1d_1 &;& \big(d_0^1,m_1\big)=1 \\
               m_0=m_1d_1 &;& d_1\neq1
             \end{array}\right.~,
   \]
   \[
    \ldots\ldots\ldots
   \]
   \[
   (s-1)\left\{\begin{array}{lcl}
               d_{s-2}=d_{s-2}^1d_{s-1} &;& \big(d_{s-2}^1,m_{s-1}\big)=1 \\
               m_{s-2}=m_{s-1}d_{s-1} &;& d_{s-1}\neq1
             \end{array}\right.~,
   \]
   \[
    (s)\left\{\begin{array}{lcl}
               d_{s-1}=d_{s-1}^1d_s &;& \big(d_{s-1}^1,m_s\big)=1 \\
               m_{s-1}=m_sd_s &;& d_s\neq1
             \end{array}\right.~.
   \]
\end{lem}
\begin{proof}
  One can build this sequence by applying lemma \ref{Lemma1}. The sequence is limited, according to lemma \ref{Lemma2}, because after $r_1$ steps, one has $d_0>d_{r_1}$, and $m_0>m_{r_1}$, and after $r_2$ steps, one has $d_{r_1}>d_{r_1+r_2}$ and $m_{r_1}>m_{r_1+r_2}$, etc., and the $m_i$ are natural numbers. One arrives at $d_s=1$ because if $d_s\neq1$ one will construct again a limited number of relations $(s+1)$, \ldots, $(s+r)$ with $d_{s+r}<d_s$.
\end{proof}

\begin{thm}\label{Thm01}
  Let us have $a,m\in\Int$, and $m\neq0$. Then
  \[
   a^{\varphi(m_s)+s}\equiv a^s \md{m}~,
  \]
  where $s$ and $m_s$, are the same ones as in the lemmas above.
\end{thm}
\begin{proof}
  Similar with the method followed previously, one can suppose $m>0$ without reducing the generality. From the sequence of relations from lemma \ref{Lemma3}, it results that:
  \[
   \begin{array}{cccccc}
     (0) & (1) & (2) & (3) &  & (s) \\
     a & =a_0d_0 & =a_0d_0^1d_1 & =a_0d_0^1d_1^1d_2 & \ldots & =a_0d_0^1d_1^1\cdots d_{s-1}^1d_s
   \end{array}
  \]
  and
  \[
   \begin{array}{cccccc}
     (0) & (1) & (2) & (3) &  & (s) \\
     m & =m_0d_0 & =m_1d_1d_0 & =m_2d_2d_1d_0 & \ldots & =m_sd_sd_{s-1}\cdots d_1^1d_0
   \end{array}
  \]
  and
  \[
   m_sd_sd_{s-1}\cdots d_1d_0=d_0d_1\cdots d_{s-1}d_sm_s~.
  \]
  From $(0)$ it results that $d_0=\big(a,m\big)$, and from $(i)$ that $d_i=\big(d_{i-1},m_{i-1}\big)$, for all $i\in\set{1, 2, \ldots, s}$.
  \begin{eqnarray*}
    d_0 &=& d_0^1d_1^1d_2^1\cdots d_{s-1}^1d_s~, \\
    d_1 &=& d_1^1d_2^1d_3^1\cdots d_{s-1}^1d_s~, \\
        & \vdots &  \\
    d_{s-1} &=& d_{s-1}^1d_s~, \\
    d_s &=& d_s~.
  \end{eqnarray*}
  Therefore
  \begin{multline*}
    d_0d_1d_2\cdots d_{s-1}d_s=(d_0^1)^1(d_1^1)^2(d_2^1)^3\cdots (d_{s-1}^1)^s(d_s^1)^{s+1} \\
    =(d_0^1)^1(d_1^1)^2(d_2^1)^3\cdots(d_{s-1}^1)^s~,
  \end{multline*}
  because $d_s=1$.

Thus $m=(d_0^1)^1(d_1^1)^2(d_2^1)^3\cdots(d_{s-1}^1)^s\cdot m_s$; therefore $m_s\mid m$~.

\begin{multline*}
  \begin{array}{lcl}
    & (s) \\
   \big(d_s,m_s\big) & = & \big(1,m_s\big)
 \end{array}
\end{multline*}
and $\big(d_{s-1}^1,m_s\big)=1$
\begin{multline*}
  \begin{array}{lclcl}
   & (s-1)  \\
   1 & = & \big(d_{s-2}^1,m_{s-1}\big) & = & \big(d_{s-2}^1,m_sd_s\big)
 \end{array}
\end{multline*}
therefore $\big(d_{s-2}^1,m_s\big)=1$~,
\[
 \begin{array}{lclclcl}
   & (s-2)  \\
   1 & = & \big(d_{s-3}^1,m_{s-2}\big) & = & \big(d_{s-3}^1,m_{s-1}d_{s-1}\big) & = & \big(d_{s-3}^1,m_sd_sd_{s-1}\big)
 \end{array}
\]
therefore $\big(d_{s-3},m_s\big)=1$~,
\[
\vdots
\]
\begin{multline*}
  1\begin{array}{c}
     (i+1) \\
     =
   \end{array}
  \big(d_i^1,m_{i+1}\big)=\big(d_i^1,m_{i+2}d_{i+2}\big)=\big(d_i^1,m_{i+3}d_{i+3}d_{i+2}\big) =\ldots \\
  =\big(d_i^1,m_sd_sd_{s-1}\cdots d_{i+2}\big)~,
\end{multline*}
thus $\big(d_i^1,m_s\big)=1$, and this is for all $i\in\set{1,2,\ldots,s-2}$~,
\[
\vdots
\]
\begin{multline*}
  \begin{array}{lclcl}
   & (0) \\
   1 & = & \big(a_0,m_0\big) & = & \big(a_0,d_1\cdots d_{s-1}d_sm_s\big)
 \end{array}~,
\end{multline*}
thus $\big(a_0,m_s\big)=1$~.

From the Euler\textquoteright{s} theorem results that: $(d_i^1)^{\varphi(m_s)}\equiv1 \md{m_s}$ for all $i\in\set{0,1,\ldots,s}$, $a_0^{\varphi(m_s)}\equiv1 \md{m_s}$, but
\[
 a_0^{\varphi(m_s)}=a_0^{\varphi(m_s)}(d_0^1)^{\varphi(m_s)}(d_1^1)^{\varphi(m_s)}\cdots(d_{s-1}^1)^{\varphi(m_s)}~,
\]
therefore $a^{\varphi(m_s)}\equiv\underbrace{1\ \cdots\ 1}_{(s+1)\ times} \md{m_s}$, then $a^{\varphi(m_s)}\equiv1 \md{m_s}$~.
We equivalence
\begin{multline*}
  a_0(d_0^1)^{s-1}(d_1^1)^{s-2}(d_2^1)^{s-3}\cdots(d_{s-2}^1)^1\cdot a^{\varphi(m_s)} \\
  \equiv a_0^s(d_0^1)^{s-1}(d_1^1)^{s-2}\cdots(d_{s-2}^1)^1\cdot 1\ \md{m_s}~.
\end{multline*}
If you multiply the
$(d_0^1)^1(d_1^1)^2(d_2^1)^3\cdots(d_{s-2}^1)^{s-1}(d_{s-1}^1)^s$
we obtain:
\begin{multline*}
  a_0^s(d_0^1)^s(d_1^1)^s\cdots(d_{s-2}^1)^s(d_{s-1}^1)^s\cdot a^{\varphi(m_s)} \\
  \equiv a_0^s(d_0^1)^s(d_1^1)^s\cdots(d_{s-2}^1)^s(d_{s-1}^1)^s \md{(d_0^1)^1\cdots(d_{s-1}^1)^s\cdot m_s}~,
\end{multline*}
but $a_0^s(d_0^1)^s(d_1^1)^s\cdots(d_{s-1}^1)^s\cdot a^{\varphi(m_s)}$ and $a_0^s(d_0^1)^s(d_1^1)^s\cdots(d_{s-1}^1)^s=a^s$ therefore $a^{\varphi(m_s)+s}\equiv a^s \md{m}$, for all $a,m\in\Int$, $m\neq0$.
\end{proof}

\begin{obs}
  If $\big(a,m\big)=1$ then $d=1$. Thus $s=0$, and according the theorem \ref{Thm01} one has $a^{\varphi(m_0)+0}\equiv a^0 \md{m}$ therefore $a^{\varphi(m_0)}\equiv 1 \md{m}$. But $m=m_0d_0=m_0\cdot1=m_0$. Thus $a^{\varphi(m)}\equiv1 \md{m}$, and one obtains Euler\textquoteright{s} theorem.
\end{obs}

\begin{obs}
  Let us have $a$ and $m$ two integers, $m\neq0$ and $\big(a,m\big)=d_0\neq1$, and $m=m_0d_0$~. If $\big(d_0,m_0\big)=1$, then $a^{\varphi(m_0)+1}\equiv a\md{m}$. Which, in fact, it results from the theorem \ref{Thm01} with $s=1$ and $m_1=m_0$. This relation has a similar to Fermat\textquoteright{s} theorem: $a^{\varphi(p)+1}\equiv a \md{p}$~.
\end{obs}

\subsection{An algorithm to solve congruences}

One will construct an algorithm to calculate $s$ and $m_s$ of the theorem \ref{Thm01}.
\begin{prog}\label{ProgS} The program is:
  \begin{tabbing}
    $S(a,m):=$\=\vline\ $s\leftarrow0$\\
    \>\vline\ $m_s\leftarrow m$\\
    \>\vline\ $d\leftarrow \gcd(a,m_s)$\\
    \>\vline\ $w$\=$hile\ d\neq1$\\
    \>\vline\ \>\vline\ $s\leftarrow s+1$\\
    \>\vline\ \>\vline\ $m_s\leftarrow\dfrac{m_s}{d}$\\
    \>\vline\ \>\vline\ $d\leftarrow \gcd(d,m_s)$\\
    \>\vline\ $return\ \left(
                          \begin{array}{c}
                            s \\
                            m_s \\
                          \end{array}
                        \right)$
   \end{tabbing}
   The program calls the function Mathcad $\gcd$ computation of the greatest common divisor.
\end{prog}

\subsection{Applications}

In the resolution of the exercises one uses the theorem \ref{Thm01} and the algorithm to calculate $s$ and $m_s$.

\emph{Example 1}: $6^{\varphi(m_s)}\equiv6^s \md{105765}$. One thus applies the algorithm to calculate $s$ and $m_s$ and then the theorem \ref{Thm01}:
\[
 a:=6\ \ \ \ m:=105765
\]
\[
 \left(\begin{array}{c}
     s \\
     m_s \\
   \end{array}\right):=S(a,m)=\left(\begin{array}{c}
                                       1 \\
                                       35255 \\
                                     \end{array}\right)
\]
\[
 \phi:=\varphi(m_s)+s=25601
\]
\[
 mod(a^\phi,m)=6\ \ \ \ \ mod(a^s,m)=6~,
\]
where we used the programs \ref{ProgS} and \ref{Progphif}.

\emph{Example 2}: $847^{\varphi(m_s)}\equiv a^s \md{283618125}$. One thus applies the algorithm to calculate $s$ and $m_s$ and then the theorem \ref{Thm01}:
\[
 a:=847\ \ \ \ m:=283618125
\]
\[
 \left(\begin{array}{c}
     s \\
     m_s \\
   \end{array}\right):=S(a,m)=\left(\begin{array}{c}
                                       5 \\
                                       16875 \\
                                     \end{array}\right)
\]
\[
 \phi:=\varphi(m_s)+s=9005
\]
\[
 mod(a^\phi,m)\rightarrow7^5\cdot9601\ \ \ \ mod(a^s,m)\rightarrow7^5\cdot9601~,
\]
where we used the programs \ref{ProgS} and \ref{Progphif}.

\emph{Example 2}: $847^{\varphi(m_s)}\equiv a^s \md{283618125}$. One thus applies the algorithm to calculate $s$ and $m_s$ and then the theorem \ref{Thm01}:
\[
 a:=847\ \ \ \ m:=283618125
\]
\[
 \left(\begin{array}{c}
     s \\
     m_s \\
   \end{array}\right):=S(a,m)=\left(\begin{array}{c}
                                       5 \\
                                       16875 \\
                                     \end{array}\right)
\]
\[
 \phi:=\varphi(m_s)+s=9005
\]
\[
 mod(a^\phi,m)\rightarrow7^5\cdot9601\ \ \ \ mod(a^s,m)\rightarrow7^5\cdot9601~,
\]
where we used the programs \ref{ProgS} and \ref{Progphif}.

\emph{Example 3}: $437^{\varphi(m_s)}\equiv a^s \md{2058557375}$. One thus applies the algorithm to calculate $s$ and $m_s$ and then the theorem \ref{Thm01}:
\[
 a:=437\ \ \ \ m:=2058557375
\]
\[
 \left(\begin{array}{c}
     s \\
     m_s \\
   \end{array}\right):=S(a,m)=\left(\begin{array}{c}
                                       3 \\
                                       300125 \\
                                     \end{array}\right)
\]
\[
 \phi:=\varphi(m_s)+s=205803
\]
\[
 mod(a^\phi,m)\rightarrow19^3\cdot23^3\ \ \ \ mod(a^s,m)\rightarrow19^3\cdot23^3~,
\]
where we used the programs \ref{ProgS} and \ref{Progphif}.

\emph{Example 4}: $4433^{\varphi(m_s)}\equiv a^s \md{789310951}$. One thus applies the
algorithm to calculate $s$ and $m_s$ and then the theorem \ref{Thm01}:
\[
 a:=4433\ \ \ \ m:=789310951
\]
\[
 \left(\begin{array}{c}
     s \\
     m_s \\
   \end{array}\right):=S(a,m)=\left(\begin{array}{c}
                                       5 \\
                                       29 \\
                                     \end{array}\right)
\]
\[
 \phi:=\varphi(m_s)+s=33
\]
\[
 mod(a^\phi,m)\rightarrow2^3\cdot11^5\cdot13^2\ \ \ \ mod(a^s,m)\rightarrow2^3\cdot11^5\cdot13^2~,
\]
where we have used the programs \ref{ProgS} and \ref{Progphif}.

\chapter{Generalization of congruence theorems}

\section{Notions introductory}

Let us consider a positive integer, which we will call \emph{modulus}. With its help we introduce in the set $\Int$ of integers a binary relation, called \emph{of congruence} and denoted $\equiv$, such that:
\begin{defn}
The integers \emph{a} and \emph{b} are congruent relative to modulus \emph{m} is and only if \emph{m} divides the difference $a-b$.
\end{defn}
Hence, we have
\begin{equation}\label{L111}
  a\equiv b\ \md{m}\ \Leftrightarrow\ a-b=k\cdot m,\ \textnormal{where}\ k\in\Int.
\end{equation}

\begin{cons}
  $a\equiv b\ \md{m}$ $\Leftrightarrow$ $a$ and $b$ give, by division trough $m$, the same residue.
\end{cons}

It is known that the congruence relation given by (\ref{L111}) is an equivalence relation (is reflexive, symmetric and transitive).
It also has following remarkable properties:
\begin{enumerate}
  \item[] $a_1\equiv b_1\ \md{m}$ and $a_2\equiv b_2\ \md{m}$ $\Rightarrow$~,
  \item[(i)] $a_1+a_2\equiv b_1+b_2\ \md{m}$~,
  \item[(ii)] $a_1-a_2\equiv b_1-b_2\ \md{m}$~,
  \item[(iii)] $a_1a_2\equiv b_1b_2\ \md{m}$~.
\end{enumerate}
More generally, if $a_i\equiv b_i\ \md{m}$, for $i=1,2,\ldots,n$, and $f(x_1,x_2,\ldots,x_n)$ is a polynomial with integer coefficients, then
\begin{enumerate}
  \item [(iv)] $f(a_1,a_2,\ldots,a_n)\equiv f(b_1,b_2,\ldots,b_n)\ \md{m}$.
\end{enumerate}
One can also prove following properties of the congruence relations:
\begin{enumerate}
  \item [(v)] $a\equiv b\ \md{m}$ and $c\in\Ns$ $\Rightarrow$ $ac\equiv bc\ \md{m}$~,
  \item [(vi)] $a\equiv b\ \md{m}$ and $n\in\Ns$, $n$ divide $m$ $\Rightarrow$ $a\equiv b\ (mod\ n)$~,
  \item [(vii)] $a\equiv b\ \md{m_i}$, $i=\overline{1,s}$, $\Rightarrow$ $a\equiv b\ \md{m}$~,
\end{enumerate}
where $m=[m_1,m_2,\ldots,m_s]=lcm(m_1,m_2,\ldots,m_s)$ is the smallest common multiple of numbers $m_i$.
\begin{enumerate}
  \item [(viii)] $a\equiv b\ \md{m}$ $\Rightarrow$ $(a,m)=(b,m)$~,
\end{enumerate}
where by $(x,y)=gcd(x,y)$ we denote the greatest common divisor of numbers $x$ and $y$.

As the relation \emph{congruence mod m} is an equivalence relation, it divides the set $\Int$ of integers into equivalence classes (classes of \emph{congruence mod m}). Two such classes either are disjoint or coincide.

As every integer provides by division through $m$ one of the residues $0$, $1$, $2$, \ldots, $m-1$, it follows that
\[
 C_0,\ C_1,\ \ldots,\ C_{m-1}
\]
are the $m$ classes of residues \emph{mod m}, where $C_i$ is the set of all integers congruent with $i\ \md{m}$.

Sometimes it is useful to consider, instead of the classes, representatives that satisfy certain conditions. Hereby, following
terminology is established.
\begin{defn}
The integers $a_1$, $a_2$, \ldots, $a_m$ compose \emph{a complete system of mod m residues} if any two of them are not congruent \emph{mod m}.
\end{defn}

It results that a complete system of \emph{mod m} residues contains a representative of each class.

If $\varphi$ is Euler's\index{Euler L.} totient function ($\varphi(n)$, denoted also $\varphi_n$, is the number of natural numbers smaller than $n$ and prime to $n$), then we also have:
\begin{defn}
The integers $a_1$, $a_2$, \ldots, $a_{\varphi(m)}$ build \emph{a reduced system of mod m residues} if each is prime with the
modulus and if any two of them are not congruent \emph{mod m}.
\end{defn}
Following result is known:
\begin{thm}\
  \begin{enumerate}
    \item If $a_1$, $a_2$, \ldots, $a_m$ is a complete system of mod m residues and a is an integer, prime to m, then the sequence
        $a\cdot a_1,\ a\cdot a_2,\ \ldots, a\cdot a_m$ is also a complete system of mod m residues.
    \item If $a_1$, $a_2$, \ldots, $a_{\varphi(m)}$ is a reduced system of mod m residues and a is an integer, prime to m, then the sequence
    $a\cdot a_1,\ a\cdot a_2,\ \ldots, a\cdot a_{\varphi(m)}$ is also a reduced system of mod m residues.
  \end{enumerate}
\end{thm}

If we denote by $Z_m$ the set of the classes of residues mod m:
\[
 Z_m=\set{C_0,C_1,\ldots,C_{m-1}}
\]
and we define the relations
\[
 +:Z_m\times Z_m\to Z_m~,\ \ \cdot:Z_m\times Z_m\to Z_m~,
\]
by
\[
 C_i+C_j=C_k,\ \textnormal{where}\ k\equiv i+j\ \md{m}~,
\]
\[
 C_i\cdot C_j=C_h,\ \textnormal{where}\ h\equiv i\cdot j\ \md{m}~,
\]
then following result holds:
\begin{thm}\
\begin{enumerate}
  \item $(Z_m,+)$ is a commutative group,
  \item $(Z_m,+,\cdot)$ is a commutative ring,
  \item $(G_m,\cdot)$ is a commutative group,
\end{enumerate}
where $G_m=\set{C_{r_1},C_{r_2},\ldots,C_{r_{\varphi(m)}}}$ the set of the classes of residues prime to the modulus.
\end{thm}
\begin{cons}
The set $Z_p$ of the classes of residues relative to a prime modulus $p$ builds a commutative field relative to the previously defined operations of addition and multiplication.
\end{cons}

\section{Theorems of congruence of the Number Theory}

In this section we will recall some congruence theorems of the Number Theory (Theorems of Fermat,\index{Fermat P.} Euler,\index{Euler L.} Wilson,\index{Wilson J.} Gauss,\index{Gauss C. F.} Lagrange,\index{Lagrange J. L.} Leibniz,\index{Leibniz G.}
Moser\index{Moser L.} and Sierpinski\index{Sierpinski W.}) which we will generalize in the next section. Equally, we will emphasize
a unifying point of view.

In 1640 Fermat\index{Fermat P.} states, without proof, the next result:
\begin{thm}[Fermat\index{Fermat P.}]\label{TFermat5}
  If integer $a$ is not divisible by the prime number $p$, then
  \begin{equation}\label{L121}
    a^{p-1}\equiv1\ \md{p}~.
  \end{equation}
\end{thm}
The first proof of this theorem was given in 1736 by Euler\index{Euler L.}.

As it is known, the reciprocal of Fermat's\index{Fermat P.} Theorem is not true. In another words, the fact that
$a^{m-1}\equiv1\ \md{m}$ and $m$ is not divisible by $a$, does not necessarily imply that $m$ is a prime number.

It is not even true that, if (\ref{L121}) holds for all numbers $a$ prime relative to $m$, then $m$ is prime, as one can remark in
the following example.
\begin{exem}
  Let $m=561=3\cdot11\cdot17$. If $a$ is an integer that is not divisible by $3$, by $11$ or by $17$, we surely have:
  \[
   a^2\equiv1\ \md{3},\ a^{10}\equiv1\ \md{11},\ a^{16}\equiv1\ \md{17}~,
  \]
  according to the direct Theorem of Fermat.\index{Fermat P.} But, as $560$ is divisible by $2$ and $10$, as well as by $16$, we deduce that:
  \[
   a^{560}\equiv1\ \md{m_i},\ i=1,2,3~,
  \]
  where $m_1=3$, $m_2=11$ and $m_3=17$.

  According to property (vii) of the previous section, it follows that
  \[
   a^{560}\equiv1\ \md{m},\ \textnormal{for}\ m=561.
  \]
Actually, it is known that $561$ is the smallest composite number that satisfies (\ref{L121}). Next numbers follow: $1105$, $1729$,
$2465$, $2821$, \ldots~.

Consequently, the congruence (\ref{L121}) can be true for a certain integer $a$ and a composite number $m$.
\end{exem}
\begin{defn}
  If relation (\ref{L121}) is satisfied for a composite number $m$ and an integer $a$, it is said that $m$ is \emph{pseudoprime relative to a}.
  If $m$ is pseudoprime relative to every integer $a$, prime to $m$, it is said that $m$ is a
  \emph{Carmichael number}\index{Carmichael R.}.
\end{defn}

The American mathematician Robert Carmichael\index{Carmichael R.} was the first who, in 1910 has emphasized such numbers, called
\emph{fake prime numbers}.

Until recently, it was not known if there exists or not an infinity of Carmichael numbers.\index{Carmichael R.} In the very
first issue of the journal "What$'$s Happening in the Mathematical Sciences", where, yearly, the most important recent results in
mathematics are emphasized, it is that three mathematicians: Alford,\index{Alford W. R.} Granville\index{Granville A.} and
Pomerance,\index{Pomerance C.} have proved that there exists an infinity of Carmichael numbers.\index{Carmichael R.}

The proof of the trio of American mathematicians is based on an heuristic remark from 1956 of the internationally known Hungarian
mathematician P. Erd\"{o}s.\index{Erd\"{o}s P.} The main idea is to chose a number $L$ for which there exist a lot of prime numbers
$p$ that do not divide $L$, but having the property that $p-1$ divides $L$. Afterwards it is shown that these prime numbers can
be multiplied among themselves in several ways such that each product is congruent with $1\ \md{L}$. It results that every such
product is a Carmichael number\index{Carmichael R.}.

For example, for $L=120$, the prime numbers that satisfy the previous condition are: $p_1=7$, $p_2=11$, $p_3=13$, $p_4=31$,
$p_5=41$, $p_6=61$. It follows that $41041=7\cdot11\cdot13\cdot41$, $172081=7\cdot13\cdot31\cdot61$
and $852841=11\cdot31\cdot41\cdot61$ are congruent with $1\ \md{120}$, and, hence, they are Carmichael numbers.\index{Carmichael R.}

We mention that the heuristic remark of P. Erd\"{o}s\index{Erd\"{o}s P.} is based on the following theorem that characterizes the Carmichael numbers,\index{Carmichael R.} proved in 1899.
\begin{thm}[A. Korselt\index{Korselt A.}]
  The number $n$ is a Carmichael number\index{Carmichael R.} if and only if following conditions hold:
  \begin{enumerate}
    \item [$(C_1)$] $n$ is squares free,
    \item [$(C_2)$] $p-1$ divides $n-1$ as long as $p$ is a prime divisor of $n$.
  \end{enumerate}
\end{thm}

The three American mathematicians have proved the following result:
\begin{thm}[Alford,\index{Alford W. R.} Granville,\index{Granville A.} Pomerance\index{Pomerance C.}]\label{TAlfordGranvillePomerance}
  There exist at least $x^{2/7}$ Carmichael numbers\index{Carmichael R.}, not greater than $x$, for $x$ sufficiently big.
\end{thm}

By means of the heuristic argument due to P. Erd\"{o}s\index{Erd\"{o}s P.} it can be proved that the exponent of Theorem \ref{TAlfordGranvillePomerance} can be replaced by any other sub unitary exponent.

\begin{thm}[Euler\index{Euler L.}]\label{TEuler}
  If $(a,m)=1$, then $a^{\varphi(m)}\equiv1\ \md{m}$.
\end{thm}
The notation $(a,m)=1$ means that the greatest common divisor of $a$ and $m$ is $1$, which means that the numbers are relatively
prime.

Theorem \ref{TEuler} generalizes Theorem \ref{TFermat5} and was proved by Euler\index{Euler L.} in 1760.

\begin{thm}[Wilson\index{Wilson J.}]\label{TWilson5}
  If $p$ is a prime number, then $(p-1)!+1\equiv0\ \md{p}$.
\end{thm}

It is known that the reciprocal of Theorem \ref{TWilson5} is true, which means that following result holds
\begin{thm}
  If $n>1$ is an integer and $(n-1)!+1\equiv0\ \md{n}$ then $n$ is prime.
\end{thm}

Theorem \ref{TWilson5}, of Wilson,\index{Wilson J.} was published in 1770 by mathematician Waring (\emph{Meditationes Algebraicae}),
but it was known long before, by Leibniz.\index{Leibniz G.}

Lagrange\index{Lagrange J. L.} generalizes Theorem \ref{TWilson5}, of Wilson,\index{Wilson J.} as follows:
\begin{thm}[Lagrange\index{Lagrange J. L.}]
  If $p$ is a prime number, then $a^{p-1}-1\equiv(a+1)(a+2)\ldots(a+p-1)\ \md{p}$.
\end{thm}

Leibniz states following theorem:
\begin{thm}[Leibniz]\label{TLeibniz}
  If $p$ is a prime number, then $(p-2)!\equiv1\ \md{p}$.
\end{thm}

The reciprocal of Theorem \ref{TLeibniz}, of Leibniz, is also true, i.e. a natural number $n>1$ is prime if and only if
$(n-2)!\equiv1\ \md{p}$.

Another result concerning congruences with prime numbers is the next theorem:
\begin{thm}[L. Moser\index{Moser L.}]
  If $p$ is a prime number, then $(p-1)!a^p+a\equiv0\ \md{p}$~.
\end{thm}

Sierpinski\index{Sierpinski W.} proves that following result
holds:
\begin{thm}[Sierpinski\index{Sierpinski W.}]
  If $p$ is a prime number, then $a^p+(p-1)!\equiv0\ \md{p}$.
\end{thm}

We remark that this statement unify the Theorems \ref{TFermat5} of Fermat\index{Fermat P.} and \ref{TWilson5} of Wilson.

In the next section we will define a function $L:\Int\times\Int\to \Int$, by means of which we will be able to prove several
results that unify all previous theorems.

\section{A unifying point of convergence theorems}

Let $A$ be the set $\set{m\in\Int /m=\pm p^\beta,\pm 2p^\beta}$ with $p$ an odd prime, $\beta\in\Ns$, or $m=\pm2^\alpha$, with $\alpha=0,1,2$, or $m=0$.

Let $m=\varepsilon\desp[\alpha]{r}$, with $\varepsilon=\pm1$, all $\alpha_i\in\Ns$ and $p_1$, $p_2$, \ldots, $p_r$ are distinct positive primes.

We construct the function $L:\Int\times\Int$,
\begin{equation}
  L(x,m)=(x+C_1)(x+C_2)\cdots(x+C_{\varphi(m)})~,
\end{equation}
where $C_1$, $C_2$, \ldots $C_{\varphi(m)}$ are all modulo $m$ rests relatively prime to $m$, and $\varphi$ is Euler\textquoteright{s} function.

If all distinct primes which divide $x$ and $m$ simultaneously are $p_{i_1}$, $p_{i_2}$, \ldots $p_{i_r}$, then:
\begin{equation}
  L(x,m)\equiv\pm1 \md{p_{i_1}^{\alpha_{i_1}}\cdot p_{i_2}^{\alpha_{i_2}}\cdots p_{i_r}^{\alpha_{i_r}}}\ \ \textnormal{when}\ \ m\in A
\end{equation}
respectively $m\notin A$, and
\begin{equation}
  L(x,m)\equiv0 \md{m/\left(p_{i_1}^{\alpha_{i_1}}\cdot p_{i_2}^{\alpha_{i_2}}\cdots p_{i_r}^{\alpha_{i_r}}\right)}~.
\end{equation}

For $d=p_{i_1}^{\alpha_{i_1}}\cdot p_{i_2}^{\alpha_{i_2}}\cdots p_{i_r}^{\alpha_{i_r}}$, and $m'=m/d$ we find
\begin{equation}
  L(x,m)\equiv\pm1+k_1^0d\equiv k_2^0m'\ \ \ \md{m}~,
\end{equation}
where $k_1^0$ and $k_2^0$ constitute a particular integer solution of the Diophantine equation $k_2m'-k_1d=\pm1$ (the signs are chosen in accordance with the affiliation of $m$ to $A$).

This result generalizes Gauss\textquoteright{} \index{Gauss C. F.} theorem, $(C_1\cdot C_2\cdots C_{\varphi(m)}\equiv\pm1 \md{m}$ when $m\in A$ respectively $m\notin A$), see \citep{Dirichlet1894}, which generalized in its turn the Wilson\textquoteright{s} \index{Wilson J.} theorem (if $p$ is prime then $(p-1)!\equiv-1 \md{m}$).

\begin{lem}\label{Lemma1FunctiaL}
  If $C_1$, $C_2$, \ldots, $C_{\varphi(p^\alpha)}$ are all modulo $p^\alpha$ rests, relatively prime to $p^\alpha$, with $p$ an integer and $\alpha\in\Ns$, then for $k\in\Int$ and $\beta\in\Ns$ we have also that $kp^\beta+C_1$, $kp^\beta+C_2$, \ldots, $kp^\beta+C_{\varphi(p^\alpha)}$ constitute all modulo $p^\alpha$ rests relatively prime to $p^\alpha$.
\end{lem}
\begin{proof}
  It is sufficient to prove that for $1\le i\le\varphi(p^\alpha)$ we have $kp^\beta+C_i$ relatively prime to $p^\alpha$, but this is obvious.
\end{proof}

\begin{lem}\label{Lemma2FunctiaL}
  If $C_1$, $C_2$, \ldots, $C_{\varphi(m)}$ are all modulo $m$ rests relatively prime to $m$, $p_i^{\alpha_i}$ divides $m$ and $p_i^{\alpha_i+1}$ does not divide $m$, then $C_1$, $C_2$, \ldots, $C_{\varphi(m)}$ constitute $\varphi(m/p_i^{\alpha_i})$ systems of all modulo $p_i^{\alpha_i}$ rests relatively prime to $p_i^{\alpha_i}$.
\end{lem}
\begin{proof}
  Proof is obvious.
\end{proof}

\begin{lem}
  If $C_1$, $C_2$, \ldots, $C_{\varphi(m)}$ are all modulo $q$ rests relatively prime to $b$ and $(b,q)=1$ then $b+C_1$, $b+C_2$, \ldots, $b+C_{\varphi(q)}$ contain a representative of the class $\widehat{0}$ modulo $q$.
\end{lem}
\begin{proof}
  Of course, because $(b,q-b)=1$ there will be a $C_{i_0}=q-b$, whence $b+C_{i_0}=\mathcal{M}q$ (multiple of $q$).
\end{proof}

From this we have:
\begin{thm}\label{Teorema1FunctiaL}
  If $\big(x,m/(p_1^{\alpha_{i_1}}\cdot p_2^{\alpha_{i_2}}\cdots p_{i_r}^{\alpha_{i_r}})\big)=1$ then
  \begin{multline*}
    L(x,m)=(x+C_1)(x+C_2)\cdots(x+C_{\varphi(m)})\\
    \equiv0 \md{m/\big(p_1^{\alpha_{i_1}}\cdot p_2^{\alpha_{i_2}}\cdots p_{i_r}^{\alpha_{i_r}}\big)}~.
  \end{multline*}
\end{thm}
\begin{proof}
  Proof is obvious.
\end{proof}

\begin{lem}\label{Lemma4FunctiaL}
  Because $C_1\cdot C_2\cdots C_{\varphi(m)}\equiv\pm1 \md{m}$ it results that
  \[
   C_1\cdot C_2\cdots C_{\varphi(m)}\equiv\pm1 \md{p_i^{\alpha_i}}~,
  \]
  for all $i$, when $m\in A$, respectively $m\notin A$.
\end{lem}
\begin{proof}
  Proof is obvious.
\end{proof}

\begin{lem}\label{Lemma5FunctiaL}
  If $p_i$ divides $x$ and $m$ simultaneously, then
  \[
   (x+C_1)(x+C_2)\cdots(x+C_{\varphi(m)})\equiv\pm1 \md{p_i^{\alpha_i}}~,
  \]
  when $m\in A$ respectively $m\notin A$.
\end{lem}
\begin{proof}
  Of course, from the lemmas \ref{Lemma2FunctiaL} and \ref{Lemma1FunctiaL}, respectively \ref{Lemma4FunctiaL}, we have
  \[
   (x+C_1)(x+C_2)\cdots(x+C_{\varphi(m)})\equiv C_1\cdot C_2\cdots C_{\varphi(m)}\equiv\pm1 \md{p_i^{\alpha_i}}~.
  \]
\end{proof}

From the lemma \ref{Lemma5FunctiaL} we obtain:

\begin{thm}\label{Teorema2FunctiaL} If $p_{i_1}$, $p_{i_2}$, \ldots, $p_{i_r}$ are all primes which divide $x$ and $m$ simultaneously then
  \[
   (x+C_1)(x+C_2)\cdots(x+C_{\varphi(m)})\equiv\pm1 \md{p_1^{\alpha_{i_1}}\cdot p_2^{\alpha_{i_2}}\cdots p_i^{\alpha_{i_r}}}~,
  \]
when $m\in A$ respectively $m\notin A$.
\end{thm}

From the theorems \ref{Teorema1FunctiaL} and \ref{Teorema2FunctiaL} it results $L(x,m)=\pm1+k_1\cdot d=k_2\cdot m'$, where $k_1, k_2\in\Int$. Because $(d,m')=1$ the Diophantine equation $k_2\cdot m'-k_1\cdot d=\pm1$ admits integer solutions (the unknowns being $k_1$ and $k_2$). Hence $k_1=m'\cdot t+k_1^0$ and $k_2=d\cdot t+k_2^0$, with $t\in\Int$, and $k_1^0,k_2^0$ constitute a particular integer solution of our equation. Thus
\[
 L(x,m)\equiv\pm1+m'\cdot d\cdot t+k_1^0\cdot d\equiv\pm1+k_1^0\ \md{m}
\]
or
\[
 L(x,m)\equiv k_2^0\cdot m'\ \md{m}~.
\]

\section{Applications}

\begin{enumerate}
  \item The theorem Lagrange\index{Lagrange J. L.} was extended of Wilson\index{Wilson J.} as follows: "if $p$ is prime, then $x^{p-1}-1\equiv(x+1)(x+2)\cdots(x+p-1)\ \md{p}$"; we shall extend this result in the following way: For any $m\neq0,\pm4$ we have for $x^2+s^2\neq0$ that
      \[
       x^{\varphi(m_s)+s}-x^s\equiv(x+1)(x+2)\cdots(x+\abs{m}-1)\ \md{m}~,
      \]
      where $m_s$ and $s$ are obtained from the algorithm:
  \begin{alg}\label{Algoritm1FunctiaL}
      \[
       (0)\ \ \left(\begin{array}{ll}
                x=x_0d_0~; & (x_0,m_0)=1 \\
                m=m_0d_0~; & d_0\neq1
              \end{array}\right.~,
      \]
      \[
       (1)\ \ \left(\begin{array}{ll}
                d_0=d_0^1d_1~; & (d_0^1,m_1)=1 \\
                m_0=m_1d_1~; & d_1\neq1
              \end{array}\right.~,
      \]
      \[
       \ldots\ldots\ldots
      \]
      \[
       (s-1)\ \ \left(\begin{array}{ll}
                d_{s-2}=d_{s-2}^1d_{s-1}~; & (d_{s-2}^1,m_{s-1})=1 \\
                m_{s-2}=m_{s-1}d_{s-1}~; & d_{s-1}\neq1
              \end{array}\right.~,
      \]
      \[
       (s)\ \ \left(\begin{array}{ll}
                d_{s-1}=d_{s-1}^1d_s~; & (d_{s-1}^1,m_s)=1 \\
                m_{s-1}=m_sd_s~; & d_s\neq1
              \end{array}\right.~,
      \]
      \citep{Smarandache1981,Smarandache1984}.
      \end{alg}
      For $m$ positive prime we have $m_s=m$, $s=0$ and $\varphi(m)=m-1$, that is Lagrange\textquoteright{s} theorem.
  \item L. Moser\index{Moser L.} enunciated the following theorem: "\emph{if} $p$ \emph{is prime, then} $(p-1)!a^p+a=\mathcal{M}p$", and \cite{Sierpinski1966}\index{Sierpinski W.}: "\emph{if} $p$ \emph{is prime then} $a^p+(p-1)!a=\mathcal{M}p$ which merges Wilson\textquoteright{s} and Fermat\textquoteright{s}\index{Fermat P.} theorems in a single one.

      The function $L$ and the algorithm \ref{Algoritm1FunctiaL} will help us to generalize them too, so: if $a$ and $m$ are integers, $m\neq0$, and $C_1$, $C_2$, \ldots, $C_{\varphi(m)}$ are all modulo rests relatively prime to $m$ then
      \[
       C_1\cdot C_2\cdots C_{\varphi(m)}a^{\varphi(m_s)+s}-L(0,m)\cdot a^s=\mathcal{M}m~,
      \]
      respectively
      \[
       -L(0,m)a^{\varphi(m_s)+s}+C_1\cdot C_2\cdots C_{\varphi(m)}\cdot a^s=\mathcal{M}m~,
      \]
      or more,
      \[
       (x+C_1)(x+C_2)\cdots(x+C_{\varphi(m)})a^{\varphi(m_s)+s}-L(x,m)\cdot a^s=\mathcal{M}m
      \]
      respectively
      \[
       -L(x,m)a^{\varphi(m_s)+s}+(x+C_1)(x+C_2)\cdots(x+C_{\varphi(m)}\cdot a^s=\mathcal{M}m~,
      \]
      which reunites Fermat\index{Fermat P.}, Euler\index{Euler L.}, Wilson\index{Wilson J.}, Lagrange\index{Lagrange J. L.} and Moser\index{Moser L.} (respectively Sierpinski\index{Sierpinski W.}).
  \item The author also obtained a partial extension of Moser\textquoteright{s} and Sierpinski\textquoteright{s} results, \citep{Smarandache1983}, so: if $m$ is positive integer, $m\neq0,4$, and $a$ is an integer, then $(a^m-a)(m-1)!=\mathcal{M}m$, reuniting Fermat\textquoteright{s} and Wilson\textquoteright{s} theorem in another way.
  \item Leibniz\index{Leibniz G.} enunciated that: "\emph{if} $p$ \emph{is prime then} $(p-2)!\equiv1\ \md{p}$"; we consider "$C_i<C_{i+1}\ \md{m}$" if $C_i<C_{i+1}$ where $0\le C_i<\abs{m}$, $0\le C_{i+1}<\abs{m}$ and $C_i\equiv C_i\ \md{m}$, $C_{i+1}\equiv C_{i+1}\ \md{m}$; one simply gives that if $C_1$, $C_2$, \ldots, $C_{\varphi(m)}$ are all modulo $m$ rests relatively prime to $m$ ($C_i<C_{i+1}\ \md{m}$ for all $i$, $m\neq0$) then $C_1\cdot C_2\cdots C_{\varphi(m)-1}\equiv\pm1\ \md{m}$ if $m\in A$ respectively $m\notin A$, because $C_{\varphi(m)}\equiv-1\ \md{m}$.
\end{enumerate}

\chapter[Analytical solving]{Analytical solving of Diophantine equations}

\section{General Diophantine equations}

A Diophantine equation is an equation in which only integer solutions are allowed.

Hilbert\textquoteright{s} 10th problem asked if an algorithm existed for determining whether an arbitrary Diophantine equation has a solution. Such an algorithm does exist for the solution of first-order Diophantine equations. However, the impossibility of obtaining a general solution was proven by \cite{Matiyasevich1970}, \cite{Davis1973}, \cite{Davis+Hersh1973}, \cite{Davis1982}, \cite{Matiyasevich1993} by showing that the relation $n=F_{2m}$ (where $F_{2}$ is the $2m$-th Fibonacci number) is Diophantine. More specifically, Matiyasevich\index{Matiyasevich Y. V.} showed that there is a polynomial $P$ in $n$, $m$, and a number of other variables $x$, $y$, $z$, \ldots having the property that $n=F_{2m}$ if there exist integers $x$, $y$, $z$, \ldots such that $P(n,m,x,y,z,\ldots)=0$.

Matiyasevich\textquoteright{s} result filled a crucial gap in previous work by Martin Davis\index{Davis M.}, Hilary Putnam\index{Putman H.}, and Julia Robinson\index{Robinson J.}. Subsequent work by Matiyasevich and Robinson proved that even for equations in thirteen variables, no algorithm can exist to determine whether there is a solution. Matiyasevich then improved this result to equations in only nine variables \cite{Jones+Matiyasevich1981}.

\cite{Ogilvy+Anderson1988} give a number of Diophantine equations with known and unknown solutions.

A linear Diophantine equation (in two variables) is an equation of the general form
\[
 m\cdot x+n\cdot y=\ell~,
\]
where solutions are sought with $m$, $n$, and $\ell$ integers. Such equations can be solved completely, and the first known solution was constructed by Brahmagupta, \citep{WeissteinDiophantineEquation}. Consider the equation
\[
 m\cdot x+n\cdot y=1~.
\]
Now use a variation of the Euclidean\index{Euclid} algorithm, letting $m=r_1$ and $n=r_2$
\begin{eqnarray*}
  r_1     &=& q_1\cdot r_2+r_3~, \\
  r_2     &=& q_2\cdot r_3+r_4~, \\
  \vdots  & & \vdots \\
  r_{n-3} &=& q_{n-3}\cdot r_{n-2}+r_{n-1}~, \\
  r_{n-2} &=& q_{n-2}\cdot r_{n-1}+1.
\end{eqnarray*}
Starting from the bottom gives
\begin{eqnarray*}
  1 &=& r_{n-2}-q_{n-2}\cdot r_{n-1} \\
  r_{n-1} &=& r_{n-3}-q_{n-3}\cdot r_{n-2},\\
  r_{n-2} &=& r_{n-4}-q_{n-4}\cdot r_{n-3}~, \\
  \vdots  & & \vdots\\
  n=r_2 &=& r_4-q_4\cdot r_3~,\\
  m=r_1 &=& r_3-q_3\cdot r_2
\end{eqnarray*}
so
\begin{eqnarray*}
  1 &=& r_{n-2}-q_{n-2}\cdot r_{n-1} \\
    &=& r_{n-2}-q_{n-2}(r_{n-3}-q_{n-3}\cdot r_{n-2}) \\
    &=& -q_{n-2}\cdot r_{n-3}+(1+q_{n-2}\cdot q_{n-3})r_{n-2} \\
    &=& -q_{n-2}\cdot r_{n-3}+(1+q_{n-2}\cdot q_{n-3})(r_{n-4}-q_{n-4}\cdot r_{n-3})\\
    &=& (1+q_{n-2}\cdot q_{n-3})r_{n-4}-(q_{n-2}+q_{n-4}+q_{n-2}\cdot q_{n-3}\cdot q_{n-4}) r_{n-3} \\
    &=& \ldots~.
\end{eqnarray*}
Continue this procedure all the way back to the top.

\section{General linear Diophantine equation}

The utility of this section is that it establishes if the number of natural solutions of a general linear equation is limited or not. We will show also a method of solving, using integer numbers, the equation $ax-by=c$ (which represents a generalization of lemmas $1$ and $2$ of \citep{Andrica+Andreescu1981}), an example of solving a linear equation with $3$ unknowns in $\Na$, and some considerations on solving, using natural numbers, equations with $n$ unknowns.

Let\textquoteright{s} consider the equation:
\begin{equation}
  a\cdot x=b~,
\end{equation}
where $a\in\Int^n$, $b\in\Int$ or in explicit form
\begin{equation}\label{Eq1}
\sum_{i=1}^na_ix_i=b~,
\end{equation}
with all $a_i,b\in\Int$, $a_i\neq0$ and the greatest common factor
\begin{equation}\label{gcf}
  d=gcf(a_1,a_2,\ldots,a_n)~.
\end{equation}

\begin{obs}
  The notion of $gcd$ (greatest common divisor) is the same with the notion of $gcf$ (greatest common factor) for numbers, $gcf$ being used for algebraic expressions.
  \begin{enumerate}
    \item the notion of $gcf$ refers to numbers and algebraic expressions, for example: $gcf(2abc,8a^2b,10abc)=2ab$.
    \item $gcd$ refers only to numbers, for example: $gcd(2,8,10)=2$.
  \end{enumerate}
  Analogously, the notion of $lcm$ (least common multiple) is the same with the notion of $lcf$ (least common factory) for numbers, $lcf$ being used for algebraic expressions.
  \begin{enumerate}
    \item the notion of $lcf$ refers to numbers and algebraic expressions, for example: $lcf(2abc,8a^2b,10abc)=40a^2b$.
    \item $lcm$ refers only to numbers, for example: $lcm(2,8,10)=40$.
  \end{enumerate}
\end{obs}

\begin{lem}\label{Lemma1OnSolving}
  The equation \emph{(\ref{Eq1})} admits at least a solution in the set of integer, if $d$, \emph{(\ref{gcf})}, divides $b$.
\end{lem}
\begin{proof}
  This result is classic.
\end{proof}

In (\ref{Eq1}), one does not diminish the generality by considering
\[
 gcf(a_1,a_2,\ldots,a_n)=1~,
\]
because in the case when $d\neq1$, one divides the equation by this number; if the division is not an integer, then the equation does not admit natural solutions.

It is obvious that each homogeneous linear equation admits solutions in $\Na$: at least the banal solution!

\subsection[The number of solutions of equation]{The number of solutions of equation\\
$a_1\cdot x_1+a_2\cdot x_2+\cdots+a_n\cdot x_n=b$}

We will introduce the following definition:
\begin{defn}
  The equation (\ref{Eq1}) has variations of sign if there are at least two coefficients $a_i,a_j$, with $1\le i,j\le n$, such that $sign(a_i,a_j)=-1$.
\end{defn}

\begin{lem}\label{Lemma2OnSolving}
  An equation \emph{(\ref{Eq1})} which has sign variations admits an infinity of natural solutions.
\end{lem}
\begin{obs}
  Lemma \ref{Lemma2OnSolving} generalization of lemma 1 of \citep{Andrica+Andreescu1981}.
\end{obs}
\begin{proof}
  From the hypothesis of the lemma it results that the equation has $h$ no null positive terms, $1\le h\le n$, and $k=n-h$ non null negative terms. We have $1\le k\le n$; it is supposed that the first $h$ terms are positive and the following $k$ terms are negative (if not, we rearrange the terms).

  We can then write:
  \[
   \sum_{i=1}^ha_ix_i-\sum_{j=h+1}^na'_jx_j=b\ \ \textnormal{where}\ \ a'_j=-a_j>0~.
  \]

 Let\textquoteright{s} consider $0<M=lcm(a_1,a_2,\ldots,a_n)$, the least common multiple, and $c_i=\abs{M/a_i}$, $i\in \I{n}=\set{1,2,\ldots,n}$.

 Let\textquoteright{s} also consider $0<P=lcm(h,k)$, the least common multiple, and $h_1=P/h$, and $k_1=P/k$. Taking
 \[\left\{
  \begin{array}{ll}
    x_t=h_1c_t\cdot z+x_t^0~, & 1\le t\le h\\
    x_j=k_1c_j\cdot z+x_j^0~, & h+1\le j\le n
  \end{array}\right.
 \]
 where $z\in\Na$,
 \[
  z\ge \max_{1\le t\le h<j\le n}\set{\left[\frac{-x_t^0}{h_1c_t}\right],\left[\frac{x_j^0}{k_1c_j}\right]}+1~,
 \]
 where $[\gamma]$ meaning integer part of $\gamma$, i.e. the greatest integer less than or equal to $\gamma$, and $x_i^0$, $i\in \I{n}$, a particular integer solution (which exists according to lemma \ref{Lemma1OnSolving}), we obtain an infinity of solutions in the set of natural numbers for the equation (\ref{Eq1}).
\end{proof}

\begin{lem}\
  \begin{enumerate}
    \item An equation \emph{(\ref{Eq1})} which does not have variations of sign has at maximum a limited number of natural solutions.
    \item In this case, for $b\neq0$, constant, the equation has the maximum number of solutions if and only if all $a_i=1$ for $i\in \I{n}$.
  \end{enumerate}
\end{lem}
\begin{proof}\
  \begin{enumerate}
    \item One considers all $a_i>0$ (otherwise, multiply the equation by $-1$).
    \begin{itemize}
      \item[] If $b<0$, it is obvious that the equation does not have any solution in $\Na$.
      \item[] If $b=0$, the equation admits only the trivial solution.
      \item[] If $b>0$, then each unknown $x_i$, takes positive integer values between $0$ and $d_i=b/a_i$ (finite), and not necessarily all these values. Thus the maximum number of solutions is lower or equal to $\prod_{i=1}^n(1+d_i)$, which is finite.
    \end{itemize}
    \item For $b\neq0$, constant, $\prod_{i=1}^n(1+d_i)$ is maximum if and only if $d_i$ are maximum, i.e. \emph{if} $\ a_i=1$ for all $i\in\I{n}$.
  \end{enumerate}
\end{proof}

\begin{thm}
  The equation \emph{(\ref{Eq1})} admits an infinity of natural solutions if and only if it has variations of sign.
\end{thm}
\begin{proof}
  This naturally follows from the previous results.
\end{proof}

\subsection{Diophantine equation of first order with two unknown}

\begin{thm}\label{Teorema2OnSolving}
  Let\textquoteright{s} consider the equation $ax-by=c$, with integer coefficients, where $a>0$ and $b>0$ and $(a,b)=gcd(a,b)=1$. Then the general solution in natural numbers of this equation is:
  \begin{equation}\label{SolEq2}
    \left\{
   \begin{array}{l}
     x=bk+x_0 \\
     y=ak+y_0
   \end{array}\right.
  \end{equation}
  where $(x_0,y_0)$ is a particular integer solution of the equation, and
  \[
   k\ge\max\set{\left\lceil\frac{-x_0}{b}\right\rceil,\left\lceil\frac{-y_0}{a}\right\rceil}
  \]
  is an integer parameter.
\end{thm}
\begin{obs}
  The theorem \ref{Teorema2OnSolving} generalization of lemma 2 of \citep{Andrica+Andreescu1981}.
\end{obs}
\begin{proof}
  It results from \citep{Creanga+Cazacu+Mihut+Opait+Reischer1965} that the general integer solution of the equation is (\ref{SolEq2}).
  Since $x$ and $y$ are natural integers, it is necessary for us to impose conditions for $k$ such that $x\ge0$ and $y\ge0$, from which it results the theorem \ref{Teorema2OnSolving}.
\end{proof}

  The solve in the set of natural numbers a linear equation with $n$ unknowns we will use the previous results in the following way:
  \begin{enumerate}
    \item[(a)] If equation does not have variations of sign, because it has a limited number of natural solutions, the solving is made by tests.
    \item[(b)] If it has variations of sign and if $b$ is divisible by $d$, then it admits an infinity of natural solutions. One finds its general integer solution, see \citep{Ion+Nita1978};
        \[
         x_i=\sum_{j=1}^{n-1}\alpha_{ij}k_j+\beta_j~,\ \ i\in\I{n}~,
        \]
        where all the $\alpha_{ij},\beta_j\in\Int$ and the $k_j$ are integer parameters.

        By applying the restriction $x_i\ge0$ for $i\in\I{n}$, one finds the conditions which must be satisfied by the parameters $k_j$:
        \begin{equation}\label{ConditiaC}
          k_j\in\Int,\ \ \textnormal{for all}\ \ j\in\I{n-1}.
        \end{equation}
  \end{enumerate}
The case $n=2$ and $n=3$ can be done by this method, but when $n$ is bigger, the conditions (\ref{ConditiaC}) becomes more and more difficult to find.

\begin{exem}
Solve in $\Na$ the equation $3x-7y+2z=-18$.

Solution: In $\Int$ one obtains the general integer solution:
\[\left\{
 \begin{array}{l}
   x=k_1 \\
   y=k_1+2k_2 \\
   z=2k_1+7k_2-9
 \end{array}\right.~,
\]
with $k_1$ and $k_2$ in $\Int$.

From the conditions (\ref{ConditiaC}) result the inequalities $x\ge0$, $y\ge0$, $z\ge0$. It results that $k_1\ge0$ and also:
\begin{itemize}
  \item[]
  \begin{itemize}
    \item[] $k_2\ge-\dfrac{k_1}{2}$ if $-\dfrac{k_1}{2}\notin\Int$,
    \item[] or
    \item[] $k_2\ge-\dfrac{k_1}{2}$ if $-\dfrac{k_1}{2}\in\Int$;
  \end{itemize}
  \item[] and
  \begin{itemize}
    \item[] $k_2\ge\dfrac{9-2k_1}{7}+1$ if $\dfrac{9-2k_1}{7}\notin\Int$,
    \item[] or
    \item[] $k_2\ge\dfrac{9-2k_1}{7}$ if $\dfrac{9-2k_1}{7}\in\Int$;
  \end{itemize}
  \item[] that is
  \begin{itemize}
    \item[] $k_2\ge\dfrac{2-2k_1}{7}+2$ if $\dfrac{2-2k_1}{7}\notin\Int$,
    \item[] or
    \item[] $k_2\ge\dfrac{2-2k_1}{7}+1$ if $\dfrac{2-2k_1}{7}\in\Int$.
  \end{itemize}
\end{itemize}
With these conditions on $k_1$ and $k_2$ we have the general solution in natural numbers of the equation.
\end{exem}

\subsubsection{Procedure for solving Diophantine equations of first order \\with two unknowns}

For automatically solving the Diophantine equations of order 1
with 2 unknowns $ax-by=c$ we need the following program.
\begin{prog} Program for finding a solution.
\begin{tabbing}
  $S12(a,b,c):=$\=\ \vline\ $return\ "Error (a,b)\neq1"\ if\ \gcd(a,b)\neq1$\\
  \>\ \vline\ $m\leftarrow10^6$\\
  \>\ \vline\ $f$\=$or\ x\in1..m$\\
  \>\ \vline\ \>\ $f$\=$or\ y\in floor(\frac{ax-c-1}{b})..ceil(\frac{ax-c+1}{b})$\\
  \>\ \vline\ \>\ \>\ $return\ \left(\begin{array}{cc}
                                   x & y
                                 \end{array}\right)^\textrm{T}\ if\ a\cdot x-b\cdot y-c\textbf{=}0$\\
  \>\ \vline\ $return\ "Not\ found\ a\ solution"$
\end{tabbing}
\end{prog}
\begin{exem} We consider the Diophantine equation on the set of natural numbers $1245x-365y=4567$. This case is solvable, as $gcd(a,b)=1$.
  \[
   a:=124\ \ \ b:=365\ \ \ c:=4567\ \ gcd(a,b)=1
  \]
  \[
   \left(\begin{array}{c}
       x_0 \\
       y_0 \\
     \end{array}\right):=S12(a,b,c)=\left(\begin{array}{c}
                                       28 \\
                                       -3 \\
                                     \end{array}\right)
  \]
  \[
   k_0:=\max\left(\left(\begin{array}{c}
                      ceil\left(\frac{-x_0}{b}\right) \\
                      ceil\left(\frac{-y_0}{a}\right)
                    \end{array}\right)\right)=1
  \]
  \[
   \begin{array}{c}
     x(k):=b\cdot k+x_0 \\
     y(k):=a\cdot k+y_0
   \end{array}
  \]
  for $k:=k_0..10$ we obtain the solutions
  \[x(k)\rightarrow
   \left(\begin{array}{r}
           393\\
           758\\
          1123\\
          1488\\
          1853\\
          2218\\
          2583\\
          2948\\
          3313\\
          3678
     \end{array}\right)
     \ \ \ y(k)\rightarrow
     \left(\begin{array}{r}
          121\\
          245\\
          369\\
          493\\
          617\\
          741\\
          865\\
          989\\
          1113\\
          1237
     \end{array}\right)~.
  \]
\end{exem}

\subsubsection{Solving of Diophantine equation $a_1\cdot x_1+a_2\cdot x_2+\ldots+a_n\cdot x_n=b$}

In this section we will present the problem of solving for integers equations of the form:
\begin{equation}\label{Ec6}
  a_1\cdot x_1+a_2\cdot x_2+\ldots+a_n\cdot x_n=b
\end{equation}
where $a_k,b\in\Int$ for $k\in\I{n}$.

We suppose that not all the numbers $a_k$, for $k\in\I{n}$, are null. Obviously, to exist an integer solution of the equation
(\ref{Ec6}) it is necessary that
\[
 d=(a_1,a_2,\ldots,a_n)=\gcd(a_1,a_2,\ldots,a_n)\mid b~.
\]
We will prove that this condition is also sufficient.

Let $a_k'=a_k/d$, $k\in\I{n}$ and $b'=b/d$. We consider following equation equivalent with (\ref{Ec6})
\begin{equation}\label{Ec7}
  a_1'\cdot x_1+a_2'\cdot x_2+\ldots+a_n'\cdot x_n=b'\
\end{equation}
then $(a_1',a_2',\ldots,a_n')=1$. Let $a_k'$, $a_j'$ be non-null numbers with $k<j$ and $\abs{a_k'}>\abs{a_j'}$. According to the
Theorem of division with remainder, \citep{Burton2010}, there exist the numbers $q$ and $r$ such that
\[
 a_k'=a_j'\cdot q+r
\]
and, by substituting $a_k'$ in (\ref{Ec7}) equation
\begin{equation}\label{Ec8}
  a_1'\cdot x_1+\ldots+r\cdot x_k+\ldots+a_j'(x_j+q\cdot x_k)+\ldots+a_n'\cdot x_n=b'\
\end{equation}
is obtained.

Equation (\ref{Ec8}) can be written as:
\begin{equation}\label{Ec9}
  a_1''\cdot x_1''+\ldots+a_n''\cdot x_n''=b'\
\end{equation}
where
\[a_i''=\left\{
 \begin{array}{ll}
   a_i'~, & i\neq k\\
   r~,    & i=k
 \end{array}\right.~,
 \ \ \ x_i''=\left\{
   \begin{array}{ll}
     x_i~, & i\neq k\\
     x_j+q\cdot x_k~, & i=k
   \end{array}\right.~.
\]

It can be easily observed that there exists a one-to-one correspondence between the solutions of equations (\ref{Ec7}) and
(\ref{Ec9}). Furthermore, knowing the solutions of equation (\ref{Ec9}) and taking into account the previous transformations,
the solutions of equation (\ref{Ec7}) can also be given.

We mention that, for every $i,k\in\I{n}$, $i\neq k$ following relations hold:
\[
 a_i''=a_i'\ \ \textnormal{and}\ \ \abs{a_k''}<\abs{a_k'}~.
\]
Additionally, we have
\[
(a_1'',\ldots,a_n'')=(a_1',\ldots, a_k'-a_j'\cdot q,\ldots,a_n')= (a_1',\ldots,a_n')=1.
\]

After summing all the previous relations, we conclude that equation (\ref{Ec7}) can be reduced to the form
\begin{equation}\label{Ec10}
  \widetilde{a_1}\cdot\widetilde{x_1}+\ldots+\widetilde{a_n}\cdot\widetilde{x_n}=b'\
\end{equation}
after a finite number of steps, where $\widetilde{a_i}$ with $i\in\I{n}$ are non-null numbers, whose absolute values are pairwise distinct.

Hence, we deduce that the numbers $\widetilde{a_i}$, $i\in\I{n}$ have only the values $0$ or $\pm1$ and are not all null. Without
any loss of generality of the problem, we suppose that $\widetilde{a_1}=1$. Then equation (\ref{Ec10}) has following solutions
\[\left\{\begin{array}{l}
           \widetilde{x_1}=b'-\widetilde{a_2}\cdot t_2-\ldots-\widetilde{a_n}\cdot t_n \\
           \widetilde{x_2}=t_2 \\
           \ldots \\
           \widetilde{x_n}=t_n
         \end{array}\right.
\]
where $t_2,t_3,\ldots,t_n$ are arbitrary integers. Using the transformations done along the previous reasonings, the solutions
of equation (\ref{Ec7}) are also obtained.

We insist on mentioning that in solving equation (\ref{Ec10}) the fact that $\widetilde{a_1}=1$ was used, and, therefore, if at a
certain step of the indicated algorithm an equation with at least one coefficient equal to $\pm1$ is obtained, the solution of this
equation can be written similarly with the solution of the equation (\ref{Ec10}).

\section{Solving the Diophantine linear systems}

More generally, every system of linear Diophantine equations may be solved by computing the \emph{Smith normal form} of its matrix, in a way that is similar to the use of the \emph{reduced row echelon form} to solve a system of linear equations over a field.

ABS algorithm for solving linear Diophantine equations, \cite{Gao+Dong2008} introduce an algorithm for solving a system of $m$
linear integer inequalities in $n$ variables, $m\le n$, with full rank coefficient matrix.

\subsection{Procedure of solving with row--reduced echelon form}

Echelon form (or row echelon form) is:
\begin{enumerate}
  \item All nonzero rows are above any rows of all zeros.
  \item Each leading entry (i.e. leftmost nonzero entry) of a row is in a column to the right of leading entry of the row above it.
  \item All entries in a column below a leading entry are zero.
\end{enumerate}
\begin{exem} Echelon forms:
  \[
   \left(\begin{array}{ccccc}
           \blacksquare & \ast & \ast & \ast & \ast \\
           0 & \blacksquare & \ast & \ast & \ast \\
           0 & 0 & 0 & 0 & 0 \\
           0 & 0 & 0 & 0 & 0 \\
         \end{array}\right)~,\ \ \
   \left(\begin{array}{ccc}
           \blacksquare & \ast & \ast \\
           0 & \blacksquare & \ast \\
           0 & 0 & \blacksquare \\
           0 & 0 & 0
     \end{array}\right)~,
  \]
  \[
   \left(\begin{array}{ccccccccccc}
           0 & \blacksquare & \ast & \ast & \ast & \ast & \ast & \ast & \ast & \ast & \ast \\
           0 & 0 & 0 & \blacksquare & \ast & \ast & \ast & \ast & \ast & \ast & \ast \\
           0 & 0 & 0 & 0 & \blacksquare & \ast & \ast & \ast & \ast & \ast & \ast \\
           0 & 0 & 0 & 0 & 0 & 0 & 0 & \blacksquare & \ast & \ast & \ast \\
           0 & 0 & 0 & 0 & 0 & 0 & 0 & 0 & \blacksquare & \ast & \ast \\
         \end{array}\right)~.
  \]
  where we noted with  $\blacksquare$ any nonzero integer and with $\ast$ any integer.
\end{exem}
Reduced echelon form: Add the following conditions to conditions $1$, $2$ and $3$ above.
\begin{enumerate}
  \item [4.] The leading entry in each nonzero row is$1$.
  \item [5.] Each leading $1$ is the only nonzero entry in its column.
\end{enumerate}
A matrix is in reduced row echelon form, also called \emph{row canonical form}, if it satisfies the following conditions, \citep{Meyer2000}.

\begin{exem} Reduced echelon form:
  \[
   \left(\begin{array}{ccccccccccc}
           0 & 1 & \ast & 0 & 0 & \ast & \ast & 0 & 0 & \ast & \ast \\
           0 & 0 & 0    & 1 & 0 & \ast & \ast & 0 & 0 & \ast & \ast \\
           0 & 0 & 0    & 0 & 1 & \ast & \ast & 0 & 0 & \ast & \ast \\
           0 & 0 & 0    & 0 & 0 & 0    & 0    & 1 & 0 & \ast & \ast \\
           0 & 0 & 0    & 0 & 0 & 0    & 0    & 0 & 1 & \ast & \ast \\
         \end{array}\right)~.
  \]
\end{exem}

The theorem uniqueness of the reduced echelon form, \citep{Nakos+Joyner1998,Meyer2000}:
\begin{thm}
  Each matrix is row-equivalent to one and only one reduced echelon matrix.
\end{thm}

\begin{defn}
  Pivot position is a position of a leading entry in an echelon form of the matrix.
\end{defn}
\begin{defn}
  Pivot is a nonzero number that either is used in a pivot position to create $0'$s or is changed into a leading $1$, which in turn is used to create $0'$s.
\end{defn}
\begin{defn}
  Pivot column is a column that contains a pivot position.
\end{defn}

The theorem existence and uniqueness, \citep{Nakos+Joyner1998,Meyer2000}.
\begin{thm}\
 \begin{enumerate}
   \item A linear system is consistent if and only if the rightmost column of augmented  matrix is not a pivot column (i.e. if and only if an echelon form of the augmented matrix has no row of the form $[0\ 0\ \ldots\ 0\ b]$, where $b\neq0$).
   \item If a linear system is consistent, then the solution contains either
   \begin{enumerate}
     \item a unique solution (when there are no free variables) or
     \item infinitely many solutions (when there is at least one free variable).
   \end{enumerate}
 \end{enumerate}
\end{thm}

\begin{alg} The algorithm using reduced row echelon form to solve linear system:
\begin{enumerate}
  \item Write the augmented matrix of the system.
  \item Use the row reduction algorithm to obtain equivalent augmented matrix in echelon form. Decide whether the system is consistent. If not stop; otherwise go to the next step.
  \item Continue row reduction to obtain the reduced echelon form.
  \item Write the system of equations corresponding to the matrix obtained in step $3$.
  \item State the solution by expressing each basic variable in terms of free variables and declare the free variables.
\end{enumerate}
\end{alg}

\begin{exem} Solving by means of symbolic computation of a Diophantine linear system using the method \emph{row reduced echelon form}.
  We consider the origin of the vectors and matrices equal to $1$.
  \[
   ORIGIN:=1
  \]

  We consider the system $A\cdot x=b$, where
  \[
    A:=\left(\begin{array}{rrrrr}
              0 & 3 & -6 & 6 & 4 \\
              3 & -7 & 8 & -5 & 8 \\
              3 & -9 & 12 & -9 & 6 \\
            \end{array}\right)\ \ \
   b:=\left(\begin{array}{r}
          -5 \\
          9 \\
          15 \\
        \end{array}\right)~.
  \]

  We concatenate matrix $A$ to the free term $b$ and we determine the number of lines and columns of matrix $E$
  \[
   E:=augment(A,b)\ \ n:=rows(E)\rightarrow3\ \ cols(E)\rightarrow6~.
  \]

  By means of the Mathcad function $rref$ we determine the \emph{row reduced echelon form} of matrix $E$.
  \[
   R:=rref(E)\rightarrow\left(\begin{array}{rrrrrr}
                                1 & 0 & -2 & 3 & 0 & 24 \\
                                0 & 1 & -2 & 2 & 0 & -7 \\
                                0 & 0 & 0 & 0 & 1 & 4 \\
                              \end{array}\right)~.
  \]

  According to this matrix, it follows that the main unknowns are $x_1$, $x_2$ and $x_5$, while $x_3$ and $x_4$ are the secondary unknowns.
  \begin{multline*}
   S(x_1,x_2,x_3,x_4,x_5):=\\
   submatrix(R,1,n,1,m-1)\cdot
   \left(\begin{array}{c}
           x_1 \\
           x_2 \\
           x_3 \\
           x_4 \\
           x_5
         \end{array}\right)-R^{\langle m\rangle}\rightarrow\\
   \left(\begin{array}{c}
       x_1-2x_3+3x_4+24 \\
       x_2-2x_3+2x_4+7 \\
       x_5-4
     \end{array}\right)
  \end{multline*}

  We determine $x_1$, $x_2$ and $x_5$ relative to $x_3$ and $x_4$.
  \begin{multline*}
   \left(\begin{array}{ccc}
       x_1 & x_2 & x_5 \\
     \end{array}\right):=\\
     S(x_1,x_2,x_3,x_4,x_5)solve
     \left(\begin{array}{ccc}
             x_1 & x_2 & x_5
           \end{array}\right)\rightarrow\\
     \left(\begin{array}{ccc}
             2x_3-3x_4-24 & 2x_3-2x_4-7 & 4 \\
       \end{array}\right)~.
  \end{multline*}

  We verify if the obtained solution satisfies the equation
  \[
    A\cdot\left(\begin{array}{c}
                 x_1 \\
                 x_2 \\
                 x_3 \\
                 x_4 \\
                 x_5
               \end{array}\right)-b=0~,
  \]
  \[
   A\cdot\left(\begin{array}{c}
                 2x_3-3x_4-24 \\
                 2x_3-2x_4-7 \\
                 x_3 \\
                 x_4 \\
                 4
               \end{array}\right)-b\rightarrow
          \left(\begin{array}{c}
                  0 \\
                  0 \\
                  0
            \end{array}\right)~.
  \]

  It is obvious that for different integer values of $x_3$ and $x_4$ there result integer solutions for the linear system.
\end{exem}

\begin{exem} Linear Diophantine system that has a unique solution. We consider the linear
system with
  \[
   A:=\left(\begin{array}{rr}
              3 & 4 \\
              2 & 5 \\
              -2 & -3
            \end{array}\right)\ \ \
     \left(\begin{array}{r}
             -3 \\
             5 \\
             1\\
           \end{array}\right)~.
  \]

  We consider matrix $E$ and we count the number of lines and columns of matrix $E$.
  \[
   E:=augment(A,b)\ \ n:=rows(E)\rightarrow3\ \ cols(E)\rightarrow3~.
  \]

  The matrix \emph{row reduced echelon form} is
  \[
   R:=rref(E)\rightarrow\left(\begin{array}{rrr}
                                1 & 0 & -5 \\
                                0 & 1 & 3  \\
                                0 & 0 & 0  \\
                              \end{array}\right)~.
  \]

  We compute
  \begin{multline*}
   S(x_1,x_2):=\\
   submatrix(R,1,n,1,m-1)\cdot\left(\begin{array}{c}
                                              x_1 \\
                                              x_2
                                    \end{array}\right)-R^{\langle m\rangle}\rightarrow
   \left(\begin{array}{c}
           x_1+5 \\
           x_2-3 \\
           0
         \end{array}\right)~.
  \end{multline*}

  From this result follows that the main unknowns are $x_1$ and $x_2$ which do not depend on any other variable and we have $x_1=-5$ and $x_2=3$.

  Indeed, it is verified that
  \[
   A\cdot\left(\begin{array}{r}
                -5 \\
                3 \\
               \end{array}\right)-b=
           \left(\begin{array}{c}
                   0 \\
                   0 \\
                   0
                 \end{array}\right)~.
  \]
\end{exem}

\begin{exem} Linear Diophantine system that has no solutions. We consider matrix $A$ and vector $b$
  \[
   A:=\left(\begin{array}{rrr}
              1 & 1 & 1 \\
              1 & 2 & 4 \\
              1 & 3 & 9 \\
              1 & 4 & 16 \\
            \end{array}\right)\ \ \
      b:=\left(
        \begin{array}{r}
          -1 \\
          3 \\
          3 \\
          5 \\
        \end{array}
      \right)~.
  \]

  Let be matrix $E$
  \[
   E:=augment(A,b)\ \ n:=rows(E)\rightarrow4\ \ m:=cols(E)\rightarrow4~.
  \]

  We compute the matrix \emph{row-reduced echelon form} by means of function $rref$
  \[
   R:=rref(E)\rightarrow\left(\begin{array}{rrrr}
                                1 & 0 & 0 & 0 \\
                                0 & 1 & 0 & 0 \\
                                0 & 0 & 1 & 0 \\
                                0 & 0 & 0 & 1 \\
                              \end{array}\right)~.
  \]

  Following calculation is done
  \begin{multline*}
   S(x_1,x_2,x_3):=\\
   submatrix(R,1,n,1,m-1)\cdot\left(\begin{array}{c}
                                              x_1 \\
                                              x_2 \\
                                              x_3
                                    \end{array}\right)-R^{\langle m\rangle}\rightarrow
   \left(\begin{array}{c}
           x_1 \\
           x_2 \\
           x_3 \\
           -1
         \end{array}\right)~.
  \end{multline*}

  We try to solve by means of symbolic computation, using function $solve$, the equation $S(x_1,x_2,x_3)\textbf{=}0$ :
  \[
   \begin{array}{ll}
     S(x_1,x_2,x_3)\textbf{=}0\ solve(x_1,x_2,x_3)&\rightarrow \\
     & \boxed{No\ solution\ was\ found}
   \end{array}~.
  \]

  The Mathcad's answer is: \emph{No solution was found}.
\end{exem}

\begin{exem} Linear Diophantine system with matrix $A$ and free term $b$
  \[A:=\left(\begin{array}{rrrrrr}
               1 &  2 &   3 &    4 &    5 &     6 \\
               1 &  4 &   9 &   16 &   25 &   -36 \\
               1 &  8 &  27 &   64 &  125 &   216 \\
               1 & 16 &  81 &  256 &  625 & -1296 \\
               1 & 32 & 243 & 1024 & 3125 &  7776
             \end{array}\right)\ \ \ b:=
       \left(\begin{array}{r}
                 104 \\
                -140 \\
                2750 \\
               -7952 \\
               87374
    \end{array}\right)~.
  \]

  Matrix $E$ is obtained by concatenating vector $b$ to matrix $A$.
  \[
   E:=augment(A,b)\ \ n:=rows(E)\rightarrow5\ \ m:=cols(E)\rightarrow7~.
  \]

  We compute the matrix \emph{row-reduced echelon form} by means of function $rref$
  \[
   R:=rref(E)\rightarrow\left(\begin{array}{rrrrrrr}
                                1 & 0 & 0 & 0 & 0 & 1980 & 17833 \\
                                0 & 1 & 0 & 0 & 0 & -3465 & -31185 \\
                                0 & 0 & 1 & 0 & 0 & 3080 & 27719 \\
                                0 & 0 & 0 & 1 & 0 & -1386 & -12469 \\
                                0 & 0 & 0 & 0 & 1 & 252 & 2272
                              \end{array}\right)~.
  \]

  We compute
  \begin{multline*}
   S(y_1,y_2,y_3,y_4,y_5,y_6):=\\
   submatrix(R,1,n,1,m-1)\cdot\left(\begin{array}{c}
                                              y_1 \\
                                              y_2 \\
                                              y_3 \\
                                              y_4 \\
                                              y_5 \\
                                              y_6 \\
                                    \end{array}\right)-R^{\langle m\rangle}\rightarrow\\
   \left(\begin{array}{l}
           y_1+1980\cdot y_6-17833 \\
           y_2-3465\cdot y_6+31185 \\
           y_3+3080\cdot y_6-27719 \\
           y_4-1386\cdot y_6+12469 \\
           y_5+252\cdot y_6-2272
         \end{array}\right)~.
  \end{multline*}

  We solve by means the symbolic computation, using function $solve$, equation $S(y_1,y_2,y_3,y_4,y_5,y_6)\textbf{=}0$ :
  \begin{multline*}
   \left(\begin{array}{c}
       y_1 \\
       y_2 \\
       y_3 \\
       y_4 \\
       y_5
     \end{array}\right):=S(y_1,y_2,y_3,y_4,y_5,y_6)\textbf{=}0\ solve
     \left(\begin{array}{c}
             y_1 \\
             y_2 \\
             y_3 \\
             y_4 \\
             y_5
           \end{array}\right)\rightarrow\\
     \left(\begin{array}{r}
             17833-1980\cdot y_6 \\
             -31185+3465\cdot y_6 \\
             27719-3080\cdot y_6 \\
             -12469+1386\cdot y_6 \\
             2272-252\cdot y_6
       \end{array}\right)~.
  \end{multline*}
\end{exem}

\subsection{Solving with Smith normal form}

Using matrix notation every system of linear Diophantine equations may be written
\[
 A\cdot X=C
\]
where $A$ is a $m\times n$ matrix of integers, $X$ is a $n\times1$ column matrix of unknowns and $C$ is a $m\times1$ column matrix of integers.

The computation of the Smith normal form of $A$ provides two unimodular matrices (that is matrices that are invertible over the integers, which have $\pm1$ as determinant) $U$ and $V$ of respective dimensions m×m and $n\times n$, such that the matrix
\[
 B=\left[b_{i,j}\right]=UAV
\]
is such that $b_{i,i}$ is not zero for $i$ not greater than some integer $k$, and all the other entries are zero. The system to be solved may thus be rewritten as
\[
 B(V^{-1}\cdot X)= U\cdot C.
\]
Calling $y_i$ the entries of $V^{-1}\cdot X$ and $d_i$ those of $D=U\cdot C$, this leads to the system
\[
 \begin{array}{ll}
   b_{i,i}\cdot y_i=d_i & \textnormal{for}\ 1\le i\le k \\
   0\cdot y_i=d_i &  \textnormal{for}\ k<i\le n
 \end{array}
\]
This system is equivalent to the given one in the following sense: $A$ column matrix of integers $x$ is a solution of the given system if and only if $x=V\cdot y$ for some column matrix of integers $y$ such that $By=D$.

It follows that the system has a solution if and only if $b_{i,i}$ divides $d_i$ for $i\le k$ and $d_i=0$ for $i> k$. If this condition is fulfilled, the solutions of the given system are
\[V\cdot\left(
          \begin{array}{c}
            \frac{d_1}{b_{1,1}} \\
            \vdots \\
            \frac{d_k}{b_{k,k}} \\
            h_{k+1} \\
            \vdots \\
            h_n \\
          \end{array}
        \right)
\]
where $h_{k+1},\ldots,h_{n}$ are arbitrary integers, \citep{Schmidt1991,Lazebnik1996,Smart1998}.

\section[Solving the Diophantine equation of order $n$]{Solving the Diophantine equation\\ of order $n$ with an unknown}

The Diophantine equation of order $n$ with a single unknown is
\begin{equation}\label{EcDiofOrdn}
  P(x)=a_nx^n+a_{n-1}x^{n-1}+\ldots+a_1x+a_0=0~,
\end{equation}
where $a_k\in\Int$, $a_n\neq0$. The problem that arises is to find the solutions that are rational numbers ($\mathbb{Q}$) or integers
($\Int$) or natural numbers ($\Na$). As it is well known, the Fundamental Theorem of Algebra, \citep{Krantz1999a}, assures the
existence of $n$ complex solutions for the algebraic equation of order $n$. Therefore, the Diophantine equation (\ref{EcDiofOrdn})
can not have more than $n$ rational, integer or natural solutions.

Leaving from "Vieta\textquoteright{s} formula", \citep{Viete1579,Girard1884}:
\[
 s_1\cdot s_2\cdots s_n=(-1)^n\frac{a_0}{a_n}~,
\]
where $s_k$ are the roots of polynomial $P$, which means $P(s_k)=0$, for $k\in\I{n}$, it results a classic method. This
method supposes finding all the divisors of $a_n$ and of $a_0$. Afterwards, the set of numbers that divide $a_n/a_0$ is generated
(finite set, $\le\sigma_0(a_n)\sigma_0(a_0)$) and it is tested which of those divisors are roots of polynomial $P$.

We give an automatic procedure for finding the rational, integer or natural solutions which uses following $3$ programs.
\begin{prog}\label{ProgDivizori} Program to find the divisors of a natural number.
  \begin{tabbing}
    $Div(m):=$\=\ \vline\ $j\leftarrow0$\\
    \>\ \vline\ $d_j\leftarrow1$\\
    \>\ \vline\ $f$\=$or\ k\in2..floor(\frac{m}{2})$\\
    \>\ \vline\ \>\ $i$\=$f\ mod(m,k)\textbf{=}0$\\
    \>\ \vline\ \>\ \>\ \vline\ $j\leftarrow j+1$\\
    \>\ \vline\ \>\ \>\ \vline\ $d_j\leftarrow k$\\
    \>\ \vline\ $d\leftarrow stack(d,m)$\\
    \>\ \vline\ $return\ d$
  \end{tabbing}
\end{prog}

\begin{prog}\label{ProgFactori} Program to find the factors (repetition excluded) of the number $a_0/a_n$. This program calls the program \ref{ProgDivizori}.
  \begin{tabbing}
   $Factori(a):=$\=\ \vline\ $d0\leftarrow Div(\abs{a_0})$\\
   \>\ \vline\ $dn\leftarrow Div(\abs{a_{last(a)}})$\\
   \>\ \vline\ $f\leftarrow d0$\\
   \>\ \vline\ $i\leftarrow last(f)+1$\\
   \>\ \vline\ $f$\=$or\ k\in 0..last(d0)$\\
   \>\ \vline\ \>\ $f$\=$or\ j\in 1..last(dn)$\\
   \>\ \vline\ \>\ \>\ \vline\ $f_i\leftarrow\frac{d0_k}{dn_j}$\\
   \>\ \vline\ \>\ \>\ \vline\ $i\leftarrow i+1$\\
   \>\ \vline\ $f\leftarrow sort(f)$\\
   \>\ \vline\ $j\leftarrow0$\\
   \>\ \vline\ $w_j\leftarrow f_0$\\
   \>\ \vline\ $f$\=$or\ k\in 1..last(f)$\\
   \>\ \vline\ \>\ $i$\=$f\ w_j\neq f_k$\\
   \>\ \vline\ \>\ \>\ \vline\ $j\leftarrow j+1$\\
   \>\ \vline\ \>\ \>\ \vline\ $w_j\leftarrow f_k$\\
   \>\ \vline\ $return\ w$
  \end{tabbing}
\end{prog}

\begin{prog} Program to find the rational solutions for the input parameter $t=0$, the integer solutions for the input parameter $t=1$ and the natural solutions for the input parameter $t=2$. This program calls the programs \ref{ProgFactori} and
\ref{Horner}.
  \begin{tabbing}
    $Sqzn(a,t):=$\=\ \vline\ $retur\ "Error\ t\neq0\wedge1\wedge2"\ if\ t\neq0\wedge t\neq1\wedge t\neq2$\\
    \>\ \vline\ $f\leftarrow Factori(a)$\\
    \>\ \vline\ $fs\leftarrow sort(f)$\\
    \>\ \vline\ $f\leftarrow stack(-reverse(fs),fs)$\\
    \>\ \vline\ $w\leftarrow f\ if\ t\textbf{=}0$\\
    \>\ \vline\ $i$\=$f\ t\textbf{=}1$\\
    \>\ \vline\ \> \vline\ $j\leftarrow0$\\
    \>\ \vline\ \> \vline\ $f$\=$or\ k\in0..last(f)$\\
    \>\ \vline\ \> \vline\ \>\ $i$\=$f\ f_k\textbf{=}trunc(f_k)$\\
    \>\ \vline\ \> \vline\ \>\ \>\ \vline\ $w_j\leftarrow f_k$\\
    \>\ \vline\ \> \vline\ \>\ \>\ \vline\ $j\leftarrow j+1$\\
    \>\ \vline\ $i$\=$f\ t\textbf{=}2$\\
    \>\ \vline\ \> \vline\ $j\leftarrow0$\\
    \>\ \vline\ \> \vline\ $f$\=$or\ k\in0..last(f)$\\
    \>\ \vline\ \> \vline\ \>\ $i$\=$f\ f_k\textbf{=}trunc(f_k)\wedge f_k\ge0$\\
    \>\ \vline\ \> \vline\ \>\ \>\ \vline\ $w_j\leftarrow f_k$\\
    \>\ \vline\ \> \vline\ \>\ \>\ \vline\ $j\leftarrow j+1$\\
    \>\ \vline\ $i\leftarrow0$\\
    \>\ \vline\ $f$\=$or\ k\in0..last(w)$\\
    \>\ \vline\ \>\ $i$\=$f\ Horner(a,w_k)\textbf{=}0$\\
    \>\ \vline\ \>\ \>\ \vline\ $s_i\leftarrow w_k$\\
    \>\ \vline\ \>\ \>\ \vline\ $i\leftarrow i+1$\\
    \>\ \vline\ $return\ s^\textrm{T}$
  \end{tabbing}
\end{prog}
\begin{exem} We consider the polynomial defined by the vector
  \[a:=\left(\begin{array}{ccccccccc}
               -96 & 776 & -1568 & 134 & 1620 & -359 & -466 & 49 & 30
             \end{array}\right)^\textrm{T}~,
  \]
  afterwards we consider the calls
  \[
   \begin{array}{l}
     S(a,2)\rightarrow \left(\begin{array}{c}
                           3
                         \end{array}\right)~,\\ \\
     S(a,1)\rightarrow \left(\begin{array}{ccc}
                           -4 & -2 & 3
                         \end{array}\right)~,\\ \\
     S(a,0)\rightarrow  \left(\begin{array}{cccccc}
                           -4 & -2 & \dfrac{1}{5} & \dfrac{1}{2} & \dfrac{2}{3} & 3
                         \end{array}\right)~.
   \end{array}
  \]
  The first call states that the polynomial defined by vector $a$ has a unique natural solution; the second call tells us that the polynomial has 3 integer solutions; while the third call shows hat the polynomial has 6 integer solutions.
\end{exem}

The second method supposes finding all the roots of the polynomial by means of formula for polynomials of order $n\le4$ and emphasizing the rational, integer or natural roots.

As the polynomials of degree $n>4$ can not be solved with square roots (Impossibility Theorem, \cite{Abel1826,Abel1881,Abel1988}
\index{Abel N. H.} and Galois\index{Galois E.} 1832, \citep{Artin1944}, (this was also shown by Ruffini\index{Ruffini P.} in 1813 \citep{Wells1986}), the roots of the polynomial will be approximated by a numerical method, usually Laguerre's
method,\index{Laguerre E.} \citep{CiraMetodeNumerice2005}(probably the best numerical method for approximating the solutions of a algebraic equation). This method is implemented in most of the mathematics softwares, such as Maple, Mathematica, Matlab, Mathcad.

The approximative roots which are "close" to rational, integer or natural numbers are verified by direct computation (Schema of
\cite{Horner1819}\index{Horner W. G.}, the fastest algorithm for computing the values of an algebraic polynomial) if those numbers
are the solutions of the equation $P(x)=0$. We give an example of how this method is applied.

\begin{exem} Let be the algebraic equation $P(x)=0$, where polynomial $P$ is defined by the vector
\begin{multline*}
  a:=(2074506308666643852,\  -4170138555243755952,\\
   3708600060698625999,\ -2371615921694294428,\\
   1144052588009550927,\ -392768652155202268,\\
   93951730922422481,\ -15744238825971732,\\
   1864646677195241,\ -156394532149220,\\
   9205044609900,\ -370727876000,\\
   9701590000,\ -148200000,\ 1000000)^\textrm{T}
\end{multline*}
The call of Mathcad function $polyroots$ will give approximations for the solutions of the algebraic equation $P(x)=0$
  \begin{multline*}
    s:=polyroots(a)\\
    =\left(
     \begin{array}{c}
       -0.00000000595668852-2.0000000743130992i\\
       -0.00000000540100408+2.00000007346564i\\
       1.0999993750452355\\
       5.097397296153272\\
       5.911537626129313\\
       6.877530938768753\\
       8.128108099781006\\
       9.08388413343586\\
       10.99946985840438\\
       13.00329871837791\\
       16.99735261469869\\
       19.00160713253058\\
       22.99980560963407\\
       29.00000860839862
     \end{array}
   \right)~.
  \end{multline*}

  Aside the fist two solution which are complex numbers,the other solutions are "close" to natural numbers. By direct computation it can be established which of these "close" integer are solutions natural numbers.
  \[
  \begin{array}{l}
    P(1)\rightarrow72560416394188800~,\\
    P(5)\rightarrow-856324819703808~,\\
    P(6)\rightarrow-202256700445800~,\\
    P(7)\rightarrow153045758638080~,\\
    P(8)\rightarrow161732639306700~,\\
    P(9)\rightarrow-274961651328000
  \end{array}
  \]
  \[
  \begin{array}{l}
    P(11)\rightarrow0~,\\
    P(13)\rightarrow0~,\\
    P(17)\rightarrow0~,\\
    P(19)\rightarrow0~,\\
    P(23)\rightarrow0~,\\
    P(29)\rightarrow0~.
  \end{array}
  \]
  The conclusion is that equation $P(x)=0$ has following solutions natural numbers: $11$, $13$, $17$, $19$, $23$ and $29$.

  In the case of algebraic equations of order $1$, $2$, $3$ and $4$ the solutions are obtained by means of symbolic calculus, by applying well known formulas.

  Equation $29x^2-490x+1469=0$ has the solutions
  \[
   29x^2-490x+1469\ solve\left(
                           \begin{array}{c}
                             13 \\ \\
                             \dfrac{113}{29}
                           \end{array}
                         \right)
  \]
  which implies that there exists a rational solution and a natural solution.

  Equation $127x^3-28829x^2-12767x+2898109=0$ has the solutions
  \[
   127x^3-28829x^2-12767x+2898109\ solve\rightarrow\left(
                                                     \begin{array}{c}
                                                       227 \\ \\
                                                       \dfrac{\sqrt{1621409}}{127} \\ \\
                                                       -\dfrac{\sqrt{1621409}}{127}
                                                     \end{array}
                                                   \right)
  \]
  hence, we have a unique solution natural number.

  Equation $3x^4-7x^3+17x^2-35x+10=$ has the solutions
  \[
   3x^4-7x^3+17x^2-35x+10\ solve\rightarrow\left(
                                             \begin{array}{c}
                                               2 \\ \\
                                               \dfrac{1}{3} \\ \\
                                               \sqrt{5}i \\ \\
                                               -\sqrt{5}i
                                             \end{array}
                                           \right)
  \]
  therefore we have a solution natural number and a solution rational number.
\end{exem}

\section[The Diophantine equation of second order]
{The Diophantine equation of second order \\and with two unknowns}

We consider the equation
\begin{equation}\label{Eq1AMethod}
  ax^2-by^2+c=0~,
\end{equation}
with $a,b\in\Ns$ and $c\in\Int^*$. It is a generalization of Pell\textquoteright{s}\index{Pell J.} equation $x^2-Dy^2=1$, \citep{Dickson2005}. Here, we show that: if the equation has an integer solution and $a\cdot b$ is not a perfect square, then (\ref{Eq1AMethod}) has an infinitude of integer solutions; in this case we find a closed expression $(x_n,y_n)$, the general positive integer solution, by an original method. More, we generalize it for any Diophantine equation of second degree and with two unknowns.

\subsection[Existence and number of solutions]{Existence and number of solutions of Diophantine\\ quadratic equations with two unknowns in $\Int$ and $\Na$}

We study the existence and number of solutions in the set of integers, $\Int$ and the set of natural numbers, $\Na$ of Diophantine equations of second degree with two unknowns of the general form (\ref{Eq1AMethod}).
\begin{thm}\label{Property1}
  The equation $x^2-y^2=c$ admits integer solutions if and only if $c$ belongs to $4\Int$ or is odd.
\end{thm}
\begin{proof}
  The equation $(x-y)(x+y)=c$ admits solutions in $\Int$ if there exist $c_1$ and $c_2$ in $\Int$ such that $x-y=c_1$, $x+y=c_2$ and $c_1c_2=c$. Therefore
  \[
   x=\frac{c_1+c_2}{2}\ \ \textnormal{and}\ \ y=\frac{c_2-c_1}{2}~.
  \]
  But $x$ and $y$ are integers if and only if $c_1+c_2\in2\Int$, i.e.:
  \begin{enumerate}
    \item or $c_1$ and $c_2$ are odd, then $c$ is odd (and reciprocally),
    \item or $c_1$ and $c_2$ are even, then $c\in4\Int$.
  \end{enumerate}

  Reciprocally, if $c\in4\Int$, then we can decompose up $c$ into two even factors $c_1$ and $c_2$, such that $c_1c_2=c$.
\end{proof}

\begin{rem}
  The theorem \ref{Property1} is true also for solving in $\Na$, because we can suppose $c\ge0$ (in the contrary case, we can multiply the equation by $-1$), and we can suppose $c_2\ge c_1\ge0$, from which $x\ge0$ and $y\ge0$.
\end{rem}

\begin{thm}
  The equation $x^2-dy^2=c^2$ (where $d$ is not a perfect square) admits infinity of solutions in $\Na$.
\end{thm}
\begin{proof}
  Let\textquoteright{s} consider $x=ck_1$, $k_1\in\Na$ and $y=ck_2$, $k_2\in\Na$, $c\in\Na$. It results that $k_1^2-dk_2^2=1$, which we can recognize as being the Pell-Fermat\textquoteright{s}\index{Pell J.}\index{Fermat P.} equation, which admits an infinity of solutions in $\Na$, $(u_n,v_n)$. Therefore $x_n=cu_n$, $y_n=cv_n$ constitute an infinity of natural solutions for our equation.
\end{proof}

\begin{thm}
  The equation \emph{(\ref{Eq1AMethod})}, $c\neq0$, where $ab=k^2$, $k\in\Int$, admits a finite number of natural solutions.
\end{thm}
\begin{proof}
  We can consider $a,b,c$ as positive numbers, otherwise, we can multiply the equation by $-1$ and we can rename the variables.

  Let us multiply the equation by $a$, then we will have:
  \begin{equation}\label{Rel1Existence}
    z^2-t^2=d\ \ \textnormal{with}\ \ z=ax\in\Na~,\ \ t=ky\in\Na\ \ \textnormal{and}\ \ d=ac>0~.
  \end{equation}

  We will solve it as in theorem \ref{Property1}, which gives $z$ and $t$. But in (\ref{Rel1Existence}) there is a finite number of natural solutions, because there is a finite number of integer divisors for a number in $\Ns$. Because the pairs $(z,t)$ are in a limited number, it results that the pairs $(z/a,t/k)$ also are in limited number, and the same for the pairs $(x,y)$.
\end{proof}

\begin{thm}
  If the equation \emph{(\ref{Eq1})}, where $ab\neq k^2$, $k\in\Int$, admits a particular nontrivial solution in $\Na$, then it admits an infinity of solutions in $\Na$.
\end{thm}
\begin{proof}
  Let\textquoteright{s} consider:
  \begin{equation}\label{Rel2Existence}
    \left\{\begin{array}{c}
             x_n=x_0\cdot u_n+b\cdot y_0\cdot v_n~, \\
             y_n=y_0\cdot u_n+a\cdot x_0\cdot v_n~,
           \end{array}\right.
  \end{equation}
  for $n\in\Na$, where $(x_0,y_0)$ is the particular natural solution for the equation (\ref{Eq1AMethod}), and $(u_n,v_n)$ is the general natural solution for the equation $u^2-abv^2=1$, called the solution Pell\index{Pell J.}, which admits an infinity of solutions. Then $ax_n^2-by_n^2=(ax_0^2-by_0^2)(u_n^2-abv_n^2)=c$. Therefore (\ref{Rel2Existence}) verifies the equation (\ref{Eq1AMethod}).
\end{proof}

\subsection[Method of solving]{Method of solving the Diophantine equation of second order}

Suppose (\ref{Eq1AMethod}) has many integer solutions. Let $(x_0,y_0)$, $(x_1,y_1)$ be the smallest positive integer solutions for (\ref{Eq1AMethod}), with $0\le x_0<x_1$. We construct the recurrent sequences:
\begin{equation}\label{RecSeqAMethod}
  \left\{\begin{array}{l}
           x_{n+1}=\alpha x_n+\beta y_n \\
           y_{n+1}=\gamma x_n+\delta y_n
         \end{array}\right.
\end{equation}
putting the condition (\ref{RecSeqAMethod}) verify (\ref{Eq1AMethod}). It results:
\begin{eqnarray}
  a\alpha\beta &=& b\gamma\delta \label{Eq4}\\
  a\alpha^2-b\gamma^2 &=& a \label{Eq5}\\
  a\beta^2-b\delta^2 &=& -b  \label{Eq6}
\end{eqnarray}
having the unknowns $\alpha,\beta,\gamma,\delta$.

We pull out $a\alpha^2$ and $a\beta^2$ from (\ref{Eq5}), respectively (\ref{Eq6}), and replace them in (\ref{Eq4}) at the square; it obtains
\begin{equation}\label{Eq7}
  a\delta^2-b\gamma^2=a~.
\end{equation}
We subtract (\ref{Eq7}) from (\ref{Eq5}) and find
\begin{equation}\label{Eq8}
  \alpha=\pm\beta~.
\end{equation}
Replacing (\ref{Eq8}) in (\ref{Eq4}) it obtains
\begin{equation}\label{Eq9}
  \beta=\pm\frac{b}{a}\gamma~.
\end{equation}

Afterwards, replacing (\ref{Eq8}) in (\ref{Eq5}), and (\ref{Eq9}) in (\ref{Eq6}) it finds the same equation:
\begin{equation}\label{Eq10}
  a\alpha^2-b\gamma^2=a~.
\end{equation}

Because we work with positive solutions only, we take
\begin{eqnarray}
  x_{n+1} &=& \alpha_0x_n+\frac{b}{a}\gamma_0y_n \\
  y_{n+1} &=& \gamma_0x_n+\alpha_0y_n
\end{eqnarray}
where $(\alpha_0,\gamma_0)$ is the smallest, positive integer solution of (\ref{Eq10}) such that $\alpha_0\gamma_0\neq0$. Let
\begin{equation}\label{MatriceaDeBaza}
  A=
    \left(
      \begin{array}{cc}
        \alpha_0 & \dfrac{b}{a}\gamma_0 \\
        \gamma_0 & \alpha_0 \\
      \end{array}
    \right)\in \mathcal{M}_2(\Int)~.
\end{equation}

Of course, if $(x',y')$ is an integer solution for (\ref{Eq1AMethod}), then
\[
 A\cdot\left(
    \begin{array}{c}
      x' \\
      y' \\
    \end{array}
  \right)~,\ \ \
 A^{-1}\cdot\left(
  \begin{array}{c}
      x' \\
      y' \\
    \end{array}
  \right)
\]
are another ones, where
\[
 A^{-1}=\frac{1}{\gamma^2b-a\alpha^2}\left(
                                      \begin{array}{cc}
                                        -a\alpha & \gamma b \\
                                        \gamma b & -a\alpha  \\
                                      \end{array}
                                    \right)
\]
is the inverse matrix of $A$, i.e. $A^{-1}\cdot A=A\cdot A^{-1}=I$ (unit matrix). Hence, if (\ref{Eq1AMethod}) has an integer solution it has an infinite ones. Clearly $A^{-1}\in\mathcal{M}_2(\Int)$.

The general positive integer solution of the equation (\ref{Eq1AMethod}) is $(x'_n,y'_n)=(\abs{x_n},\abs{y_n})$.
\begin{equation}
  \left(
    \begin{array}{c}
      x_n \\
      y_n \\
    \end{array}
  \right)=A^n\cdot\left(
                    \begin{array}{c}
                      x_0 \\
                      y_0 \\
                    \end{array}
                  \right)\ \ \ \textnormal{for all}\ n\in\Int~,
\end{equation}
where by conversion $A^0=I$ and
\[
 A^{-k}=\underbrace{A^{-1}\cdots A^{-1}}_{k\ times}~.
\]

In problems it is better to write general solution as
\begin{equation}
  \left(
    \begin{array}{c}
      x'_n \\
      y'_n \\
    \end{array}
  \right)=A^n\cdot
  \left(
    \begin{array}{c}
      x_0 \\
      y_0 \\
    \end{array}
  \right)\ \ \ n\in\Na
\end{equation}
and
\begin{equation}\label{GS2}
  \left(
    \begin{array}{c}
      x''_n \\
      y''_n \\
    \end{array}
  \right)=A^n\cdot
  \left(
    \begin{array}{c}
      x_1 \\
      y_1 \\
    \end{array}
  \right)\ \ \ n\in\Ns~.
\end{equation}

We proof, by \emph{reduction ad absurdum}, (\ref{GS2}) is a general positive integer solution for (\ref{Eq1AMethod}).

Let $(u,v)$ be a positive integer particular solution for (\ref{Eq1AMethod}). If
\[
 \exists\ k_0\in\Na\ :\ (u,v)=A^{k_0}\cdot\left(
                                            \begin{array}{c}
                                              x_0 \\
                                              y_0 \\
                                            \end{array}
                                          \right)
\]
or
\[
 \exists\ k_1\in\Ns\ :\ (u,v)=A^{k_1}\cdot\left(
                                            \begin{array}{c}
                                              x_0 \\
                                              y_0 \\
                                            \end{array}
                                          \right)~,
\]
then $(u,v)\in$(\ref{GS2}). Contrary to this, we calculate
\[
 (u_{i+1},v_{i+1})=A^{-1}\cdot\left(
                                \begin{array}{c}
                                  u_i \\
                                  v_i \\
                                \end{array}
                              \right)
\]
for $i=0,1,2,\ldots$, where $u_0=u$, $v_0=v$. Clearly $u_{i+1}<u_i$ for all $i$. After a certain rank $x_0<u_{i_0}<x_1$ it finds either $0u_{i_0}<x_0$ but that is absurd.

It is clear we can put
\begin{equation}\label{GS3}
  \left(
    \begin{array}{c}
      x_n \\
      y_n \\
    \end{array}
  \right)=A^n\cdot
  \left(
    \begin{array}{c}
      x_0 \\
      \varepsilon\cdot y_0 \\
    \end{array}
  \right)\ \ \ n\in\Na~,\ \ \textnormal{where}\ \ \varepsilon=\pm1~.
\end{equation}

We shall now transform the general solution (\ref{GS3}) in closed expression.

Let $\lambda$ be real number, then $det(A-\lambda I)=0$ involves the solutions $\lambda_{1,2}$ and the proper vectors $v_{1,2}$ i.e.
\[
 A\cdot v_i=\lambda_i\cdot v_i~,\ \ \textnormal{for} \ \ i\in\I{2}~.
\]
Note
\[
 P=\left(
     \begin{array}{cc}
       v_1 & v_2 \\
     \end{array}
   \right)\in\mathcal{M}_2(\Real)~.
\]
Then
\[
 P^{-1}\cdot A\cdot P=\left(
                        \begin{array}{cc}
                          \lambda_1 & 0 \\
                          0 & \lambda_2 \\
                        \end{array}
                      \right)~,
\]
whence
\[
 A^n=P\cdot\left(\begin{array}{cc}
                   \lambda_1^n & 0 \\
                   0 & \lambda_2^n \\
                  \end{array}\right)\cdot P^{-1}
\]
and replacing it in (\ref{GS3}) and doing the calculus we find a closed expression for (\ref{GS3}).

\subsection[Procedure for solving]{Procedure for solving of Diophantine equation \\of second order with two unknowns}

We will present an automatic procedure for solving Diophantine equations (\ref{Eq1AMethod}). We have two programs that establish the basis matrix (\ref{MatriceaDeBaza}) and the particular minimal solution. The input variables are the integer constants $a$, $b$
and $c$. Finding the basis matrix and the minimal solution is done up to a given limit (in our case up to $m=10^6$, obviously this
limit can be augmented).

\begin{prog} Program for finding the basis matrix.
\begin{tabbing}
  $M(a,b):=$\=\vline\ $m\leftarrow10^6$\\
  \>\vline\ $f$\=$or\ \alpha\in2..m$\\
  \>\vline\ \>\vline\ $q\leftarrow\sqrt{\dfrac{b}{a}(\alpha-1)(\alpha+1)}$\\
  \>\vline\ \>\vline\ $break\ if\ q\textbf{=}trunc(q)\wedge\dfrac{a}{b}q\textbf{=}trunc\left(\dfrac{a}{b}q\right)$\\
  \>\vline\ $return\ \left(
                       \begin{array}{cc}
                         \alpha & q \\
                         \dfrac{a}{b}q & \alpha \\
                       \end{array}
                     \right)\ if\ \alpha<m$\\
  \>\vline\ $return\ "Error\ Matrix\ A\ was\ not\ found"\ otherwise$
\end{tabbing}
\end{prog}

\begin{prog} Program for finding the minimal solutions.
  \begin{tabbing}
    $SM(a,b,c):=$\=\vline\ $m\leftarrow10^6$\\
    \>\vline\ $f$\=$or\ y\in1..m$\\
    \>\vline\ \>\vline\ $d\leftarrow\dfrac{b\cdot y^2-c}{a}$\\
    \>\vline\ \>\vline\ $i$\=$f\ d\ge0$\\
    \>\vline\ \>\vline\ \>\vline\ $x\leftarrow\sqrt{d}$\\
    \>\vline\ \>\vline\ \>\vline\ $break\ if\ x\textbf{=}trunc(x)$\\
    \>\vline\ $return\ \left(
                         \begin{array}{cc}
                           x & x \\
                           y & -y \\
                         \end{array}
                       \right)\ if\ y<m$\\
    \>\vline\ $return\ "Error\ S\ was\ not\ found"\ otherwise$\\
  \end{tabbing}
\end{prog}

\begin{exem}
  For the Diophantine equation $2x^2-3y^2=5$, we give the sequence of symbolic Mathcad commands which completely solve analytically the Diophantine equation. Origin indices are considered $1$ by using the command $ORIGIN:=1$.

  We initialize the constants $a$, $b$ and $c$
  \[
   a:=2\ \ b:=3\ \ c:=-5
  \]

  The determine the basis matrix $A$ and the eigenvalues of the matrix by means of the Mathcad function $eigenvals$
  \[
   A:=M(a,b)=\left(
               \begin{array}{cc}
                 5 & 6 \\
                 4 & 5 \\
               \end{array}
             \right)\ \
             \lambda:=eigenvals(A)\rightarrow\left(
                                     \begin{array}{c}
                                       5+2\sqrt{6} \\
                                       5-2\sqrt{6} \\
                                     \end{array}
                                   \right)
  \]

  We determine the eigenvectors of matrix $A$ with the aid of the Mathcad function $eigenvec$ and we build matrix $V$
  \[
   V:=augment(eigenvec(A,\lambda_1),eigenvec(A,\lambda_2))\rightarrow\left(
                                                                       \begin{array}{cc}
                                                                         \dfrac{\sqrt{6}}{2} & -\dfrac{\sqrt{6}}{2} \\
                                                                         1 & 1 \\
                                                                       \end{array}
                                                                     \right)
  \]

  We have matrix $P(n)$ given by formula (where $P(n)=A^n$):
  \[
   P(n):=V\cdot\left(
                 \begin{array}{cc}
                   (\lambda_1)^n & 0 \\
                   0 & (\lambda_2)^n \\
                 \end{array}
               \right)\cdot V^{-1}
  \]

  We determine the minimal solutions
  \[
    SM(a,b,c)\rightarrow\left(
                         \begin{array}{cc}
                           2 & 2 \\
                           1 & -1 \\
                         \end{array}
                       \right)
  \]
  \[
   S_0:=SM(a,b,c)^{\langle1\rangle}\rightarrow\left(
                                                \begin{array}{c}
                                                  2 \\
                                                  1 \\
                                                \end{array}
                                              \right)\ \
   S_1:=SM(a,b,c)^{\langle2\rangle}\rightarrow\left(
                                                \begin{array}{c}
                                                  2 \\
                                                  -1 \\
                                                \end{array}
                                              \right)
  \]

  The general solutions of the Diophantine equation are $S0(n)$ and $S1(n)$
  \[
                       S0(n):=\left(A^n\cdot S_0\right)^\textrm{T}\ \
                       S1(n):=\left(A^n\cdot S_1\right)^\textrm{T}
  \]

  The explicit formulas for the general solutions are $T0(n)$ and $T1(n)$
  \begin{multline*}
    T0(n):=P(n)\cdot S_0\ factor\rightarrow \\
    \left(
      \begin{array}{c}
        \dfrac{(\sqrt{6}+4)(5+2\sqrt{6})^n-(\sqrt{6}-4)(5-2\sqrt{6})^n}{4} \\
        \dfrac{(2\sqrt{6}-3)(5-2\sqrt{6})^n-(2\sqrt{6}+3)(5+2\sqrt{6})^n}{6} \\
      \end{array}
    \right)
  \end{multline*}
   \begin{multline*}
    T1(n):=P(n)\cdot S_1\ factor\rightarrow \\
    \left(
      \begin{array}{c}
        \dfrac{(\sqrt{6}+4)(5-2\sqrt{6})^n-(\sqrt{6}-4)(5+2\sqrt{6})^n}{4} \\
        \dfrac{(2\sqrt{6}+3)(5-2\sqrt{6})^n-(2\sqrt{6}-3)(5+2\sqrt{6})^n}{6} \\
      \end{array}
    \right)
  \end{multline*}

  Let $n=0,1,2,\ldots,10$
  \[
   n:=0..10
  \]

  We display the solutions for $n$
  \[
   S0(n)\rightarrow
   \left(
     \begin{array}{cc}
       2 & 1 \\
       16 & 13 \\
       158 & 129 \\
       1564 & 1277 \\
       15482 & 12641 \\
       153256 & 125133 \\
       1517078 & 1238689 \\
       15017524 & 12261757 \\
       148658162 & 121378881 \\
       1471564096 & 1201527053 \\
       14566982798 & 11893891649 \\
     \end{array}
   \right)
  \]
  \[
   S1(n)\rightarrow
   \left(
     \begin{array}{cc}
       2 & -1 \\
       4 & 3 \\
       38 & 31 \\
       376 & 307 \\
       3722 & 3039 \\
       36844 & 30083 \\
       364718 & 297791 \\
       3610336 & 2947827 \\
       35738642 & 29180479 \\
       353776084 & 288856963 \\
       3502022198 & 2859389151 \\
     \end{array}
   \right)
  \]

  The displayed solutions can be tested if we verify the Diophantine equation by the aid of the sequences:
  \begin{multline*}
    a\cdot\left(S0(n)_1\right)^2-b\cdot\left(S0(n)_2\right)^2+c\rightarrow \\
    \left(\begin{array}{ccccccccccc}
           0 & 0 & 0 & 0 & 0 & 0 & 0 & 0 & 0 & 0 & 0
         \end{array}\right)^\textrm{T}
  \end{multline*}
  \begin{multline*}
    a\cdot\left(S1(n)_1\right)^2-b\cdot\left(S1(n)_2\right)^2+c\rightarrow \\
    \left(\begin{array}{ccccccccccc}
           0 & 0 & 0 & 0 & 0 & 0 & 0 & 0 & 0 & 0 & 0
         \end{array}\right)^\textrm{T}
  \end{multline*}
  The solutions given by the expressions $T0(n)$ and $T1(n)$ can also be displayed, as follows:
  \[
   \left(
      \begin{array}{cc}
        T0(n)_1 & T0(n)_2 \\
      \end{array}
    \right)\rightarrow
    \left(
     \begin{array}{cc}
       2 & 1 \\
       16 & 13 \\
       158 & 129 \\
       1564 & 1277 \\
       15482 & 12641 \\
       153256 & 125133 \\
       1517078 & 1238689 \\
       15017524 & 12261757 \\
       148658162 & 121378881 \\
       1471564096 & 1201527053 \\
       14566982798 & 11893891649 \\
     \end{array}
   \right)~,
  \]
  \[
   \left(
      \begin{array}{cc}
        T1(n)_1 & T1(n)_2 \\
      \end{array}
    \right)\rightarrow
    \left(
     \begin{array}{cc}
       2 & -1 \\
       4 & 3 \\
       38 & 31 \\
       376 & 307 \\
       3722 & 3039 \\
       36844 & 30083 \\
       364718 & 297791 \\
       3610336 & 2947827 \\
       35738642 & 29180479 \\
       353776084 & 288856963 \\
       3502022198 & 2859389151 \\
     \end{array}
   \right)~.
  \]
  Obviously these solutions are identical with those given by $S0(n)$ and $S1(n)$.
\end{exem}

Basically, any Diophantine equation of the type (\ref{Eq1AMethod}) can be completely solved with this set of commands.

\begin{exem}
  Let us consider the equation $13x^2-17y^2+2636=0$. The basis
  matrix is
  \[
   A:=M(13,17)=\left(
                 \begin{array}{cc}
                   1665 & 1904 \\
                   1456 & 1665 \\
                 \end{array}
               \right)~.
  \]
  The minimal solutions are:
  \[
   S_0\rightarrow\left(\begin{array}{c}
                         19 \\
                         11 \\
                       \end{array}\right)\ \ \
   S_1\rightarrow\left(\begin{array}{c}
                         19 \\
                         -11 \\
                       \end{array}\right)~.
  \]
  the solutions are given by the formulas:
  \[
   S0(n):=\left(A^n\cdot S_0\right)^\textrm{T}\ \ \  S1(n):=\left(A^n\cdot S_1\right)^\textrm{T}~.
  \]
  The values given by $S0$ for $n=0,1,2$ and $10$ are:
  \[
  \boxed{19,~11}, \boxed{52579,~45979}, \boxed{175088051,~153110059}
  \]
  and
  \begin{multline*}
    \vline\underline{\overline{2647342081327033989423041791914721331,}}\\
    \underline{\overline{2315033492863349726442025803342919339}}\vline~,
  \end{multline*}
  and the values provided by $S1$ for $n=0,1,2$ and $10$ are:
  \[
  \boxed{19,-11}, \boxed{10691,~9349}, \boxed{35601011,~31132181}
  \]
  and
  \begin{multline*}
    \vline\underline{\overline{538289472181531211118549596688006131,}}\\
    \underline{\overline{470720488200496189286367630993971861}}\vline~.
  \end{multline*}
  The explicit solutions are:
  \begin{multline*}
    T0(n):=P(n)\cdot S_0\ factor\rightarrow \\
    \left(\begin{array}{c}
        \dfrac{(11\sqrt{221}+247)\theta_1^n-(11\sqrt{221}-247)\theta_2^n}{26} \\
        \dfrac{(19\sqrt{221}+187)\theta_1^n-(19\sqrt{221}-187)\theta_2^n}{34} \\
      \end{array}\right)
  \end{multline*}
  where
  \[
   \theta_1=1665+112\sqrt{221}~,\ \ \theta_2=1665-112\sqrt{221}
  \]
  and
   \begin{multline*}
    T1(n):=P(n)\cdot S_1\ factor\rightarrow \\
    \left(\begin{array}{c}
        \dfrac{(11\sqrt{221}+247)\theta_2^n-(11\sqrt{221}-247)\theta_1^n}{26} \\
        \dfrac{(19\sqrt{221}+187)\theta_2^n-(19\sqrt{221}-187)\theta_1^n}{34}
      \end{array}\right)~,
  \end{multline*}
  for $n\in\Na$.
\end{exem}

\subsection{Generalizations}

If $f(x,y)=0$ is a Diophantine equation of second degree and with two unknowns, by linear transformations it becomes (\ref{Eq1AMethod}).

If $a\cdot b\ge0$ the equation has at most a finite number of integer solutions which can be found attempts. It is easier to present an example.

The Diophantine equation
\begin{equation}\label{Eq13AMethod}
  18x^2+12xy-26y^2-12x-32y+40=0
\end{equation}
becomes
\begin{equation}\label{Eq14AMethod}
  2u^2-7v^2+45=0~,
\end{equation}
where (unfortunately, finding these substitutions is a difficult problem)
\begin{equation}\label{Eq15AMethod}
  \left\{\begin{array}{lcl}
           u &=& 3x+y-1~, \\
           v &=& 2y+1~.
         \end{array}\right.
\end{equation}

The basis matrix for the Diophantine equation (\ref{Eq14AMethod}) is
\[
 A:=M(2,7)=\left(
             \begin{array}{cc}
               15 & 28 \\
               8 & 15 \\
             \end{array}
           \right)
\]
and the minimal solutions are $S_0=(3\ 3)^\textrm{T}$ and $S_1=(3\ -3)^\textrm{T}$. In this conditions we obtain the solutions
$S0(n)=A^n\cdot S_0$ and $S1(n)=A^n\cdot S_1$. Formula $S1(n)$ produces as solutions negative integers. Back from the solutions
obtain with formula $S0(n)$ to variables $x$ and $y$ by means of the substitutions
\[
 \left\{\begin{array}{l}
          x=\dfrac{2u-v-3}{6} \\ \\
          y=\dfrac{v-1}{2}
        \end{array}\right.
\]
we obtain the solution of the Diophantine equation (\ref{Eq13AMethod}). The first $11$ positive integer solutions
are:
\[
\left(
  \begin{array}{cc}
    1 & 1 \\
    32 & 34 \\
    945 & 1033 \\
    28304 & 30970 \\
    848161 & 928081 \\
    25416512 & 27811474 \\
    761647185 & 833416153 \\
    22823999024 & 24974673130 \\
    683958323521 & 748406777761 \\
    20495925706592 & 22427228659714 \\
    614193812874225 & 672068453013673 \\
  \end{array}
\right)~.
\]

We solve (\ref{Eq14AMethod}). Thus:
\begin{equation}\label{Eq16AMethod}
  \left\{\begin{array}{lcl}
           u_{n+1} &=& 15u_n+28v_n~,\\
           v_{n+1} &=& 8u_n+15v_n~,
         \end{array}\right.
\end{equation}
$n\in\Na$, with $(u_0,\ v_0)=(3,\ 3\varepsilon)$.

\subsubsection{First solution}

By induction we proof that: for all $n\in\Na$ we have $v_n$ is odd, and $u_n$ as well as $v_n$ are multiple of $3$. Clearly $v_0=3\varepsilon\cdot u_0$. For $n+1$ we have $v_{n+1}=8u_n+15v_n=even+odd=odd$, and of course $u_{n+1}$, $v_{n+1}$ are multiples of $3$ because $u_n$, $v_n$ are multiple $3$, too.

Hence, there exist $x_n$, $y_n$, in positive integers for all $n\in\Na$:
\begin{equation}\label{Eq17AMethod}
  \left\{\begin{array}{lcl}
           x_n &=& \dfrac{2u_n-v_n+3}{6}~, \\ \\
           y_n &=& \dfrac{v_n-1}{2}~,
         \end{array}\right.
\end{equation}
(from (\ref{Eq15AMethod})). Now we find the (\ref{GS3}) for (\ref{Eq14AMethod}) as closed expression, and by means of (\ref{Eq17AMethod} it results the general integer solution of the equation (\ref{Eq13AMethod}).

\subsubsection{Second solution}

Another expression of the (\ref{GS3}) for (\ref{Eq13AMethod}) we obtain if we transform (\ref{Eq15AMethod}) as: $u_n=3x_n+y_n-1$ and $v_n=2y_n+1$, for all $n\in\Na$. Whence, using (\ref{Eq16AMethod}) and doing the calculus, it finds
\begin{equation}\label{Eq18AMethod}
  \left\{\begin{array}{l}
           x_{n+1}=11x_n+\dfrac{52}{3}y_n+\dfrac{11}{3}~, \\ \\
           y_{n+1}=12x_n+19y_n+3~,
         \end{array}\right.
\end{equation}
for $n\in\Na$, with $(x_0,y_0)=(1,\ 1)$ or $(2,\ -2)$ (two infinitude of integer solutions). Let
\[
A=\left(\begin{array}{ccc}
     11 & \dfrac{52}{3} & \dfrac{11}{3} \\ \\
     12 & 9 & 3 \\ \\
     0 & 0 & 1
   \end{array}\right)~.
\]
Then
\[
 \left(
   \begin{array}{c}
     x_n \\
     y_n \\
     1 \\
   \end{array}
 \right)=A^n\cdot\left(
                   \begin{array}{c}
                     1 \\
                     1 \\
                     1 \\
                   \end{array}
                 \right)
\]
or
\begin{equation}\label{Eq19AMethod}
  \left(
   \begin{array}{c}
     x_n \\
     y_n \\
     1 \\
   \end{array}
 \right)=A^n\cdot\left(
                   \begin{array}{c}
                     2 \\
                     -2 \\
                     1 \\
                   \end{array}
                 \right)~,
\end{equation}
always $n\in\Na$.

From (\ref{Eq18AMethod}) we have always $y_{n+1}\equiv y_n\equiv\ldots\equiv y_0\equiv1\ \md{3}$, hence always $x_n\in\Int$. Of course (\ref{Eq19AMethod}) and (\ref{Eq17AMethod}) are equivalent as general integer solution (\ref{Eq13AMethod}).

This method can be generalized for Diophantine equations
\begin{equation}\label{Eq20AMethod}
  \sum_{i=1}^na_i\cdot x_i^2=b~,
\end{equation}
will all $a_i,b\in\Int$.

It always $a_i\cdot a_j\ge0$ $1\le i\le j<n$, the equation (\ref{Eq20AMethod}) has most finite number of integer solution.

Now, we suppose $\exists i_0,j_0\in\I{n}$ for which $a_{i_0}\cdot a_{j_0}<0$ (the equation presents at least a variation of sign). Analogously, for $n\in\Na$. We define the recurrent sequence:
\begin{equation}\label{Eq21AMethod}
  x_h^{(n+1)}=\sum_{i=1}^n a_{ih}\cdot x_i^{(n)}~\ \ \ 1h\in\I{n}
\end{equation}
considering $(x_1^0,x_2^0,\ldots,x_n^0)$ the smallest positive integer solution of (\ref{Eq20AMethod}). It replaces (\ref{Eq21AMethod}) in (\ref{Eq20AMethod}), it identifies the coefficients and it look for the $n^2$ unknowns $a_{ih}$, where $i,h\in\I{n}$. This calculus is very intricate, but it can done by means of a computer. The method goes on similarly, but the calculus becomes more and more intricate -- for example to calculate $A^n$. It must computer may be.

Other results referring to Diophantine equations can be found in the papers
\citep{Landau1955,Long1965,Ogibvy+Anderson1966,Mordell1969,Hardy+Wright1984,Bencze1985,Borevich+Shafarevich1985,Perez+Amaya+Corres2013}.

\section{The Diophantine equation $x^2-2y^4+1=0$}

In this section we present a method of solving this Diophantine equation, method which is different from Ljunggren\textquoteright{s}, Mordell\textquoteright{s} and Guy\textquoteright{s}.

In the book \citep[pp. 84-85]{Guy1981} to shows that equation $x^2=2y^4-1$ has, in the set of positive integers, only solutions \boxed{1,1} and \boxed{239,13}; \cite{Ljunggren1966} has proved it in a complicated way. But \cite{Mordell1964} gave an easier proof.

We\textquoteright{l}l note $t=y^2$. The general integer solution for $x^2-2t^2+1$ is
\[\left\{\begin{array}{c}
           x_{n+1}=3x_n+4t_n~, \\
           t_{n+1}=2x_n+3t_n
         \end{array}\right.
\]
for all $n\in\Na$, where $(x_0,y_0)=(1,\varepsilon)$, with $\varepsilon=\pm1$ or
\[
 \left(
   \begin{array}{c}
     x_n \\
     t_n \\
   \end{array}
 \right)=
 \left(
   \begin{array}{cc}
     3 & 4 \\
     2 & 3 \\
   \end{array}
 \right)^n\cdot
 \left(
   \begin{array}{c}
     1 \\
     \varepsilon \\
   \end{array}
 \right)~,
\]
for all $n\in\Na$, where a matrix to the power zero is equal to the unit matrix $I$.

Let\textquoteright{s} consider
\[A=\left(
      \begin{array}{cc}
        3 & 4 \\
        2 & 3\\
      \end{array}
    \right)~,
\]
and $\lambda\in\Real$. Then $det(A-\lambda\cdot I)=0$ implies $\lambda_{1,2}=3\pm\sqrt{2}$, whence if $v$ is a vector of dimension two, then $Av=\lambda_{1,2}\cdot v$.

Let\textquoteright{s} consider
\[P=\left(
      \begin{array}{cc}
        2 & 2 \\
        \sqrt{2} & -\sqrt{2} \\
      \end{array}
    \right)
\]
and
\[D=\left(
      \begin{array}{cc}
        3+\sqrt{2} &0 \\
        0 & 3-\sqrt{2} \\
      \end{array}
    \right)~.
\]
We have $P^{-1}\cdot A\cdot P=D$, or
\[A^n=P\cdot D^n\cdot P^{-1}=
 \left(\begin{array}{cc}
         \dfrac{a_n+b_n}{2} & \dfrac{\sqrt{2}(a_n-b_n)}{2} \\
         \dfrac{\sqrt{2}(a_n-b_n)}{4} & \dfrac{a_n+b_n}{2} \\
       \end{array}\right)~.
\]
where $a_n=(3+2\sqrt{2})^n$ and $b_n=(3-2\sqrt{2})^n$.  Hence, we find
\[
 \left(
   \begin{array}{c}
     x_n \\
     t_n \\
   \end{array}
 \right)=
 \left(
   \begin{array}{l}
     \dfrac{1+\varepsilon\sqrt{2}}{2}a_n+\dfrac{1-\varepsilon\sqrt{2}}{2}b_n \\
     \dfrac{2\varepsilon+\sqrt{2}}{4}a_n+\dfrac{2\varepsilon-\sqrt{2}}{4}b_n \\
   \end{array}
 \right)~.
\]
for all $n\in\Na$.

Or $y_n^2=t_n$, for all $n\in\Na$. For $n=0$, $\varepsilon=1$ we obtain $y_0^2=1$ (whence $x_0=1$), and for $n=3$, $\varepsilon=1$ we obtain $y_3^2=169$ (whence $x_3=239$).
\begin{equation}\label{Rel1AMethod}
  y_n^2=\varepsilon\sum_{k=0}^{\left[\frac{n}{2}\right]}C_n^{2k}\cdot 3^{n-2k}\cdot2^k+
\sum_{k=0}^{\left[\frac{n-1}{2}\right]}C_n^{2k+1}3^{n-2k-1}\cdot2^{3k+1}~.
\end{equation}

We still prove that $y_n^2$ is perfect square if and only if $n=0,3$. We can use a similar method the Diophantine equation $x^2=Dy^4\pm1$, or more generally: $C\cdot X^{2a}=DY^{2b}+E$, with $a,b\in\Ns$ and $C,D,E\in\Int^*$; denoting $x^a=U$, $y^b=V$, and applying the results from \citep{Smarandache1988}, the relation (\ref{Rel1AMethod}) becomes very complicated.

May be found following works \citep{Mordell1964,Ljunggren1966,Cohn1978} and \citep[pp 84-85]{Guy1981}.

\chapter[Partial empirical solving]{Partial empirical solving \\of $\eta$--Diophantine equations}

\section{Empirical determination of solutions}
A method often used to find some solutions of Diophantine equations is the empirical search, \cite{Alanen1972}, of certain numbers that satisfy the Diophantine equation, \citep{Abraham+Sanyal+Sanglikar2010}, \citep{Cohen2007,Niven+Zuckerman+Hugh1991,Rossen1987}.

The empirical search or exhaustive search, also known as generating and testing, is a very general technique of problem
solving, which consists in systematically enumerating all possible candidates as solutions and testing if they verify the problem.

An algorithm of empirical search for finding the divisors of a natural number $n$ would enumerate all integers from 1 to
$\lfloor\sqrt{n}\rfloor$, and verify each number if it divides $n$.

An empirical search is easy to implement, and it will always find a solution in the case that those solutions exist, its cost being
proportional with the number of candidate solutions -- which, in may practical problems, tends to grow very fast along with the
problem's dimension. Therefore, the empirical search is used when the dimension of the problem is limited, or when, for specific
heuristic causes, the problem can be reduced to a more manageable dimension. The method is also used when the simplicity of the
implementation is more important than the speed of the problem solving.

For example, this is the case of critical applications, when any error in the algorithm would have serious consequences; or when a
computer is used to prove a mathematical theorem. The empirical search is also useful as a basic method when benchmarks or other
meta-heuristic algorithm are used. Indeed, the empirical search can be viewed as the simplest meta-heuristic algorithm. The
empirical search should not be confounded with the backtracking, where a great number of solutions can be avoided without being
explicitly enumerated. The empirical search method is useful for finding an element in a table -- namely, it verifies sequentially
all inputs -- that's why it is a linear search.

A possibility to accelerate the empirical algorithm is to reduce the search space, that is the set of candidate solutions, by using
heuristic techniques that are specific to the problem.

By means of a brief analysis we can often bring dramatic reductions to the number of candidate solutions, and solving the problem can turn from a difficult issue into a trivial one.

In some cases, the analysis can reduce the candidate solutions toe a set of viable solutions. This can be obtained with an algorithm
that enumerates directly all candidates, without losing time on testing, and generates also invalid candidates. For example, for
the problem "\emph{find all integers between $1$ and $10^9$, divisible by $571$}", a naive solution would generate all
integers, testing afterwards each of them for divisibility by $571$. However, this problem can be solved more efficiently by
beginning with la $571$ and, repeatedly, adding $571$ up to number $10^9$ -- which would necessitate only $1751314$ ($=10^9/571$)
steps and tests.

\subsection{Partial empirical solving of Diophantine equations}

Examples of problems solved by means of the empirical search:
\begin{enumerate}
  \item D. Wilson\index{Wilson D.}, \citep[A030052]{Sloane2014}, has compiled a list of the smallest $n$th powers of
  positive integers that are the sums of the $n$th powers of distinct smaller positive integers. The first few are:
      \begin{eqnarray*}
        3^1 &=& 1^1+2^1~, \\
        5^2 &=& 3^2+4^2~, \\
        6^3 &=& 3^3+4^3+5^3~, \\
        15^4 &=& 4^4+6^4+8^4+9^4+14^4~,\\
        12^5 &=& 4^5+5^5+6^5+7^5+9^5+11^5~,
      \end{eqnarray*}
      \begin{multline*}
           25^6=1^6+2^6+3^6+5^6+6^6+7^6+8^6+9^6+10^6+12^6 \\
                +13^6+15^6+16^6+17^6+18^6+23^6~,
      \end{multline*}
      \begin{multline*}
           40^7=1^7+3^7+5^7+9^7+12^7+14^7+16^7+17^7+18^7 \\
           +20^7+21^7+22^7+25^7+28^7+39^7~,
      \end{multline*}
      \begin{multline*}
        84^8=1^8+2^8+3^8+5^8+7^8+9^8+10^8+11^8+12^8+13^8 \\
             +14^8+15^8+16^8+17^8+18^8+19^8+21^8+23^8 \\
             +24^8+25^8+26^8+27^8+29^8+32^8+33^8+35^8 \\
             +37^8+38^8+39^8+41^842^8+43^8+45^8+46^8 \\
             +47^8+48^8+49^8+51^8+52^8+53^8+57^8+58^8 \\
             +59^8+61^8+63^8+69^8+73^8~,
      \end{multline*}
      \begin{multline*}
          47^9=1^9+2^9+4^9+7^9+11^9+14^9+15^9+18^9++26^9+27^9 \\
              +30^9+31^9+32^9+33^9+36^9+38^9+39^9+43^9~,
      \end{multline*}
      \begin{multline*}
          63^{10}=1^{10}+2^{10}+4^{10}+5^{10}+6^{10}+8^{10}+12^{10}+15^{10}+16^{10} \\
          +17^{10}+20^{10}+21^{10}+25^{10}+26^{10}+27^{10}+28^{10}+30^{10}\\
          +36^{10}+37^{10}+38^{10}+40^{10}+51^{10}+62^{10}~.
      \end{multline*}
  \item The first prime number with the special property that the result of the addition to its reverse is also a prime number is $229$. We will call the prime numbers with this property numbers having \emph{the 229 property}. By means of an empirical search algorithm were found all 50598 prime numbers having \emph{the 229 property}, for $p$ prime, $p<10^7$, in approximately 25 seconds on a computer with Intel processor of 2.20GHz with RAM of 4.00GB (3.46GB usable) \citep{Cira+Smarandache2014}.

      The list of solutions begins with the prime numbers:\\
      229, 239, 241, 257, 269, 271, 277, 281, 439, 443, 463, 467, 479, 499, 613, 641, 653, 661, 673, 677, 683, 691, 811, 823, 839, 863, 881 \ldots\\
      and ends with the prime numbers:\\
      8998709, 8998813, 8998919, 8999099, 8999161, 8999183, 8999219, 8999311, 8999323, 8999339, 8999383, 8999651, 8999671, 8999761, 8999899, 8999981~.
  \item The natural numbers that satisfy the Diophantine relation $\overline{c_{n-1}c_{n-2}\ldots c_0}=c_{n-1}^n+c_{n-2}^n+\ldots+c_0^n$, are called \emph{narcissistic numbers}, \citep{Cira+Cira2010}.
      \begin{enumerate}
        \item Solutions in base $3$ numeral system:
           \begin{enumerate}
             \item For $n=1$ we have the trivial solutions: $1=1^1$, $2=2^1$, out of 2 possible cases, and solution $0=0^1$.
             \item For $n=2$ we have the solutions:
                 \[
                  \begin{array}{l}
                    12=1^2+2^2=1+11~, \\
                    22=2^2+2^2=11+11~,
                  \end{array}
                 \]
                out of 7 possible cases.
             \item For $n=3=10_3$ we have a sole solution:
                 \[
                  122=1^{10}+2^{10}+2^{10}=1+22+22~,
                 \]
                 out of 19 possible cases.
             \item[(4-7)] for $n=4=11_3$, $n=5=12_3$ $n=6=20_3$ and $n=7=21_3$ there do not exist solutions, out of, respectively 55, 163, 487 and
             1459 possible cases.
           \end{enumerate}
        \item Solutions in base $4$ numeral system:
            \begin{enumerate}
              \item For $n=1$ we have the trivial solutions: $1=1^1$, $2=2^1$, $3=3^1$, out of 3 possible cases, and solution $0=0^1$.
              \item For $n=2$ we do not have solutions, out of 13 possible cases.
              \item For $n=3$ we have $6$ solutions:
                  \[
                   \begin{array}{l}
                     130=1^3+3^3+0^3=1+123+0~, \\
                     131=1^3+3^3+1^3=1+123+1~, \\
                     203=2^3+0^3+3^3=20+0+123~, \\
                     223=2^3+2^3+3^3=20+20+123~, \\
                     313=3^3+1^3+3^3=123+1+123~, \\
                     332=3^3+3^3+2^3=123+123+20~,
                   \end{array}
                  \]
                  out of 49 possible cases.
              \item For $n=4=10_4$ we $2$ solutions:
                  \[
                   \begin{array}{l}
                     1103=1^{10}+1^{10}+0^{10}+3^{10}=1+1+0+1101~, \\
                     3303=3^{10}+3^{10}+0^{10}+3^{10}=1101+1101+0+1101~,
                   \end{array}
                  \]
                  out of 193 possible cases.
              \item[(5-13)] For $n=5=11_4$, $n=6=12_4$, \ldots, $n=13=31_4$ we do not have solutions, out of 769, 3073, \ldots,
                  503316493 possible cases.
      \end{enumerate}
        \item etc.
      \end{enumerate}
  \item Conjecture of Erd\"{o}s-Straus: \emph{for} $n$ \emph{natural number} $n\ge2$ \emph{the equation}
  \[
   \frac{4}{n}=\frac{1}{x}+\frac{1}{y}+\frac{1}{z}
  \]
  \emph{admits at least a solution} $(x,y,z)\in\Ns\times\Ns\times\Ns$. Important theoretical results were obtained by
  \cite{Tao2011} and \cite{Elsholtz+Tao2012}, but the previous statement was not yet proved. \cite{Swett2006} announced that he has verified the statement for $n\le10^{14}$. We give solutions of this equation in the form $\boxed{x,y,z}$:
  \begin{itemize}
    \item[n=2] \boxed{1,2,2};
    \item[n=3] \boxed{1,4,12}, \boxed{1,6,6}, \boxed{2,2,3};
    \item[n=4] \boxed{2,3,6}, \boxed{2,4,4}, \boxed{3,3,3};
    \item[n=5] \boxed{2,4,20}, \boxed{2,5,10};
    \item[n=6] \boxed{2,7,42}, \boxed{2,8,24}, \boxed{2,9,18}, \boxed{2,10,15}, \boxed{2,12,12}, \boxed{3,4,12}, \boxed{3,6,6}, \boxed{4,4,6};
    \item[n=7] \boxed{2,18,63}, \boxed{2,21,42}, \boxed{2,28,28}, \boxed{3,6,14}, \boxed{4,4,14};
    \item[n=8] \boxed{3,7,42}, \boxed{3,8,24}, \boxed{3,9,18}, \boxed{3,10,15}, \boxed{3,12,12}, \boxed{4,5,20}, \boxed{4,6,12}, \boxed{4,8,8}, \boxed{5,5,10}, \boxed{6,6,6};
    \item[n=9] \boxed{3,10,90}, \boxed{3,12,36}, \boxed{3,18,18}, \boxed{4,6,36}, \boxed{4,9,12}, \boxed{6,6,9};
    \item[n=10] \boxed{3,18,90}, \boxed{3,20,60}, \boxed{3,24,40}, \boxed{3,30,30}, \boxed{4,8,40}, \boxed{4,10,20}, \boxed{4,12,15}, \boxed{5,6,30}, \boxed{5,10,10}, \boxed{6,6,15};
    \item[n=11] \boxed{3,66,66}, \boxed{4,11,44}, \boxed{4,12,33}, \boxed{6,6,33};
    \item[n=12] \boxed{4,14,84}, \boxed{4,15,60}, \boxed{4,16,48,}, \boxed{4,18,36}, \boxed{4,20,30}, \boxed{4,21,28}, \boxed{4,24,24}, \boxed{5,9,45}, \boxed{5,10,30}, \boxed{5,12,20}, \boxed{5,15,15}, \boxed{6,7,42}, \boxed{6,8,24}, \boxed{6,9,18}, \boxed{6,10,15}, \boxed{6,12,12}, \boxed{7,7,21}, \boxed{8,8,12}, \boxed{9,9,9};
    \item[n=13] \boxed{4,26,52}.
  \end{itemize}
  Obviously, these solutions verify the equation, as the solution for $n=13$
  \[
   \frac{4}{13}=\frac{1}{4}+\frac{1}{26}+\frac{1}{52}~.
  \]
\end{enumerate}

\section{The $\eta$--Diophantine equations}
Let $m,n\in\Ns$ fixed and $x$ and $y$ unknown positive integers. The Diophantine equations in which function $\eta$ is
involved are called $\eta$--Diophantine. The list of $\eta$--Diophantine equations, considered from \citep{Smarandache1999a}, which we have into consideration to solve empirically are:
\begin{enumerate}
  \item[(2069)] $\eta(m\cdot x+n)=x$~,
  \item[(2070)] $\eta(m\cdot x+n)=m+n\cdot x$~,
  \item[(2071)] $\eta(m\cdot x+n)=x!$~,
  \item[(2072)] $\eta(x^m)=x^n$~,
  \item[(2073)] $\eta(x)^m=\eta(x^n)$~,
  \item[(2074)] $\eta(m\cdot x+n)=\eta(x)^y$~,
  \item[(2075)] $\eta(x)+y=x+\eta(y)$~, where $x\neq y$, $x$ and $y$ are not prime,
  \item[(2076)] $\eta(x)+\eta(y)=\eta(x+y)$, where $x$ and $y$ are not siblings prime,
  \item[(2077)] $\eta(x+y)=\eta(x)\cdot\eta(y)$~,
  \item[(2078)] $\eta(x\cdot y)=\eta(x)\cdot\eta(y)$~,
  \item[(2079)] $\eta(m\cdot x+n)=x^y$~,
  \item[(2080)] $\eta(x)\cdot y=x\cdot\eta(y)$, where $x$ and $y$ are not
  prime,
  \item[(2081)] $\eta(x)\cdot\eta(y)=x\cdot y$~, where $x$ and $y$ are not prime,
  \item[(2082)] $\eta(x)^y=x^{\eta(y)}$, where $x$ and $y$ are not prime,
  \item[(2083)] $\eta(x)^{\eta(y)}=\eta(x^y)$~,
  \item[(2084)] $\eta(x^y)-\eta(z^w)=1$, with $y\neq1\neq w$,
  \item[(2085)] $\eta(x^y)=y$, with $y\ge2$~,
  \item[(2086)] $\eta(x^x)=y^y$~,
  \item[(2087)] $\eta(x^y)=y^x$~,
  \item[(2088)] $\eta(x)=y!$~,
  \item[(2089)] $\eta(m\cdot x)=m\cdot\eta(x)$, with $m\ge2$~,
  \item[(2090)] $m^{\eta(x)}+\eta(x)^n=m^n$~.
  \item[(2091)] $n\cdot\eta(x^2)\pm m\cdot\eta(y^2)=m\cdot n$~,
  \item[(2092)] $\eta(x_1^{y_1}+x_2^{y_2}+\ldots+x_r^{y_r})=\eta(x_1)^{y_1}+\eta(x_2)^{y_2}+\ldots+\eta(x_r)^{y_r}$~,
  \item[(2093)] $\eta(x_1!+x_2!+\ldots+x_r!)=\eta(x_1)!+\eta(x_2)!+\ldots+\eta(x_r)!$~,
  \item[(2094)] $\big(x,y\big)=\big(\eta(x),\eta(y)\big)$, where by $\big(\cdot,\cdot\big)$ we understand is the greatest common divisor and $x$ and
  $y$ are not prime,
  \item[(2095)] $\big[x,y\big]=\big[\eta(x),\eta(y)\big]$, where by $\big[\cdot,\cdot\big]$ we understand is the smallest common multiple and $x$
  and $y$ are not prime.
\end{enumerate}

\subsection{Partial empirical solving of $\eta$--Diophantine equations}

For all Diophantine equations solved in this section the file $\eta.prn$ is read, generated by the program \ref{ProgEta}, by
means of Mathcad function $READPRN$
\[
 \eta:=READPRN("...\backslash \eta.prn")\ \ last(\eta)=10^6
\]
where the command $last(\eta)$ indicates the last index of vector $\eta$.

\subsection{The equation 2069}

\begin{prog}  Given vector $\eta$, the equation $\eta(mx+n)=x$ is equivalent with the relation $\eta_{mx+n}=x$. The program
to find the solutions of equation $(2069)$ is:
  \begin{tabbing}
    $Ed2069(a_m,b_m,a_n,b_n,a_x,b_x):=$\=\ \vline\ $S\leftarrow("m"\ "n"\ "x")$\\
    \>\ \vline\ $u\leftarrow last(\eta)$\\
    \>\ \vline\ $f$\=$or\ m\in a_m..b_m$\\
    \>\ \vline\ \>\ $f$\=$or\ n\in a_n..b_n$\\
    \>\ \vline\ \>\ \>\ $f$\=$or\ x\in a_x..b_x$\\
    \>\ \vline\ \>\ \>\ \>\ \vline\ $\eta\leftarrow m\cdot x+n$\\
    \>\ \vline\ \>\ \>\ \>\ \vline\ $q\leftarrow\eta\le u\wedge \eta_\eta\textbf{=}x$\\
    \>\ \vline\ \>\ \>\ \>\ \vline\ $S\leftarrow stack[S,(m\ n\ x)]\ if\ q$\\
    \>\ \vline\ $return\ S$\\
   \end{tabbing}
   The call of the program is done by the sequence:
   \[
    a_m:=2\ b_m:=10\ \ \ a_n:=1\ b_n:=10\ \ \ a_x:=1\ b_x:=16~,
   \]
   hence, the search domain is
   \[
    D_c=\set{2,3,\ldots,10}\times\set{1,2,\ldots,10}\times\set{1,2,\ldots,16}~.
   \]
   The total number of verified cases is:
   \[
   (b_m-a_m+1)(b_n-a_n+1)(b_x-a_x+1)=1440~.
   \]
   The call of the program Ed2069:
   \[
    t_0:time(0)\ \ Sol:=Ed2069(a_m,b_m,a_n,b_n,a_x,b_x)\ \ t_1:=time(1)
   \]
   The execution time in seconds and the number of solutions follow from:
   \[
    (t_1-t_0)\cdot s=0.011\cdot s\ \ \ \ \ rows(Sol)-1=36
   \]
   For $m\in\set{2,3,\ldots,10}$, $n\in\set{1,2,\ldots,10}$ and $x\in\set{1,2,\ldots,16}$ the $36$ solutions of the Diophantine equation    $\eta(m\cdot x+n)=x$, given as $\boxed{m,n,x}$, are:
\begin{itemize}
  \item[] \boxed{2,4,4} \boxed{2,4,6} \boxed{2,5,5} \boxed{2,5,10} \boxed{2,6,6} \boxed{2,7,7} \boxed{2,9,9} \boxed{2,10,5}~;
  \item[] \boxed{3,5,5} \boxed{3,7,7} \boxed{3,7,14} \boxed{3,8,8}~;
  \item[] \boxed{4,7,7} \boxed{4,8,4} \boxed{4,10,5} \boxed{4,10,10}~;
  \item[] \boxed{5,4,4} \boxed{5,5,5} \boxed{5,6,6} \boxed{5,7,7} \boxed{5,9,9}~;
  \item[] \boxed{6,9,6} \boxed{6,10,5}~;
  \item[] \boxed{7,3,6} \boxed{7,5,5} \boxed{7,5,10} \boxed{7,6,6} \boxed{7,7,7} \boxed{7,8,8}~;
  \item[] \boxed{8,5,15} \boxed{8,7,7} \boxed{8,9,9}~;
  \item[] \boxed{9,7,7} \boxed{9,10,10}~;
  \item[] \boxed{10,7,14} \boxed{10,10,5}~.
\end{itemize}
The maximum value of solutions $x$ is $15$.

By a similar call, the program $Ed2069$ provides the $40$ solutions of the Diophantine equation $\eta(m\cdot x+n)=x$ of the search domain
\[
 D_c=\set{97,98,\ldots,100}\times\set{11,12,\ldots,99}\times\set{43,44,\ldots,89}
\]
in the form $\boxed{m,n,x}$:
\begin{itemize}
  \item[]
  \begin{flushleft}
  \boxed{97,43,43} \boxed{97,47,47} \boxed{97,53,53} \boxed{97,59,59} \boxed{97,61,61} \boxed{97,67,67} \boxed{97,71,71}
  \boxed{97,73,73} \boxed{97,79,79} \boxed{97,83,83} \boxed{97,86,43} \boxed{97,89,89} \boxed{97,94,47}~;
  \end{flushleft}
  \item[]
  \begin{flushleft}
  \boxed{98,43,43} \boxed{98,47,47} \boxed{98,53,53} \boxed{98,59,59} \boxed{98,61,61} \boxed{98,67,67} \boxed{98,71,71} \boxed{98,73,73} \boxed{98,79,79} \boxed{98,83,83} \boxed{98,86,43} \boxed{98,89,89} \boxed{98,94,47}~;
  \end{flushleft}
  \item[]
  \begin{flushleft}
  \boxed{99,43,43} \boxed{99,47,47} \boxed{99,53,53} \boxed{99,59,59} \boxed{99,61,61} \boxed{99,67,67} \boxed{99,71,71} \boxed{99,73,73} \boxed{99,79,79} \boxed{99,83,83} \boxed{99,89,89}~;
  \end{flushleft}
  \item[] \boxed{100,86,43} \boxed{100,87,58} \boxed{100,94,47}~.
\end{itemize}
The maximum value of solutions $x$ is $89$.
\end{prog}

\subsection{The equation 2070}

\begin{prog}  Given vector $\eta$, the equation $\eta(mx+n)=m+nx$ is equivalent with the relation $\eta_{mx+n}=m+nx$. The program for finding the solutions of the equation $(2070)$ is:
  \begin{tabbing}
    $Ed2070(a_m,b_m,a_n,b_n,a_x,b_x):=$\=\ \vline\ $S\leftarrow("m"\ "n"\ "x")$\\
    \>\ \vline\ $u\leftarrow last(\eta)$\\
    \>\ \vline\ $f$\=$or\ m\in a_m..b_m$\\
    \>\ \vline\ \>\ $f$\=$or\ n\in a_n..b_n$\\
    \>\ \vline\ \>\ \>\ $f$\=$or\ x\in a_x..b_x$\\
    \>\ \vline\ \>\ \>\ \>\ \vline\ $\eta\leftarrow m\cdot x+n$\\
    \>\ \vline\ \>\ \>\ \>\ \vline\ $q\leftarrow\eta\le u\wedge \eta_\eta\textbf{=}m+n\cdot x$\\
    \>\ \vline\ \>\ \>\ \>\ \vline\ $S\leftarrow stack[S,(m\ n\ x)]\ if\ q$\\
    \>\ \vline\ $return\ S$\\
   \end{tabbing}
The call of the program is done by the sequence:
   \[
    a_m:=2\ b_m:=20\ \ \ a_n:=1\ b_n:=20\ \ \ a_x:=1\ b_x:=16~,
   \]
hence, the search domain is
   \[
    D_c=\set{2,3,\ldots,20}\times\set{1,2,\ldots,20}\times\set{1,2,\ldots,16}~.
   \]
The total number of verified cases is:
   \[
   (b_m-a_m+1)(b_n-a_n+1)(b_x-a_x+1)=7220~.
   \]
The call of the program Ed2070:
   \[
    t_0:time(0)\ \ Sol:=Ed2070(a_m,b_m,a_n,b_n,a_x,b_x)\ \ t_1:=time(1)
   \]
The execution time in seconds and the number of solutions follow from:
   \[
    (t_1-t_0)\cdot s=0.853\cdot s\ \ \ \ \ rows(Sol)-1=14
   \]
For $m\in\set{2,3,\ldots,20}$, $n\in\set{1,2,\ldots,20}$ and $x\in\set{2,3,\ldots,20}$ the $14$ solutions of the Diophantine equation $\eta(m\cdot x+n)=m+n\cdot x$, given as $\boxed{m,n,x}$, are:
   \begin{itemize}
     \item[] \boxed{2,1,4}~;
     \item[] \boxed{4,1,2} \boxed{4,1,6}~;
     \item[] \boxed{6,1,4} \boxed{6,1,8}~;
     \item[] \boxed{8,1,6} \boxed{8,1,6}~;
     \item[] \boxed{10,1,12}~;
     \item[] \boxed{12,1,10} \boxed{12,1,14}~;
     \item[] \boxed{14,1,12}~;
     \item[] \boxed{16,1,18}~;
     \item[] \boxed{18,1,16} \boxed{18,1,20}~;
     \item[] \boxed{20,1,18}~.
   \end{itemize}
 The maximum value of solutions $x$ is $20$.
\end{prog}

\subsection{The equation 2071}

\begin{prog}  Given vector $\eta$, the equation $\eta(mx+n)=x!$ is equivalent with the relation $\eta_{mx+n}=x!$. The program for finding the solutions of the equation $(2071)$ is:
  \begin{tabbing}
    $Ed2071(a_m,b_m,a_n,b_n,a_x,b_x):=$\=\ \vline\ $S\leftarrow("m"\ "n"\ "x")$\\
    \>\ \vline\ $u\leftarrow last(\eta)$\\
    \>\ \vline\ $f$\=$or\ m\in a_m..b_m$\\
    \>\ \vline\ \>\ $f$\=$or\ n\in a_n..b_n$\\
    \>\ \vline\ \>\ \>\ $f$\=$or\ x\in a_x..b_x$\\
    \>\ \vline\ \>\ \>\ \>\ \vline\ $\eta\leftarrow m\cdot x+n$\\
    \>\ \vline\ \>\ \>\ \>\ \vline\ $q\leftarrow\eta\le u\wedge \eta_\eta\textbf{=}x!$\\
    \>\ \vline\ \>\ \>\ \>\ \vline\ $S\leftarrow stack[S,(m\ n\ x)]\ if\ q$\\
    \>\ \vline\ $return\ S$\\
   \end{tabbing}
The call of the program is done by the sequence:
   \[
    a_m:=2\ b_m:=15\ \ \ a_n:=1\ b_n:=15\ \ \ a_x:=1\ b_x:=19~,
   \]
hence, the search domain is
   \[
    D_c=\set{2,3,\ldots,15}\times\set{1,2,\ldots,15}\times\set{1,2,\ldots,19}~.
   \]
 The total number of verified cases is:
   \[
   (b_m-a_m+1)(b_n-a_n+1)(b_x-a_x+1)=3990~.
   \]
 The call of the program Ed2071:
   \[
    t_0:time(0)\ \ Sol:=Ed2071(a_m,b_m,a_n,b_n,a_x,b_x)\ \ t_1:=time(1)
   \]
 The execution time in seconds and the number of solutions follow
from:
   \[
    (t_1-t_0)\cdot s=0.02\cdot s\ \ \ \ \ rows(Sol)-1=24
   \]
For $m\in\set{2,3,\ldots,15}$, $n\in\set{1,2,\ldots,15}$ and $x\in\set{1,2,\ldots,19}$ the $24$ solutions of the Diophantine
equation $\eta(m\cdot x+n)=x!$, given as $\boxed{m,n,x}$, are:
   \begin{itemize}
     \item[] \boxed{2,3,3} \boxed{2,10,3} \boxed{2,12,3}~;
     \item[] \boxed{3,7,3} \boxed{3,9,3}~;
     \item[] \boxed{4,4,3} \boxed{4,6,3}~;
     \item[] \boxed{5,1,3} \boxed{5,3,3}~;
     \item[] \boxed{7,15,3}~;
     \item[] \boxed{8,12,3}~;
     \item[] \boxed{9,9,3}~;
     \item[] \boxed{10,6,3} \boxed{10,15,3}~;
     \item[] \boxed{11,3,3} \boxed{11,12,3} \boxed{11,15,3}~;
     \item[] \boxed{12,9,3} \boxed{12,12,3}~;
     \item[] \boxed{13,6,3} \boxed{13,9,3}~;
     \item[] \boxed{14,3,3} \boxed{14,6,3}~;
     \item[] \boxed{15,3,3}~.
   \end{itemize}
The maximum value of solutions $x$ is $3$.
\end{prog}

\subsection{The equation 2072}

\begin{prog}  Given vector $\eta$, the equation $\eta(x^m)=x^n$ is equivalent with the relation $\eta_{x^m}=x^n$.
  The program for finding the solutions of the equation $(2072)$ is:
  \begin{tabbing}
    $Ed2072(a_m,b_m,a_n,b_n,a_x,b_x):=$\=\ \vline\ $S\leftarrow("m"\ "n"\ "x")$\\
    \>\ \vline\ $u\leftarrow last(\eta)$\\
    \>\ \vline\ $f$\=$or\ m\in a_m..b_m$\\
    \>\ \vline\ \>\ $f$\=$or\ n\in a_n..b_n$\\
    \>\ \vline\ \>\ \>\ $f$\=$or\ x\in a_x..b_x$\\
    \>\ \vline\ \>\ \>\ \>\ \vline\ $\eta\leftarrow x^m$\\
    \>\ \vline\ \>\ \>\ \>\ \vline\ $q\leftarrow\eta\le u\wedge \eta_\eta\textbf{=}x^n$\\
    \>\ \vline\ \>\ \>\ \>\ \vline\ $S\leftarrow stack[S,(m\ n\ x)]\ if\ q$\\
    \>\ \vline\ $return\ S$\\
   \end{tabbing}
   The call of the program is done by the sequence:
   \[
    a_m:=2\ b_m:=9\ \ \ a_n:=2\ b_n:=9\ \ \ a_x:=2\ b_x:=10~,
   \]
   hence, the search domain is
   \[
    D_c=\set{2,3,\ldots,9}\times\set{2,3,\ldots,9}\times\set{2,3,\ldots,10}~.
   \]
   The total number of verified cases is:
   \[
    (b_m-a_m+1)(b_n-a_n+1)(b_x-a_x+1)=576~.
   \]
   The call of the program Ed2072:
   \[
    t_0:time(0)\ \ Sol:=Ed2072(a_m,b_m,a_n,b_n,a_x,b_x)\ \ t_1:=time(1)
   \]
   The execution time in seconds and the number of solutions follow
  from:
   \[
    (t_1-t_0)\cdot s=0.11\cdot s\ \ \ \ \ rows(Sol)-1=12~.
   \]
   For $m\in\set{2,3,\ldots,9}$, $n\in\set{2,3,\ldots,9}$ and $x\in\set{1,2,\ldots,10}$ the $12$ solutions of the Diophantine equation $\eta(x^m)=x^n$, given as $\boxed{m,n,x}$, are:
   \begin{itemize}
     \item[] \boxed{2,2,2}~;
     \item[] \boxed{3,2,2} \boxed{3,2,3}
     \item[] \boxed{4,2,3}~;
     \item[] \boxed{5,2,5} \boxed{5,3,2}~;
     \item[] \boxed{6,2,4} \boxed{6,2,5} \boxed{6,3,2}~;
     \item[] \boxed{7,2,4} \boxed{7,2,7} \boxed{7,3,2}~.
   \end{itemize}
  The maximum value of solutions $x$ is $7$.
\end{prog}

\subsection{The equation 2073}

\begin{prog}
  Given vector $\eta$, the equation $\eta(x)^m=\eta(x^n)$ is equivalent with the relation $\left(\eta_x\right)^m=\eta_{x^n}$. The program for finding the solutions of the equation $(2073)$ is:
  \begin{tabbing}
    $Ed2073(a_m,b_m,a_n,b_n,a_x,b_x):=$\=\ \vline\ $S\leftarrow("m"\ "n"\ "x")$\\
    \>\ \vline\ $u\leftarrow last(\eta)$\\
    \>\ \vline\ $f$\=$or\ m\in a_m..b_m$\\
    \>\ \vline\ \>\ $f$\=$or\ n\in a_n..b_n$\\
    \>\ \vline\ \>\ \>\ $f$\=$or\ x\in a_x..b_x$\\
    \>\ \vline\ \>\ \>\ \>\ \vline\ $q\leftarrow x^n\le u\wedge (\eta_x)^m\textbf{=}\eta_{x^n}$\\
    \>\ \vline\ \>\ \>\ \>\ \vline\ $S\leftarrow stack[S,(m\ n\ x)]\ if\ q$\\
    \>\ \vline\ $return\ S$\\
   \end{tabbing}
  The call of the program is done by the sequence:
   \[
    a_m:=2\ b_m:=9\ \ \ a_n:=2\ b_n:=9\ \ \ a_x:=2\ b_x:=25~,
   \]
  hence, the search domain is
   \[
    D_c=\set{2,3,\ldots,9}\times\set{2,3,\ldots,9}\times\set{2,3,\ldots,25}~.
   \]
  The total number of verified cases is:
   \[
   (b_m-a_m+1)(b_n-a_n+1)(b_x-a_x+1)=1536~.
   \]
  The call of the program Ed2073:
   \[
    t_0:time(0)\ \ Sol:=Ed2073(a_m,b_m,a_n,b_n,a_x,b_x)\ \ t_1:=time(1)
   \]
  The execution time in seconds and the number of solutions follow from:
   \[
    (t_1-t_0)\cdot s=0.014\cdot s\ \ \ \ \ rows(Sol)-1=20
   \]
For $m\in\set{2,3,\ldots,9}$, $n\in\set{2,3,\ldots,9}$ and $x\in\set{2,3,\ldots,25}$ the $20$ solutions the Diophantine equation $\eta(x)^m=\eta(x^n)$, in the form $\boxed{m,n,x}$ are:
   \begin{itemize}
     \item[] \boxed{2,2,2}~;
     \item[] \boxed{2,3,2} \boxed{2,3,3} \boxed{2,3,6}~;
     \item[] \boxed{2,4,3} \boxed{2,4,6} \boxed{2,4,8} \boxed{2,4,24}~;
     \item[] \boxed{2,5,5} \boxed{2,5,8} \boxed{2,5,10} \boxed{2,5,15}~;
     \item[] \boxed{2,6,4} \boxed{2,6,5} \boxed{2,6,10}~;
     \item[] \boxed{2,7,4} \boxed{2,7,7}~;
     \item[] \boxed{3,5,2} \boxed{3,6,2} \boxed{3,7,2}~.
   \end{itemize}
  The maximum value of solutions $x$ is $24$.
\end{prog}

\subsection{The equation 2074}

\begin{prog}
  Given vector $\eta$, the equation $\eta(mx+n)=\eta(x)^m$ is equivalent with the relation $\eta_{mx+n}=(\eta_x)^m$. The program for finding the solutions of the equation $(2074)$ is:
  \begin{tabbing}
    $Ed2074(a_m,b_m,a_n,b_n,a_x,b_x):=$\=\ \vline\ $S\leftarrow("m"\ "n"\ "x")$\\
    \>\ \vline\ $u\leftarrow last(\eta)$\\
    \>\ \vline\ $f$\=$or\ m\in a_m..b_m$\\
    \>\ \vline\ \>\ $f$\=$or\ n\in a_n..b_n$\\
    \>\ \vline\ \>\ \>\ $f$\=$or\ x\in a_x..b_x$\\
    \>\ \vline\ \>\ \>\ \>\ \vline\ $\eta\leftarrow m\cdot x+n$\\
    \>\ \vline\ \>\ \>\ \>\ \vline\ $q\leftarrow \eta\le u\wedge \eta_\eta\textbf{=}\left(\eta_x\right)^m$\\
    \>\ \vline\ \>\ \>\ \>\ \vline\ $S\leftarrow stack[S,(m\ n\ x)]\ if\ q$\\
    \>\ \vline\ $return\ S$\\
   \end{tabbing}
   The call of the program is done by the sequence:
   \[
    a_m:=1\ b_m:=6\ \ \ a_n:=1\ b_n:=9\ \ \ a_x:=1\ b_x:=10^5~,
   \]
   then the search domain is
   \[
    D_c=\set{1,2,\ldots,6}\times\set{1,2,\ldots,9}\times\set{1,2,\ldots,10^5}~.
   \]
   The total number of verified cases is:
   \[
   (b_m-a_m+1)(b_n-a_n+1)(b_x-a_x+1)=5400000~.
   \]
   The call of the program Ed2074:
   \[
    t_0:time(0)\ \ Sol:=Ed2074(a_m,b_m,a_n,b_n,a_x,b_x)\ \ t_1:=time(1)
   \]
   The execution time in seconds and the number of solutions follow
   from:
   \[
    (t_1-t_0)\cdot s=0.017\cdot s\ \ \ \ \ rows(Sol)-1=24
   \]
   For $m\in\set{1,2,\ldots,6}$, $n\in\set{1,2,\ldots,9}$ and $x\in\set{1,2,\ldots,10^5}$ the $24$ solutions of the Diophantine equation
   $\eta(mx+n)=\eta(x)^m$, in the form $\boxed{m,n,x}$, are:
   \begin{itemize}
     \item[] \boxed{1,2,16}~;
     \item[] \boxed{1,3,3} \boxed{1,3,45}~;
     \item[] \boxed{1,4,4} \boxed{1,4,8}~;
     \item[] \boxed{1,5,5} \boxed{1,5,10} \boxed{1,5,15}~;
     \item[] \boxed{1,7,7} \boxed{1,7,9} \boxed{1,7,14} \boxed{1,7,21} \boxed{1,7,28} \boxed{1,7,35} \boxed{1,7,56} \boxed{1,7,63} \boxed{1,7,105}~;
     \item[] \boxed{1,8,4} \boxed{1,8,72} \boxed{1,8,1792}~;
     \item[] \boxed{1,9,9} \boxed{1,9,36}~;
     \item[] \boxed{2,4,2} \boxed{2,8,2}~.
   \end{itemize}
   The maximum value of solutions $x$ is $1792$.
\end{prog}

The equation (2074) has a similar version in the form of the equation $(2074')$ $\eta(mx+n)=\eta(x)^n$.

\begin{prog}
  Given vector $\eta$, the equation $\eta(mx+n)=\eta(x)^n$ is equivalent with the relation $\eta_{mx+n}=(\eta_x)^n$. The program for finding the solutions of the equation $(2074')$ is:
  \begin{tabbing}
    $Ed20741(a_m,b_m,a_n,b_n,a_x,b_x):=$\=\ \vline\ $S\leftarrow("m"\ "n"\ "x")$\\
    \>\ \vline\ $u\leftarrow last(\eta)$\\
    \>\ \vline\ $f$\=$or\ m\in a_m..b_m$\\
    \>\ \vline\ \>\ $f$\=$or\ n\in a_n..b_n$\\
    \>\ \vline\ \>\ \>\ $f$\=$or\ x\in a_x..b_x$\\
    \>\ \vline\ \>\ \>\ \>\ \vline\ $\eta\leftarrow m\cdot x+n$\\
    \>\ \vline\ \>\ \>\ \>\ \vline\ $q\leftarrow \eta\le u\wedge \eta_\eta\textbf{=}\left(\eta_x\right)^n$\\
    \>\ \vline\ \>\ \>\ \>\ \vline\ $S\leftarrow stack[S,(m\ n\ x)]\ if\ q$\\
    \>\ \vline\ $return\ S$\\
   \end{tabbing}
   The call of the program is done by the sequence:
   \[
    a_m:=1\ b_m:=9\ \ \ a_n:=1\ b_n:=9\ \ \ a_x:=1\ b_x:=10^5~,
   \]
   hence, the search domain is
   \[
    D_c=\set{1,2,\ldots,9}\times\set{1,2,\ldots,9}\times\set{1,2,\ldots,10^5}~.
   \]
   The total number of verified cases is:
   \[
   (b_m-a_m+1)(b_n-a_n+1)(b_x-a_x+1)=8100000~.
   \]
   The call of the program Ed20741:
   \[
    t_0:time(0)\ \ Sol:=Ed20741(a_m,b_m,a_n,b_n,a_x,b_x)\ \ t_1:=time(1)
   \]
   The execution time in seconds and the number of solutions follow from:
   \[
    (t_1-t_0)\cdot s=7.566\cdot s\ \ \ \ \ rows(Sol)-1=3
   \]
   For $m\in\set{1,2,\ldots,6}$, $n\in\set{1,2,\ldots,9}$ and $x\in\set{1,2,\ldots,10^5}$ the $3$ solutions of the Diophantine equation
   $\eta(mx+n)=\eta(x)^m$, in the form $\boxed{m,n,x}$, are:
   \begin{itemize}
     \item[] \boxed{1,2,2} \boxed{3,2,2} \boxed{5,2,2}~.
   \end{itemize}
   The maximum value of solutions $x$ is $2$.
\end{prog}

\subsection{The equation 2075}

\begin{prog}
  Given vector $\eta$, the equation $\eta(x)+y=x+\eta(y)$ cu $x\neq y$, where $x$ and $y$ are not prime, is equivalent with the relation $\eta_x+y=x+\eta_y$, with $x<y$ (we consider condition $x<y$ instead of $x\neq y$ for symmetry reasons of the equation relative to $x$ and $y$), where $x$ and $y$ are not prime. The program for finding the solutions of the equation $(2075)$ is:
  \begin{tabbing}
    $Ed2075(a_{xy},b_{xy}):=$\=\ \vline\ $S\leftarrow("x"\ "y")$\\
    \>\ \vline\ $f$\=$or\ x\in a_{xy}..b_{xy}-1$\\
    \>\ \vline\ \>\ $f$\=$or\ y\in x+1..b_{xy}$\\
    \>\ \vline\ \>\ \>\ \vline\ $u\leftarrow Tp\eta(x)=0$\\
    \>\ \vline\ \>\ \>\ \vline\ $v\leftarrow Tp\eta(y)=0$\\
    \>\ \vline\ \>\ \>\ \vline\ $notprime\leftarrow u\wedge v$\\
    \>\ \vline\ \>\ \>\ \vline\ $q\leftarrow \eta_x+y\textbf{=}x+\eta_y$\\
    \>\ \vline\ \>\ \>\ \vline\ $S\leftarrow stack[S,(x\ y)]\ if\ notprime\wedge q$\\
    \>\ \vline\ $return\ S$\\
   \end{tabbing}
   The program $Ed2075$ calls the program \ref{ProgTpEta} to run the primality test for $x$ and $y$.The call of the program is done by the sequence:
   \[
    a_{xy}:=2\ b_{xy}:=1000
   \]
   hence, the search domain is
   \[
    D_c=\set{2,3,\ldots,999}\times\set{3,4,\ldots,1000},\ \ \textnormal{with}\ \ x<y~.
   \]
   The total number of verified cases is:
   \[
   \sum_{x=2}^{999}\sum_{y=x+1}^{1000}1=498501~.
   \]
   The call of the program Ed2075:
   \[
    t_0:time(0)\ \ Sol:=Ed2075(a_{xy},b_{xy})\ \ t_1:=time(1)
   \]
   The execution time in seconds and the number of solutions follow from:
   \[
    (t_1-t_0)\cdot s=5.765\cdot s\ \ \ \ \ rows(Sol)-1=157
   \]

   For $x\in\set{2,3,\ldots,999}$ and $y\in\set{3,4,\ldots,1000}$ the $157$ solutions of the Diophantine equation $\eta(x)+y=x+\eta(y)$ in the form of pairs $\boxed{x,y}$ are:

\begin{itemize}
  \item[]
  \begin{flushleft}
  \boxed{6,9} \boxed{15,16} \boxed{20,25} \boxed{40,42} \boxed{40,49} \boxed{42,49} \boxed{45,52} \boxed{60,66} \boxed{63,64} \boxed{72,77} \boxed{75,78} \boxed{80,111} \boxed{84,88} \boxed{90,98} \boxed{96,99}~;
  \end{flushleft}
  \item[]
  \begin{flushleft}
  \boxed{108,110} \boxed{108,121} \boxed{110,121} \boxed{120,138} \boxed{126,136} \boxed{140,147} \boxed{140,152} \boxed{144,161} \boxed{147,152} \boxed{154,156} \boxed{154,169} \boxed{156,169} \boxed{160,171} \boxed{162,170} \boxed{168,184} \boxed{175,176} \boxed{180,203} \boxed{192,207} \boxed{195,196} \boxed{198,204}~;
  \end{flushleft}
  \item[]
  \begin{flushleft}
  \boxed{200,209} \boxed{210,232} \boxed{216,230} \boxed{220,228} \boxed{225,258} \boxed{231,242} \boxed{234,238} \boxed{243,245} \boxed{250,282} \boxed{256,287} \boxed{260,266} \boxed{264,276} \boxed{270,290} \boxed{272,289} \boxed{280,286} \boxed{288,294}~;
  \end{flushleft}
  \item[]
  \begin{flushleft}
  \boxed{300,319} \boxed{312,322} \boxed{320,325} \boxed{320,338} \boxed{325,338} \boxed{330,348} \boxed{336,376} \boxed{340,342} \boxed{340,361} \boxed{342,361} \boxed{343,345} \boxed{350,357} \boxed{352,363} \boxed{352,372} \boxed{360,413} \boxed{363,372} \boxed{378,410} \boxed{384,423} \boxed{390,406}~;
  \end{flushleft}
  \item[]
  \begin{flushleft}
  \boxed{408,414} \boxed{416,434} \boxed{420,472} \boxed{432,470} \boxed{441,488} \boxed{448,450} \boxed{455,459} \boxed{456,460} \boxed{462,492} \boxed{480,531} \boxed{486,553}~;
  \end{flushleft}
  \item[]
  \begin{flushleft}
  \boxed{500,582} \boxed{504,568} \boxed{506,529} \boxed{507,518} \boxed{510,522} \boxed{525,618} \boxed{528,564} \boxed{540,590} \boxed{544,558} \boxed{546,574} \boxed{560,632} \boxed{561,578} \boxed{567,589} \boxed{570,580} \boxed{572,602} \boxed{576,639} \boxed{588,615} \boxed{594,605} \boxed{594,636}~;
  \end{flushleft}
  \item[]
  \begin{flushleft}
  \boxed{600,649} \boxed{605,636} \boxed{608,620} \boxed{616,625} \boxed{624,658} \boxed{630,637} \boxed{630,712} \boxed{637,712} \boxed{640,711} \boxed{648,710} \boxed{660,708} \boxed{663,665} \boxed{672,747} \boxed{675,684} \boxed{675,686} \boxed{684,686} \boxed{690,696} \boxed{693,713}~;
  \end{flushleft}
  \item[]
  \begin{flushleft}
  \boxed{702,742} \boxed{714,738} \boxed{720,729} \boxed{735,824} \boxed{736,744} \boxed{748,774} \boxed{756,830} \boxed{770,782} \boxed{780,826} \boxed{792,852} \boxed{798,820}~;
  \end{flushleft}
  \item[]
  \begin{flushleft}
  \boxed{800,869} \boxed{810,890} \boxed{812,841} \boxed{816,846} \boxed{819,837} \boxed{825,851} \boxed{832,845} \boxed{836,860} \boxed{840,850} \boxed{840,867} \boxed{850,867} \boxed{874,888} \boxed{875,903} \boxed{880,948} \boxed{882,899} \boxed{896,925}~;
  \end{flushleft}
  \item[]
  \begin{flushleft}
  \boxed{900,979} \boxed{910,920} \boxed{912,940} \boxed{918,954} \boxed{924,996} \boxed{928,930} \boxed{928,961} \boxed{930,961} \boxed{936,994} \boxed{966,984} \boxed{968,989} \boxed{975,999}~.
  \end{flushleft}
\end{itemize}
  The maximum value of solutions $x$ is $975$ and of $y$ is $999$.
\end{prog}

\subsection{The equation 2076}

\begin{prog}
  Given vector $\eta$, the equation $\eta(x)+\eta(y)=\eta(x+y)$ cu $x\neq y$, where $x$ and $y$ are not prime, is equivalent with the relation $\eta_x+\eta_y=\eta_{x+y}$, with $x<y$ (we consider condition $x<y$ instead of $x\neq y$ for symmetry reasons of the equation relative to $x$ and $y$), where $x$ and $y$ are not prime. The program for finding the solutions of the equation $(2076)$ is:
  \begin{tabbing}
    $Ed2076(a_{xy},b_{xy}):=$\=\ \vline\ $S\leftarrow("x"\ "y")$\\
    \>\ \vline\ $f$\=$or\ x\in a_{xy}..b_{xy}-1$\\
    \>\ \vline\ \>\ $f$\=$or\ y\in x+1..b_{xy}$\\
    \>\ \vline\ \>\ \>\ \vline\ $u\leftarrow Tp\eta(x)\textbf{=}0$\\
    \>\ \vline\ \>\ \>\ \vline\ $v\leftarrow Tp\eta(y)\textbf{=}0$\\
    \>\ \vline\ \>\ \>\ \vline\ $notprime\leftarrow u\wedge v$\\
    \>\ \vline\ \>\ \>\ \vline\ $q\leftarrow \eta_x+\eta_y\textbf{=}\eta_{x+y}$\\
    \>\ \vline\ \>\ \>\ \vline\ $S\leftarrow stack[S,(x\ y)]\ if\ notprime\wedge q$\\
    \>\ \vline\ $return\ S$\\
   \end{tabbing}
   The program $Ed2076$ calls the program \ref{ProgTpEta} to run the primality test for $x$ and $y$. The call of the program is done by the sequence:
   \[
    a_{xy}:=4\ b_{xy}:=1000
   \]
   hence, the search domain is
   \[
    D_c=\set{4,5,\ldots,999}\times\set{5,6,\ldots,1000},\ \ \textnormal{with}\ \ x<y~.
   \]
   The total number of verified cases is:
   \[
   \sum_{x=4}^{999}\sum_{y=x+1}^{1000}1=496506~.
   \]
   The call of the program Ed2076:
   \[
    t_0:time(0)\ \ Sol:=Ed2076(a_{xy},b_{xy})\ \ t_1:=time(1)
   \]
   The execution time in seconds and the number of solutions follow from:
   \[
    (t_1-t_0)\cdot s=3.877\cdot s\ \ \ \ \ rows(Sol)-1=1277
   \]
  For $x\in\set{4,5,\ldots,999}$ and $y\in\set{5,6,\ldots,1000}$ the $1277$ solutions of the Diophantine equation $\eta(x)+\eta(y)=\eta(x+y)$, with $y>x$, $x$ and $y$ non-prime numbers in the form of pairs $\boxed{x, y}$ (the first $96$ and the last $75$ solutions) are:
\begin{itemize}
  \item[] \boxed{4,84} \boxed{4,234} \boxed{4,455} \boxed{4,456}~;
  \item[] \boxed{6,8} \boxed{6,48} \boxed{6,147} \boxed{6,150} \boxed{6,192}~;
  \item[]
  \begin{flushleft}
  \boxed{8,14} \boxed{8,24} \boxed{8,26} \boxed{8,38} \boxed{8,74} \boxed{8,86} \boxed{8,125} \boxed{8,134} \boxed{8,135} \boxed{8,158} \boxed{8,168} \boxed{8,194} \boxed{8,206} \boxed{8,218} \boxed{8,254} \boxed{8,326} \boxed{8,386} \boxed{8,446} \boxed{8,458} \boxed{8,468} \boxed{8,475} \boxed{8,554} \boxed{8,614} \boxed{8,626} \boxed{8,698} \boxed{8,758} \boxed{8,794} \boxed{8,878} \boxed{8,910} \boxed{8,912} \boxed{8,914} \boxed{8,926} \boxed{8,974} \boxed{8,998}~;
  \end{flushleft}
  \item[]
  \begin{flushleft}
  \boxed{9,56} \boxed{9,110} \boxed{9,143} \boxed{9,221} \boxed{9,297} \boxed{9,368} \boxed{9,390} \boxed{9,420} \boxed{9,612} \boxed{9,620} \boxed{9,851}~;
  \end{flushleft}
  \item[] \boxed{10,15} \boxed{10,40} \boxed{10,45} \boxed{10,144} \boxed{10,224} \boxed{10,294} \boxed{10,595} \boxed{10,640}~;
  \item[]
  \begin{flushleft}
  \boxed{12,15} \boxed{12,21} \boxed{12,27} \boxed{12,39} \boxed{12,57} \boxed{12,111} \boxed{12,129} \boxed{12,201} \boxed{12,237} \boxed{12,252} \boxed{12,260} \boxed{12,291} \boxed{12,309} \boxed{12,327}~;
      \boxed{12,378} \boxed{12,381} \boxed{12,489} \boxed{12,494} \boxed{12,579} \boxed{12,669} \boxed{12,687} \boxed{12,702} \boxed{12,729} \boxed{12,831} \boxed{12,921} \boxed{12,939} \boxed{12,960}~;
  \end{flushleft}
  \item[] \boxed{14,35} \boxed{14,84} \boxed{14,90} \boxed{14,280} \boxed{14,363} \boxed{14,700}~;
  \item[] \ \vdots
  \item[]
  \begin{flushleft}
  \boxed{630,840} \boxed{630,900} \boxed{636,637} \boxed{637,696} \boxed{637,705} \boxed{640,767} \boxed{640,880} \boxed{646,798} \boxed{650,702} \boxed{650,841} \boxed{658,720} \boxed{660,792} \boxed{666,703} \boxed{672,756} \boxed{672,924} \boxed{675,864} \boxed{680,765} \boxed{684,760} \boxed{690,897} \boxed{693,880} \boxed{696,986} \boxed{697,984}~;
  \end{flushleft}
  \item[]
  \begin{flushleft}
  \boxed{700,731} \boxed{700,851} \boxed{700,899} \boxed{700,910} \boxed{702,819} \boxed{704,990} \boxed{715,975} \boxed{720,749} \boxed{720,759} \boxed{720,819} \boxed{720,840} \boxed{722,779} \boxed{725,957} \boxed{726,840} \boxed{729,792} \boxed{735,944} \boxed{738,943} \boxed{750,837} \boxed{754,833} \boxed{754,928} \boxed{756,765} \boxed{759,828} \boxed{768,798} \boxed{768,927} \boxed{770,924} \boxed{777,810} \boxed{779,902} \boxed{780,910} \boxed{782,805} \boxed{783,980} \boxed{784,952}~;
  \end{flushleft}
  \item[]
  \begin{flushleft}
  \boxed{805,1000} \boxed{810,867} \boxed{810,900} \boxed{812,870} \boxed{816,918} \boxed{819,896} \boxed{820,861} \boxed{825,990} \boxed{832,858} \boxed{836,968} \boxed{836,969} \boxed{845,918} \boxed{845,1000}\boxed{850,884} \boxed{855,950} \boxed{860,989} \boxed{867,988} \boxed{875,896} \boxed{884,972} \boxed{891,924} \boxed{891,972}~;
  \end{flushleft}
  \item[] \boxed{903,946} \boxed{930,992}~.
\end{itemize}
  The maximum value of solutions $x$ is $930$ and of $y$ is $1000$.
\end{prog}

\subsection{The equation 2077}

Let us consider the Diophantine equation $\eta(x+y)=\eta(x)\cdot\eta(y)$ with $x\neq y$ on the set $\set{1,2,\ldots,5\cdot10^4-1}\times\set{2,3,\ldots,5\cdot10^4}$. Given vector $\eta$, the Diophantine equation (2077) is equivalent
with relation $\eta_{x+y}=\eta_x\cdot\eta_y$, with $x<y$ (we use condition $x<y$ instead of $x\neq y$ for symmetry reasons of the
equation relative to $x$ and $y$). The search of the solutions was done by means of the program:
\begin{tabbing}
    $Ed2077(a_{xy},b_{xy}):=$\=\ \vline\ $S\leftarrow("x"\ "y")$\\
    \>\ \vline\ $f$\=$or\ x\in a_{xy}..b_{xy}-1$\\
    \>\ \vline\ \>\ $f$\=$or\ y\in x+1..b_{xy}$\\
    \>\ \vline\ \>\ \>\ \vline\ $q\leftarrow \eta_{x+y}\textbf{=}\eta_x\cdot\eta_y$\\
    \>\ \vline\ \>\ \>\ \vline\ $S\leftarrow stack[S,(x\ y)]\ if\ q$\\
    \>\ \vline\ $return\ S$\\
\end{tabbing}
The search time was of $853.816s$ of $1249975000$ possible cases. There was no solution found on the search domain
$\set{1,2,\ldots,5\cdot10^4-1}\times\set{2,3,\ldots,5\cdot10^4}$.

Instead of the Diophantine equation (2077), which seems not to have any solutions, we propose equation $\eta(x\cdot
y)=\eta(x)+\eta(y)$ for $x\neq y$, where $x$ and $y$ are not prime. Vector $\eta$ allows us to write the equivalent relation to
the Diophantine equation $\eta_{x\cdot y}=\eta_x+\eta_y$ for $x<y$ (we use condition $x<y$ instead of $x\neq y$ for symmetry reasons
of the equation relative to $x$ and $y$). The program for finding the solutions of the Diophantine equation $\eta(x\cdot
y)=\eta(x)+\eta(y)$ is:
  \begin{tabbing}
    $Ed20771(a_{xy},b_{xy}):=$\=\ \vline\ $S\leftarrow("x"\ "y")$\\
    \>\ \vline\ $f$\=$or\ x\in a_{xy}..b_{xy}-1$\\
    \>\ \vline\ \>\ $f$\=$or\ y\in x+1..b_{xy}$\\
    \>\ \vline\ \>\ \>\ \vline\ $u\leftarrow Tp\eta(x)\textbf{=}0$\\
    \>\ \vline\ \>\ \>\ \vline\ $v\leftarrow Tp\eta(y)\textbf{=}0$\\
    \>\ \vline\ \>\ \>\ \vline\ $notprime\leftarrow u\wedge v$\\
    \>\ \vline\ \>\ \>\ \vline\ $q\leftarrow \eta_{x\cdot y}\textbf{=}\eta_x+\eta_y$\\
    \>\ \vline\ \>\ \>\ \vline\ $S\leftarrow stack[S,(x\ y)]\ if\ notprime\wedge q$\\
    \>\ \vline\ $return\ S$\\
   \end{tabbing}
The program $Ed20771$ calls the program \ref{ProgTpEta} to run the primality test for $x$ and $y$. The call of the program is done by
the sequence:
   \[
    a_{xy}:=4\ b_{xy}:=10^3
   \]
   hence, the search domain is
   \[
    D_c=\set{4,5,\ldots,999}\times\set{5,6,\ldots,1000},\ \ \textnormal{with}\ \ x<y~.
   \]
   The total number of verified cases is:
   \[
   \sum_{x=4}^{999}\sum_{y=x+1}^{1000}1=496506~.
   \]
   The call of the program Ed20772:
   \[
    t_0:time(0)\ \ Sol:=Ed20771(a_{xy},b_{xy})\ \ t_1:=time(1)
   \]
   The execution time in seconds and the number of solutions follow from:
   \[
    (t_1-t_0)\cdot s=4.979\cdot s\ \ \ \ \ rows(Sol)-1=9659
   \]
For $x\in\set{4,5,\ldots,999}$ and $y\in\set{5,6,\ldots,1000}$ the $9659$ solutions of the Diophantine equation $\eta(x\cdot
y)=\eta(x)+\eta(y)$, with $x<y$, $x$ and $y$ being non-prime in the form of pairs $\boxed{x, y}$ (the first $131$ and the last
$55$ solutions) are:
\begin{itemize}
  \item[] \boxed{4, 8} \boxed{4,24} \boxed{4,128} \boxed{4,384} \boxed{4,640} \boxed{4,896}~;
  \item[]
  \begin{flushleft}
  \boxed{6,9} \boxed{6,18} \boxed{6,36} \boxed{6,45} \boxed{6,72} \boxed{6,81} \boxed{6,90} \boxed{6,144} \boxed{6,162} \boxed{6,180} \boxed{6,243} \boxed{6,324} \boxed{6,360} \boxed{6,405} \boxed{6,486} \boxed{6,567} \boxed{6,648} \boxed{6,720} \boxed{6,729} \boxed{6,810} \boxed{6,972}~;
  \end{flushleft}
  \item[]
  \begin{flushleft}
  \boxed{8,12} \boxed{8,24} \boxed{8,64} \boxed{8,128} \boxed{8,192} \boxed{8,256} \boxed{8,320} \boxed{8,384} \boxed{8,448} \boxed{8,512} \boxed{8,576} \boxed{8,640} \boxed{8,768} \boxed{8,896} \boxed{8,960}~;
  \end{flushleft}
  \item[]
  \begin{flushleft}
  \boxed{9,81} \boxed{9,162} \boxed{9,243} \boxed{9,324} \boxed{9,405} \boxed{9,486} \boxed{9,567} \boxed{9,648} \boxed{9,810} \boxed{9,972}~;
  \end{flushleft}
  \item[]
  \begin{flushleft}
  \boxed{10,15} \boxed{10,20} \boxed{10,25} \boxed{10,30} \boxed{10,40} \boxed{10,50} \boxed{10,60} \boxed{10,75} \boxed{10,100} \boxed{10,120} \boxed{10,125} \boxed{10,150} \boxed{10,175} \boxed{10,200} \boxed{10,225} \boxed{10,250} \boxed{10,300} \boxed{10,350} \boxed{10,375} \boxed{10,400} \boxed{10,450} \boxed{10,500} \boxed{10,525} \boxed{10,600} \boxed{10,625} \boxed{10,675} \boxed{10,700}
      \boxed{10,750} \boxed{10,800} \boxed{10,875} \boxed{10,900} \boxed{10,1000}~;
  \end{flushleft}
  \item[] \boxed{12,24} \boxed{12,128} \boxed{12,384} \boxed{12,640} \boxed{12,896}~;
  \item[]
  \begin{flushleft}
  \boxed{14,21} \boxed{14,28} \boxed{14,35} \boxed{14,42} \boxed{14,49} \boxed{14,56} \boxed{14,63} \boxed{14,70} \boxed{14,84} \boxed{14,98} \boxed{14,105} \boxed{14,112} \boxed{14,126} \boxed{14,140} \boxed{14,147} \boxed{14,168} \boxed{14,196} \boxed{14,210} \boxed{14,245} \boxed{14,252} \boxed{14,280} \boxed{14,294} \boxed{14,315} \boxed{14,336} \boxed{14,343} \boxed{14,392} \boxed{14,420} \boxed{14,441} \boxed{14,490} \boxed{14,504} \boxed{14,539} \boxed{14,560} \boxed{14,588} \boxed{14,630} \boxed{14,637} \boxed{14,686} \boxed{14,735} \boxed{14,784} \boxed{14,840} \boxed{14,882} \boxed{14,980}~;
  \end{flushleft}
  \item[] \ \vdots
  \item[]
  \begin{flushleft}
  \boxed{888,925} \boxed{888,962} \boxed{888,999} \boxed{890,979} \boxed{891,924} \boxed{891,968} \boxed{891,990} \boxed{893,940} \boxed{893,987} \boxed{896,960} \boxed{897,920} \boxed{897,966} \boxed{899,930} \boxed{899,961} \boxed{899,992}~;
  \end{flushleft}
  \item[]
  \begin{flushleft}
  \boxed{900,1000} \boxed{901,954} \boxed{902,943} \boxed{902,984} \boxed{903,946} \boxed{903,989} \boxed{910,936} \boxed{910,975} \boxed{912,931} \boxed{912,950} \boxed{912,969} \boxed{912,988} \boxed{913,996} \boxed{915,976} \boxed{918,935} \boxed{918,952} \boxed{920,966} \boxed{923,994} \boxed{924,968} \boxed{924,990} \boxed{925,962} \boxed{925,999} \boxed{928,957} \boxed{928,986} \boxed{930,961} \boxed{930,992} \boxed{931,950} \boxed{931,969} \boxed{931,988} \boxed{935,952} \boxed{936,975} \boxed{940,987} \boxed{943,984} \boxed{946,989} \boxed{950,969} \boxed{950,988} \boxed{957,986} \boxed{961,992} \boxed{962,999} \boxed{968,990}~;
  \end{flushleft}
\end{itemize}

\subsection{The equation 2078}

The Diophantine equation $\eta(x\cdot y)=\eta(x)\cdot\eta(y)$ for $x\in\set{1,2,\ldots,10^3-1}$ and $y\in\set{2,3,\ldots,10^3}$ with
$x\neq y$ has the only the $999$ trivial solutions, in the form $\boxed{x,y}$: $\boxed{1,2}$ $\boxed{1,3}$ \ldots $\boxed{1,1000}$~.

\subsection{The equation 2079}

\begin{prog}  Given vector $\eta$, the equation $\eta(mx+n)=x^y$ is equivalent with the relation $\eta_{mx+n}=x^y$. The program for finding the solutions of the equation $(2079)$ is:
  \begin{tabbing}
    $Ed2079(a_m,b_m,a_n,b_n,a_x,b_x,y):=$\=\ \vline\ $S\leftarrow("m"\ "n"\ "x")$\\
    \>\ \vline\ $u\leftarrow last(\eta)$\\
    \>\ \vline\ $f$\=$or\ m\in a_m..b_m$\\
    \>\ \vline\ \>\ $f$\=$or\ n\in a_n..b_n$\\
    \>\ \vline\ \>\ \>\ $f$\=$or\ x\in a_x..b_x$\\
    \>\ \vline\ \>\ \>\ \>\ \vline\ $\eta\leftarrow m\cdot x+n$\\
    \>\ \vline\ \>\ \>\ \>\ \vline\ $q\leftarrow\eta\le u\wedge \eta_\eta\textbf{=}x^y$\\
    \>\ \vline\ \>\ \>\ \>\ \vline\ $S\leftarrow stack[S,(m\ n\ x)]\ if\ q$\\
    \>\ \vline\ $return\ S$\\
   \end{tabbing}
The call of the program is done by the sequence:
   \[
    y:=2\ \ a_m:=1\ b_m:=10\ \ \ a_n:=1\ b_n:=10\ \ \ a_x:=1\ b_x:=10^4~,
   \]
hence, the search domain is
   \begin{equation}\label{Dc2079}
     D_c=\set{1,2,\ldots,10}\times\set{1,2,\ldots,10}\times\set{1,2,\ldots,10^4}~.
   \end{equation}
The total number of verified cases is:
   \[
   (b_m-a_m+1)(b_n-a_n+1)(b_x-a_x+1)=10^6~.
   \]
The call of the program Ed2079:
   \[
    t_0:time(0)\ \ Sol:=Ed2079(a_m,b_m,a_n,b_n,a_x,b_x,y)\ \ t_1:=time(1)
   \]
The execution time in seconds and the number of solutions follow from:
   \[
    (t_1-t_0)\cdot s=0.834\cdot s\ \ \ \ \ rows(Sol)-1=16
   \]
For $m\in\set{1,1,\ldots,10}$, $n\in\set{1,2,\ldots,10}$ and $x\in\set{1,2,\ldots,10^4}$ the $16$ solutions of the Diophantine
equation $\eta(m\cdot x+n)=x^2$, in the form $\boxed{m,n,x}$, are:
\begin{itemize}
  \item[] \boxed{1,2,2} \boxed{1,6,2} \boxed{1,10,2}~;
  \item[] \boxed{2,4,2} \boxed{2,8,2}~;
  \item[] \boxed{3,2,2} \boxed{3,6,2}~;
  \item[] \boxed{4,4,2}~;
  \item[] \boxed{5,2,2}~;
  \item[] \boxed{6,9,3}~;
  \item[] \boxed{7,6,3} \boxed{7,10,2}~;
  \item[] \boxed{8,3,3} \boxed{8,8,2}~;
  \item[] \boxed{9,6,2}~;
  \item[] \boxed{10,4,2}~.
\end{itemize}
The maximum value of solutions $m$, $n$ and $x$ is, respectively, $10$, $10$ and $2$.

For $y=3$ the Diophantine equation $2079$ does not have solutions in the search domain (\ref{Dc2079}).

It should be noted that for $y=\frac{1}{2}$ ($y$ is not an integer!) the $\eta$--Diophantine equation $\eta(mx+n)=x^{1/2}$
has $8$ solutions in the search domain (\ref{Dc2079}). The $\eta$--Diophantine equation $\eta(mx+n)=x^{1/2}$ is equivalent
with the relation $\eta_{mx+n}=\sqrt{x}$. The call of the program is done by the sequence:
   \[
    y:=\frac{1}{2}\ \ a_m:=1\ b_m:=10\ \ \ a_n:=1\ b_n:=10\ \ \ a_x:=1\ b_x:=10^4~,
   \]
   hence we have the same search domain $D_c$ given by (\ref{Dc2079}). The total number of verified cases is:
   \[
   (b_m-a_m+1)(b_n-a_n+1)(b_x-a_x+1)=10^6~.
   \]
   The call of the program Ed2079:
   \[
    t_0:time(0)\ \ Sol:=Ed2079(a_m,b_m,a_n,b_n,a_x,b_x,y)\ \ t_1:=time(1)
   \]
   The execution time in seconds and the number of solutions follow from:
   \[
    (t_1-t_0)\cdot s=0.928\cdot s\ \ \ \ \ rows(Sol)-1=16
   \]
   For $m\in\set{1,2,\ldots,10}$, $n\in\set{1,2,\ldots,10}$ and $x\in\set{1,2,\ldots,10^4}$ the $8$ solutions of the Diophantine equation $\eta(m\cdot x+n)=x^{1/2}$, in the form $\boxed{m,n,x}$, are:
   \begin{itemize}
     \item[] \boxed{1,5,25} \boxed{1,7,49} \boxed{1,8,16} \boxed{1,9,36}~;
     \item[] \boxed{2,7,49} \boxed{2,8,36} \boxed{2,10,25}~;
     \item[] \boxed{5,7,49}~.
   \end{itemize}
The maximum value of solutions $m$, $n$ and $x$ is, respectively $5$, $10$ and $49$.
\end{prog}

\subsection{The equation 2080}

\begin{prog} Given vector $\eta$, the equation $\eta(x)\cdot y=x\cdot\eta(y)$ is equivalent with the relation $\eta_x\cdot y=x\cdot\eta_y$. The program for finding the solutions of the equation $(2080)$ is:
  \begin{tabbing}
    $Ed2080(a_{xy},b_{xy}):=$\=\vline\ $S\leftarrow("x"\ "y")$\\
    \>\vline\ $fo$\=$r\ x\in a_{xy}..b_{xy}-1$\\
    \>\vline\ \>\vline\ $u\leftarrow Tp\eta(x)=0$\\
    \>\vline\ \>\vline\ $fo$\=$r\ y\in x+1..b_{xy}$\\
    \>\vline\ \>\vline\ \>\vline\ $q\leftarrow u\wedge Tp\eta(y)\textbf{=}0$\\
    \>\vline\ \>\vline\ \>\vline\ $S\leftarrow stack[S,(x\ y)]\ if\ q\wedge \eta_x\cdot y\textbf{=}x\cdot \eta_y$\\
    \>\vline\ $return\ S$
  \end{tabbing}
The program calls the subprogram \ref{ProgTpEta} of establishing the primality of $x$ and $y$. The call of the program is done by
$a_{xy}:=2$, $b_{xy}:=10^3$ and hereby it results that we have the search domain
\[
 D_c=\set{2,3,\ldots,999}\times\set{3,4,\ldots,1000}~,
\]
with $x<y$, where $x$ and $y$ are not prime. Hence, we obtain a number of $498501$ possible cases. These cases have been ran
through by the program $Ed2080$ in $5.444$ seconds. A number of $13200$ solutions was obtained, of which we present the first $95$
and the last $81$:
\begin{itemize}
  \item[]
  \begin{flushleft}
  \boxed{6,8} \boxed{6,10} \boxed{6,14} \boxed{6,22} \boxed{6,26} \boxed{6,34} \boxed{6,38} \boxed{6,46} \boxed{6,58} \boxed{6,62} \boxed{6,74} \boxed{6,82} \boxed{6,86} \boxed{6,94}~;
  \end{flushleft}
  \item[]
  \begin{flushleft}
  \boxed{6,106} \boxed{6,118} \boxed{6,122} \boxed{6,134} \boxed{6,142} \boxed{6,146} \boxed{6,158} \boxed{6,166} \boxed{6,178} \boxed{6,194}~;
  \end{flushleft}
  \item[]
  \begin{flushleft}
  \boxed{6,202} \boxed{6,206} \boxed{6,214} \boxed{6,218} \boxed{6,226} \boxed{6,254} \boxed{6,262} \boxed{6,274} \boxed{6,278} \boxed{6,298}~;
  \end{flushleft}
  \item[]
  \begin{flushleft}
  \boxed{6,302} \boxed{6,314} \boxed{6,326} \boxed{6,334} \boxed{6,346} \boxed{6,358} \boxed{6,362} \boxed{6,382} \boxed{6,386} \boxed{6,394} \boxed{6,398}~;
  \end{flushleft}
  \item[] \boxed{6,422} \boxed{6,446} \boxed{6,454} \boxed{6,458} \boxed{6,466} \boxed{6,478} \boxed{6,482}~;
  \item[] \boxed{6,502} \boxed{6,514} \boxed{6,526} \boxed{6,538} \boxed{6,542} \boxed{6,554} \boxed{6,562} \boxed{6,566} \boxed{6,586}~;
  \item[] \boxed{6,614} \boxed{6,622} \boxed{6,626} \boxed{6,634} \boxed{6,662} \boxed{6,674} \boxed{6,694} \boxed{6,698}~;
  \item[] \boxed{6,706} \boxed{6,718} \boxed{6,734} \boxed{6,746} \boxed{6,758} \boxed{6,766} \boxed{6,778} \boxed{6,794}~;
  \item[] \boxed{6,802} \boxed{6,818} \boxed{6,838} \boxed{6,842} \boxed{6,862} \boxed{6,866} \boxed{6,878} \boxed{6,886} \boxed{6,898}~;
  \item[] \boxed{6,914} \boxed{6,922} \boxed{6,926} \boxed{6,934} \boxed{6,958} \boxed{6,974} \boxed{6,982} \boxed{6,998}~;
  \item[]\ \vdots
  \item[]
  \begin{flushleft}
  \boxed{900,990} \boxed{902,946} \boxed{903,987} \boxed{905,955} \boxed{905,965} \boxed{905,985} \boxed{905,995} \boxed{906,942} \boxed{906,978} \boxed{908,916} \boxed{908,932} \boxed{908,956} \boxed{908,964} \boxed{909,927} \boxed{909,963} \boxed{909,981}~;
  \end{flushleft}
  \item[]
  \begin{flushleft}
  \boxed{910,980} \boxed{913,979} \boxed{914,922} \boxed{914,926} \boxed{914,934} \boxed{914,958} \boxed{914,974} \boxed{914,982} \boxed{914,998} \boxed{916,932} \boxed{916,956} \boxed{916,964} \boxed{917,959} \boxed{917,973}~;
  \end{flushleft}
  \item[]
  \begin{flushleft}
  \boxed{921,933} \boxed{921,939} \boxed{921,951} \boxed{921,993} \boxed{922,926} \boxed{922,934} \boxed{922,958} \boxed{922,974} \boxed{922,982} \boxed{922,998} \boxed{923,949} \boxed{926,934} \boxed{926,958} \boxed{926,974} \boxed{926,982} \boxed{926,998} \boxed{927,963} \boxed{927,981} \boxed{928,992}~;
  \end{flushleft}
  \item[]
  \begin{flushleft}
  \boxed{932,956} \boxed{932,964} \boxed{933,939} \boxed{933,951} \boxed{933,993} \boxed{934,958} \boxed{934,974} \boxed{934,982} \boxed{934,998} \boxed{938,994} \boxed{939,951} \boxed{939,993}~;
  \end{flushleft}
  \item[] \boxed{942,978} \boxed{943,989} \boxed{944,976} \boxed{948,996}~;
  \item[]
  \begin{flushleft}
  \boxed{951,993} \boxed{955,965} \boxed{955,985} \boxed{955,995} \boxed{956,964} \boxed{958,974} \boxed{958,982} \boxed{958,998} \boxed{959,973}~;
  \end{flushleft}
  \item[] \boxed{963,981} \boxed{965,985} \boxed{965,995}~;
  \item[] \boxed{974,982} \boxed{974,998}~;
  \item[] \boxed{982,998} \boxed{985,995}~.
\end{itemize}
\end{prog}

Taking into consideration the great number of solutions of equation (2080), we propose the same equation for
$x\in\set{a_x,a_x+m_x,a_x+2m_x,\ldots,b_x}$, $x$ non-prime and $y\in\set{a_y,a_y+m_y,a_y+2m_y,\ldots,b_y}$, $y$ non-prime. The
program for finding the solutions of the equation (2080) becomes:
 \begin{tabbing}
    $Ed20801(a_x,m_x,b_x,a_y,m_y,b_y):=$\=\vline\ $S\leftarrow("x"\ "y")$\\
    \>\vline\ $fo$\=$r\ x\in a_x,a_x+m_x..b_x$\\
    \>\vline\ \>\vline\ $u\leftarrow Tp\eta(x)=0$\\
    \>\vline\ \>\vline\ $fo$\=$r\ y\in a_y,a_y+m_y..b_y$\\
    \>\vline\ \>\vline\ \>\vline\ $q\leftarrow u\wedge Tp\eta(y)\textbf{=}0$\\
    \>\vline\ \>\vline\ \>\vline\ $i$\=$f\ q\wedge \eta_x\cdot y\textbf{=}x\cdot \eta_y$\\
    \>\vline\ \>\vline\ \>\vline\ \>\ $S\leftarrow stack[S,(x\ y)]$\\
    \>\vline\ $return\ S$
  \end{tabbing}
The program calls the subprogram \ref{ProgTpEta} for establishing the primality of $x$ and $y$. For $a_x:=2$, $m_x:=111$,
$b_x:=3\cdot 10^3$, $a_y:=3$, $m_y:=203$, $b_y:=2\cdot10^4$, then the search set is
\begin{multline*}
  \set{2,113,224,335,446,557,668,779,890,\ldots,2999} \\
  \times\set{3,206,409,612,815,1018,1221,1424,1627,\ldots,19897}
\end{multline*}
and the solutions, in the form $\boxed{x,y}$, are:
\begin{itemize}
  \item[] \boxed{335,815} \boxed{335,2845} \boxed{335,6905} \boxed{335,8935}~;
  \item[]
  \begin{flushleft}
  \boxed{446,206} \boxed{446,1018} \boxed{446,2642} \boxed{446,5078} \boxed{446,10762} \boxed{446,13198} \boxed{446,14822} \boxed{446,15634} \boxed{446,17258}~;
  \end{flushleft}
  \item[] \boxed{668,7108}~;
  \item[] \boxed{779,2033} \boxed{779,17461}~;
  \item[] \boxed{1112,6296} \boxed{1112,9544} \boxed{1112,19288}~;
  \item[] \boxed{1334,8326}~;
  \item[] \boxed{1556,7108}~;
  \item[] \boxed{2222,3454} \boxed{2222,12386}~;
  \item[] \boxed{2444,2236}~;
  \item[] \boxed{2888,13604}~.
\end{itemize}

\subsection{The equation 2081}

\begin{prog}  Given vector $\eta$, the equation $\eta(x)\cdot\eta(y)=x\cdot y$, with $x\neq y$ and $x,y$ non-prime, is equivalent to the relation $\eta_x\cdot\eta_y=x\cdot y$, cu $x<y$ (for symmetry reasons of the equation relative to $x$ and $y$, in order to elapse the
equivalent solutions, condition $x<y$ was considered). The program for finding the solutions of the equation $(2081)$, where $x<y$,
with $x$ and $y$ non-prime, is:
\begin{tabbing}
  $Ed2081(a_{xy},b_{xy}):=$\=\vline\ $S\leftarrow("x"\ "y")$\\
  \>\vline\ $fo$\=$r\ x\in a_{xy}..b_{xy}-1$\\
  \>\vline\ \>\vline\ $u\leftarrow Tp\eta(x)\textbf{=}0$\\
  \>\vline\ \>\vline\ $fo$\=$r\ y\in x+1..b_{xy}$\\
  \>\vline\ \>\vline\ \>\vline\ $q\leftarrow u\wedge Tp\eta(y)\textbf{=}0$\\
  \>\vline\ \>\vline\ \>\vline\ $S\leftarrow stack(S,(x\ y))\ if\ q\wedge \eta_x\cdot\eta_y\textbf{=}x\cdot y$\\
  \>\vline\ $return\ S$
\end{tabbing}
The program calls the subprogram \ref{ProgTpEta} to establish the primality of $x$ and $y$. The call of the program is done by the
sequence:
   \[
    \ a_{xy}:=2\ \ \ b_{xy}:=10^4~,
   \]
   hence, the search domain is
   \[
    D_c=\set{2,3\ldots,9999}\times\set{3,4,\ldots,10000}
   \]
   cu $x<y$. The total number of verified cases is:
   \[
   \sum_{x=2}^{999}\sum_{y=x+1}^{10000}=49985001~,
   \]
   and no solution has been found.
\end{prog}

\subsection{The equation 2082}

The $\eta$--Diophantine equation $\eta(x)^y=x^{\eta(y)}$, with $x$ and $y$ non-prime, is equivalent with the relation
$(\eta_x)^y=x^{\eta_y}$, with $x$ and $y$ non-prime, taking into consideration the meaning of vector $\eta$.

The search program is
\begin{tabbing}
  $Ed2082(a_x,b_x,a_y,b_y):=$\=\vline\ $S\leftarrow("x"\ "y")$\\
  \>\vline\ $fo$\=$r\ x\in a_x..b_x$\\
  \>\vline\ \>\vline\ $u\leftarrow Tp\eta(x)\textbf{=}0$\\
  \>\vline\ \>\vline\ $fo$\=$r\ y\in a_y..b_y$\\
  \>\vline\ \>\vline\ \>\vline\ $q\leftarrow u\wedge Tp\eta(y)\textbf{=}0$\\
  \>\vline\ \>\vline\ \>\vline\ $S\leftarrow stack(S,(x\ y))\ if\ q\wedge (\eta_x)^y\textbf{=}x^{\eta_y}$\\
  \>\vline\ $return\ S$
\end{tabbing}
The program calls the subprogram \ref{ProgTpEta} to establish the primality of $x$ and $y$.

The solutions of the problem for the search domain defined by $a_x:=2$, $b_x:=64$, $x$ non-prime and $a_y:=2$, $b_y:=80$, $y$
non-prime in the form $\boxed{x,y}$, are:
\begin{itemize}
  \item[] \boxed{4,4} \boxed{8,9} \boxed{27,9}~;
  \item[]
  \begin{flushleft}
  \boxed{36,6} \boxed{36,8} \boxed{36,10} \boxed{36,14} \boxed{36,22} \boxed{36,26} \boxed{36,34} \boxed{36,38} \boxed{36,46} \boxed{36,58} \boxed{36,62} \boxed{36,74}~;
  \end{flushleft}
  \item[]
  \begin{flushleft}
  \boxed{64,6} \boxed{64,8} \boxed{64,10} \boxed{64,14} \boxed{64,22} \boxed{64,26} \boxed{64,34} \boxed{64,38} \boxed{64,46} \boxed{64,58} \boxed{64,62} \boxed{64,74}~;
  \end{flushleft}
\end{itemize}
The number of ran through cases is $4977$, and the execution time is less then one second.

\subsection{The equation 2083}

The equation $\eta(x)^{\eta(y)}=\eta(x^y)$ is equivalent with the relation $(\eta_x)^{\eta_y}=\eta_{x^y}$ taking into consideration
the significance of vector $\eta$.

The search program will have to take into consideration the fact that $x^y$ should not be greater then the last index of file
$\eta$ and $x<y$, as the role of $x$ and $y$ is symmetric.
\begin{tabbing}
  $Ed2083(a_x,b_x,a_y,b_y):=$\=\vline\ $S\leftarrow("x"\ "y")$\\
  \>\vline\ $fo$\=$r\ x\in a_x..b_x-1$\\
  \>\vline\ \>\ $fo$\=$r\ y\in x+1..b_y$\\
  \>\vline\ \>\ \>\vline\ $u\leftarrow x^y\le last(\eta)$\\
  \>\vline\ \>\ \>\vline\ $S\leftarrow stack[S,(x\ y)]\ if\ u\wedge (\eta_x)^{\eta_y}\textbf{=}\eta_{x^y}$\\
  \>\vline\ $return\ S$
\end{tabbing}

for $a_x:=2$, $b_x:=100$, $a_y:=2$, $b_y:120$, with $x<y$, the number of ran trough cases is $6831$, with the remark that some
cases are excluded, namely when $x^y\ge10^6$. We have only 2 solutions: $\boxed{2,6}$ and $\boxed{2,12}$ that satisfy equation
$(2083)$
\[
 \eta(2)^{\eta(6)}=\eta(2^6)\ \ \textnormal{and}\ \ \eta(2)^{\eta(12)}=\eta(2^{12})
\]
or the equivalent relation $(\eta_x)^{\eta_y}=\eta_{x^y}$
\[
 \left(\eta_2\right)^{\eta_6}=\eta_{2^6}\ \ \textnormal{and}\ \ \left(\eta_2\right)^{\eta_{12}}=\eta_{2^{12}}~.
\]

\subsection{The equation 2084}

The $\eta$--Diophantine equation $\eta(x^y)=\eta(z^w)$, with $x\neq z$, is equivalent with the relation $\eta_{x^y}=\eta_{z^w}$
taking into consideration the significance of the file $\eta$. The program $Ed20840(a_x,b_x,a_y,b_y,a_z,b_z,a_w,b_w)$ that run trough
all 6561 situations of the search domain
\begin{equation}\label{Dc2084}
  D_c=\set{2,3,\ldots,10}\times\set{2,3,\ldots,10}\times\set{2,3,\ldots,10}\times\set{2,3,\ldots,10}~,
\end{equation}
where $x\neq z$. The real number of cases that have been ran trough is 4033 for reason of restrictions $x^y\le last(\eta)=10^6$
and $z^w\le last(\eta)=10^6$.

\begin{tabbing}
  $Ed20840$\=$(a_x,b_x,a_y,b_y,a_z,b_z,a_w,b_w):=$\\ \\
  \>\vline\ $S\leftarrow("x"\ "y"\ "z"\ "w")$\\
  \>\vline\ $u\leftarrow last(\eta)$\\
  \>\vline\ $fo$\=$r\ x\in a_x..b_x$\\
  \>\vline\ \>\ \vline\ $fo$\=$r\ y\in a_y..b_y$\\
  \>\vline\ \>\ \vline\ $X\leftarrow x^y$\\
  \>\vline\ \>\ \vline\ $fo$\=$r\ z\in a_z..b_z$\\
  \>\vline\ \>\ \vline\ \>\vline\ $fo$\=$r\ w\in a_w..b_w$\\
  \>\vline\ \>\ \vline\ \>\vline\ $Z\leftarrow z^w$\\
  \>\vline\ \>\ \vline\ \>\vline\ $q\leftarrow x\neq z\wedge X\le u\wedge Z\le u$\\
  \>\vline\ \>\ \vline\ \>\vline\ $S\leftarrow stack[S,(x\ y\ z\ w)]\ if\ q\wedge \eta_X\textbf{=}\eta_Z$\\
  \>\vline\ $return\ S$
\end{tabbing}

The solution of the Diophantine equation $\eta(x^y)=\eta(z^w)$, in the form $\boxed{x,y,z,w}$, are:
\begin{itemize}
  \item[]
  \begin{flushleft}
  \boxed{2,5,4,3} \boxed{2,6,4,3} \boxed{2,7,4,3} \boxed{2,9,3,5} \boxed{2,9,4,5} \boxed{2,9,6,5} \boxed{2,9,8,3} \boxed{2,10,3,5} \boxed{2,10,4,5} \boxed{2,10,6,5} \boxed{2,10,8,3}~;
  \end{flushleft}
  \item[]
  \begin{flushleft}
  \boxed{3,2,2,4} \boxed{3,3,6,4} \boxed{3,3,9,2} \boxed{3,4,9,2} \boxed{3,5,2,9} \boxed{3,5,2,10} \boxed{3,7,4,8} \boxed{3,7,9,4} \boxed{3,8,6,7} \boxed{3,8,9,4} \boxed{3,10,9,5}~;
  \end{flushleft}
  \item[]
  \begin{flushleft}
  \boxed{4,3,2,5} \boxed{4,3,2,6} \boxed{4,3,2,7} \boxed{4,3,8,2} \boxed{4,4,2,8} \boxed{4,4,5,2} \boxed{4,4,10,2} \boxed{4,5,2,9} \boxed{4,5,2,10} \boxed{4,5,8,3} \boxed{4,6,8,5} \boxed{4,7,8,5} \boxed{4,8,3,7} \boxed{4,8,6,7} \boxed{4,9,8,6}~;
  \end{flushleft}
  \item[] \boxed{5,2,2,8} \boxed{5,3,3,6} \boxed{5,4,4,9} \boxed{5,4,8,6} \boxed{5,5,10,6}~;
  \item[]
  \begin{flushleft}
  \boxed{6,2,2,4} \boxed{6,3,3,4} \boxed{6,3,9,2} \boxed{6,4,9,2} \boxed{6,5,2,9} \boxed{6,5,2,10} \boxed{6,5,8,3} \boxed{6,6,5,3} \boxed{6,6,9,3} \boxed{6,6,10,3} \boxed{6,7,3,8} \boxed{6,7,4,8} \boxed{6,7,9,4}~;
  \end{flushleft}
  \item[] \boxed{7,3,3,9} \boxed{7,5,5,8}~;
  \item[]
  \begin{flushleft}
  \boxed{8,2,2,5} \boxed{8,2,2,6} \boxed{8,2,2,7} \boxed{8,2,4,3} \boxed{8,3,2,9} \boxed{8,3,2,10} \boxed{8,3,3,5} \boxed{8,3,4,5} \boxed{8,3,6,5} \boxed{8,4,4,6} \boxed{8,4,4,7} \boxed{8,5,4,6} \boxed{8,5,4,7} \boxed{8,6,4,9} \boxed{8,6,5,4} \boxed{8,6,10,4}~;
  \end{flushleft}
  \item[]
  \begin{flushleft}
  \boxed{9,2,3,4} \boxed{9,2,6,3} \boxed{9,2,6,4} \boxed{9,3,3,6} \boxed{9,4,3,7} \boxed{9,4,3,8} \boxed{9,4,4,8} \boxed{9,4,6,7} \boxed{9,5,3,10}~;
  \end{flushleft}
  \item[] \boxed{10,2,2,8} \boxed{10,3,3,6} \boxed{10,4,4,9} \boxed{10,4,8,6} \boxed{10,5,5,6}~.
\end{itemize}
We have 87 solutions of the Diophantine equation $\eta(x^y)=\eta(z^w)$ in the search domain $D_c$, given by (\ref{Dc2084}), with $x\neq z$, $x^y\le10^6$ and $z^w\le10^6$ .

Similarly, we can consider the Diophantine equations $\eta(x^y)-\eta(z^w)=k$, where $k=1,2,3,4,5\ldots$. The number of
solutions for $k=1$ is 61, for $k=2$, 67, for $k=3$, 67, for $k=4$, 66 and for $k=5$ we have 51 solutions, etc.

For example, the solutions of the equation $x^y-z^w=23$ with $x\neq y\neq z\neq w$, $\boxed{x,y,z,w}\in\set{2,3,\ldots,10}^4$
are:
\begin{itemize}
  \item[] \boxed{5,8,2,9} \boxed{5,8,2,10} \boxed{7,5,2,9} \boxed{7,5,2,10} \boxed{7,5,8,3} \boxed{9,6,2,3}~.
\end{itemize}

\subsection{The equation 2085}

Instead of the Diophantine equation (2085), $\eta(x^y)=y$, proposed in \citep{Smarandache1999a}, we solve the more general
equation $\eta(x^y)-y=k$. The solutions of this equation with $\boxed{x,y,k}\in
D_c\in\set{2,3,\ldots,10^2}^2\times\set{0,1,\ldots,10}$ \big(where $\set{2,3,\ldots,10^3}^2=\set{2,3,\ldots,10^3}\times\set{2,3,\ldots,10^3}$\big)are:
\begin{itemize}
  \item[] \boxed{2,3,1} \boxed{2,7,1} \boxed{2,15,1}~;
  \item[] \boxed{2,2,2} \boxed{2,4,2} \boxed{2,6,2} \boxed{2,8,2} \boxed{2,10,2} \boxed{2,14,2} \boxed{2,16,2} \boxed{2,18,2}~;
  \item[] \boxed{2,5,3} \boxed{2,9,3} \boxed{2,11,3} \boxed{2,13,3} \boxed{2,17,3} \boxed{2,19,3}~;
  \item[] \boxed{2,12,4} \boxed{3,2,4} \boxed{4,2,4} \boxed{6,2,4} \boxed{12,2,4}~;
  \item[] \boxed{3,4,5} \boxed{4,3,5} \boxed{6,4,5}~;
  \item[] \boxed{3,3,6} \boxed{4,4,6} \boxed{6,3,6} \boxed{8,2,6} \boxed{12,3,6} \boxed{12,4,6} \boxed{24,2,6}~;
  \item[] \boxed{3,5,7} \boxed{4,5,7} \boxed{6,5,7} \boxed{9,2,7} \boxed{12,5,7} \boxed{18,2,7} \boxed{36,2,7} \boxed{72,2,7}~;
  \item[]
  \begin{flushleft}
  \boxed{5,2,8} \boxed{10,2,8} \boxed{15,2,8} \boxed{16,2,8} \boxed{20,2,8} \boxed{30,2,8} \boxed{40,2,8} \boxed{45,2,8} \boxed{48,2,8} \boxed{60,2,8} \boxed{80,2,8} \boxed{90,2,8}~;
  \end{flushleft}
  \item[] \boxed{3,6,9} \boxed{4,7,9} \boxed{6,6,9} \boxed{8,3,9} \boxed{24,3,9}~;
  \item[] \boxed{3,8,10} \boxed{4,6,10} \boxed{4,8,10} \boxed{32,2,10} \boxed{96,2,10}~.
\end{itemize}
Let us remark that in $D_c$ there exists no solution of equation $\eta(x^y)=y$.

The number of analyzed cases is $107811$ of which only $3047$ have satisfied the condition $x^y<last(\eta)=10^6$.

Knowing the significance of vector $\eta$, the Diophantine equation $\eta(x^y)-y=k$ is equivalent with the relation
$\eta_{x^y}-y=k$. In these conditions, the program for finding the solution is:
\begin{tabbing}
  $Ed2085$\=$(a_k,b_k,a_x,b_x,a_y,b_y):=$\\ \\
  \>\vline\ $j\leftarrow0$\\
  \>\vline\ $S\leftarrow("x"\ "y"\ "k")$\\
  \>\vline\ $fo$\=$r\ k\in a_k..b_k$\\
  \>\vline\ \>\ $fo$\=$r\ x\in a_x..b_x$\\
  \>\vline\ \>\ \>\ $fo$\=$r\ y\in a_y..b_y$\\
  \>\vline\ \>\ \>\ \>\vline\ $\eta\leftarrow x^y$\\
  \>\vline\ \>\ \>\ \>\vline\ $if$\=$\ \eta\le last(\eta)$\\
  \>\vline\ \>\ \>\ \>\vline\ \>\vline\ $j\leftarrow j+1$\\
  \>\vline\ \>\ \>\ \>\vline\ \>\vline\ $S\leftarrow stack[S,(x\ y\ k)]\ if\ \eta_\eta-y\textbf{=}k$\\
  \>\vline\ $return\ stack[S, (j\ j\ j)]$
\end{tabbing}
By means of variable $j$ we count the number of concrete analyzed cases.

\subsection{The equation 2086}

The Diophantine equation $\eta(x^x)=y^y$, which is equivalent with the relation $\eta_{x^x}=y^y$, taking into consideration the
significance of vector $\eta$, has only the trivial solutions $\boxed{x,y}=\boxed{1,1}$ and $\boxed{x,y}=\boxed{2,2}$ for
$\boxed{x,y}\in D_c=\set{1,2,\ldots,10^2}^2$ of which there were in fact analyzed only 700 cases in which $x^x\le last(\eta)=10^6$.

\subsection{The equation 2087}

The Diophantine equation $\eta(x^y)=y^x$, equivalent with the relation $\eta_{x^y}=y^x$, taking into consideration the
significance of vector $\eta$, has only the trivial solutions $\boxed{1,1}$ and $\boxed{2,2}$ for $\boxed{x,y}\in
D_c=\set{1,2,\ldots,10^2}^2$ of which there were in fact analyzed only 476 cases in which $x^y\le last(\eta)=10^6$.

\subsection{The equation 2088}

The 26 solutions of the $\eta$--Diophantine equation $\eta(x)=y!$, in the form $\boxed{x,y}$, are:
\begin{itemize}
  \item[] \boxed{2,2}~;
  \item[]
  \begin{flushleft}
  \boxed{9,3} \boxed{16,3} \boxed{18,3} \boxed{36,3} \boxed{45,3} \boxed{48,3} \boxed{72,3} \boxed{80,3} \boxed{90,3} \boxed{144,3} \boxed{180,3} \boxed{240,3} \boxed{360,3} \boxed{720,3}~;
  \end{flushleft}
  \item[]
  \begin{flushleft}
  \boxed{59049,4} \boxed{118098,4} \boxed{236196,4} \boxed{295245,4} \boxed{413343,4} \boxed{472392,4} \boxed{590490,4} \boxed{649539,4} \boxed{767637,4} \boxed{826686,4}~.
  \end{flushleft}
\end{itemize}

The program for finding the solutions of the equation (2088) is
\begin{tabbing}
  $Ed2088(a_x,b_x,a_y,b_y):=$\=\vline\ $S\leftarrow("x"\ "y")$\\
  \>\vline\ $fo$\=$r\ x\in a_x..b_x$\\
  \>\vline\ \>\ $fo$\=$r\ y\in a_y..b_y$\\
  \>\vline\ \>\ \>\ $S\leftarrow stack[S,(x\ y)]\ if\ \eta_x\textbf{=}y!$\\
  \>\vline\ $return\ S$
\end{tabbing}

The search domain is $D_c=\set{1,2,\ldots,10^6}\times\set{1,2,\ldots,19}$, where we have considered that $19!$ has 17 significant digits and $20!$ has 18 significant digits. Therefore, for $y>19$ some errors will be produced, connected to the representation of the numbers in the intern memory of classic computers.

\subsection{The equation 2089}

The $\eta$--Diophantine equation $\eta(m\cdot x)=m\cdot\eta(x)$, if we take into consideration the significance of vector $\eta$,
is equivalent with the relation
\[
 \eta_{m\cdot x}=m\cdot \eta_x~.
\]
In these conditions, the empirical search program for the solutions of the $\eta$--Diophantine equation is:
\begin{tabbing}
  $Ed2089(a_m,b_m,a_x,b_x):=$\=\vline\ $j\leftarrow0$\\
  \>\vline\ $S\leftarrow("m"\ "x")$\\
  \>\vline\ $fo$\=$r\ m\in a_m..b_m$\\
  \>\vline\ \>\ $fo$\=$r\ x\in a_x..b_x$\\
  \>\vline\ \>\ \>\ $if$\=$\ m\cdot x\le last(\eta)$\\
  \>\vline\ \>\ \>\ \>\vline\ $j\leftarrow j+1$\\
  \>\vline\ \>\ \>\ \>\vline\ $S\leftarrow stack(S,(m\ x))\ if\ \eta_{m\cdot x}\textbf{=}m\cdot \eta_x$\\
  \>\vline\ $return\ stack(S,(j\ j))$
\end{tabbing}
The search domain is
\[
 D_c=\set{2,3,\ldots,10^2}\times\set{2,3,\ldots,10^6}
\]
which imply that the number of cases is $98999901$, but, due to condition $m\cdot x\le last(\eta)=10^6$, the number of real
analyzed cases is $4187241$, counted in the program by means of variable $j$.

On this search domain the equation has only a sole solution
\[
 \boxed{m,x}=\boxed{2,2}~.
\]
Obviously, there exist also the trivial solutions where $m=1$, $x\in\set{1,2,\ldots,10^6}$ and $m\in\set{1,2,\ldots,10^6}$ and
$x=1$, but those were avoided by choosing the search domain.

\subsection{The equation 2090}

The $\eta$--Diophantine equation $m^{\eta(x)}+\eta(x)^n=m^n$ is equivalent with the relation $m^{\eta_x}+(\eta_x)^n=m^n$ if we
consider vector $\eta$ which contains all the values of function $\eta(k)$ for $k\in\set{1,2,\ldots,10^6}$. Let us consider the
next empirical search domain
\[
 D_c=\set{2,3,\ldots,100}^2\times\set{2,3,\ldots10^4}
\]
for the triplet $(m,n,x)$. As we have operations to raise at power that can generate numbers greater than $10^{17}$, a condition has
been imposed to avoid the floating overflow errors and the numbers greater than $10^{17}$. The search program is:
\begin{tabbing}
  $Ed2090$\=$(a_m,b_m,a_n,b_n,a_x,b_x):=$\\ \\
  \>\vline\ $j\leftarrow0$\\
  \>\vline\ $fo$\=$r\ m\in a_m..b_m$\\
  \>\vline\ \>\ $fo$\=$r\ n\in a_n..b_n$\\
  \>\vline\ \>\ \>\ $fo$\=$r\ x\in a_x..b_x$\\
  \>\vline\ \>\ \>\ \>\vline\ $\eta\leftarrow \eta_x$\\
  \>\vline\ \>\ \>\ \>\vline\ $N\leftarrow\infty\ on\ error\ N\leftarrow m^\eta+\eta^n$\\
  \>\vline\ \>\ \>\ \>\vline\ $if$\=$\ N\le10^{17}$\\
  \>\vline\ \>\ \>\ \>\vline\ \>\vline\ $j\leftarrow j+1$\\
  \>\vline\ \>\ \>\ \>\vline\ \>\vline\ $S\leftarrow stack[S,(m\ n\ x)]\ if\ N\textbf{=}m^n$\\
  \>\vline\ $return\ stack[S,(j\ j\ j)]$
\end{tabbing}

The search domain has $3799620$ possible cases but in reality there were considered $236363$ due to restriction $m^\eta+\eta^n\le10^{17}$. In these conditions, the $\eta$--Diophantine equation (2090) has no solutions on this search domain.

\subsection{The equation 2091}

The Diophantine equation (2091) is $\eta(x^2)/m\pm\eta(y^2)/n=1$. The Diophantine equation (2091), variant with $+$,
$\eta(x^2)/m+\eta(y^2)/n=1$ has no solutions on this search domain $D_c=\set{1,2,\ldots,10}^2\times\set{2,3,\ldots,10^3}^2$, with
$m\neq n$ and $x\neq y$. The number of total possible cases is 80838081 of which, due to restriction $m\neq n$ and $x\neq y$, the
number of real analyzed cases is 71784144. The search time was 37.397 seconds.

On the contrary, the Diophantine equation $\eta(x^2)/m-\eta(y^2)/n=1$, which is equivalent with equation
$n\cdot\eta(x^2)-m\cdot\eta(y^2)=m\cdot n$, are 54370 de solutions pe search domain $D_c=\set{1,2,\ldots,10}^2\times\set{2,3,\ldots,10^3}^2$. The number of possible cases is 80838081 of which, due to restriction
$m\neq n$ and $x\neq y$, the number of real analyzed cases is 71784144. The search time was 169.662 seconds. The first 49 and
the last 28 solutions, in the form $\boxed{m,n,x,y}$, are:
\begin{itemize}
  \item[] \boxed{2,3,3,4} \boxed{2,3,3,6} \boxed{2,3,3,12}~;
  \item[] \boxed{2,3,4,3} \boxed{2,3,4,6} \boxed{2,3,4,12}~;
  \item[] \boxed{2,3,5,32} \boxed{2,3,5,96} \boxed{2,3,5,160} \boxed{2,3,5,288} \boxed{2,3,5,480}~;
  \item[] \boxed{2,3,6,3} \boxed{2,3,6,4} \boxed{2,3,6,12}~;
  \item[]
  \begin{flushleft}
  \boxed{2,3,7,81} \boxed{2,3,7,162} \boxed{2,3,7,256} \boxed{2,3,7,324} \boxed{2,3,7,405} \boxed{2,3,7,567} \boxed{2,3,7,648} \boxed{2,3,7,768} \boxed{2,3,7,810}~;
  \end{flushleft}
  \item[] \boxed{2,3,8,9} \boxed{2,3,8,18} \boxed{2,3,8,36} \boxed{2,3,8,72}~;
  \item[] \boxed{2,3,10,32} \boxed{2,3,10,96} \boxed{2,3,10,160} \boxed{2,3,10,288} \boxed{2,3,10,480}~;
  \item[] \boxed{2,3,12,3} \boxed{2,3,12,4} \boxed{2,3,12,6}~;
  \item[]
  \begin{flushleft}
  \boxed{2,3,14,81} \boxed{2,3,14,162} \boxed{2,3,14,256} \boxed{2,3,14,324} \boxed{2,3,14,405} \boxed{2,3,14,567} \boxed{2,3,14,648} \boxed{2,3,14,768} \boxed{2,3,14,810}~;
  \end{flushleft}
  \item[] \boxed{2,3,15,32} \boxed{2,3,15,96} \boxed{2,3,15,160} \boxed{2,3,15,288} \boxed{2,3,15,480}~;
  \item[]\ \vdots
  \item[] \boxed{10,9,512,9} \boxed{10,9,512,18} \boxed{10,9,512,36} \boxed{10,9,512,72}~;
  \item[] \boxed{10,9,525,9} \boxed{10,9,525,18} \boxed{10,9,525,36} \boxed{10,9,525,72}~;
  \item[] \boxed{10,9,600,9} \boxed{10,9,600,18} \boxed{10,9,600,36} \boxed{10,9,600,72}~;
  \item[] \boxed{10,9,675,9} \boxed{10,9,675,18} \boxed{10,9,675,36} \boxed{10,9,675,72}~;
  \item[] \boxed{10,9,700,9} \boxed{10,9,700,18} \boxed{10,9,700,36} \boxed{10,9,700,72}~;
  \item[] \boxed{10,9,800,9} \boxed{10,9,800,18} \boxed{10,9,800,36} \boxed{10,9,800,72}~;
  \item[] \boxed{10,9,900,9} \boxed{10,9,900,18} \boxed{10,9,900,36} \boxed{10,9,900,72}~.
\end{itemize}

\subsection{The equation 2092}

The $\eta$--Diophantine equation
\[
 \eta(x_1^{y_1}+x_2^{y_2}+\ldots+x_r^{y_r})=\eta(x_1)^{y_1}+\eta(x_2)^{y_2}+\ldots+\eta(x_r)^{y_r}~,
\]
is equivalent with the relation
\[
\eta_{x_1^{y_1}}+\eta_{x_2^{y_2}}+\ldots+\eta_{x_r^{y_r}}= (\eta_{x_1})^{y_1}+(\eta_{x_2})^{y_2}+\ldots+(\eta_{x_r})^{y_r}~,
\]
if we use vector $\eta$ that contains all the values of function $\eta(k)$ for $k\in\set{1,2,\ldots,10^6}$. Unfortunately, equation
(2092) surpasses slightly our possibilities of finding the solutions. It is sufficient to chose $r>3$ and the powers $y_1$,
$y_2$, \ldots, $y_r$ to be natural numbers $>2$. In the case of this equation two cases were considered, that have no solutions
(except the trivial solution $x_1=1$, $x_2=1$, \ldots, $x_r=1$) on the given search domain:
\begin{enumerate}
  \item Equation (2092), with $r=2$, $y_1=2$ and $y_2=3$
      \[
       \eta(x_1^2+x_2^3)=\eta(x_1)^2+\eta(x_2)^3
      \]
  for $\set{x_1,x_2}\in D_c=\set{2,3,\ldots,1000}\times\set{2,3,\ldots,100}$.
  \item Equation (2092), with $r=3$, $y_1=1$, $y_2=2$ and $y_3=4$
      \[
       \eta(x_1+x_2^2+x_3^4)=\eta(x_1)+\eta(x_2)^2+\eta(x_3)^4~,
      \]
  for
  \[
   \set{x_1,x_2,x_3}\in D_c=\set{2,3,\ldots,10^4}\times\set{2,3,\ldots,10^3}\times\set{2,3,\ldots,31}.
  \]
\end{enumerate}

\subsection{The equation 2093}

The $\eta$--Diophantine equation
\[
 \eta(x_1!+x_2!+\ldots+x_r!)=\eta(x_1)!+\eta(x_2)!+\ldots+\eta(x_r)!~,
\]
is equivalent with the relation
\[
\eta_{x_1!}+\eta_{x_2!}+\ldots+\eta_{x_r!}= \eta_{x_1}!+\eta_{x_2}!+\ldots+\eta_{x_r}!~,
\]
if we use vector $\eta$ that contains all the values of function $\eta(k)$ for $k\in\set{1,2,\ldots,10^6}$. This equation also
surpasses slightly our possibilities of finding the solutions for $r>2$. Let equation $\eta$--Diophantine for $r=2$
\begin{equation}\label{Ec20932}
  \eta(x_1!+x_2!)=\eta(x_1)!+\eta(x_2)!~,
\end{equation}
and equivalent relation
\[
\eta_{x_1!}+\eta_{x_2!}= \eta_{x_1}!+\eta_{x_2}!~.
\]
As for $n=9$, $n!=362880$ and for $n=10$, $n!=3628800>10^6$, Then it follows that the biggest search domain for $x_1$ and $x_2$ is
$D_c\set{2,3,\ldots,9}^2$. Equation (\ref{Ec20932}) is symmetric relative to both unknowns. To avoid symmetric solutions we impose
also the condition $x_1<x_2$. We have avoided the values $x_1=1$ and $x_2=1$ because they lead to the trivial solution. The program
for finding the solutions of the equation (\ref{Ec20932}) is:
\begin{prog} Program $Ed20932$
  \begin{tabbing}
    $Ed20932(a_{xy},b_{xy}):=$\=\vline\ $S\leftarrow("x"\ "y")$\\
    \>\vline\ $fo$\=$r\ x\in a_{xy}..b_{xy}-1$\\
    \>\vline\ \>\ $fo$\=$r\ y\in x+1..b_{xy}$\\
    \>\vline\ \>\ \>\vline\ $c\leftarrow \eta_{x!}+\eta_{y}\textbf{=}\eta_{x}!+\eta_{y}!$\\
    \>\vline\ \>\ \>\vline\ $S\leftarrow stack[S,(x\ y)]\ if\ c$\\
    \>\vline\ $return\ S$
  \end{tabbing}
With this program the solution $\boxed{2,6}$ has be determined on the search domain $D_c\set{2,3,\ldots,9}^2$, solution for which we
have $\eta(2!)+\eta(6!)=\eta(2)!+\eta(6)!$. Obviously, there exist also the trivial solutions $\boxed{1,1}$ and $\boxed{2,2}$ which
verify:
\[
 \eta(1!)+\eta(1!)=\eta(1)!+\eta(1)!\ \ \textnormal{and}\ \ \eta(2!)+\eta(2!)=\eta(2)!+\eta(2)!~.
\]
\end{prog}

\subsection{The equation 2094}

The $\eta$--Diophantine equation
$\big(x,y\big)=\big(\eta(x),\eta(y)\big)$, \big(where by (x,y) was denoted the greatest common divisor of $x$ and $y$\big), is
equivalent with the relation $\gcd(x,y)=\gcd(\eta_x,\eta_y)$, if we use vector $\eta$ that contains all values of function
$\eta(k)$ for $k\in\set{1,2,\ldots,10^6}$ and Mathcad function $\gcd(n_1,n_2,\ldots,n_\ell)$ for computing the greatest common
divisor of $n_1,n_2,\ldots,n_\ell$ has 4799 solutions. The search domain ales is $D_c=\set{2,3,\ldots,10^3}^2$ (with $x<y$ as the
role of $x$ and $y$ is symmetric) and $(x,y)\neq1$, i.e. $x$ and $y$ not relative prime. The number of possible cases is $498501$
of which the analyzed ones are $193351$ due to restrictions $x<y$ and $(x,y)\neq1$. We give all the solution for $x=4,6,8$, in the
form $\boxed{x,y}$:
\begin{itemize}
  \item[]
  \begin{flushleft}
  \boxed{4,8} \boxed{4,12} \boxed{4,18} \boxed{4,24} \boxed{4,32} \boxed{4,50} \boxed{4,64} \boxed{4,90} \boxed{4,96} \boxed{4,98} \boxed{4,128} \boxed{4,150} \boxed{4,160} \boxed{4,192} \boxed{4,224} \boxed{4,242} \boxed{4,288} \boxed{4,294} \boxed{4,320} \boxed{4,338} \boxed{4,350} \boxed{4,384} \boxed{4,448} \boxed{4,450} \boxed{4,480} \boxed{4,490} \boxed{4,512} \boxed{4,576} \boxed{4,578} \boxed{4,640} \boxed{4,672} \boxed{4,722} \boxed{4,726} \boxed{4,882} \boxed{4,896} \boxed{4,960} \boxed{4,972}~;
  \end{flushleft}
  \item[]
  \begin{flushleft}
  \boxed{6,9} \boxed{6,27} \boxed{6,45} \boxed{6,81} \boxed{6,135} \boxed{6,189} \boxed{6,243} \boxed{6,375} \boxed{6,405} \boxed{6,567} \boxed{6,729} \boxed{6,945}~;
  \end{flushleft}
  \item[]
  \begin{flushleft}
  \boxed{8,12} \boxed{8,18} \boxed{8,50} \boxed{8,90} \boxed{8,98} \boxed{8,150} \boxed{8,242} \boxed{8,294} \boxed{8,338} \boxed{8,350} \boxed{8,450} \boxed{8,490} \boxed{8,578} \boxed{8,722} \boxed{8,726} \boxed{8,882} \boxed{8,972}~.
  \end{flushleft}
\end{itemize}
For $x\ge860$ we give the solutions on $x=86\ast$, $87\ast$,
$88\ast$, $89\ast$, $90\ast$, $91\ast$, $92\ast$, $93\ast$,
$94\ast$, $95\ast$, $96\ast$, in the same form $\boxed{x,y}$:
\begin{itemize}
  \item[]
  \begin{flushleft}
  \boxed{860,903} \boxed{860,989} \boxed{861,902} \boxed{861,943} \boxed{867,884} \boxed{867,935} \boxed{867,952} \boxed{868,899} \boxed{868,961} \boxed{869,948}~;
  \end{flushleft}
  \item[] \boxed{871,938} \boxed{872,981} \boxed{873,970} \boxed{874,897} \boxed{876,949}~;
  \item[] \boxed{880,891} \boxed{882,968} \boxed{884,935} \boxed{885,944} \boxed{888,925}~;
  \item[]
  \begin{flushleft}
  \boxed{890,979} \boxed{891,968} \boxed{893,940} \boxed{893,987} \boxed{896,972} \boxed{897,920} \boxed{899,930} \boxed{899,961} \boxed{899,992}~;
  \end{flushleft}
  \item[] \boxed{901,954} \boxed{902,943} \boxed{903,946} \boxed{903,989}~;
  \item[] \boxed{912,931} \boxed{913,996} \boxed{915,976} \boxed{918,935}~;
  \item[] \boxed{923,994} \boxed{925,962} \boxed{925,999} \boxed{928,957}~;
  \item[] \boxed{930,961} \boxed{931,950} \boxed{931,969} \boxed{931,988} \boxed{935,952}~;
  \item[] \boxed{940,987} \boxed{943,984} \boxed{946,989}~;
  \item[] \boxed{950,969} \boxed{957,986}~;
  \item[] \boxed{961,992} \boxed{962,999} \boxed{969,988}~.
\end{itemize}

The empirical search domain is:
\begin{tabbing}
  $Ed2094(a_{xy},b_{xy}):=$\=\vline\ $j\leftarrow0$\\
  \>\vline\ $S\leftarrow("x"\ "y")$\\
  \>\vline\ $fo$\=$r\ x\in a_{xy}..b_{xy}-1$\\
  \>\vline\ \>\vline\ $u\leftarrow Tp\eta(x)=0$\\
  \>\vline\ \>\vline\ $fo$\=$r\ y\in x+1..b_{xy}$\\
  \>\vline\ \>\vline\ \>\vline\ $notprime\leftarrow u\wedge Tp\eta(y)\textbf{=}0$\\
  \>\vline\ \>\vline\ \>\vline\ $q\leftarrow \gcd(x,y)$\\
  \>\vline\ \>\vline\ \>\vline\ $if$\=$\ notprime\wedge q\neq1$\\
  \>\vline\ \>\vline\ \>\vline\ \>\vline\ $j\leftarrow j+1$\\
  \>\vline\ \>\vline\ \>\vline\ \>\vline\ $S\leftarrow stack(S,(x\ y))\ if\ q\textbf{=}\gcd(\eta_x,\eta_y)$\\
  \>\vline\ $return\ stack(S,(j\ j))$
  \end{tabbing}
The program calls the subprogram \ref{ProgTpEta} to establish the primality of $x$ and $y$.

\subsection{The equation 2095}

We consider the search domain $D_c=\set{2,3,\ldots,10^3}^2$, $x,y$ non-prime with $x<y$ and $(x,y)\neq1$. The number of possible cases is $498501$. Due to conditions $x<y$ and $(x,y)=1$ the number of analyzed cases is in fact $193351$. We have $145$ solutions with $x=4$ and solution $\boxed{x,y}=\boxed{6,12}$ that verify the $\eta$--Diophantine equation $[x,y]=[\eta(x),\eta(y)]$, where by $[n_1,n_2]$ was denoted the smallest common multiple of numbers $n_1$ and $n_2$. This $\eta$--Diophantine equation is equivalent with the relation $lcm(x,y)=lcm\big(\eta_x,\eta_y\big)$, if we use vector $\eta$ with the values of function $\eta$ and Mathcad function $lcm(n_1,n_2,\ldots,n_\ell)$ to compute the the smallest common multiple of $n_1$, $n_2$, \ldots $n_\ell$. The solutions, in the form $\boxed{x,y}$, are:
\begin{itemize}
  \item[]
  \begin{flushleft}
  \boxed{4,6} \boxed{4,10} \boxed{4,14} \boxed{4,20} \boxed{4,22} \boxed{4,26} \boxed{4,28} \boxed{4,34} \boxed{4,38} \boxed{4,44} \boxed{4,46} \boxed{4,52} \boxed{4,58} \boxed{4,62} \boxed{4,68} \boxed{4,74} \boxed{4,76} \boxed{4,82} \boxed{4,86} \boxed{4,92} \boxed{4,94}~;
  \end{flushleft}
  \item[]
  \begin{flushleft}
  \boxed{4,116} \boxed{4,118} \boxed{4,122} \boxed{4,124} \boxed{4,134} \boxed{4,142} \boxed{4,146} \boxed{4,148} \boxed{4,158} \boxed{4,164} \boxed{4,166} \boxed{4,172} \boxed{4,178} \boxed{4,188} \boxed{4,194}~;
  \end{flushleft}
  \item[]
  \begin{flushleft}
  \boxed{4,202} \boxed{4,206} \boxed{4,212} \boxed{4,214} \boxed{4,218} \boxed{4,226} \boxed{4,236} \boxed{4,244} \boxed{4,254} \boxed{4,262} \boxed{4,268} \boxed{4,274} \boxed{4,278} \boxed{4,284} \boxed{4,292} \boxed{4,298}~;
  \end{flushleft}
  \item[]
  \begin{flushleft}
  \boxed{4,302} \boxed{4,314} \boxed{4,316} \boxed{4,326} \boxed{4,332} \boxed{4,334} \boxed{4,346} \boxed{4,356} \boxed{4,358} \boxed{4,362} \boxed{4,382} \boxed{4,386} \boxed{4,388} \boxed{4,394} \boxed{4,398}~;
  \end{flushleft}
  \item[]
  \begin{flushleft}
  \boxed{4,404} \boxed{4,412} \boxed{4,422} \boxed{4,428} \boxed{4,436} \boxed{4,446} \boxed{4,452} \boxed{4,454} \boxed{4,458} \boxed{4,466} \boxed{4,478} \boxed{4,482}~;
  \end{flushleft}
  \item[]
  \begin{flushleft}
  \boxed{4,502} \boxed{4,508} \boxed{4,514} \boxed{4,524} \boxed{4,526} \boxed{4,538} \boxed{4,542} \boxed{4,548} \boxed{4,554} \boxed{4,556} \boxed{4,562} \boxed{4,566} \boxed{4,586} \boxed{4,596}~;
  \end{flushleft}
  \item[]
  \begin{flushleft}
  \boxed{4,604} \boxed{4,614} \boxed{4,622} \boxed{4,626} \boxed{4,628} \boxed{4,634} \boxed{4,652} \boxed{4,662} \boxed{4,668} \boxed{4,674} \boxed{4,692} \boxed{4,694} \boxed{4,698}~;
  \end{flushleft}
  \item[]
  \begin{flushleft}
  \boxed{4,706} \boxed{4,716} \boxed{4,718} \boxed{4,724} \boxed{4,734} \boxed{4,746} \boxed{4,758} \boxed{4,764} \boxed{4,766} \boxed{4,772} \boxed{4,778} \boxed{4,788} \boxed{4,794} \boxed{4,796}~;
  \end{flushleft}
  \item[]
  \begin{flushleft}
  \boxed{4,802} \boxed{4,818} \boxed{4,838} \boxed{4,842} \boxed{4,844} \boxed{4,862} \boxed{4,866} \boxed{4,878} \boxed{4,886} \boxed{4,892} \boxed{4,898}~;
  \end{flushleft}
  \item[]
  \begin{flushleft}
  \boxed{4,908} \boxed{4,914} \boxed{4,916} \boxed{4,922} \boxed{4,926} \boxed{4,932} \boxed{4,934} \boxed{4,956} \boxed{4,958} \boxed{4,964} \boxed{4,974} \boxed{4,982} \boxed{4,998}~;
  \end{flushleft}
  \item[] \boxed{6,12}~.
\end{itemize}

The empirical search domain is:
\begin{tabbing}
  $Ed2095$\=$(a_{xy},b_{xy}):=$\\ \\
  \>\vline\ $j\leftarrow0$\\
  \>\vline\ $S\leftarrow("x"\ "y")$\\
  \>\vline\ $fo$\=$r\ x\in a_{xy}..b_{xy}-1$\\
  \>\vline\ \>\vline\ $u\leftarrow Tp\eta(x)=0$\\
  \>\vline\ \>\vline\ $fo$\=$r\ y\in x+1..b_{xy}$\\
  \>\vline\ \>\vline\ \>\vline\ $notprime\leftarrow u\wedge Tp\eta(y)\textbf{=}0$\\
  \>\vline\ \>\vline\ \>\vline\ $if$\=$\ notprime\wedge \gcd(x,y)\neq 1$\\
  \>\vline\ \>\vline\ \>\vline\ \>\vline\ $j\leftarrow j+1$\\
  \>\vline\ \>\vline\ \>\vline\ \>\vline\ $S\leftarrow stack(S,(x\ y))\ if\ lcm(x,y)\textbf{=}lcm(\eta_x,\eta_y)$\\
  \>\vline\ $return\ stack(S,(j\ j))$
\end{tabbing} The program calls the subprogram \ref{ProgTpEta} to establish the primality $x$ and $y$.

\section[The $\eta$--$s$--Diophantine equations]{The $\eta$--$s$--Diophantine equations}

The function $s:\Na\to\Na$, where $s(n)$ is the sum of the divisors of $n$ without $n$ (the sum of the aliquot parts), for
example $s(12)=1+2+3+4+6=16$.
\begin{figure}
  \centering
  \includegraphics[scale=0.8]{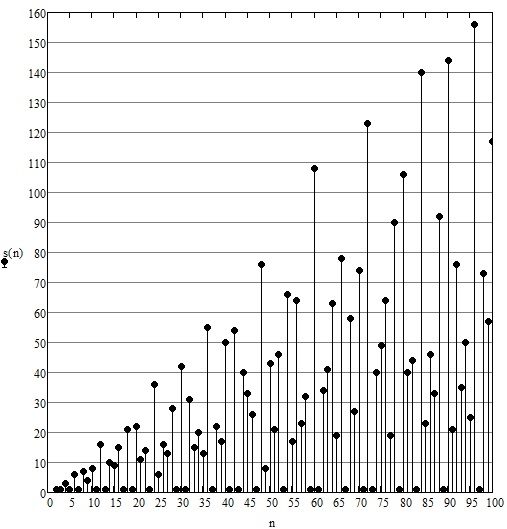}\\
  \caption{The function $s$}\label{FigFunctias}
\end{figure}
The function $s(n)$ can be defined by means of function $\sigma(n)=\sigma_1(n)$ which is the sum of the divisors of $n$,
$s(n)=\sigma(n)-n$ (see figure \ref{FigFunctias}). Keeping the numbering from the paper \citep{Smarandache1999a}, the
$\eta$--$s$--Diophantine equations are:
\begin{enumerate}
  \item[(2124)] $\eta(x)=s(m\cdot x+n)$~,
  \item[(2125)] $\eta(x)^m=s(x^n)$~,
  \item[(2126)] $\eta(x)+y=x+s(y)$~,
  \item[(2127)] $\eta(x)\cdot y=x\cdot s(y)$~,
  \item[(2128)] $\eta(x)/y=x/s(y)$~,
  \item[(2129)] $\eta(x)^y=x^{s(y)}$~,
  \item[(2130)] $\eta(x)^y=s(y)^x$~,
\end{enumerate}

\subsection[Empirical solving of $\eta$--$s$--Diophantine equations]{Partial empirical solving of $\eta$--$s$--Diophantine equations}

To solve empirically the $\eta$--$s$--Diophantine equations (2124--2130), we will read the files $\eta.prn$ and $s.prn$ which
contain the values of functions $\eta(n)$ and $s(n)$ for $n=1,2,\ldots,10^6$, these files being generated previously. The
files $\eta.prn$ and $s.prn$ will be assigned to vectors $\eta$ and $s$ with by means of the Mathcad sequences:
\[
 \eta:=READPRN("...\backslash \eta.prn")\ \ last(\eta)=10^6
\]
\[
 s:=READPRN("...\backslash s.prn")\ \ last(s)=10^6~,
\]
where the commands $last(\eta)$ and $last(s)$ indicate the last index of vectors $\eta$ and $s$. These manoeuvres are necessary to
diminish the empirical search time of the solutions of the Diophantine equations.

\subsection{The equation 2124}

Equation $\eta$--$s$--Diophantine $\eta(x)=s(m\cdot x+n)$ is equivalent with the relation $\eta_x=s_{m\cdot x+n}$ if we take
into consideration the significance of vectors $\eta$ and $s$. Let be search domain
\begin{equation}\label{Dc2124}
  D_c=\set{1,2,\ldots,10}^2\times\set{a_x,a_x+r_x,a_x+2r_x,\ldots,\left\lfloor\frac{b_x-a_x}{r_x}\right\rfloor\cdot r_x}
\end{equation}
with $a_x=1$, $r_x=1$, $b_x=10^6$ and additional conditions $m\cdot x+n\le last(s)$ and $\gcd(m,n)=1$, i.e. $m$ and $n$ are
relatively prime numbers. Then, the number of possible cases is $10^8$. Due to the additional conditions, the number of analyzed
cases is $21271688$. The number of found solutions is $554$, of which we present the first $80$ and the last $85$:
\begin{itemize}
  \item[]
      \begin{flushleft}
        \boxed{1,1,1} \boxed{1,1,3} \boxed{1,1,7} \boxed{1,1,8} \boxed{1,1,26} \boxed{1,1,31} \boxed{1,1,76} \boxed{1,1,124} \boxed{1,1,127} \boxed{1,1,610} \boxed{1,1,1072} \boxed{1,1,2032} \boxed{1,1,2186} \boxed{1,1,4912} \boxed{1,1,8191} \boxed{1,1,16806} \boxed{1,1,23488} \boxed{1,1,25240} \boxed{1,1,49900} \boxed{1,1,50248} \boxed{1,1,68920} \boxed{1,1,78124} \boxed{1,1,99100} \boxed{1,1,131071} \boxed{1,1,142090} \boxed{1,1,205378} \boxed{1,1,213052} \boxed{1,1,357100} \boxed{1,1,357910} \boxed{1,1,371292} \boxed{1,1,516496} \boxed{1,1,524287} \boxed{1,1,545146} \boxed{1,1,704968} \boxed{1,1,791716} \boxed{1,1,929230}~;
      \end{flushleft}
  \item[]
      \begin{flushleft}
        \boxed{1,2,1} \boxed{1,2,2289} \boxed{1,2,15947} \boxed{1,2,19337} \boxed{1,2,39447} \boxed{1,2,52395} \boxed{1,2,111045} \boxed{1,2,135795} \boxed{1,2,170255} \boxed{1,2,186801} \boxed{1,2,400449} \boxed{1,2,485225} \boxed{1,2,787461} \boxed{1,2,996495}~;
      \end{flushleft}
  \item[]
      \begin{flushleft}
        \boxed{1,3,296} \boxed{1,3,382} \boxed{1,3,1940} \boxed{1,3,5174} \boxed{1,3,7258} \boxed{1,3,12824} \boxed{1,3,101594} \boxed{1,3,133748} \boxed{1,3,210014} \boxed{1,3,336374} \boxed{1,3,441364} \boxed{1,3,576884} \boxed{1,3,855316}~;
      \end{flushleft}
  \item[]
      \begin{flushleft}
        \boxed{1,4,1} \boxed{1,4,51} \boxed{1,4,217} \boxed{1,4,22593} \boxed{1,4,33657} \boxed{1,4,34705} \boxed{1,4,95109} \boxed{1,4,205897} \boxed{1,4,348609} \boxed{1,4,355893} \boxed{1,4,383443} \boxed{1,4,433945} \boxed{1,4,510147} \boxed{1,4,578085} \boxed{1,4,684697} \boxed{1,4,926833} \boxed{1,4,950137}~;
      \end{flushleft}
  \item[]\ \vdots
  \item[]
      \begin{flushleft}
        \boxed{7,6,1} \boxed{7,6,34793} \boxed{7,6,61993} \boxed{7,8,4869} \boxed{7,8,13403} \boxed{7,8,123767} \boxed{7,9,74050} \boxed{7,9,82850} \boxed{7,10,1} \boxed{7,10,843} \boxed{7,10,1477} \boxed{7,10,89851} \boxed{7,10,131679}~;
      \end{flushleft}
  \item[]
      \begin{flushleft}
        \boxed{8,1,134} \boxed{8,1,254} \boxed{8,1,614} \boxed{8,1,2936} \boxed{8,1,3155} \boxed{8,1,6281} \boxed{8,1,8615} \boxed{8,1,64562} \boxed{8,1,88121} \boxed{8,3,1} \boxed{8,3,37} \boxed{8,3,1603} \boxed{8,5,1} \boxed{8,5,622} \boxed{8,5,2104} \boxed{8,5,18372} \boxed{8,5,48663} \boxed{8,5,89193} \boxed{8,5,114132} \boxed{8,7,69023} \boxed{8,7,80450} \boxed{8,7,88187} \boxed{8,9,1} \boxed{8,9,604} \boxed{8,9,4910} \boxed{8,9,23386} \boxed{8,9,36775}~;
      \end{flushleft}
  \item[]
      \begin{flushleft}
        \boxed{9,2,1} \boxed{9,2,4383} \boxed{9,4,1} \boxed{9,4,56683} \boxed{9,5,1084} \boxed{9,5,1374} \boxed{9,5,3246} \boxed{9,5,43256} \boxed{9,8,1} \boxed{9,8,1317} \boxed{9,8,3787} \boxed{9,8,13155} \boxed{9,8,23955} \boxed{9,8,46521} \boxed{9,8,81827} \boxed{9,10,1} \boxed{9,10,71} \boxed{9,10,16703} \boxed{9,10,30769} \boxed{9,10,31419} \boxed{9,10,102417}~;
      \end{flushleft}
  \item[]
      \begin{flushleft}
        \boxed{10,1,1} \boxed{10,1,61} \boxed{10,1,2524} \boxed{10,1,4990} \boxed{10,1,6892} \boxed{10,1,9910} \boxed{10,1,14209} \boxed{10,1,35710} \boxed{10,1,35791} \boxed{10,1,92923} \boxed{10,3,1} \boxed{10,3,194} \boxed{10,7,1} \boxed{10,7,43} \boxed{10,7,139} \boxed{10,7,11947} \boxed{10,7,25975} \boxed{10,7,26581} \boxed{10,7,64360} \boxed{10,9,1} \boxed{10,9,3928} \boxed{10,9,29420} \boxed{10,9,51835} \boxed{10,9,57995}~.
      \end{flushleft}
\end{itemize}

\subsection{The equation 2125}

The $\eta$--$s$--Diophantine equation $\eta(x)^m=s(x^n)$, is equivalent with the relation $(\eta_x)^m=s_{x^n}$ if we take into
consideration the significance of vectors $\eta$ and $s$. On the search domain given by (\ref{Dc2124}) with $a_x=2$, $r_x=1$,
$b_x=10^6$ and additional conditions $(\eta_x)^m<10^{17}$ and $x^n\le last(s)$, the equation has a sole solution
$\boxed{2,1,12}$, after having analyzed $4766682$ cases of $99999900$ possible situations. Obviously, this solution verifies
equation (2125)
\[
 \eta(12)^2=4^2=16\ \ \ s(12^1)=1+2+3+4+6=16~.
\]

\subsection{The equation 2126}

The $\eta$--$s$--Diophantine equation $\eta(x)+y=x+s(y)$ is equivalent with the relation $\eta_x+y=x+s_y$ if we take into
consideration the significance of vectors $\eta$ and $s$. Let
\begin{multline}\label{Dc2126}
  D_c=\set{a_x,a_x+r_x,a_x+2r_x,\ldots,\left\lfloor\frac{b_x-a_x}{r_x}\right\rfloor\cdot r_x}\\
      \times\set{a_y,a_y+r_y,a_y+2r_y,\ldots,\left\lfloor\frac{b_y-a_y}{r_y}\right\rfloor\cdot r_y}~.
\end{multline}
be the search domain. As equation (2126) has a great number of solutions on the search domain $D_c$ given by (\ref{Dc2126}), with
$a_x=2$, $r_x=1$, $b_x=10^6$, $a_y=3$, $r_y=1$ and $b_y=10^6$, we considered a restricted search domain with $a_x=2$, $r_x=113$,
$b_x=10^6$, $a_y=3$, $r_y=127$, $b_y=10^6$ and the additional condition of $x$ and $y$ being relative prime. This domain has
$69684900$ possible cases and $42366956$ cases to be analyzed as the cases when $(x,y)\neq1$ are excluded. On this search domain
equation (2126) has $44$ de solutions given in the form $\boxed{x,y}$:
\begin{itemize}
  \item[]
      \begin{flushleft}
      \boxed{3957,3305} \boxed{13449,288928} \boxed{21133,99825} \boxed{33111,65662} \boxed{38309,76838} \boxed{43733,693296} \boxed{67689,480190} \boxed{71757,143386} \boxed{93227,790070} \boxed{102041,190503} \boxed{103849,653926}\boxed{108369,192662} \boxed{116957,303914} \boxed{207357,276482} \boxed{239449,239017} \boxed{311091,358397} \boxed{320131,832234} \boxed{321487,605158} \boxed{323747,411229} \boxed{328267,601983} \boxed{332787,336299} \boxed{339567,227333} \boxed{349737,237239} \boxed{357873,257051} \boxed{407141,910593} \boxed{430871,369319} \boxed{431097,348491} \boxed{433131,434089} \boxed{459573,310391} \boxed{471325,940438} \boxed{505677,370843} \boxed{507259,822963} \boxed{508389,509273} \boxed{509745,534673} \boxed{516299,696471} \boxed{517881,722125} \boxed{520593,433835} \boxed{523983,533657} \boxed{567827,849633} \boxed{631785,681739} \boxed{635627,637289} \boxed{737779,955551} \boxed{765803,739397} \boxed{897561,897893}~.
      \end{flushleft}
\end{itemize}

\subsection{The equation 2127}

The $\eta$--$s$--Diophantine equation $\eta(x)\cdot y=x\cdot s(y)$ is equivalent with the relation $\eta_x\cdot y=x\cdot s_y$ if we
take into consideration the significance of vectors $\eta$ and $s$. On the search domain $D_c$ given by (\ref{Dc2126}) with
$a_x=100$, $r_x=1$, $b_x=200$, $a_y=1$, $r_y=1$, $b_x=10^5$ and $(x,y)=1$ (i.e. $x$ and $y$ are relative prime) equation has $83$
solutions. The number of possible cases is $10099798$ an the number of analyzed cases is $6151818$. We give the solutions in
the form $\boxed{x,y}$:
\begin{itemize}
  \item[] \boxed{101,6} \boxed{101,28} \boxed{101,496} \boxed{101,8128}~;
  \item[] \boxed{103,6} \boxed{103,28} \boxed{103,496} \boxed{103,8128}~;
  \item[] \boxed{107,6} \boxed{107,28} \boxed{107,496} \boxed{107,8128}~;
  \item[] \boxed{109,6} \boxed{109,28} \boxed{109,496} \boxed{109,8128}~;
  \item[] \vdots
\end{itemize}
As it is known, (\ref{L01}), for every prime number $p$ we have $\eta(p)=p$. The numbers $6=P_1$, $28=P_2$, $496=P_3$ and
$8128=P_4$ $33550336=P_5$ \ldots, are perfect numbers, \citep[A000396 ]{Sloane2014}, for which we have $s(P_k)=P_k$. As
the solutions can also be given as pairs \fbox{p,$P_k$}, where $p$ is a prime number and $P_k$ is a perfect number, we can propose
next theorem to be proved:
\begin{thm}
The solutions of the $\eta$--$s$--Diophantine equation $\eta(x)\cdot y=x\cdot s(y)$ are the pairs of numbers \fbox{p,$P_k$}, where $p$ is a prime number and $P_k$, $k=1,2,\ldots$, is a perfect number.
\end{thm}

\subsection{The equation 2128}

The $\eta$--$s$--Diophantine equation $\eta(x)/y=x/s(y)$ $\Leftrightarrow$ $\eta(x)\cdot s(y)=x\cdot y$ is equivalent with
the relation $\eta_x\cdot s_y=x\cdot y$ if we take into consideration the significance of vectors $\eta$ and $s$. On the
search domain $D_c$ given by (\ref{Dc2126}) with $a_x=2$, $r_x=1$, $b_x=100$, $a_y=6$, $r_y=1$, $b_x=8128$ and $(x,y)=1$ (i.e. $x$
and $y$ are relative prime) the equation has $93$ solutions. The search domain has $804177$ possible cases, but, due to the
additional condition $(x,y)=1$ the number of analyzed cases is $485989$. We give the solutions in the form $\boxed{x,y}$:
\begin{itemize}
  \item[] \boxed{3,28} \boxed{3,496} \boxed{3,8128}~;
  \item[] \boxed{5,6} \boxed{5,28} \boxed{5,496} \boxed{5,8128}~;
  \item[] \boxed{7,6} \boxed{7,496} \boxed{7,8128}~;
  \item[] \boxed{11,6} \boxed{11,28} \boxed{11,496} \boxed{11,8128}~;
  \item[] \boxed{13,6} \boxed{13,28} \boxed{13,496} \boxed{13,8128}~;
  \item[] \boxed{17,6 17,28} \boxed{17,496} \boxed{17,8128}~;
  \item[] \boxed{19,6} \boxed{19,28} \boxed{19,496} \boxed{19,8128}~;
  \item[] \ \vdots
\end{itemize}
By a solutions' analysis, we remark that we have the same pairs \fbox{p,$P_k$}, where $p$ is a prime and $P_k$ perfect number, as
in equation (2127). Also, it can be observed that there are missing pairs of solutions,due to the additional condition
$(x,y)=1$, for example, the pair $\boxed{3,6}$ is missing, as $(3,6)=3$.

\subsection{The equation 2129}

The $\eta$--$s$--Diophantine equation $\eta(x)^y=x^{s(y)}$ is equivalent with the relation $(\eta_x)^y=x^{s_y}$, if we take into
consideration the significance of vectors $\eta$ and $s$. The solutions of the equation are pairs \fbox{p,$P_k$}, where $p$ is a
prime number and $P_k$ a perfect number as in equation (2127).

\subsection{The equation 2130}

The $\eta$--$s$--Diophantine equation $\eta(x)^y=s(y)^y$ is equivalent with the relation $(\eta_x)^y=(s_y)^x$, if we take into
consideration the significance of vectors $\eta$ and $s$. This equation has no solutions on the search domain
$D_c=\set{2,3,\ldots,10^6}\times\set{1,2,\ldots,10^6}$. which has $999999000000$ possible cases of which only $2190820$ where
analyzed due to the restrictive conditions $(\eta_x)^y<10^{307}$, $(s_y)^x<10^{307}$ and $(x,y)=1$.

\section{The $\eta$--$\pi$--Diophantine equations}

Let $m,n,k\in\Ns$ be fixed and $x$ and $y$ unknown positive integers. The Diophantine equation in which function $\eta$ and
$\pi$ are involved, given by formula (\ref{FormulaPiEta}), are called $\eta$--$\pi$--Diophantine equations. The list of
$\eta$--$\pi$--Diophantine equations, as given in \citep{Smarandache1999a}, which we intend to solve empirically
are:
\begin{enumerate}
  \item[(2152)] $\eta(x)=\pi(m\cdot x+n)$~,
  \item[(2153)] $\eta(x)^m=\pi(x^n)$~,
  \item[(2154)] $\eta(x)+y=x+\pi(y)$~,
  \item[(2155)] $\eta(x)\cdot y=x\cdot\pi(y)$~,
  \item[(2156)] $\eta(x)/y=x/\pi(y)$~,
  \item[(2157)] $\eta(x)^y=x^{\pi(y)}$~,
  \item[(2158)] $\eta(x)^y=\pi(y)^x$~,
\end{enumerate}

\subsection[Empirical solving of $\eta$--$\pi$--Diophantine equations]{Partial empirical solving of $\eta$--$\pi$--Diophantine equations}

To solve empirically the $\eta$--$\pi$--Diophantine equations (2152--2158) we will read the file $\eta.prn$ which contains the
values of functions $\eta(n)$ for $n=1,2,\ldots,10^6$, this file being generated previously. The file $\eta.prn$ will be attributed
to vector $\eta$ by means of the Mathcad sequence:
\[
 \eta:=READPRN("...\backslash \eta.prn")\ \ last(\eta)=10^6
\]
where the command $last(\eta)$ indicates the last index of vectors $\eta$. This manoeuver is necessary to reduce a lot the time due
to the empirical search of the solutions of the Diophantine equations.

\subsection{The equation 2152}

The $\eta$--$\pi$--Diophantine equation $\eta(x)=\pi(m\cdot x+n)$ is equivalent with the relation $\eta_x=\pi(m\cdot x+n)$ if we
take into consideration the significance of vector $\eta$ and formula (\ref{FormulaPiEta}). Let us consider the search domain
\begin{equation}\label{Dc2153}
  D_c=\set{2,3,\ldots,20}^2\times\set{1,2,3,\ldots,10^3}
\end{equation}
with $(m,n)=1$, $m\cdot x+n\le last(\eta)$ and $x$ non-prime. Then, the number of possible cases is $361000$. Due to the
additional conditions, the number of analyzed cases is $179712$.
\begin{prog}
The search program is:
\begin{tabbing}
  $Ed2152$\=$(a_m,b_m,a_n,b_n,a_x,r_x,b_x):=$\\ \\
  \>\vline\ $j\leftarrow0$\\
  \>\vline\ $S\leftarrow("m"\ "n"\ "x")$\\
  \>\vline\ $f$\=$or\ m\in a_m..b_m$\\
  \>\vline\ \>\ $f$\=$or\ n\in a_n..b_n$\\
  \>\vline\ \>\ \>\ $f$\=$or\ x\in a_x,a_x+r_x..b_x$\\
  \>\vline\ \>\ \>\ \>\ $if$\=$\ Tp\eta(x)\textbf{=}0\wedge m\cdot x+n\le last(\eta)\wedge \gcd(m,n)\textbf{=}1$\\
  \>\vline\ \>\ \>\ \>\ \>\vline\ $j\leftarrow j+1$\\
  \>\vline\ \>\ \>\ \>\ \>\vline\ $S\leftarrow stack[S,(m\ n\ x)]\ if\ \eta_x\textbf{=}\pi(m\cdot x+n)$\\
  \>\vline\ $return\ stack[S,(j\ j\ j)]$
\end{tabbing}
\end{prog}
The program calls the subprogram \ref{ProgTpEta} for establishing the primality of $x$ and $y$. On the search domain $D_c$ given by
(\ref{Dc2152}) we have $35$ solutions, given in the form \boxed{m,n,x}:
\begin{itemize}
\item[] \boxed{2,3,62} \boxed{2,3,543} \boxed{2,5,62} \boxed{2,7,573} \boxed{2,9,74} \boxed{2,9,573} \boxed{2,11,74} \boxed{2,11,573} \boxed{2,13,74} \boxed{2,13,573} \boxed{2,13,579} \boxed{2,15,82} \boxed{2,15,573} \boxed{2,15,579} \boxed{2,17,579} \boxed{2,19,86} \boxed{2,19,579} \boxed{2,19,591}~;
\item[] \boxed{3,2,362} \boxed{3,4,362} \boxed{3,7,382} \boxed{3,8,382} \boxed{3,10,382} \boxed{3,11,382} \boxed{3,13,382} \boxed{3,13,386} \boxed{3,14,382} \boxed{3,14,386} \boxed{3,16,382} \boxed{3,16,386} \boxed{3,17,386} \boxed{3,19,386} \boxed{3,19,394} \boxed{3,20,386} \boxed{3,20,394}~.
\end{itemize}
The necessary time for searching the solutions is of approximately 360 seconds on a computer with an Intel processor of 2.20GHz with
RAM of 4.00GB (3.46GB usable).

\subsection{The equation 2153}

The $\eta$--$\pi$--Diophantine equation $\eta(x)^m=\pi\big(x^n\big)$ is equivalent with the relation
$\big(\eta_x\big)^m=\pi\big(x^n\big)$ if we take into consideration the significance of vector $\eta$ and formula (\ref{FormulaPiEta}). Fie search domain
\begin{equation}\label{Dc2152}
  D_c=\set{2,3,\ldots,10}^2\times\set{1,2,3,\ldots,10^3}
\end{equation}
with $x^n\le last(\eta)$. Then, the number of possible cases is $80919$. Due to the additional condition, the number of analyzed
cases is $10494$.
\begin{prog}
The search program is:
\begin{tabbing}
  $Ed2153$\=$(a_m,b_m,a_n,b_n,a_x,r_x,b_x):=$\\ \\
  \>\vline\ $j\leftarrow0$\\
  \>\vline\ $S\leftarrow("m"\ "n"\ "x")$\\
  \>\vline\ $f$\=$or\ m\in a_m..b_m$\\
  \>\vline\ \>\ $f$\=$or\ n\in a_n..b_n$\\
  \>\vline\ \>\ \>\ $f$\=$or\ x\in a_x,a_x+r_x..b_x$\\
  \>\vline\ \>\ \>\ \>\ $if$\=$x^n\le last(\eta)$\\
  \>\vline\ \>\ \>\ \>\ \>\vline\ $j\leftarrow j+1$\\
  \>\vline\ \>\ \>\ \>\ \>\vline\ $S\leftarrow stack[S,(m\ n\ x)]\ if\ \big(\eta_x\big)^m\textbf{=}\pi\big(x^n\big)$\\
  \>\vline\ $return\ stack[S,(j\ j\ j)]$
\end{tabbing}
\end{prog}
on the search domain $D_c$ given by (\ref{Dc2153}) we have $3$ de solutions, given in the form \boxed{m,n,x}: \boxed{2,2,10},
\boxed{2,3,2} and \boxed{2,3,3}. The necessary time for searching the solutions is of approximately $1580$ seconds on a computer
with an Intel processor of 2.20GHz with RAM of 4.00GB (3.46GB usable).

\subsection{The equation 2154}

The $\eta$--$\pi$--Diophantine equation $\eta(x)+y=x+\pi(y)$ is equivalent with the relation $\eta_x+y=x+\pi(y)$ if we take into
consideration the significance of vector $\eta$ and formula (\ref{FormulaPiEta}). Fie search domain
\begin{multline}
  D_c=\set{a_x,a_x+r_x,a_x+2r_x,\ldots,\left\lfloor\frac{b_x-a_x}{r_x}\right\rfloor\cdot r_x}\\
      \times\set{a_y,a_y+r_y,a_y+2r_y,\ldots,\left\lfloor\frac{b_y-a_y}{r_y}\right\rfloor\cdot r_y}~.
\end{multline}
with $a_x=2$, $r_x=3$, $b_x=10^3$, $a_y=2$, $r_y=5$ and $b_y=10^3$. Then, the number of possible and analyzed cases is
$66600$. In these conditions we have $52$ solutions given in the form \boxed{x,y}:
\begin{itemize}
  \item[] \boxed{14,12} \boxed{65,72} \boxed{74,52} \boxed{95,102}~;
  \item[] \boxed{116,117} \boxed{128,157} \boxed{140,172} \boxed{152,172} \boxed{158,107} \boxed{176,212}~;
  \item[] \boxed{230,262} \boxed{245,292} \boxed{290,327}~;
  \item[] \boxed{305,307} \boxed{329,352} \boxed{350,422} \boxed{368,427} \boxed{374,442} \boxed{380,447}~;
  \item[] \boxed{404,377} \boxed{410,457} \boxed{416,497} \boxed{422,267} \boxed{434,497} \boxed{455,542}~;
  \item[] \boxed{500,592} \boxed{527,607} \boxed{533,602} \boxed{548,507} \boxed{551,637} \boxed{572,682}~;
  \item[] \boxed{602,682} \boxed{614,382} \boxed{623,652} \boxed{635,622} \boxed{656,747} \boxed{668,612} \boxed{674,417} \boxed{680,802} \boxed{689,772} \boxed{698,432}~;
  \item[] \boxed{725,842} \boxed{746,462} \boxed{779,892}~;
  \item[] \boxed{803,882} \boxed{836,982} \boxed{848,957} \boxed{860,982} \boxed{866,532}~;
  \item[] \boxed{917,947} \boxed{956,867}~.
\end{itemize}
The necessary time for searching the solutions is of approximately $112$ seconds on a computer with an Intel processor of 2.20GHz
with RAM of 4.00GB (3.46GB usable).

\subsection{The equation 2155}

The $\eta$--$\pi$--Diophantine equation $\eta(x)\cdot y=x\cdot\pi(y)$ is equivalent with the relation $\eta_x\cdot y=x\cdot\pi(y)$ if we take into consideration the significance of vector $\eta$ and formula (\ref{FormulaPiEta}). Fie search domain
\begin{multline}\label{Dc2155}
  D_c=\set{a_x,a_x+r_x,a_x+2r_x,\ldots,\left\lfloor\frac{b_x-a_x}{r_x}\right\rfloor\cdot r_x}\\
      \times\set{a_y,a_y+r_y,a_y+2r_y,\ldots,\left\lfloor\frac{b_y-a_y}{r_y}\right\rfloor\cdot r_y}~.
\end{multline}
with $a_x=2$, $r_x=1$, $b_x=10^3$, $a_y=2$, $r_y=1$ and $b_y=10^3$. Then, the number of possible and analyzed cases is $998001$. In these conditions we have $985$ solutions of which we give the first $45$ and the last $50$ solutions in the form \boxed{x,y}:
\begin{itemize}
  \item[] \boxed{6,4} \boxed{6,8} \boxed{8,4} \boxed{8,6}~;
  \item[] \boxed{10,4} \boxed{10,6} \boxed{10,8} ~;
  \item[] \boxed{12,3} \boxed{12,27} \boxed{12,30} \boxed{12,33} ~;
  \item[] \boxed{14,4} \boxed{14,6} \boxed{14,8}~;
  \item[] \boxed{15,3} \boxed{15,27} \boxed{15,30} \boxed{15,33} ~;
  \item[] \boxed{16,24}~;
  \item[] \boxed{18,3} \boxed{18,27} \boxed{18,30} \boxed{18,33}~;
  \item[] \boxed{20,96} \boxed{20,100} \boxed{20,120}~;
  \item[] \boxed{21,3} \boxed{21,27} \boxed{21,30} \boxed{21,33}~;
  \item[] \boxed{22,4} \boxed{22,6} \boxed{22,8}~;
  \item[] \boxed{25,10} \boxed{25,15} \boxed{25,20}~;
  \item[] \boxed{26,4} \boxed{26,6} \boxed{26,8}~;
  \item[] \boxed{27,3} \boxed{27,30} \boxed{27,33}~;
  \item[] \boxed{28,96} \boxed{28,100} \boxed{28,120}~;
  \item[] \ \vdots
  \item[] \boxed{951,3} \boxed{951,27} \boxed{951,30} \boxed{951,33}~;
  \item[] \boxed{955,330} \boxed{955,335} \boxed{955,340} \boxed{955,350} \boxed{955,355} \boxed{955,360}~;
  \item[] \boxed{956,96} \boxed{956,100} \boxed{956,120}~;
  \item[] \boxed{958,4} \boxed{958,6} \boxed{958,8}~;
  \item[] \boxed{964,96} \boxed{964,100} \boxed{964,120}~;
  \item[] \boxed{965,330} \boxed{965,335} \boxed{965,340} \boxed{965,350} \boxed{965,355} \boxed{965,360}~;
  \item[] \boxed{974,4} \boxed{974,6} \boxed{974,8}~;
  \item[] \boxed{982,4} \boxed{982,6} \boxed{982,8}~;
  \item[] \boxed{985,330} \boxed{985,335} \boxed{985,340} \boxed{985,350} \boxed{985,355} \boxed{985,360}~;
  \item[] \boxed{993,3} \boxed{993,27} \boxed{993,30} \boxed{993,33}~;
  \item[] \boxed{995,330} \boxed{995,335} \boxed{995,340} \boxed{995,350} \boxed{995,355} \boxed{995,360}~;
  \item[] \boxed{998,4} \boxed{998,6} \boxed{998,8}~.
\end{itemize}
The necessary time for searching the solutions is of approximately $116$ seconds on a computer with an Intel processor of 2.20GHz
with RAM of 4.00GB (3.46GB usable).

\subsection{The equation 2156}

The $\eta$--$\pi$--Diophantine equation $\eta(x)/y=x/\pi(y)$ is equivalent with the relation
\begin{equation}\label{Ec2156}
  \eta_x\cdot\pi(y)=x\cdot y~,
\end{equation}
if we take into consideration the significance of vector $\eta$ and formula (\ref{FormulaPiEta}). Fie search domain given by
(\ref{Dc2155}), with $a_x=2$, $r_x=1$, $b_x=10^3$, $a_y=2$, $r_y=1$ and $b_y=10^3$. Then, the number of possible and analyzed
cases is $998001$. In these conditions, equation (\ref{Ec2156}) not has solutions.

\subsection{The equation 2157}

The $\eta$--$\pi$--Diophantine equation $\eta(x)^y=x^{\pi(y)}$ is equivalent with the relation
\begin{equation}\label{Ec2157}
  (\eta_x)^y=x^{\pi(y)}~,
\end{equation}
if we take into consideration the significance of vector $\eta$ and formula (\ref{FormulaPiEta}). Let us consider the search
domain given by (\ref{Dc2155}), with $a_x=2$, $r_x=1$, $b_x=10^3$, $a_y=2$, $r_y=1$ and $b_y=10^3$ with following restrictions
$(x,y)=1$, $(\eta_x)^y\le10^{307}$ and $x^{\pi(y)}\le10^{307}$. Then, the number of possible cases is $998001$ and the number of
analyzed cases is $112809$. In these conditions, equation (\ref{Ec2157}) has $3$ solutions: \boxed{32,5} \boxed{81,4}, \boxed{81,8}.
\begin{prog}The search program for the solutions of relation (\ref{Ec2157})
   \begin{tabbing}
     $Ed2157$\=$(a_x,r_x,b_x,a_y,r_y,b_y):=$\\
     \>\vline\ $j\leftarrow0$\\
     \>\vline\ $S\leftarrow("x"\ "y")$\\
     \>\vline\ $f$\=$or\ x\in a_x,a_x+r_x..b_x$\\
     \>\vline\ \>\ $f$\=$or\ y\in a_y,a_y+r_y..b_y$\\
     \>\vline\ \>\ \>\vline\ $break\ on\ error\ (\eta_x)^y$\\
     \>\vline\ \>\ \>\vline\ $break\ on\ error\ x^{\pi(y)}$\\
     \>\vline\ \>\ \>\vline\ $i$\=$f\ \gcd(x,y)\textbf{=}1$\\
     \>\vline\ \>\ \>\vline\ \>\vline\ $j\leftarrow j+1$\\
     \>\vline\ \>\ \>\vline\ \>\vline\ $S\leftarrow stack[S,(x\ y)]\ if\ (\eta_x)^y\textbf{=}x^{\pi(y)}$\\
     \>\vline\ $return\ stack[S,(j\ j)]$
   \end{tabbing}
\end{prog}

\subsection{The equation 2158}

The $\eta$--$\pi$--Diophantine equation $\eta(x)^y=\pi(y)^x$ is equivalent with the relation
\begin{equation}\label{Ec2158}
  (\eta_x)^y=\pi(y)^x~,
\end{equation}
considering the meaning of vector $\eta$ and formula (\ref{FormulaPiEta}). Let the search domain be given by
(\ref{Dc2155}), with $a_x=2$, $r_x=1$, $b_x=10^3$, $a_y=2$, $r_y=1$ and $b_y=10^3$ with restrictions $(x,y)=1$,
$(\eta_x)^y\le10^{307}$ and $\pi(y)^x\le10^{307}$. Then, the number of possible cases is $998001$ and the number of analyzed
cases are $35743$. In these conditions, equation (\ref{Ec2158}) has no solutions.

\section{The $\eta$--$\sigma_k$--Diophantine equations}

Let us consider $m,n,k\in\Ns$ fixed and $x$ and $y$ unknown positive integers. The Diophantine equations where functions
$\eta$ and $\sigma_k$ are involved, ar called $\eta$--$\sigma_k$--Diophantine equations. The list of
$\eta$--$\sigma_k$--Diophantine equations, as in \citep{Smarandache1999a}, which we intend to solve empirically, is:
\begin{enumerate}
  \item[(2166)] $\eta(x)=\sigma_k(m\cdot x+n)$~,
  \item[(2167)] $\eta(x)^m=\sigma_k(x^n)$~,
  \item[(2168)] $\eta(x)+y=x+\sigma_k(y)$~,
  \item[(2169)] $\eta(x)\cdot y=x\cdot\sigma_k(y)$~,
  \item[(2170)] $\eta(x)/y=x/\sigma_k(y)$~,
  \item[(2171)] $\eta(x)^y=x^{\sigma_k(y)}$~,
  \item[(2172)] $\eta(x)^y=\sigma_k(y)^x$~,
\end{enumerate}

\subsection[Empirical solving of $\eta$--$\sigma_k$--Diophantine equations]
{Partial empirical solving of $\eta$--$\sigma_k$--Diophantine equations}

For every Diophantine equation solved in this section, the file $\eta.prn$ is read, generated by the program \ref{ProgEta}, by
means of the Mathcad function $READPRN$
\[
 \eta:=READPRN("...\backslash \eta.prn")\ \ last(\eta)=10^6
\]
where the command $last(\eta)$ indicates the last index of vector $\eta$. The reading time is of about $10$ seconds, and, therefore,
an important saving of the search time can be remarked.

The file $\sigma0.prn$ is read and the values will be assigned to vector $\sigma0$, by means of the Mathcad function $READPRN$, with
the sequence:
\[
 \sigma0:=READPRN("...\backslash \sigma0.prn")\ \ last(\sigma0)=10^6~,
\]
where each component $\sigma0_k$ contain the number of divisors of $k$, for $k=1,2,\ldots,10^6$. The values were generated with the
program \ref{ProgSigmak}. The time for generating the file $\sigma0.prn$ was \emph{2:4:48.362hhmmss}~. If we have a
Diophantine equation in which function $\sigma_0(x)$ is involved, we will use vector $\sigma0$. The time saving that results is important.

If we have a Diophantine equation in which function $\sigma(x)$ is involved and we proceed to a systematic empirical search for many
values $x$, we will read the file $\sigma1.prn$ in vector $\sigma1$ with the sequence:
\[
 \sigma1:=READPRN("...\backslash \sigma1.prn")\ \ last(\sigma1)=10^6~,
\]
where each component $\sigma1_k$ contains the sum of the divisors of $k$, for $k=1,2,\ldots,10^6$. The necessary time for generating
the file $\sigma1.prn$ was \emph{2:5:32.155hhmmss}.

By the same reasons as in reading files $\sigma0.prn$ or $\sigma1.prn$, we will read the file $\sigma2.prn$ in vector $\sigma2$ with the sequence:
\[
 \sigma2:=READPRN("...\backslash \sigma2.prn")\ \ last(\sigma2)=10^6~.
\]
where each component $\sigma2_k$ contains the sum of the squares of the divisors of $k$, for $k=1,2,\ldots,10^6$. The necessary
time for generating the file  $\sigma2.prn$ was \emph{2:4:50.197hhmmss}. Thereby, an important time saving will be obtained in the empirical search.

\subsection{The equation 2166}

The Diophantine equation $(2166)$ for $k=0$ is $\eta(x)=\sigma_0(m\cdot x+n)$ and is equivalent with the relation
$\eta_x=\sigma0_{m\cdot x+n}$, considering the meaning of files $\eta$ and $\sigma0$. Let
\begin{equation}\label{Dc2166}
  D_c=\set{1,2,\ldots,10}\times\set{1,2,\ldots,10}\times\set{1,2,\ldots,10^6}~,
\end{equation}
be the search domain. We have $10^8$ possible cases. As the component $\sigma0_{m\cdot x+n}$ with $m\cdot x+n>10^6$ can not be
addressed, for the search domain $D_c$ given by (\ref{Dc2166}), we consider the additional condition $m\cdot x+n\le10^6$. With this
condition, the number of analyzed cases is $29289486$. In these conditions we have $3869$ solutions. We present the first $79$
solutions in the form $\boxed{m,n,x}$:
\begin{itemize}
  \item[]
  \begin{flushleft}
  \boxed{1,1,2} \boxed{1,1,3} \boxed{1,1,15} \boxed{1,1,63} \boxed{1,1,175} \boxed{1,1,384} \boxed{1,1,896} \boxed{1,1,2240} \boxed{1,1,3024} \boxed{1,1,4095} \boxed{1,1,6720} \boxed{1,1,13824} \boxed{1,1,24255} \boxed{1,1,25920} \boxed{1,1,39375} \boxed{1,1,39424} \boxed{1,1,68607} \boxed{1,1,81920} \boxed{1,1,93555} \boxed{1,1,281600} \boxed{1,1,425984} \boxed{1,1,552500} \boxed{1,1,630784} \boxed{1,1,649539}~;
  \end{flushleft}
  \item[]
  \begin{flushleft}
  \boxed{1,2,4} \boxed{1,2,8} \boxed{1,2,12} \boxed{1,2,16} \boxed{1,2,18} \boxed{1,2,24} \boxed{1,2,48} \boxed{1,2,64} \boxed{1,2,90} \boxed{1,2,128} \boxed{1,2,240} \boxed{1,2,288} \boxed{1,2,320} \boxed{1,2,640} \boxed{1,2,720} \boxed{1,2,960} \boxed{1,2,1440} \boxed{1,2,1470} \boxed{1,2,2240} \boxed{1,2,2688} \boxed{1,2,2880} \boxed{1,2,3150} \boxed{1,2,3402} \boxed{1,2,4800} \boxed{1,2,5346} \boxed{1,2,5632} \boxed{1,2,5760} \boxed{1,2,8064} \boxed{1,2,20160} \boxed{1,2,21384} \boxed{1,2,26730} \boxed{1,2,32768} \boxed{1,2,48750} \boxed{1,2,73728} \boxed{1,2,85050} \boxed{1,2,90112} \boxed{1,2,95550} \boxed{1,2,98304} \boxed{1,2,114688} \boxed{1,2,135168} \boxed{1,2,138240} \boxed{1,2,200704} \boxed{1,2,270336} \boxed{1,2,308750} \boxed{1,2,319488} \boxed{1,2,401408} \boxed{1,2,409600} \boxed{1,2,442368} \boxed{1,2,630784} \boxed{1,2,708750} \boxed{1,2,716800} \boxed{1,2,737280} \boxed{1,2,802816} \boxed{1,2,901120} \boxed{1,2,921600}~;
  \end{flushleft}
  \item[]\ \vdots
\end{itemize}
The search time was $108.645$ seconds.

As the number of solutions is big, we propose to restrict the search domain as follows: let $m$ and $n$ be relative prime and
$x$ to take the values of an arithmetic progression with the first term $a_x=1$, the ratio $r_x=113$ and the last term $\le
b_x=10^6$. These restrictions are obtained by the composed condition, written in the Mathcad syntax, $\gcd(m,n)=1$ and
$x\in\set{a_x,a_x+r_x..b_x}$. In these conditions we have $D_c=\set{1,2,\ldots,10}^2\times\set{1,14,27,\ldots,999938}$ with
$(m,n)=1$. The number of analyzed cases is $188273$ and the number of solutions is $16$:
\begin{itemize}
  \item[] \boxed{1,2,1470} \boxed{1,9,41472}~;
  \item[] \boxed{2,9,4860} \boxed{2,9,41472}~;
  \item[] \boxed{3,4,41472} \boxed{3,8,4860}~;
  \item[] \boxed{4,3,41472}~;
  \item[] \boxed{5,3,41472} \boxed{5,4,41472} \boxed{5,8,4860} \boxed{5,9,4860}~;
  \item[] \boxed{7,8,4860} \boxed{7,9,4860}~;
  \item[] \boxed{8,9,4860}~;
  \item[] \boxed{9,4,4860}~;
  \item[] \boxed{10,9,4860}~.
\end{itemize}
The empirical search domain of the solutions of equation $(2166)$ for $k=0$ is:
\begin{tabbing}
  $Ed2166$\=$(a_m,b_m,a_n,b_n,a_x,r_x,b_x):=$\\ \\
  \>\vline\ $j\leftarrow0$\\
  \>\vline\ $S\leftarrow("m"\ "n"\ "x")$\\
  \>\vline\ $fo$\=$r\ m\in a_m..b_m$\\
  \>\vline\ \>\ $fo$\=$r\ n\in a_n..b_n$\\
  \>\vline\ \>\ \>\ $fo$\=$r\ x\in a_x,a_x+r_x..b_x$\\
  \>\vline\ \>\ \>\ \>\ $if$\=$\ m\cdot x+n\le last(\sigma_0)\wedge \gcd(m,n)\textbf{=}1$\\
  \>\vline\ \>\ \>\ \>\ \>\vline\ $j\leftarrow j+1$\\
  \>\vline\ \>\ \>\ \>\ \>\vline\ $S\leftarrow stack[S,(m\ n\ x)]\ if\ \eta_x\textbf{=}\sigma0_{m\cdot x+n}$\\
  \>\vline\ $return\ stack[S,(j\ j\ j)]$
\end{tabbing}

For $k=1$ and $k=2$ equation (2166) has no solutions on the search domain $D_c$ given by (\ref{Dc2166}).

As the values $\sigma1_k$ overpass the values $\eta_k$, we have wondered if there exist solutions for the equation Diophantine
\[
 \eta(m\cdot x+n)=\sigma(x)~.
\]
The equivalent relation is $\eta_{m\cdot x+n}=\sigma1_x$, considering the meaning of vectors $\eta$ and $\sigma1$. On the search domain
\begin{equation}\label{Dc2166prim}
   D_c=\set{1,2,\ldots,10}^2\times\set{a_x,a_x+r_x,a_x+2r_x,\ldots,\left\lfloor\frac{b_x-a_x}{r_x}\right\rfloor\cdot r_x}~,
\end{equation}
with $a_x=1$, $r_x=1$ and $b_x=10^6$ and $(m,n)=1$ (i.e. $m$ and $n$ relative prime), there exist for this equation $29$ solutions,
given in the form $\boxed{m,n,x}$:
\begin{itemize}
  \item[]
     \begin{flushleft}
       \boxed{1,1,2} \boxed{1,1,3} \boxed{1,3,4} \boxed{1,4,2} \boxed{1,4,5} \boxed{1,4,9} \boxed{1,5,3} \boxed{1,6,25} \boxed{1,9,3} \boxed{1,10,4}~;
     \end{flushleft}
  \item[] \boxed{3,1,5} \boxed{3,2,4} \boxed{3,10,13}~;
  \item[] \boxed{4,3,9} \boxed{4,5,4}~;
  \item[] \boxed{5,1,4} \boxed{5,7,9} \boxed{5,8,4} \boxed{5,9,3}~;
  \item[] \boxed{6,5,25}~;
  \item[] \boxed{7,1,5} \boxed{7,2,9} \boxed{7,3,3} \boxed{7,10,5}~;
  \item[] \boxed{8,3,4} \boxed{8,5,5}~;
  \item[] \boxed{9,1,7} \boxed{9,10,9}~;
  \item[] \boxed{10,1,9}~.
\end{itemize}
The search domain has $10^8$ possible cases, but the analyzed are of only $21271688$ due to restrictions $m\cdot x+n\le last(\eta)$
and $\gcd(m,n)=1$.

Analogously, the Diophantine equation
\[
 \eta(m\cdot x+n)=\sigma_2(x)~.
\]
was considered, which, on the search domain $D_c$ given by (\ref{Dc2166prim}) with $a_x=1$, $r_x=1$, $b_x=10^6$ and additional conditions $m\cdot x+n\le last(\eta)$ and $\gcd(m,n)=1$ has $11$ solutions:
\begin{itemize}
  \item[]
     \begin{flushleft}
       \boxed{1,3,2} \boxed{1,8,2} \boxed{2,1,2} \boxed{3,4,2} \boxed{4,7,2} \boxed{6,7,3} \boxed{7,1,2} \boxed{7,4,3} \boxed{7,6,2} \boxed{8,1,3} \boxed{9,2,2}~.
     \end{flushleft}
\end{itemize}

\subsection{The equation 2167}

The Diophantine equation $\eta(x)^m=\sigma_k(x^n)$ with $k=0,1,2$, can be equivalent with the relation $(\eta_x)^m=\sigma k_{x^n}$,
with $k=0,1,2$, if we consider the significance of vectors $\eta$ and $\sigma k$, with $k=0,1,2$. If we impose the condition $m\neq
n$ the number of solutions is $22$, for $k=0$, on the search domain
\begin{equation}
  D_c=\set{1,2,\ldots,10}^2\times\set{2,3,\ldots,10^6}~.
\end{equation}
The number of possible cases is $99999900$, as a consequence of conditions $m\neq n$ and $x^n\le10^6$, the number of analyzed
cases is $9010449$, and the search time is under one minute. The $22$ solutions are presented, in the form $\boxed{m,n,x}$,
\begin{itemize}
  \item[]
  \begin{flushleft}
   \boxed{1,2,3} \boxed{1,2,81} \boxed{1,2,686} \boxed{1,4,5} \boxed{1,6,7}~;
  \end{flushleft}
  \item[]
  \begin{flushleft}
  \boxed{2,1,13440} \boxed{2,1,272160} \boxed{2,1,340200} \boxed{2,1,374220} \boxed{2,1,427680} \boxed{2,1,534600} \boxed{2,1,598752} \boxed{2,1,777600} \boxed{2,1,935550} \boxed{2,3,2} \boxed{2,3,96} \boxed{2,4,10} \boxed{2,4,15} \boxed{2,5,8} \boxed{2,8,3}~;
  \end{flushleft}
  \item[]
  \begin{flushleft}
  \boxed{3,4,30} \boxed{3,7,2}~.
  \end{flushleft}
\end{itemize}

The Diophantine equations $\eta(x)^m=\sigma_k(x^n)$, for $k=1,2$ have $90$ solutions on the search domain
\[
 D_c=\set{1,2,\ldots,10}^2\times\set{1,2,\ldots,10^6}
\]
and restriction $x^n\le10^6$. All solutions are trivial only with $x=1$.

\subsection{The equation 2168}

The $\eta$--$\sigma_0$--Diophantine equation $\eta(x)+y=x+\sigma_0(y)$ has on thee search domain $D_c=\set{2,3,\ldots,10^4}^2$ $99980001$ cases to be analyzed. This equation has $9893$ solutions, obtained in approximately 50 seconds, of which we present the first $26$ and the last $26$ solutions in the form $\boxed{x,y}$:
\begin{itemize}
  \item[]
     \begin{flushleft}
       \boxed{2,2} \boxed{3,2} \boxed{4,2} \boxed{5,2} \boxed{6,5} \boxed{7,2} \boxed{8,8} \boxed{9,5}~;
     \end{flushleft}
  \item[]
     \begin{flushleft}
       \boxed{10,7} \boxed{11,2} \boxed{13,2} \boxed{15,14} \boxed{16,14} \boxed{17,2} \boxed{18,18} \boxed{19,2}~;
     \end{flushleft}
  \item[]
     \begin{flushleft}
       \boxed{20,17} \boxed{21,20} \boxed{22,13} \boxed{22,15} \boxed{22,16} \boxed{23,2} \boxed{25,17} \boxed{27,22} \boxed{28,23} \boxed{29,2}~;
     \end{flushleft}
  \item[]\ \vdots
  \item[]
     \begin{flushleft}
       \boxed{9971,9928} \boxed{9972,9697} \boxed{9973,2} \boxed{9976,9937} \boxed{9977,9076} \boxed{9977,9086} \boxed{9978,8317} \boxed{9978,8323}~;
     \end{flushleft}
  \item[]
     \begin{flushleft}
       \boxed{9981,8888} \boxed{9981,8904} \boxed{9982,9963} \boxed{9982,9975} \boxed{9983,9838} \boxed{9983,9846} \boxed{9984,9973} \boxed{9986,4997} \boxed{9987,6662} \boxed{9989,8566} \boxed{9989,8570}~;
     \end{flushleft}
  \item[]
     \begin{flushleft}
       \boxed{9990,9957} \boxed{9991,9896} \boxed{9993,6668} \boxed{9994,9733} \boxed{9995,8016} \boxed{9996,9983} \boxed{9999,9902}~.
     \end{flushleft}
\end{itemize}
The search program uses the equivalent relation $\eta_x+y=x+\sigma0_y$, taking into consideration the meaning of vectors $\eta$ and $\sigma0$. The simple search algorithm can be deduced from the source text of the program:
\begin{tabbing}
  $Ed2168$\=$(a_x,r_x,b_x,a_y,r_y,b_y):=$\\ \\
  \>\vline\ $S\leftarrow ("x"\ "y")$\\
  \>\vline\ $f$\=$or\ x\in a_x,a_x+r_x..b_x$\\
  \>\vline\ \>\ $f$\=$or\ y\in a_y,a_y+r_y..b_y$\\
  \>\vline\ \>\ \>\ $S\leftarrow stack[S,(x\ y)]\ if\ \eta_x+y\textbf{=}x+\sigma0_y$\\
  \>\vline\ $return\ S$
\end{tabbing}
on the search domain
\[
 D_c=\set{2,505,1008,\ldots,999966}\times\set{3,604,1205,\ldots,999466}~,
\]
the Diophantine equation $\eta(x)+y=x+\sigma_0(y)$ has two solutions $\boxed{262065,244610}$ and $\boxed{741927,494626}$.

The $\eta$--$\sigma_k$--Diophantine equation $\eta(x)+y=x+\sigma_k(y)$, with $k=1,2$ does not have solutions on the search domain $D_c=\set{1,2,\ldots,10^6}^2$, other than the $78500$ trivial solutions when $y=1$ and $x$ is a prime number or $x=4$.

\subsection{The equation 2169}

The $\eta$--$\sigma_k$--Diophantine equation $\eta(x)\cdot y=x\cdot\sigma_k(y)$ is equivalent with the relation $\eta_x\cdot
y=x\cdot\sigma k_y$ if we consider the meaning of vectors $\eta$ and $\sigma k$ for $k=0,1,2$. We consider the search domain
\begin{equation}\label{Dc2169}
  D_c=\set{a_x,a_x+r_x,a_x+2r_x,\ldots,b_x}\times\set{a_y,a_y+r_y,a_y+2r_y,\ldots,b_y}~,
\end{equation}
where $a_x=2$, $r_x=113$, $b_x=10^6$, $a_y=3$, $r_y=127$ and $b_y=10^6$. The number of possible cases is given by formula
\[
 \left(\left\lfloor\frac{b_x-a_x}{r_x}\right\rfloor+1\right)\left(\left\lfloor\frac{b_y-a_y}{r_y}\right\rfloor+1\right)~,
\]
and, therefore, for the previously given values we have $69684900$ possible cases. The additional condition for the search domain is
$x\neq y$. In this condition, the number of analyzed cases of the search domain (\ref{Dc2169}) is $69684830$. We have $81$ solutions
given in the form $\boxed{x,y}$:
\begin{itemize}
  \item[]
     \begin{flushleft}
       \boxed{4296,384} \boxed{8816,8512} \boxed{15144,384} \boxed{20568,384} \boxed{22715,36960} \boxed{36162,47628} \boxed{47688,384} \boxed{53112,384} \boxed{64525,52200} \boxed{80232,384} \boxed{82944,829440} \boxed{87238,1908} \boxed{96504,384}~;
     \end{flushleft}
  \item[]
     \begin{flushleft}
       \boxed{111194,1908} \boxed{145885,29340} \boxed{161592,384} \boxed{167016,384} \boxed{177864,384} \boxed{194136,384}~;
     \end{flushleft}
  \item[]
     \begin{flushleft}
       \boxed{230974,1908} \boxed{251766,61344} \boxed{264648,384} \boxed{291768,384}~;
     \end{flushleft}
  \item[]
      \begin{flushleft}
        \boxed{302842,1908} \boxed{307588,12576} \boxed{324312,384} \boxed{326798,1908} \boxed{356856,384} \boxed{369964,65408} \boxed{370755,623700} \boxed{372450,280800} \boxed{378552,384} \boxed{392225,52200} \boxed{394824,384} \boxed{398666,1908}~;
      \end{flushleft}
  \item[]
     \begin{flushleft}
       \boxed{405672,384} \boxed{421944,384} \boxed{427368,384} \boxed{459912,384} \boxed{476184,384} \boxed{478783,5464}~;
     \end{flushleft}
  \item[]
     \begin{flushleft}
       \boxed{514265,29340} \boxed{519576,384} \boxed{544436,12576} \boxed{573816,384} \boxed{590314,1908}~;
     \end{flushleft}
  \item[]
     \begin{flushleft}
       \boxed{617208,384} \boxed{622632,384} \boxed{627152,8512} \boxed{634158,47628} \boxed{649752,384} \boxed{653820,97920} \boxed{655176,384} \boxed{655854,67440} \boxed{660939,494160} \boxed{662182,1908} \boxed{666024,384} \boxed{682296,384} \boxed{686138,1908} \boxed{698455,29340}~;
     \end{flushleft}
  \item[]
     \begin{flushleft}
       \boxed{703992,384} \boxed{718795,36960} \boxed{720264,384} \boxed{732242,108080} \boxed{736536,384} \boxed{758006,1908} \boxed{779928,384} \boxed{787499,5464}~;
     \end{flushleft}
  \item[]
     \begin{flushleft}
       \boxed{800042,11052} \boxed{805918,1908} \boxed{823546,110112} \boxed{829874,1908} \boxed{830552,236096} \boxed{833264,8512} \boxed{834168,384} \boxed{859254,372240} \boxed{882984,384} \boxed{893832,384} \boxed{898352,122304}~;
     \end{flushleft}
  \item[]
     \begin{flushleft}
       \boxed{923890,260480} \boxed{948072,384}~.
     \end{flushleft}
\end{itemize}
The number of solutions for equation (2169) with $k=0$ on the search domain (\ref{Dc2169}) with $a_1=1$, $r_x=1$, $b_x=10^6$, $a_y=1$, $r_y=1$ and $b_y=10^6$ is huge.

The search program is:
\begin{tabbing}
  $Ed2169$\=$(a_x,r_x,b_x,a_y,r_y,b_y):=$\\ \\
  \>\vline\ $j\leftarrow 0$\\
  \>\vline\ $S\leftarrow("x"\ "y")$\\
  \>\vline\ $f$\=$or\ x\in a_x,a_x+r_x..b_x$\\
  \>\vline\ \>\ $f$\=$or\ y\in a_y,a_y+r_y..b_y$\\
  \>\vline\ \>\ \>\ $i$\=$f\ x\neq y$\\
  \>\vline\ \>\ \>\ \>\ \vline\ $j\leftarrow j+1$\\
  \>\vline\ \>\ \>\ \>\ \vline\ $S\leftarrow stack[S,(x\ y)]\ if\ \eta_x\cdot y\textbf{=}x\cdot\sigma0_y$\\
  \>\vline\ $return\ stack[S,(j\ j)]$\\
\end{tabbing}

The Diophantine equation (2169) has, for $k=1$ and $k=2$, only trivial solutions in which $y=1$.

\subsection{The equation 2170}

The $\eta$--$\sigma$--Diophantine equation (2170) can be written in the form $\eta(x)\cdot\sigma_k(y)=x\cdot y$, equation which is
equivalent with the relation $\eta_x\cdot\sigma k_y=x\cdot y$ if we take into consideration the significance of vector $\eta$ and
$\sigma k$ for $k=0,1,2$. The search program for equation $\eta(x)\cdot\sigma_0(y)=x\cdot y$ is
\begin{tabbing}
  $Ed2170(a_x,r_x,b_x,a_y,r_y,b_y):=$\=\vline\ $S\leftarrow("x"\ "y")$\\
  \>\vline\ $f$\=$or\ x\in a_x,a_x+r_x..b_x$\\
  \>\vline\ \>\ $f$\=$or\ y\in a_y,a_y+r_y..b_y$\\
  \>\vline\ \>\ \>\ $S\leftarrow stack[S,(x\ y)]\ if\ \eta_x\cdot\sigma0_y\textbf{=}x\cdot y$\\
  \>\vline\ $return\ S$\\
\end{tabbing}
The search domain is similar with the domain (\ref{Dc2169}) in which, in order to cover all possible cases, we should have
$a_x=1$, $r_x=1$, $b_x=10^6$, $a_y=1$, $r_y=1$ and $b_y=10^6$. Therefore we will have $10^{12}$ possible cases. The equation
$\eta(x)\cdot\sigma_0(y)=x\cdot y$ does not have solutions on the search domain $D_c$ given by (\ref{Dc2169}), where $a_=1$,
$r_x=1$, $b_x=10^4$, $a_y=1$, $r_y=1$ and $b_y=10^4$.

The Diophantine equation $\eta(x)\cdot\sigma(y)=x\cdot y$ on the search domain given by (\ref{Dc2169}), where $a_x=2$, $r_x=1$,
$b_x=10^4$, $a_y=3$, $r_y=1$ and $b_y=10^4$, with $99970002$ analyzed cases has $3625$ solutions. As the number of solutions is
that big, we intend that, on the same search domain, to consider the additional condition $(x,y)=1$ (i.e. $x$ and $y$ are
relatively prime). Hence, the number of analyzed cases has decreased to $60769973$ and, thereby, we have found $3$ solutions,
presented in the form $\boxed{x,y}$: $\boxed{25,24}$, $\boxed{49,4320}$ and $\boxed{49,4680}$.
\begin{prog}\label{Ed21701} The empirical search program is:
\begin{tabbing}
  $Ed2170$\=$(a_x,r_x,b_x,a_y,r_y,b_y):=$\\ \\
  \>\vline\ $j\leftarrow0$\\
  \>\vline\ $S\leftarrow("x"\ "y")$\\
  \>\vline\ $f$\=$or\ x\in a_x,a_x+r_x..b_x$\\
  \>\vline\ \>\ $f$\=$or\ y\in a_y,a_y+r_y..b_y$\\
  \>\vline\ \>\ \>\ $i$\=$f\ \gcd(x,y)\textbf{=}1$\\
  \>\vline\ \>\ \>\ \>\vline\ $j\leftarrow j+1$\\
  \>\vline\ \>\ \>\ \>\vline\ $S\leftarrow stack[S,(x\ y)]\ if\ \eta_x\cdot\sigma1_y\textbf{=}x\cdot y$\\
  \>\vline\ $return\ stack[S,(j\ j)]$\\
\end{tabbing}
\end{prog}

The Diophantine equation $\eta(x)\cdot\sigma_2(y)=x\cdot y$ on the given search domain (\ref{Dc2169}), where $a_x=2$, $r_x=1$,
$b_x=10^4$, $a_y=3$, $r_y=1$, $b_y=10^4$ and $(x,y)=1$, with $60769973$ analyzed cases, has $211$ solutions. We give in the
form $\boxed{x,y}$ the first $60$ and the last $57$ solutions:
\begin{itemize}
  \item[] \boxed{125,6}~;
  \item[]
      \begin{flushleft}
        \boxed{221,10} \boxed{247,10} \boxed{299,10} \boxed{377,10} \boxed{403,10} \boxed{481,10} \boxed{533,10} \boxed{559,10} \boxed{611,10} \boxed{689,10} \boxed{767,10} \boxed{793,10} \boxed{871,10} \boxed{923,10} \boxed{949,10} \boxed{1027,10} \boxed{1079,10} \boxed{1157,10} \boxed{1261,10} \boxed{1313,10} \boxed{1339,10} \boxed{1391,10} \boxed{1417,10} \boxed{1469,10} \boxed{1651,10} \boxed{1703,10} \boxed{1781,10} \boxed{1807,10} \boxed{1937,10} \boxed{1963,10} \boxed{1972,65} \boxed{2041,10} \boxed{2119,10} \boxed{2171,10} \boxed{2327,10} \boxed{2353,10} \boxed{2249,10} \boxed{2483,10} \boxed{2509,10} \boxed{2561,10} \boxed{2587,10} \boxed{2743,10} \boxed{2899,10}~;
      \end{flushleft}
  \item[] \boxed{175,12} \boxed{245,12}~;
  \item[] \boxed{1547,60} \boxed{1729,60} \boxed{2093,60} \boxed{2639,60} \boxed{2821,60}~;
  \item[] \boxed{1292,65} \boxed{1564,65} \boxed{2108,65} \boxed{2312,65} \boxed{2516,65} \boxed{2788,65}~;
  \item[] \boxed{2125,84} \boxed{2375,84} \boxed{2875,84}~;
  \item[] \ \vdots
  \item[]
      \begin{flushleft}
        \boxed{7423,10} \boxed{7501,10} \boxed{7631,10} \boxed{7709,10} \boxed{7787,10} \boxed{7813,10} \boxed{7891,10} \boxed{7969,10} \boxed{8021,10} \boxed{8047,10} \boxed{8203,10} \boxed{8333,10} \boxed{8359,10} \boxed{8411,10} \boxed{8489,10} \boxed{8567,10} \boxed{8593,10} \boxed{8749,10} \boxed{8801,10} \boxed{8879,10} \boxed{8983,10} \boxed{9113,10} \boxed{9217,10} \boxed{9347,10} \boxed{9451,10} \boxed{9529,10} \boxed{9607,10} \boxed{9659,10} \boxed{9763,10} \boxed{9841,10} \boxed{9893,10} \boxed{9997,10}~;
      \end{flushleft}
  \item[] \boxed{7553,60} \boxed{8099,60} \boxed{8827,60} \boxed{9191,60} \boxed{9373,60} \boxed{9737,60} \boxed{9919,60}~;
  \item[] \boxed{7684,65} \boxed{8636,65} \boxed{8908,65} \boxed{9316,65} \boxed{9452,65}~;
  \item[] \boxed{7625,84} \boxed{8375,84} \boxed{8875,84} \boxed{9125,84} \boxed{9875,84}~;
  \item[] \boxed{8029,150} \boxed{8897,150} \boxed{9331,150}~;
  \item[] \boxed{7626,175}  \boxed{7998,175} \boxed{8742,175} \boxed{9858,175}~;
  \item[] \boxed{8211,260}~.
\end{itemize}
The empirical search program is similar with the program
\ref{Ed21701}.

\subsection{The equation 2171}

The Diophantine equation $\eta(x)^y=x^{\sigma_k(y)}$ is equivalent with the relation $(\eta_x)^y=x^{\sigma k_y}$ for $k=0,1,2$ if we
take into consideration the meaning of vectors $\eta$ and $\sigma k$, with $k=0,1,2$. For $k=0$ let us consider the search domain
(\ref{Dc2169}) where $a_x=2$, $r_x=1$, $b_x=10^4$, $a_y=3$, $r_y=1$ and $b_y=10^4$. Then, the number of considered cases is
$99970002$. Due to conditions $(\eta_x)^y>10^{307}$ and $x^{\sigma k_y}>10^{307}$ (upper floating overflow) the number of analyzed
cases is of only $1460765$. The number of found solutions is $45$ and are presented in the form $\boxed{x,y}$:
\begin{itemize}
  \item[]
     \begin{flushleft}
       \boxed{8,3} \boxed{8,6} \boxed{27,3} \boxed{27,6} \boxed{36,8} \boxed{36,12} \boxed{64,8} \boxed{64,12} \boxed{81,8} \boxed{81,12} \boxed{100,8} \boxed{100,12} \boxed{196,8} \boxed{196,12} \boxed{484,8} \boxed{484,12} \boxed{676,8} \boxed{676,12} \boxed{1156,8} \boxed{1156,12} \boxed{1444,8} \boxed{1444,12} \boxed{2116,8} \boxed{2116,12} \boxed{3125,5} \boxed{3125,10} \boxed{3364,8} \boxed{3364,12} \boxed{3375,9} \boxed{3375,24} \boxed{3844,8} \boxed{3844,12} \boxed{4096,9} \boxed{4096,18} \boxed{4096,24} \boxed{5476,8} \boxed{5476,12} \boxed{6724,8} \boxed{6724,12} \boxed{7396,8} \boxed{7396,12} \boxed{8836,8} \boxed{8836,12} \boxed{9261,9} \boxed{9261,18}~.
     \end{flushleft}
\end{itemize}

For the Diophantine equation $\eta(x)^y=x^{\sigma(y)}$ on the search domain (\ref{Dc2169}), with $a_x=2$, $r_x=1$, $b_x=10^5$,
$a_y=3$, $r_y=1$ and $b_y=10^5$, there are $9999700002$ possible cases, of which only $2839629$ cases were analyzed, by reasons of
upper floating overflow generated by the raising to power $\eta(x)^y$ or $x^{\sigma(y)}$. The equation has no solutions.

The Diophantine equation $\eta(x)^y=x^{\sigma_2(y)}$, on the search domain (\ref{Dc2169}) with $a_x=2$, $r_x=1$, $b_x=10^5$,
$a_y=3$, $r_y=1$ and $b_y=10^5$ has $9999700002$ possible cases, of which only $507015$ cases were analyzed, by reasons of upper
floating overflow generated by the raising to power $\eta(x)^y$ or $x^{\sigma_2(y)}$, has no solutions.

\subsection{The equation 2172}

The equation $\eta(x)^y=\sigma_k(y)^x$ is equivalent with the relation $\big(\eta_x\big)^y=\big(\sigma k_y\big)^x$ for $k=0,1,2$
if we take into consideration the meaning of vectors $\eta$ and $\sigma k$, with $k=0,1,2$. For the case $k=0$ fie search domain
is (\ref{Dc2169}), where $a_x=2$, $r_x=1$, $b_x=10^6$, $a_y=3$, $r_y=1$ and $b_y=10^6$. Then, the number of considered cases is
$999997000002$. Due to conditions $(\eta_x)^y>10^{307}$ and $x^{\sigma k_y}>10^{307}$ (upper floating overflow) the number of
analyzed cases is $72776$. The solutions found for $k=0$ are: \boxed{8,8}, \boxed{18,18}, \boxed{45,45} and $\boxed{128,128}$.

On the same search domain, for $k=1$ the number of analyzed cases is $38794$, for $k=2$ the number of analyzed cases is only of
$24736$ and the Diophantine equations do not have solutions.

\section{The $\eta$--$\varphi$--Diophantine equations}

Let us consider $m,n\in\Ns$ fixed and $x$ and $y$ unknown positive integers. The Diophantine equations in which functions $\eta$ and
$\varphi$ are involved are said to be $\eta$--$\varphi$--Diophantine equations. The list of $\eta$--$\varphi$--Diophantine equation, considered from \citep{Smarandache1999a}, and which we intend to solve empirically, is:
\begin{enumerate}
  \item[(2187)] $\eta(x)=\varphi(m\cdot x+n)$~,
  \item[(2188)] $\eta(x)^m=\varphi(x^n)$~,
  \item[(2189)] $\eta(x)+y=x+\varphi(y)$~,
  \item[(2190)] $\eta(x)\cdot y=x\cdot\varphi(y)$~,
  \item[(2191)] $\eta(x)/y=x/\varphi(y)$~,
  \item[(2192)] $\eta(x)^y=x^{\varphi(y)}$~,
  \item[(2193)] $\eta(x)^y=\varphi(y)^x$~.
\end{enumerate}

\subsection[Empirical solving of $\eta$--$\varphi$--Diophantine equations]
{Partial empirical solving of $\eta$--$\varphi$--Diophantine equations}

For all Diophantine equations solved in this section the file $\eta.prn$ will be read, generated by the program \ref{ProgEta},
by means of the Mathcad function $READPRN$
\[
 \eta:=READPRN("...\backslash \eta.prn")\ \ last(\eta)=10^6
\]
where command $last(\eta)$ indicates the last index of vector $\eta$. The reading time is of about $10$ seconds, therefore an
important saving for the execution time of the search is obtained.

The file $\varphi.prn$ will be read and the values will be attributed to vector $\varphi$, by means of the Mathcad function
$READPRN$, with the sequence:
\[
 \varphi:=READPRN("...\backslash \varphi.prn")\ \ last(\varphi)=10^6~,
\]
where each component $\varphi$ contains the number of factors relatively prime to $k$, for $k=1,2,\ldots,10^6$. The values were
generated with the program \ref{Progphif}. The generating time for the file $\varphi.prn$ was \emph{5:30:33.558hhmmss}~. If we have a
Diophantine equation in which function $\varphi(x)$ is involved, we will use the vector $\varphi$. The time saving that results is
important.

\subsection{The equation 2187}

The Diophantine equation $\eta(x)=\varphi(m\cdot x+n)$ is equivalent with the relation $\eta_x=\varphi_{m\cdot x+n}$ if we
take into account the significance of vectors $\eta$ and $\varphi$. Let the search domain be
\begin{equation}\label{Dc2187}
  D_c=\set{1,2,\ldots,10}^2\times\set{1,2,\ldots,10^6}
\end{equation}
and additional condition $(m,n)=1$. The number of possible cases is $10^8$ and we have only $21271688$ analyzed cases. Under these
conditions equation (2187) has $13$ solutions given in the form $\boxed{m,n,x}$:
\begin{itemize}
  \item[]
      \begin{flushleft}
         \boxed{1,1,1} \boxed{1,1,2} \boxed{1,1,4} \boxed{1,2,2} \boxed{1,2,8} \boxed{1,2,16} \boxed{1,4,2} \boxed{1,4,4} \boxed{1,4,8} \boxed{1,5,9} \boxed{1,6,4} \boxed{1,8,4} \boxed{1,9,9}~.
      \end{flushleft}
\end{itemize}

We can also consider the Diophantine equation $\eta(m\cdot x+n)=\varphi(x)$ which is equivalent with the relation
$\eta_{m\cdot x+n}=\varphi_x$ if we take into account the significance of vectors $\eta$ and $\varphi$. On the search domain
(\ref{Dc2187}) we have 26 solutions:
\begin{itemize}
  \item[]
      \begin{flushleft}
        \boxed{1,2,7} \boxed{1,2,10} \boxed{1,2,14} \boxed{1,2,30} \boxed{1,3,5} \boxed{1,3,22} \boxed{1,4,8} \boxed{1,4,14} \boxed{1,7,5} \boxed{1,7,9} \boxed{1,8,24} \boxed{1,9,7} \boxed{1,9,9}~;
      \end{flushleft}
  \item[] \boxed{2,3,11} \boxed{2,9,18} \boxed{3,4,20} \boxed{4,9,9}~;
  \item[] \boxed{5,1,7} \boxed{5,2,14} \boxed{5,3,9} \boxed{5,6,23} \boxed{5,8,24}~;
  \item[] \boxed{7,9,9} \boxed{9,1,11} \boxed{9,2,22} \boxed{9,8,24}~.
\end{itemize}

\subsection{The equation 2188}

The $\eta$--$\varphi$--Diophantine equation $\eta(x)^m=\varphi(x^n)$ is equivalent with the relation
$(\eta_x)^m=\varphi_{x^n}$ if we take into account the significance of vectors $\eta$ and $\varphi$. Let the search
domain be (\ref{Dc2187}) with $a_m=2$, $a_n=2$, $a_x=2$ and additional conditions $x^n\le last(\varphi)$, $(\eta_x)^n<10^{17}$
and $(m,n)=1$. Then we have $80999919$ possible cases and $4362$ analyzed cases. Under these conditions, the Diophantine equation
(2188) has $12$ solutions:
\begin{itemize}
  \item[]
     \begin{flushleft}
       \boxed{2,3,2} \boxed{3,2,32} \boxed{3,2,50} \boxed{3,4,2} \boxed{4,3,8} \boxed{4,5,2} \boxed{5,6,2} \boxed{6,7,2} \boxed{7,5,8} \boxed{7,8,2} \boxed{8,9,2} \boxed{9,10,2}~.
     \end{flushleft}
\end{itemize}

\subsection{The equation 2189}

The $\eta$--$\varphi$--Diophantine equation $\eta(x)+y=x+\varphi(y)$ is equivalent with the relation
$\eta_x+y=x+\varphi_y$ if we take into account the significance of vectors $\eta$ and $\varphi$. Let the search domain be
\begin{multline}\label{Dc2189}
  D_c=\set{a_x,a_x+r_x,a_x+2r_x,\ldots,a_x+\left\lfloor\frac{b_x-a_x}{r_x}\right\rfloor r_x}\\
      \times\set{a_y,a_y+r_y,a_y+2r_y,\ldots,a_y+\left\lfloor\frac{b_y-a_y}{r_y}\right\rfloor r_y}~,
\end{multline}
where $a_x=2$, $r_x=113$, $b_x=10^6$, $a_y=3$, $r_y=127$ and $b_y=10^6$. This search domain has $69684900$ possible cases which
are also analyzed cases. In these conditions, equation (2189) has $60$ solutions given in the form $\boxed{x,y}$:
\begin{itemize}
  \item[]
      \begin{flushleft}
        \boxed{2036,518671} \boxed{3618,828551} \boxed{7234,907037} \boxed{8816,18291} \boxed{20568,610111} \boxed{21585,35182} \boxed{25540,993143} \boxed{25653,37722} \boxed{32094,575567} \boxed{33676,54105} \boxed{34128,743969} \boxed{39552,276101} \boxed{47462,304549} \boxed{47914,586997} \boxed{66672,695201} \boxed{67915,118748} \boxed{71418,416563} \boxed{72435,131956} \boxed{81136,403355} \boxed{81362,200409} \boxed{85656,217173} \boxed{91758,430025} \boxed{92888,276609}~;
      \end{flushleft}
  \item[]
      \begin{flushleft}
        \boxed{123624,864619} \boxed{129274,643385} \boxed{130404,597665} \boxed{131421,241684} \boxed{135715,199012} \boxed{139218,697487} \boxed{143173,271148} \boxed{144755,213744} \boxed{146337,193678} \boxed{167355,299596} \boxed{184079,212982} \boxed{189051,375542} \boxed{194814,828805} \boxed{196735,393322}~;
      \end{flushleft}
  \item[]
      \begin{flushleft}
        \boxed{208035,415928} \boxed{208939,273180} \boxed{212216,593601} \boxed{229731,381130} \boxed{244647,366906} \boxed{246455,347094} \boxed{248376,910085} \boxed{259789,426850} \boxed{277078,822201} \boxed{277530,443360} \boxed{283745,564518} \boxed{293011,405006}~;
      \end{flushleft}
  \item[] \boxed{352675,503304} \boxed{364427,533784} \boxed{370755,627256} \boxed{387592,974855}~;
  \item[] \boxed{401491,710314} \boxed{430080,851792} \boxed{437425,839854} \boxed{440137,702440} \boxed{453019,742318}~;
  \item[] \boxed{745915,994540} \boxed{754503,838330}~.
\end{itemize}

\subsection{The equation 2190}

The $\eta$--$\varphi$--Diophantine equation $\eta(x)\cdot y=x\cdot\varphi(y)$ can be assimilated to relation $\eta_x\cdot
y=x\cdot\varphi_y$, taking into account the significance of vectors $\eta$ and $\varphi$. On the search domain
$\set{1,2,\ldots,10^6}^2$ the equation has many solutions. Thus, we will consider a restricted search domain. Let us consider the
search domain (\ref{Dc2189}) with $a_x=2$, $r_x=1$, $b_x=10^4$, $a_y=3$, $r_y=1$, $b_y=10^4$ and additional condition
$\gcd(x,y)=1$ (i.e. $x$ and $y$ are relatively prime), then we have $99970002$ possible cases, $60769973$ analyzed cases and $0$
solutions.

\subsection{The equation 2191}

The $\eta$--$\varphi$--Diophantine equation $\eta(x)/y=x/\varphi(y)$ is equivalent with relation $\eta_x\cdot\varphi_y=x\cdot y$ if we take into account the significance of vectors $\eta$ and $\varphi$. The equation does not have solutions on the search domain $D_c$ given by (\ref{Dc2189}) with $a_x=2$, $r_x=1$, $b_x=10^4$, $a_y=3$, $r_y=1$ $b_y=10^4$ and additional condition $(x,y)=1$. The number of possible cases is $99970002$ and the number of analyzed cases is $607699734$, due to the additional condition.

\subsection{The equation 2192}

The $\eta$--$\varphi$--Diophantine equation $\eta(x)^y=x^{\varphi(y)}$ is equivalent with relation $(\eta_x)^y=x^{\varphi_y}$  if we take into account the significance of vectors $\eta$ and $\varphi$. Let us consider the search domain $D_c$ given by (\ref{Dc2189}) with $a_x=2$, $r_x=1$, $b_x=10^4$, $a_y=3$, $r_y=1$ $b_y=10^4$ and additional conditions $(x,y)=1$, $(\eta_x)^y<10^{17}$ and $x^{\varphi_y}<10^{17}$. The number of possible cases is $99970002$ and the number of analyzed cases is $534464$, due to the additional conditions. The equation has $11$ solutions given as $\boxed{x,y}$:
\begin{itemize}
  \item[] \boxed{8,3} \boxed{8,9} \boxed{8,27} \boxed{8,81} \boxed{8,243}~,
  \item[] \boxed{81,4} \boxed{81,8} \boxed{81,16} \boxed{81,32} \boxed{81,64} \boxed{81,128}~.
\end{itemize}

\subsection{The equation 2193}

The $\eta$--$\varphi$--Diophantine equation $\eta(x)^y=\varphi(y)^x$ is equivalent with relation $(\eta_x)^y=(\varphi_y)^x$ if we take into account the significance of vectors $\eta$ and $\varphi$. The equation does not have solutions on the search domain $D_c$ given by
(\ref{Dc2189}) with $a_x=2$, $r_x=1$, $b_x=10^4$, $a_y=3$, $r_y=1$ $b_y=10^4$ and additional conditions $(x,y)=1$, $(\eta_x)^y<10^{17}$ and $(\varphi_y)^x<10^{17}$. The number of possible cases is $99970002$ and the number of analyzed cases is of only $26936$ cases, due to the additional conditions.

\section{Guy type Diophantine equations}

\cite{Guy2004} considered the $\varphi$--$\sigma$--Diophantine equation $\varphi(\sigma(n))=n$.  F. Helenius\index{Helenius F.}
determined 365 solutions. Similarly, the next Diophantine equations in which function $\eta$ is involved were considered:
\begin{align}
  \eta(\varphi(x))=&x \label{Guy1}\\
  \varphi(\eta(x))=&x \label{Guy2}\\
  \eta(\varphi(x))=&\varphi(\eta(x)) \label{Guy3}\\
  \eta(\sigma_0(x))=&x \label{Guy4}\\
  \sigma_0(\eta(x))=&x \label{Guy5}\\
  \eta(\sigma_0(x))=&\sigma_0(\eta(x)) \label{Guy6}\\
  \eta(\sigma(x))=&x \label{Guy7}\\
  \sigma(\eta(x))=&x \label{Guy8}\\
  \eta(\sigma(x))=&\sigma(\eta(x)) \label{Guy9}\\
  \eta(s(x))=&x \label{Guy10}\\
  s(\eta(x))=&x \label{Guy11}\\
  \eta(s(x))=&s(\eta(x)) \label{Guy12}\\
  \eta(\pi(x)=&x \label{Guy13}\\
  \pi(\eta(x)=&x \label{Guy14}\\
  \eta(\pi(x))=&\pi(\eta(x)) \label{Guy15}
\end{align}

\subsection[Empirical solving Guy type Diophantine equations]{Partial empirical solving Guy type Diophantine equations}

For solving these equations we will use the files $\eta.prn$, $\varphi.prn$, $\sigma0.prn$, $\sigma1.prn$ and $s.prn$ which were
generated by programs \ref{ProgGenVfS}, \ref{ProgGenPhi} and \ref{ProgGenSigmak}. To solve a Diophantine equation we will read
the files as in the next sequence:
\[
 \eta:=READPRN("...\backslash\eta.prn")\ \ last(\eta)=10^6~.
\]
Hence, we will have the vectors $\eta$, $\varphi$, $\sigma0$, $\sigma1$ and $s$ with the values of the functions $\eta$,
$\varphi$, $\sigma_0$, $\sigma$ and $s$.

Let us consider the search domain
\begin{equation}\label{DcGuy}
  D_c=\set{1,2,\ldots,10^6}.
\end{equation}
Equations (\ref{Guy1}), (\ref{Guy2}), (\ref{Guy7}), (\ref{Guy8}) , (\ref{Guy10}) and (\ref{Guy11}) have a sole solution, the trivial
solution $x=1$, equations (\ref{Guy4}) and (\ref{Guy5}) have two trivial solutions $x=1$ and $x=2$. Equations (\ref{Guy3}), (\ref{Guy6}), (\ref{Guy9}) and (\ref{Guy12}) have more solutions.

\subsection{The equation \ref{Guy3}}

The $\eta$--$\varphi$--Diophantine equation $\eta(\varphi(x))=\varphi(\eta(x))$ is equivalent with relation $\eta_{\varphi_x}=\varphi_{\eta_x}$ if we take into account the significance of vectors $\eta$ and $\varphi$. In the search domain (\ref{DcGuy}) equation (\ref{Guy3}) has $842$ solutions. We give the first 60 solutions and the last 60 solutions.\\
1, 2, 3, 4, 5, 6, 10, 15, 20, 27, 30, 54, 63, 105, 108, 112, 126, 135, 140, 168, 210, 216, 252, 270, 275, 315, 432, 504, 540, 550, 630, 825, 1100 1650 1925, 2200, 2475, 2783, 2816, 3125, 3159, 3300, 3328, 3520, 3850, 4160, 4224, 4400, 4950, 4992, 5280, 5566, 5775, 6240, 6250, 6318, 6600, 6656, 7371, 7425, \ldots\\
864864, 866320, 868296, 868725, 870205, 875160, 876645, 881280, 884000, 886464, 890560, 891072, 893142, 895068, 897600, 898909, 900000, 900315, 901689, 901692, 904475, 904932, 905177, 914166, 918750, 919931, 926640, 928200, 933504, 934375 935088, 940032, 941868, 942480, 942761, 943488, 944794, 946220, 950000, 951786, 952000, 954569, 954720, 956250, 959616, 960336, 969570, 969657, 972400, 974050, 975000, 976661, 976833, 979200, 980343, 982800, 990080, 992380, 993531, 994520~.

\subsection{The equation \ref{Guy6}}

The Diophantine equation $\eta(\sigma_0(x))=\sigma_0(\eta(x))$ has 82655 solutions. In order to solve empirically the equation, the equivalent relation $\eta_{\sigma0_x}=\sigma0_{\eta_x}$ is used , taking into account the significance of vectors $\eta$ and $\sigma0$. 120 solutins are listed, the first 60 solutions and the last 60 solutions.\\
1, 2, 3, 4, 5, 7, 11, 12, 13, 17, 19, 23, 29, 31, 37, 41, 43, 47, 53, 59, 61, 67, 71, 72, 73, 79, 83, 89, 90, 96, 97, 101, 103, 107,
109, 113, 125, 127, 128, 131, 137, 139, 149, 150, 151, 157, 160, 163, 167, 173, 179, 181, 191, 193, 197, 199, 200, 211, 223, 224, \ldots\\
999101, 999133, 999149, 999169, 999181, 999199, 999217, 999221, 999233, 999239, 999269, 999287, 999307, 999329, 999331, 999359, 999370, 999371, 999377, 999389, 999431, 999433, 999437, 999451, 999491, 999499, 999521, 999529, 999541, 999553 999563, 999599, 999611, 999613, 999623, 999631, 999653, 999667, 999671, 999683, 999721, 999727, 999749, 999763, 999769, 999773, 999809, 999853, 999863, 999883, 999907, 999917, 999931, 999949, 999953, 999959, 999961, 999979, 999983, 999998~.

\subsection{The equation \ref{Guy9}}

The Diophantine equation $\eta(\sigma(x))=\sigma(\eta(x))$ has 648 solutions. To solve empirically the equation, the equivalent relation is used $\eta_{\sigma1_x}=\sigma1_{\eta_x}$, if we take into account the significance of vectors $\eta$ and $\sigma1$. The first 60 solutions were listed:\\
1, 2, 3, 4, 6, 10, 12, 21, 30, 40, 42, 52, 84, 105, 120, 156, 168, 210, 260, 364, 416, 420, 468, 572, 780, 840, 976, 1092, 1248, 1404, 1525, 1716, 1813, 1820 2080, 2340, 2860, 2912, 2928, 3050, 3125, 3159, 3276, 3626, 3744, 4004, 4477, 4575, 4576, 4880, 5148, 5439, 5460, 6100, 6240, 6250, 6318, 6832, 7020, 7252, \ldots\\
as well as the last 60 solutions:\\
420616, 425315, 425475, 426512, 436150, 437675, 440176, 440559, 446875, 447811, 452925, 455975, 458689, 459025, 459375, 462315, 470085, 473193, 478125, 486475, 492575, 498575, 501725, 503125, 505827, 507825, 514855, 520025, 523075, 523809 531471, 532763, 542087, 559625, 565775, 571875, 574925, 578347, 581371, 584375, 585599, 589225, 595441, 596275, 614575, 618233, 620675, 629825, 635221, 649165, 653125, 666425, 683501, 687775, 690625, 693935, 708883, 718153, 720797, 730639~.

\subsection{The equation \ref{Guy12}}

\begin{table}
  \centering
  \begin{tabular}{|r|r|r|r|r|}
    \hline
    $x$ & $s(x)$ & $\eta(x)$ & $s(\eta(x))$ & $\eta(s(x))$ \\
    \hline
       4 &      3 &  4 &  3 &  3 \\
      64 &     63 &  8 &  7 &  7 \\
      90 &    144 &  6 &  6 &  6 \\
     224 &    280 &  8 &  7 &  7 \\
     441 &    300 & 14 & 10 & 10 \\
    5145 &   4455 & 21 & 11 & 11 \\
   71148 & 141120 & 22 & 14 & 14 \\
  166012 & 206388 & 22 & 14 & 14 \\
    \hline
  \end{tabular}
  \caption{The check of the solutions of equation \ref{Guy12}}\label{TabSolEcGuy12}
\end{table}

The Diophantine equation $\eta(s(x))=s(\eta(x))$ is equivalent with relation
\begin{equation}
  \eta_{s_x}=s_{\eta_x}~,
\end{equation}
if we take into account the significance of vectors $\eta$ and $s$. We consider the search domain $D_c$ given by (\ref{DcGuy}).
The equation has following solutions: the 78498 prime numbers $<10^6$ and 8 solutions non-prime numbers: 4, 64, 90, 224, 441,
5145, 71148, 166012~. The check of the non-prime solutions is presented in table \ref{TabSolEcGuy12}.

\subsection{The equations \ref{Guy13}--\ref{Guy14}}

The $\eta$--$\pi$--Diophantine equations $\eta(\pi(x))=x$ and $\pi(\eta(x))=x$ are equivalent with the relations
$\eta_{\pi(x)}=x$ and $\pi(\eta_x)=x$ if we take into account the significance of vector $\eta$ and formula (\ref{FormulaPiEta}).
These relations do not have solutions on the search domain $D_c=\set{4,5,\ldots,10^3}$.

\subsection{The equation \ref{Guy15}}

\begin{table}
  \centering
  \begin{tabular}{|r|r|r|r|r|}
    \hline
    $x$ & $\pi(x)$ & $\eta(x)$ & $\eta(\pi(x))$ & $\pi(\eta(x))$ \\ \hline
     4 &   2 &  4 &  2 &  2\\ \hline
    15 &   6 &  5 &  3 &  3\\ \hline
    16 &   6 &  6 &  3 &  3\\ \hline
    21 &   8 &  7 &  4 &  4\\ \hline
    26 &   9 & 13 &  6 &  6\\ \hline
    65 &  18 & 13 &  6 &  6\\ \hline
    96 &  24 &  8 &  4 &  4\\ \hline
   133 &  32 & 19 &  8 &  8\\ \hline
   156 &  36 & 13 &  6 &  6\\ \hline
   176 &  40 & 11 &  5 &  5\\ \hline
   187 &  42 & 17 &  7 &  7\\ \hline
   232 &  50 & 29 & 10 & 10\\ \hline
   236 &  51 & 59 & 17 & 17\\ \hline
   253 &  54 & 23 &  9 &  9\\ \hline
   364 &  72 & 13 &  6 &  6\\ \hline
   416 &  80 & 13 &  6 &  6\\ \hline
   527 &  99 & 31 & 11 & 11\\ \hline
   598 & 108 & 23 &  9 &  9\\ \hline
   660 & 120 & 11 &  5 &  5\\ \hline
   726 & 128 & 22 &  8 &  8\\ \hline
   738 & 130 & 41 & 13 & 13\\ \hline
   744 & 132 & 31 & 11 & 11\\ \hline
   870 & 150 & 29 & 10 & 10\\ \hline
   885 & 153 & 59 & 17 & 17\\ \hline
   899 & 154 & 31 & 11 & 11\\ \hline
   966 & 162 & 23 &  9 &  9\\ \hline
  \end{tabular}
  \caption{The check of the solutions of equation \ref{Guy15}}\label{TabSolEcGuy15}
\end{table}
The $\eta$--$\pi$--Diophantine equation $\eta(\pi(x))=\pi(\eta(x))$ is equivalent with the relation
\begin{equation}\label{RelatiaGuy15}
  \eta_{\pi(x)}=\pi(\eta_x)~,
\end{equation}
if we take into account the significance of vector $\eta$ and formula (\ref{FormulaPiEta}). Let $D_c=\set{4,5,\ldots,10^3}$ be
the search domain, then relation (\ref{RelatiaGuy15}) has $26$ solutions: 4, 15, 16, 21, 26, 65, 96, 133, 156, 176, 187, 232,
236, 253, 364, 416, 527, 598, 660, 726, 738, 744, 870, 885, 899, 966~.

In table \ref{TabSolEcGuy15} the check of these solutions is presented.

\backmatter

\chapter{Conclusions}

We, the authors, hope that this book offers a valuable insight into a fascinating range of Number Theory problems. The approaches and algorithms proposed are not intended to just solve proposed problems, but also to inspire the reader to find better, more efficient or more beautiful solutions that further enrich our understanding of this field of mathematics.

The "partial results" of over $62~\eta$--Diophantine equations presented in the last chapter could prove to be an excellent starting point for anyone motivated to explore more solutions. By simply running the proposed programs and algorithms on machines with better hardware capabilities, one can extend the set numbers that verify this equations. The mathematicians interested in analytical solutions are encouraged to explore Chapter 6 where, various analytical approaches into solving Diophantine equations are presented. At the same time, students or researchers unfamiliar with the details of such mathematical problems will discover, in the first five chapters some of the most important concepts, algorithms and tools to help them in their quest to learn about Diophantine equations.

The content of this book is a result of our collective mathematical and computing expertise. While the choice of the problems is a subjective one, our intent was to cover the most interesting Diophantine equations involving Smarandache's function $\eta$. We encourage anyone to approach us with comments and observation regarding the content of this book and also with other fresh problems related to it.

\chapter{Indexes}

\section*{\LARGE{Index of notations}}
\small\begin{description}
  \item $\Na=\set{0,1,2,\ldots}$;
  \item $\Ns=\Na\setminus\set{0}=\set{1,2,\ldots}$;
  \item $\Int=\set{\ldots,-2,-1,0,1,2,\ldots}$;
  \item $\I{s}=\set{1,2,\ldots,s}$ : the set of indexes;
  \item $\mathbb{Q}=\set{\frac{p}{q}\ \mid\ p,q\in\Int,\ q\neq0}$ ;
  \item $\Real$ : the real numbers;
  \item $\pi(x)$ : the number of prime numbers up to $x$;
  \item $\pi_m(x)$ : the function which is the lower bound of function $\pi(x)$;
  \item $\pi_M(x)$ : the function which is the upper bound of function $\pi(x)$;
  \item $[x]$ : the integer part of number $x$;
  \item $\{x\}$ : the fractional part of $x$;
  \item $\sigma_k(n)$ : the sum of the powers of order $k$ of the divisors of $n$;
  \item $\sigma(n)$ : the sum of the divisors of $n$; $\sigma(n)=\sigma_1(n)$;
  \item $s(n)$ : the sum of the divisors of $n$ without $n$; $s(n)=\sigma(n)-n$;
  \item $\lfloor a\rfloor$ : the lower integer part of $a$; the greatest integer, smaller than $a$;
  \item $\lceil a\rceil$ : the upper integer part of $a$; the smallest integer, greater than $a$;
  \item $n\mid m$ : $n$ divides $m$;
  \item $n\nmid m$ : $n$ does not divide $m$;
  \item $(m,n)$ : the greatest common divisor of $m$ and $n$; $(m,n)=gcd(m,n)$;
  \item $[m,n]$ : the smallest common multiple of $m$ and $n$; $[m,n]=lcd(m,n)$;
  \item $\set{n_1,n_2,\ldots,n_m}^k$ : the cartesian product $k\ times$
  \begin{multline*}
    \set{n_1,n_2,\ldots,n_m}^k\\
    =\underbrace{\set{n_1,n_2,\ldots,n_m}\times\set{n_1,n_2,\ldots,n_m}\times\ldots\times\set{n_1,n_2,\ldots,n_m}}_{k\ ori} \\
    =\left\{\left(\underbrace{n_1,n_1,\ldots,n_1}_{k\ elemente}\right), \left(\underbrace{n_1,n_1,\ldots,n_2}_{k\ elemente}\right), \ldots,
     \left(\underbrace{n_1,n_1,\ldots,n_m}_{k\ elemente}\right),\right.\\
     \left.\left(\underbrace{n_1,n_1,\ldots,n_2,n_1}_{k\ elemente}\right), \left(\underbrace{n_1,n_1,\ldots,n_2,n_2}_{k\ elemente}\right), \ldots\right.\\
     \left.\ldots, \left(\underbrace{n_m,n_m,\ldots,n_m}_{k\ elemente}\right)\right\}~;
  \end{multline*}
  \item $\abs{f(x)}$ : the module or the absolute value of $f(x)$;
  \item $f(x)\equiv g(x)$ : $f$ is asymptotic to $g$ if
  \[
   \frac{f(x)}{g(x)}\to1\ \ \textnormal{when}\ \ x\to\infty~;
  \]
  \item $f(x)=o\big(g(x)\big)$ :
  \[
   \frac{f(x)}{g(x)}\to0\ \ \textnormal{when}\ \ x\to\infty~;
  \]
  \item $f(x)=O\big(g(x)\big)$ : there exists a constant $c$ such that $\abs{f(x)}<c\cdot g(x)$, for all $x$; this property is also denoted $f(x)\ll g(x)$;
\end{description}

\newpage
\section*{\LARGE{Mathcad functions}}

{\small\begin{description}
  \item[]
  \item[]
  \item $augment(M,N)$ : concatenates matrices $M$ and $N$ that have the same number of lines;
  \item $ceil(x)$ : the upper integer part function;
  \item $cols(M)$ : the number of columns of matrix $M$;
  \item $eigenvals(M)$ : the eigenvalues of matrix $M$;
  \item $eigenvec(M,\lambda)$ : the eigenvector of matrix $M$ relative to the eigenvalue $\lambda$;
  \item $eigenvecs(M)$ : the matrix of the eigenvectors of matrix $M$;
  \item $n\ factor\rightarrow$ : symbolic computation function that factorizes $n$;
  \item $floor(x)$ : the lower integer part function;
  \item $gcd(n_1,n_2,\ldots)$ : the function which computes the greatest common divisor of $n_1,n_2,\ldots$;
  \item $last(v)$ : the last index of vector $v$;
  \item $lcm(n_1,n_2,\ldots)$ : the function which computes the smallest common multiple of $n_1,n_2,\ldots$;
  \item $max(v)$ : the maximum of vector $v$;
  \item $min(v)$ : the minimum of vector $v$;
  \item $mod(m,n)$ : the rest of the division of $m$ by $n$;
  \item $ORIGIN$ : the variable dedicated to the origin of indexes, $0$ being an implicit value;
  \item $rref(M)$ : determines the matrix \emph{row-reduced echelon form};
  \item $rows(M)$ : the number of lines of matrix $M$;
  \item $solve$ : the function of symbolic solving the equations;
  \item $stack(M,N)$ : concatenates matrices $M$ and $N$ that have the same care number of columns;
  \item $submatrix(M,k_r,j_r,k_c,j_c)$ : extracts from matrix $M$, from line $k_r$ to line $j_r$ and from column $k_c$ to column $j_c$, a submatrix;
  \item $trunc(x)$ : the truncation function;
  \item $\sum v$ : the function that sums the components of vector $v$~.
\end{description}}

\renewcommand\indexname{\LARGE{Index of name}}
\printindex

\bibliographystyle{plainnat}
\bibliography{SDE}
\addcontentsline{toc}{chapter}{Bibliography}
\end{document}